%% file: preprint.tex
\documentclass[11pt,pagebackref=true]{article}
\usepackage{fullpage}

\usepackage{hyperref}

\usepackage[round]{natbib}
\input{macros}

\hypersetup{
	colorlinks   = true, 
	urlcolor     = {niceblue}, 
	linkcolor    = {niceblue}, 
	citecolor   = {niceblue}
}

\usepackage{aligned-overset}
\usepackage{url}

\usepackage{float}

\usepackage{bbding}
\usepackage{xcolor}
\definecolor{table-green}{rgb}{0., 0.43, 0.2}

\newcommand{\lya}{H}

\DeclareSymbolFont{extraup}{U}{zavm}{m}{n}
\DeclareMathSymbol{\varheart}{\mathalpha}{extraup}{86}
\DeclareMathSymbol{\vardiamond}{\mathalpha}{extraup}{87}

\newsavebox{\test} 

\title{\textbf{Tight Lower Bounds and Optimal Algorithms for Stochastic Nonconvex Optimization\\ with Heavy-Tailed Noise}}

\author{ Adrien Fradin$^{1,2}$\thanks{The work of Adrien Fradin was performed during a summer research internship in the Optimization and Machine Learning Lab at KAUST led by Peter Richt\'arik.
	} \qquad Abdurakhmon Sadiev$^{1}$\thanks{ Correspondence to: \texttt{abdurakhmon.sadiev@kaust.edu.sa}.
	} \qquad Laurent Condat$^{1}$ \\ Peter Richtárik$^{1}$\\
	\phantom{x}
	\\$^{1}$King Abdullah University of Science and Technology (KAUST)\\ Thuwal, Saudi Arabia\\$^{2}$\'Ecole Polytechnique, Paris, France\\
}
\date{October 9, 2025}

\begin{document}
	\maketitle

	\begin{abstract}
		We study stochastic nonconvex optimization under heavy-tailed noise. In this setting, the stochastic gradients only have bounded $p$-th central moment ($p$-BCM) for some $p \in \intof{1}{2}$. Building on the foundational work of \cite{arjevani2022lower} in stochastic optimization, we establish tight sample complexity lower bounds for all first-order methods under \emph{relaxed} mean-squared smoothness ($q$-WAS) and $\delta$-similarity ($(q, \delta)$-S) assumptions, allowing any exponent $q \in [1,2]$ instead of the standard $q = 2$. These results substantially broaden the scope of existing lower bounds. To complement them, we show that Normalized Stochastic Gradient Descent with Momentum Variance Reduction (\algname{NSGD-MVR}), a known algorithm, matches these bounds in expectation. Beyond expectation guarantees, we introduce a new algorithm, Double-Clipped \algname{NSGD-MVR}, which allows the derivation of high-probability convergence rates under weaker assumptions than in previous works. Finally, for second-order methods with stochastic Hessians satisfying bounded $q$-th central moment assumptions for some exponent $q \in [1, 2]$ (allowing $q \neq p$), we establish sharper lower bounds than previous works while improving over \cite{sadiev2025second} (where only $p = q$ is considered) and yielding stronger convergence exponents. Together, these results provide a nearly complete complexity characterization of stochastic nonconvex optimization in heavy-tailed regimes.
	\end{abstract}
	
	\newpage
	{\small
		\renewcommand\baselinestretch{0}
		\tableofcontents
		\renewcommand\baselinestretch{1}
	}
	
	\section{Introduction}
	We consider the stochastic optimization problem
	\[ \min_{x \in \R^d} F(x), \qquad F(x) \eqdef \ExpSub{\xi \sim \cD}{f(x, \xi)}, \numberthis\label{eq:optim-problem} \]
	where $d\geq 1$ is the dimension, $F \colon \R^d \to \R$ is a smooth possibly nonconvex objective function, and $\xi$ is a random variable drawn from an unknown distribution $\cD$. In the nonconvex setting, our goal is to find an $\eps$--stationary point~\citep{nesterov2018lectures}; that is, a (random) vector $\bar{x}$ such that $\mathbb{E}[\norm{\nabla F(\bar{x})}] \le \eps$. Problem of the form~\eqref{eq:optim-problem} are pervasive in machine learning where $f\left(x, \xi \right)$ denotes the loss of a model with weights $x$ on a data sample $\xi \sim \mathcal{D}$, and $\cD$ is the distribution of the training samples from the dataset~\citep{bottou2018optimizationmethodsML}. While Gradient Descent (\algname{GD}) is known to achieve the optimal rate $\cO(\eps^{-2})$ for finding an $\eps$--stationary point \citep{carmon2020lower}, it requires access to exact gradients $\nabla F(\cdot)$ which is infeasible in practice. Therefore, reliance on noisy gradients has become a gold standard approach, giving rise to stochastic gradient methods like Stochastic Gradient Descent (\algname{SGD})~\citep{MR42668}. Under the standard $\sigma_1^2$--bounded variance assumption, \algname{SGD} is provably optimal~\citep{ghadimi2013stochastic}, matching the lower bound $\Omega(\nicefrac{L_1 \Delta}{\eps^2} + \nicefrac{L_1 \Delta \sigma_1^2}{\eps^4})$~\citep{arjevani2022lower} for $L_1$--smooth functions with 
	$\smash{\Delta \eqdef F(x^0) - F^{\inf}}$.
	
	However, empirical observations in modern machine learning, e.g., image classification~\citep{simsekli2019heavytailedtheorystochasticgradient,pmlr-v97-simsekli19a,pmlr-v238-battash24a}, large language models~\citep{10.5555/3495724.3497014,ahn2024linear} and reinforcement learning~\citep{pmlr-v139-garg21b}, have shed light on the importance of \emph{heavy-tailed} noise as a more realistic setting than the standard bounded variance assumption. That is, noisy gradients only have bounded $p$--th central moment, $p \in \intof{1}{2}$. Such a relaxed moment assumption introduces challenges on both the algorithmic and theoretical sides, as \algname{SGD} may fail to converge when $p < 2$. In the heavy-tailed regime,~\citet{10.5555/3495724.3497014} establish a lower bound of $\Omega(T^{-\nicefrac{(p - 1)}{(3 p - 2)}})$ on the convergence rate of first-order methods, where $T$ denotes the number of iterations. Yet, their analysis is confined to the standard stochastic setting, where algorithms access smooth (potentially nonconvex) functions via unbiased stochastic gradients. Crucially, they did not address more structured assumptions on the gradient noise such as \textit{mean-squared smoothness} or \textit{$\delta$--similarity}, which play a central to the design and analysis of \emph{variance-reduced} methods, but they remain largely unexplored in the heavy-tailed regime.
	
	In this work, we aim to: $(a)$ \textbf{develop lower complexity bounds for first-order variance-reduced methods in the heavy-tailed regime}, thereby closing an important open question from~\cite{pmlr-v195-liu23c}, while advancing beyond the results of~\citet{sun2024gradientnormalizationprovablybenefits}; $(b)$ \textbf{extend existing high-probability analysis to these new assumptions}, and finally, for second-order methods; $(c)$ \textbf{characterize how the tail indices $p$ (of gradients) and $q$ (of Hessians) jointly affect convergence in the heavy-tailed regime}.
	
	Beyond the development of lower bounds, a natural question arises:
	
	\vspace{0.1cm}
	{\centering\textit{Does there exist algorithm(s) which provably match (in expectation) the lower bounds developed in $(a)$?}\par}
	\vspace{0.2cm}
	Here, we answer this question affirmatively by revisiting a classical technique: gradient normalization. Remarkably, our results show that this simple method remains effective even in the heavy-tailed regime.
	
	\subsection{Our Contributions}
	
	\begin{table}[b!]
		\centering
		\caption{Sample complexities of stochastic methods for finding an $\varepsilon$-stationary point (in expectation or with high probability). Stochastic gradients satisfy $p$-BCM with $p \in (1,2]$ (\Cref{ass:p-bounded-central-moment-gradient}). The column “$q$” specifies the relevant moment assumptions: \Cref{ass:mean-squared-smoothness} ($q$-WAS), \Cref{ass:mean-squared-smoothness-2} ($(q,\delta)$-S), or \Cref{ass:q-bounded-central-moment-hessian} ($q$-BCM); $q=\infty$ corresponds to bounded noise. The column “\textbf{HP?}” indicates whether a high-probability guarantee with only polylogarithmic dependence on $\nicefrac{1}{\beta}$ is available.}
		\label{tab:summary}
		\begin{threeparttable}
			\resizebox{\textwidth}{!}{%
				\begin{tabular}{c|cccc}
					
					\bf Setup& \bf Algorithm & \bf Sample Complexity & $q$& \bf HP?\\ \hline \hline  
					\multirow{10}{*}{\makecell{$q$-WAS}} &\begin{tabular}{c}  \algname{AccNSGD}  \\  \cite{pmlr-v195-liu23c} \end{tabular} & $\left(\frac{\sqrt{\bar L\Delta}+ \sigma_1}{\varepsilon}\right)^{2 +\frac{1}{p-1}}$ & {\color{BrickRed} $\ \ \ \ \infty$}$^{\color{blue}(1)}$ &  \ \cmark \\ 
					&\begin{tabular}{c}  \algname{NSFOM with RM}  \\  \cite{he2025complexitynormalizedstochasticfirstorder} \end{tabular} & $\left(\frac{\Delta +\sigma_1^p+L_1+L_1^p+\bar{L}^p}{\varepsilon}\right)^{2 +\frac{1}{p-1}}$ $^{\color{blue}(2)}$ & {\color{BrickRed} $\ p=q$}&  \ \xmark \\ 
					&
					\cellcolor{linen} \begin{tabular}{c} \algname{NSGD-MVR}  \\  Theorem~\ref{thm:nsgd-mvr-convergence-analysis} \end{tabular} & \cellcolor{linen}$\frac{\bar L\Delta}{\varepsilon^2}+\frac{\bar L\Delta}{\varepsilon^2}\left(\frac{\sigma_1}{\varepsilon}\right)^{\frac{p}{q(p-1)}}+ \left(\frac{\sigma_1}{\varepsilon}\right)^{\frac{p}{p-1}} $ & \cellcolor{linen} {\color{PineGreen}$(1,2]$} & \cellcolor{linen} \xmark \\ 
					&\cellcolor{linen} \begin{tabular}{c} \algname{D-Clip-NSGD-MVR}  \\  Theorem~\ref{thm:clipped-nsgd-mvr-convergence-analysis} \end{tabular} & \cellcolor{linen}$\left(\frac{\sqrt{\bar L\Delta}+ \sigma_1}{\varepsilon}\right)^{2 +\max\left\{\frac{1}{p-1}, \frac{p}{q(p-1)}\right\}}$ & \cellcolor{linen} {\color{PineGreen}$(1,2]$} & \cellcolor{linen} \cmark \\ \cline{2-5}
					& \cellcolor{linen} \begin{tabular}{c} Lower Bound \\ \Cref{thm:lower-bound-sofo-mean-squared-smoothness} \end{tabular} & \cellcolor{linen}$\frac{\bar L\Delta}{\varepsilon^2}+\frac{\bar L\Delta}{\varepsilon^2}\left(\frac{\sigma_1}{\varepsilon}\right)^{\frac{p}{q(p-1)}}+ \left(\frac{\sigma_1}{\varepsilon}\right)^{\frac{p}{p-1}} $ & \cellcolor{linen} {\color{PineGreen}$(1,2]$} & \cellcolor{linen} - \\ 
					\hline
					\multirow{6}{*}{\makecell{$(q,\delta)$-S}} 
					&
					\cellcolor{linen} \begin{tabular}{c} \algname{NSGD-MVR}  \\  Theorem~\ref{thm:nsgd-mvr-convergence-analysis-2} \end{tabular} & \cellcolor{linen}$\frac{(L_1+\delta)\Delta}{\varepsilon^2}+\frac{\delta\Delta}{\varepsilon^2}\left(\frac{\sigma_1}{\varepsilon}\right)^{\frac{p}{q(p-1)}}+ \left(\frac{\sigma_1}{\varepsilon}\right)^{\frac{p}{p-1}} $ & \cellcolor{linen} {\color{PineGreen}$(1,2]$} & \cellcolor{linen} \xmark \\ 
					&\cellcolor{linen} \begin{tabular}{c} \algname{D-Clip-NSGD-MVR}  \\  Theorem~\ref{thm:clipped-nsgd-mvr-convergence-analysis-2} \end{tabular} & \cellcolor{linen}$\left(\frac{\sqrt{L_1\Delta}+ \sigma_1}{\varepsilon}\right)^{2 +\max\left\{\frac{1}{p-1}, \frac{p}{q(p-1)}\right\}}+ \left(\frac{\sqrt{\delta\Delta}}{\varepsilon}\right)^{2 + \frac{p}{q(p-1)}}$ & \cellcolor{linen} {\color{PineGreen}$(1,2]$} & \cellcolor{linen} \cmark \\ \cline{2-5}
					& \cellcolor{linen} \begin{tabular}{c} Lower Bound \\ \Cref{thm:lower-bound-sofo-mean-squared-smoothness-2} \end{tabular} & \cellcolor{linen}$\frac{(L_1+\delta)\Delta}{\varepsilon^2}+\frac{\delta\Delta}{\varepsilon^2}\left(\frac{\sigma_1}{\varepsilon}\right)^{\frac{p}{q(p-1)}}+ \left(\frac{\sigma_1}{\varepsilon}\right)^{\frac{p}{p-1}} $ & \cellcolor{linen} {\color{PineGreen}$(1,2]$} & \cellcolor{linen} - \\ 
					\hline
					\multirow{10}{*}{\makecell{$q$-BCM}} &\begin{tabular}{c}  \algname{NSGDHess}  \\  \cite{sadiev2025second} \end{tabular} & $\frac{(L_1 +\sigma_2)\Delta}{\varepsilon^2} + \left(\frac{(L_1 + \sigma_2)\Delta }{\varepsilon^2}  + \frac{\sigma_1}{\varepsilon}\right)\left(\frac{\sigma_1}{\varepsilon}\right)^{\frac{1}{p-1}} $ & {\color{BrickRed} $\ p=q$}&  \ \xmark \\ 
					&\begin{tabular}{c}  \algname{Clip NSGDM-Hess}  \\  \cite{sadiev2025second} \end{tabular} & $\left(\frac{\sqrt{(L_1+\sigma_2)\Delta}+ \sigma_1}{\varepsilon}\right)^{2 +\frac{1}{p-1}}$ & {\color{BrickRed} $\ p=q$}&  \ \cmark \\ 
					&
					\cellcolor{linen} \begin{tabular}{c} \algname{NSGD-Hess}  \\  Theorem~\ref{thm:nsgd-mvr-hess-convergence-analysis} \end{tabular} & \cellcolor{linen}$\frac{\sigma_2\Delta}{\varepsilon^2}\left(\frac{\sigma_1}{\varepsilon}\right)^{\frac{p}{q(p-1)}}+ \left(\frac{\sigma_1}{\varepsilon}\right)^{\frac{p}{p-1}} +\frac{\sqrt{L_2} \Delta \sigma_1^{\nicefrac{1}{4}}}{\eps^{\nicefrac{7}{4}}} \left( \frac{\sigma_1}{\eps} \right)^{\frac{1}{4 (p - 1)}}$ $^{{\color{blue}(3)}}$ & \cellcolor{linen} {\color{PineGreen}$(1,2]$} & \cellcolor{linen} \xmark \\ 
					&\cellcolor{linen} \begin{tabular}{c} \algname{Clip NSGD-Hess}  \\  Theorem~\ref{thm:clipped-nsgd-mvr-convergence-analysis-hess} \end{tabular} & \cellcolor{linen}$\left(\frac{\sqrt{L_1\Delta}+ \sigma_1}{\varepsilon}\right)^{2 +\max\left\{\frac{1}{p-1}, \frac{p}{q(p-1)}\right\}}+ \left(\frac{\sqrt{\sigma_2\Delta}}{\varepsilon}\right)^{2 + \frac{p}{q(p-1)}}$ $^{{\color{blue}(3)}}$ & \cellcolor{linen} {\color{PineGreen}$(1,2]$} & \cellcolor{linen} \cmark \\ \cline{2-5}
					& \cellcolor{linen} \begin{tabular}{c} Lower Bound \\ \Cref{thm:lower-bound-sofo-bounded-central-moments} \end{tabular} & \cellcolor{linen}$\min\left\{\frac{\sigma_2\Delta }{\varepsilon^2}\left(\frac{\sigma_1}{\varepsilon}\right)^{\frac{p}{q(p-1)}}, \frac{L_1\Delta}{\varepsilon^{2}}\left(\frac{\sigma_1}{\varepsilon}\right)^{\frac{p}{p-1}} , \frac{\sqrt{L_2}\Delta}{\varepsilon^{\nicefrac{3}{2}}}\left(\frac{\sigma_1}{\varepsilon}\right)^{\frac{p}{p-1}} \right\}$$^{{\color{blue}(3)}}$ & \cellcolor{linen} {\color{PineGreen}$(1,2]$} & \cellcolor{linen} - \\ \hline
				\end{tabular}
			}
			\begin{tablenotes}
				{\small \item [{\color{blue}(1)}] 
					\cite{pmlr-v195-liu23c} provide analysis under stronger assumptions, which implies the bounded stochastic Hessian.  
					\item [{\color{blue}(2)}] \cite{he2025complexitynormalizedstochasticfirstorder} establish a near-optimal bound in terms of $\varepsilon$ for the case $p=q$. 
					\item [{\color{blue}(3)}] For simplicity we present only the stochastic part of the complexity.
				}
			\end{tablenotes}
		\end{threeparttable}
	\end{table}
	
	We revisit the analysis of Normalized \algname{SGD} (\algname{NSGD}) with Momentum Variance Reduction (\algname{MVR}) in the heavy-tailed regime. Our contributions are:
	
	$\bullet$ \textbf{Tight lower bounds.} We obtain novel and tight lower bounds on the sample complexity of any first-order methods in the $p$--BCM noise model, under \emph{relaxed} mean-squared smoothness and $\delta$-similarity assumptions, allowing any exponent $q \in \intff{1}{2}$ instead of $q = 2$. Our bound, $\smash{\Omega(\eps^{-\nicefrac{(p (2 q + 1) - 2 q)}{q (p - 1)}})}$, recovers the known $\smash{\Omega(\eps^{-3})}$ when $p = q = 2$ and achieves an improvement over the minimax complexity for first-order methods~\citep{10.5555/3495724.3497014,sun2024gradientnormalizationprovablybenefits,hübler2025gradientclippingnormalizationheavy}, thereby demonstrating the effectiveness of variance reduction in the heavy-tailed regime. Moreover, our result \emph{continuously interpolates} between variance-reduced and non-variance-reduced regimes, as in the limit $q \to 1$ we recover the known $\Omega(\eps^{-\nicefrac{(3 p - 2)}{(p - 1)}})$ complexity. 
	\emph{Our proofs are based on a generalization of~\citet[Lemma~10]{arjevani2022lower} to the $p$--BCM noise assumption, which resolves an open question from~\citet{pmlr-v195-liu23c}}.
	
	$\bullet$ \textbf{Optimal first-order methods.} Revisiting \algname{NSGD-MVR}, we show that, with a suitable choice of parameters, the method matches the lower bounds (up to constant factors) with respect to all parameters, thereby establishing the optimality of our analysis and the robustness of gradient normalization in the context of heavy-tailed noises. 
	We further provide convergence rates with \emph{unknown} tail indices $p$ and $q$, recovering the best known rate of $\cO(\eps^{-\nicefrac{2 p}{(p - 1)}})$~\citep{liu2025nonconvex}.
	
	$\bullet$ \textbf{High-probability upper bounds.} We propose a clipped variant of \algname{NSGD-MVR} leading to a \emph{new} algorithm: Double-Clipped \algname{NSGD-MVR}, which provably achieves high-probability convergence rates under weaker assumptions than previous works \citep{pmlr-v195-liu23c}.
	
	$\bullet$ \textbf{Extension to second-order methods.} Lastly, for second-order methods using stochastic Hessians with bounded $q$-th central moment, for some exponent $q \in \intff{1}{2}$, we establish sharper lower bounds than previous works~\citep{pmlr-v125-arjevani20a,sadiev2025second}, extend previous analysis to the more general setting where $p \neq q$ (unlike~\citet{sadiev2025second}), and derive upper bounds with stronger convergence exponents, thereby offering a more complete characterization of the limits of stochastic second-order optimization under heavy-tailed noise. Overall, our results reveal a striking parallel in the rates of variance-reduced methods and second-order optimization. 
	
	See summary in \Cref{tab:summary}.

	\subsection{Related Works}
	
	\textbf{Lower bound in the heavy-tailed regime:} In nonconvex optimization,~\cite{10.5555/3495724.3497014,liu2025nonconvex}
	establish a lower bound of $\smash{\Omega\big(\frac{L_1 \Delta}{\eps^2} + \frac{L_1 \Delta}{\eps^2} \left( \frac{\sigma_1}{\eps}\right)^{\nicefrac{p}{(p - 1)}}\big)}$ under $p$--BCM noise and standard $L_1$--smoothness, recovering the well-known $\Omega(\eps^{-4})$ lower bound when $p = 2$.
	Other line of works focused on algorithm-dependent lower bounds: \algname{NSGD}~\citep{10.5555/3666122.3669370} and recently \algname{SGD}~\citep{fatkhullin2025sgdhandleheavytailednoise}, highlighting the ineffectiveness of \algname{SGD} under heavy-tailed noise with only bounded $p$-th (non-central) moments. Beyond first-order methods,~\cite{sadiev2025second} obtain near-optimal lower bounds for all second-order methods, extending the work of~\cite{pmlr-v125-arjevani20a} to $p$--BCM noise.\medskip
	
	\textbf{Convergence under heavy-tailed noise:} A substantial body of work investigates upper bounds for stochastic optimization with heavy-tailed noise, both in expectation and high probability. For smooth and nonconvex objectives, clipped or normalized variants of \algname{SGD} achieve the rate $\cO(\eps^{-\nicefrac{(3 p - 2)}{(p - 1)}})$ up to logarithmic factors~\citep{10.5555/3495724.3497014,NEURIPS2021_26901deb,pmlr-v195-liu23c,pmlr-v202-sadiev23a}. Among these,~\citet{10.5555/3495724.3497014} obtained the sharpest in-expectation guarantees, later shown to be tight by \citet{hübler2025gradientclippingnormalizationheavy} via \algname{minibatch-NSGD}, while~\citet{NEURIPS2023_4c454d34} established improved high-probability guarantees, avoiding extra $\cO(\log T)$ factors. Later works extended known high-probability analyses to settings with higher-order smoothness~\citep{sadiev2025second} and to more general nonlinear \algname{SGD}-type methods~\citep{DBLP:journals/corr/abs-2410-13954,pmlr-v258-armacki25a}. In parallel,~\citet{sun2024gradientnormalizationprovablybenefits} studied \algname{clip-SGD} under stronger smoothness assumptions, obtaining a rate of $\cO(\eps^{-\nicefrac{(2 p - 1)}{(p - 1)}})$, while~\citet{liu2025nonconvex} derived in-expectation bounds for \algname{NSGD-Mom} under a general $(\sigma_0,\sigma_1)$--affine $p$--BCM model, albeit without matching lower bounds.\medskip

	\textbf{Gradient clipping and normalization:} Gradient clipping and normalization are two closely related techniques that have become central in modern optimization. Gradient clipping, originally popularized to stabilize training across various machine learning applications~\citep{10.5555/3042817.3043083,schulman2017proximalpolicyoptimizationalgorithms}, has been analyzed extensively, providing robustness under relaxed moment assumptions~\citep{PolyTsy79,jakovetic2023nonlinear}, high-order smoothness~\citep{sadiev2025second}, and enabling high-probability convergence guarantees with only logarithmic dependence on the failure probability. These results hold in both convex~\citep{Nazin2019,10.5555/3495724.3496985,JMLR:v22:20-821,10.5555/3692070.3692710,liu2023stochasticnonsmoothconvexoptimization,Gorbunov2024,pmlr-v238-puchkin24a,armacki2024highprobabilityconvergenceboundsnonlinear,armacki2025optimalhighprobabilityconvergencenonlinear} and nonconvex settings~\citep{10.5555/3495724.3497014,NEURIPS2021_26901deb,pmlr-v202-sadiev23a,NEURIPS2023_4c454d34,pmlr-v195-liu23c,sadiev2025second}. Importantly, high-probability guarantees are valuable both theoretically and practically, as they capture the behavior of individual runs rather than merely characterizing the average-case performance.
	
	Normalized gradient methods date back to Nesterov's pioneering work~\citep{nesterov1984minimization} and were later extended to the smooth and stochastic regimes~\citep{kiwiel2001,hazan2015beyond,NIPS2017_ce5140df,pmlr-v89-nacson19b}. In deep learning, normalization addresses exploding and vanishing gradients~\citep{you2017largebatchtrainingconvolutional,you2020large}, though rigorous nonconvex guarantees were first obtained by~\citet{pmlr-v119-cutkosky20b}, showing how Polyak's momentum ensures convergence without large batches. Later works explored \algname{NSGD}'s strengths and limits, including saddle-point escape~\citep{levy2016powernormalizationfasterevasion}, lower bounds~\citep{10.5555/3666122.3669370} and parameter-agnostic convergence~\citep{pmlr-v238-hubler24a}. Recently, extensions to heavy-tailed noise have been investigated~\citep{NEURIPS2021_26901deb,pmlr-v195-liu23c,sun2024gradientnormalizationprovablybenefits}, though typically under stronger smoothness assumptions. 
	
	{\centering\emph{Despite strong empirical performance and complementarity with clipping, the theory of gradient normalization under heavy-tailed noise remains incomplete, motivating our contributions}.\par}\medskip

	\textbf{Variance reduction for stochastic optimization:} Variance reduction originated as a tool to accelerate convergence in convex finite-sum optimization, with seminal works such as~\citet{10.5555/2999325.2999432,10.5555/2999611.2999647,10.5555/2567709.2502598,10.5555/3042817.3043024,10.5555/2968826.2969010} introducing various algorithms, e.g., \algname{SAG}, \algname{SVRG}, \algname{SAGA}. Building on these foundations and the key advances of~\citet{10.1145/3055399.3055448}, a sequence of works 
	\citep{lan2019unified,zhou2019direct,song2020variance,kovalev2020don} developed algorithms attaining near-optimal or optimal rates under various regimes. In nonconvex optimization, variance reduction can also improve convergence: in the general stochastic setting~\eqref{eq:optim-problem}, several works~\citep{SPIDER,cutkosky2019momentum,tran2019hybrid,liu2020optimal,PAGE} established an $\cO(\eps^{-3})$ convergence rate in expectation, which improves upon the classical $\Theta(\eps^{-4})$ of \algname{SGD} and matches the minimax lower bound $\Omega(\eps^{-3})$ under mean-squared smoothness~\citep{arjevani2022lower}. More recently, there has been growing interest in extending variance reduction techniques to the heavy-tailed regime~\citep{pmlr-v195-liu23c,sun2024gradientnormalizationprovablybenefits,he2025complexitynormalizedstochasticfirstorder}, where normalization and clipping have emerged as effective mechanisms to guarantee robustness under weaker moment assumptions. 
	
	{\centering\emph{In this work, we characterize the fundamental limits of first-order \say{variance-reduced} methods in the heavy-tailed regime and provide optimal methods which attain these limits}.\par}

	\section{Notation and  Assumptions}
	
	We review here the basic notation and assumptions needed in this paper (see~\Cref{appdx-sec:notation} for more details).
	
	\textbf{Notations.} For integer $n > 0$, $[n] \eqdef \ens{1, 2, \ldots, n}$. We let $d \ge 1$ be the dimension, $\ps{\cdot}{\cdot}$ the standard dot product on $\R^d$, $\norm{\cdot}$ the $\ell^2$--norm and $\normop{\cdot}$ the canonical spectral/operator norm. Here, $\nabla f(\cdot, \cdot)$ and $\nabla^2 f(\cdot, \cdot)$ denote the stochastic gradient and Hessian oracles; $a \wedge b \eqdef \min\ens{a, b}$ and $a \vee b \eqdef \max\ens{a, b}$. We use the standard $\cO(\cdot)$ and $\Omega(\cdot)$ for complexity notation. To avoid confusion, we use subscript $1$ for gradient parameters (e.g., $L_1$, $\sigma_1$) and $2$ for Hessian ones.
	
	In this work, we make the following assumptions.
	\begin{assumption}[Lower Boundedness]
		\label{ass:lower-boundedness}
		The objective $F$ is lower bounded: $F^{\inf} \eqdef \inf_{x \in \R^d} F(x) > -\infty$. 
	\end{assumption}
	We then let $\Delta \eqdef F(x^0) - F^{\inf}$ be the initial suboptimality where $x^0$ is the starting point.
	
	\begin{assumption}[$L_1$--Lipschitz Gradients]
		\label{ass:L-lipschitz-gradients}
		The objective $F$ is differentiable over $\R^d$ and  its gradient is $L_1$--Lipschitz for some $L_1 \ge 0$, i.e., for all $x, y \in \R^d$,
		\[ \norm{\nabla F(x) - \nabla F(y)} \le L_1 \norm{x - y}. \]
	\end{assumption}
	
	\begin{assumption}[$p$--BCM for Gradients]
		\label{ass:p-bounded-central-moment-gradient}
		We have access to unbiased stochastic gradients $\nabla f(x, \xi)$, which have $p$-bounded central moment for some $p \in \intof{1}{2}$, 
		i.e., there exists $\sigma_1 > 0$ such that for all $x \in \R^d$,
		\begin{equation*}
			\ExpSub{\xi \sim \cD}{\nabla f(x, \xi)} = \nabla F(x) \quad \text{and} \quad \ExpSub{\xi \sim \cD}{\norm{\nabla f(x, \xi) - \nabla F(x)}^p} \le \sigma_1^p.
		\end{equation*}
	\end{assumption}
	
	\begin{assumption}[$L_2$--Lipschitz Hessians]
		\label{ass:L-lipschitz-hessians}
		The objective $F$ is twice continuously differentiable over $\R^d$ 
		and we have access to stochastic Hessian-vector products $\nabla^2 f(x, \xi) \cdot v$ for any $v \in \R^d$. Moreover, the Hessian of $F$ is $L_2$--Lipschitz, i.e., for all $x, y \in \R^d$,
		\[ \normop{\nabla^2 F(x) - \nabla^2 F(y)} \le L_2 \norm{x - y}. \] 
	\end{assumption}
	
	\begin{assumption}[$q$--BCM for Hessians]
		\label{ass:q-bounded-central-moment-hessian}
		The stochastic Hessians $\nabla^2 f(x, \xi)$ are unbiased and have $q$-bounded central moment for some $q \in \intff{1}{2}$,
		i.e., there exists $\sigma_2 > 0$ such that for all $x \in \R^d$
		\begin{equation*}
			\ExpSub{\xi \sim \cD}{\nabla^2 f(x, \xi)} = \nabla^2 F(x), \quad \text{and} \quad \ExpSub{\xi \sim \cD}{\normop{\nabla^2 f(x, \xi) - \nabla^2 F(x)}^q} \le \sigma_2^q. 
		\end{equation*}
	\end{assumption}

	\begin{assumption}[{$q$-Weak Average Smoothness~\citep{he2025complexitynormalizedstochasticfirstorder}}]\label{ass:mean-squared-smoothness}
		For some $q \in [1;2]$, there exists a finite constant $\bar{L} \ge 0$ such that, for all $x, y \in \R^d$ we have
		\[ \ExpSub{\xi \sim \cD}{\norm{\nabla f(x, \xi) - \nabla f(y, \xi)}^q} \le \bar{L}^q \norm{x - y}^q. \]
	\end{assumption}
	
	The next assumption has been introduced in~\cite{pmlr-v125-arjevani20a} in the special case $q = 2$.
	\begin{assumption}[$(q, \delta)$--Similarity]\label{ass:mean-squared-smoothness-2}
		For some exponent $q \in [1;2]$, there exists a finite constant $\delta \ge 0$ such that, for all $x, y \in \R^d$ we have
		\begin{equation*}
			\ExpSub{\xi \sim \cD}{\norm{\left[ \nabla f(x, \xi) - \nabla f(y, \xi) \right] - \left[ \nabla F(x) - \nabla F(y) \right]}^q} \leq \delta^q \norm{x - y}^q.
		\end{equation*}
	\end{assumption}
	
	\section{Lower Complexity Bounds}

	To establish our lower bounds, we rely on a technique developed in~\citet{carmon2020lower}, where a \say{worst-case} nonconvex function $F_T \colon \R^T \to \R$ is introduced, which is hard to optimize for any zero-respecting algorithm. Our main novelty here is~\Cref{appdx-lem:lower-bound-global-stichastic-model-mean-squared-smoothness-p-BCM}, an important result which lies at the core of all our lower bounds; it is a generalization of~\citet[Lemma~10]{arjevani2022lower} to the $p$--BCM assumption, yielding an improvement from $\Omega(\nicefrac{\sigma_1^2}{\eps^2})$ to $\Omega((\nicefrac{\sigma_1}{\eps})^{\nicefrac{p}{(p - 1)}})$. This lemma handles a critical case in all~\Cref{thm:lower-bound-sofo-mean-squared-smoothness,thm:lower-bound-sofo-mean-squared-smoothness-2,thm:lower-bound-sofo-bounded-central-moments} when the parameters force the dimension $T$ of the \say{hard instance} $F_T$ to be too small. More details along with the main motivations behind this generalization are provided in~\Cref{appdx-rem:lower-bound-global-stochastic-model-motivations}. We recall all the necessary definitions (function class, oracle class\textellipsis) and the \say{hard instance} $F_T$ from~\citet{carmon2020lower} in~\Cref{appdx-sec:lower-bound-def}. The proofs of the following three lower bounds can be found in~\Cref{appdx-subsec:proof-lower-bound-sofo-mean-squared-smoothness,appdx-subsec:proof-lower-bound-sofo-mean-squared-smoothness-2,appdx-subsec:proof-lower-bound-sofo-bounded-central-moments}.
	
	\subsection{Under $q$-Weak Average Smoothness}
	
	We first establish a lower bound on the oracle complexity of any zero-respecting algorithm under~\Cref{ass:mean-squared-smoothness}, an extension of the mean-squared smoothness (MSS) assumption. This result generalizes one of the core contributions of~\citet{arjevani2022lower}, who derived the optimal lower bound in the $p = q = 2$ setting.
	\begin{theorem}\label{thm:lower-bound-sofo-mean-squared-smoothness}
		Given $\Delta, \bar{L} > 0$, $\sigma_1 \ge 0$ and $0 < \eps \le c_1 \sqrt{\bar{L} \Delta}$ for some universal constant $c_1 > 0$. Then, for any algorithm $A \in \mathcal{A}_{\texttt{zr}}$, there exists a function $f \in \mathcal{F}\left( \Delta \right)$, an oracle and a distribution $(O, \cD) \in \mathcal{O}\left( f, \bar{L}^q, \sigma_1^p \right)$ such that
		\begin{equation*}
			\mathfrak{m}^{\normalfont\texttt{zr}}_{\eps}\left( K, \bar{L}, \Delta, \sigma_1^p \right) \geq \Omega(1) \cdot \left( \left( \frac{\sigma_1}{\eps} \right)^{\frac{p}{p - 1}} + \frac{\bar{L} \Delta}{\eps^2} + \frac{\bar{L} \Delta}{\eps^2} \left( \frac{\sigma_1}{\eps} \right)^{\frac{p}{q (p - 1)}} \right).
		\end{equation*}
	\end{theorem}
	In the special case $p = q = 2$, our result recovers the optimal complexity lower bound $\Omega(\nicefrac{\sigma_1^2}{\eps^2} + \nicefrac{\bar{L} \Delta}{\eps^2} + \nicefrac{\bar{L} \Delta \sigma_1}{\eps^3})$ derived by~\citet{arjevani2022lower}. We further establish the tightness of this bound in~\Cref{thm:nsgd-mvr-convergence-analysis}, where we prove a matching upper bound.
	
	\subsection{Under $(q, \delta)$--Similarity}
	
	The previous bound captures limits of first-order methods under~\Cref{ass:mean-squared-smoothness}, but this assumption couples smoothness and noise together into $\bar{L}$. Assuming separately $(q, \delta)$–similarity with $L_1$-smoothness decouples these effects, yielding sharper lower bounds.
	\begin{theorem}\label{thm:lower-bound-sofo-mean-squared-smoothness-2}
		Given $\Delta, L_1, \delta > 0$, $\sigma_1 \ge 0$ and $0 < \eps \le c_1 \sqrt{L_1 \Delta}$ for some universal constant $c_1 > 0$. Then, for any algorithm $A \in \mathcal{A}_{\texttt{zr}}$, there exists a function $f \in \mathcal{F}\left( \Delta, L_1 \right)$, an oracle and a distribution $(O, \cD) \in \mathcal{O}\left( f, \delta^q, \sigma_1^p \right)$ such that
		\begin{alignat*}{2}
			&\mathfrak{m}^{\normalfont\texttt{zr}}_{\eps}\left( K, L_1, \Delta, \delta, \sigma_1^p \right) \\
			&\qquad \ge \Omega(1) \cdot \min\left\{ \frac{L_1 \Delta}{\eps^2} + \frac{L_1 \Delta}{\eps^2} \left( \frac{\sigma_1}{\eps} \right)^{\frac{p}{p - 1}},  \left( \frac{\sigma_1}{\eps} \right)^{\frac{p}{p - 1}} + \frac{(L_1 + \delta) \Delta}{\eps^2} + \frac{\delta \Delta}{\eps^2} \left( \frac{\sigma_1}{\eps} \right)^{\frac{p}{q (p - 1)}} \right\}.
		\end{alignat*}
	\end{theorem}
	It is worth noting that~\Cref{ass:mean-squared-smoothness} is in fact equivalent to~\Cref{ass:mean-squared-smoothness-2,ass:L-lipschitz-gradients}. Specifically, the latter two assumptions imply~\Cref{ass:mean-squared-smoothness} with $\bar{L}^q = 2^{q - 1} \left( L_1^q + \delta^q \right)$, while conversely, if~\Cref{ass:mean-squared-smoothness} holds then both~\Cref{ass:L-lipschitz-gradients,ass:mean-squared-smoothness-2} hold with $L_1 = \bar{L}$ and $\delta = \bar{L}$. Hence, $L_1 \le \bar{L}$, $\delta \le \bar{L}$, and $\bar{L}^q \le 2^{q - 1} \left( L_1^q + \delta^q \right) \le 2^q \bar{L}^q$. Expressing the rate in terms of $L_1$ and $\delta$ is more precise than using $\bar{L}$ alone; in particular, when $\delta \ll \bar{L}$, we show in~\Cref{thm:nsgd-mvr-convergence-analysis-2} that \algname{NSGD-MVR} can achieve improved complexity, even matching the second term of the above $\min\ens{\ldots}$.
	
	\subsection{Under Bounded Central Moments}
	
	Compared to the previous results, which focused on first-order methods under $q$-weak average smoothness and $(q, \delta)$–similarity, we consider here the setting where the oracle's noise is controlled only via bounded central moments, allowing distinct exponents $p$ and $q$ for the stochastic gradients and Hessians.
	\begin{theorem}\label{thm:lower-bound-sofo-bounded-central-moments}
		Given $\Delta, L_1, L_2 > 0$, $\sigma_1, \sigma_2 \ge 0$ and $0 < \eps \le c_1 \min\{\sqrt{L_1 \Delta}, L_2^{\nicefrac{1}{3}} \Delta^{\nicefrac{2}{3}} \}$ for some universal constant $c_1 > 0$. Then, for any algorithm $A \in \mathcal{A}_{\texttt{zr}}$, there exists a function $f \in \mathcal{F}\left( \Delta, L_1, L_2 \right)$, an oracle and a distribution $(O, \cD) \in \mathcal{O}\left( f, \sigma_1^p, \sigma_2^q \right)$ such that\footnote{For clarity, we omit the deterministic term in this lower bound. The full lower bound can be found in the proof of~\Cref{thm:lower-bound-sofo-bounded-central-moments} in~\Cref{appdx-subsec:proof-lower-bound-sofo-bounded-central-moments}.}
		\begin{alignat*}{2}
			&\mathfrak{m}^{\normalfont\texttt{zr}}_{\eps}\left( K, L_1, L_2, \Delta, \sigma_1^p, \sigma_2^q \right) \\
			&\quad \ge \Omega(1) \cdot \min\left\{ \frac{L_1 \Delta}{\eps^2} \left( \frac{\sigma_1}{\eps} \right)^{\frac{p}{p - 1}}, \frac{L_2^{\nicefrac{1}{2}} \Delta}{\eps^{\nicefrac{3}{2}}} \left( \frac{\sigma_1}{\eps} \right)^{\frac{p}{p - 1}},  \left( \frac{\sigma_1}{\eps} \right)^{\frac{p}{p - 1}} + \frac{\sigma_2 \Delta}{\eps^2} + \frac{\sigma_2 \Delta}{\eps^2} \left( \frac{\sigma_1}{\eps} \right)^{\frac{p}{q (p - 1)}} \right\}.
		\end{alignat*}
	\end{theorem}

	\begin{remark}
		In the noiseless setting, i.e., $\sigma_1 = \sigma_2 = 0$, the complexity of~\Cref{thm:lower-bound-sofo-bounded-central-moments} reduces to
		\[ \min\ens{\frac{L_1 \Delta}{\eps^2}, \frac{L_2^{\nicefrac{1}{2}} \Delta}{\eps^{\nicefrac{3}{2}}}}, \numberthis\label{b26a6d07-2917-471c-a3c4-5c9b1b33f5cc} \]
		which is matched by the combination of gradient descent (\algname{GD}) and cubic regularized Newton method~\citep{Nesterov2006}, hence is optimal. This is slightly better than the $\smash{\cO(L_1^{\nicefrac{1}{2}} L_2^{\nicefrac{1}{4}} \Delta \eps^{-\nicefrac{7}{4}})}$ bound 
		achieved in~\citet{10.5555/3305381.3305449}, since
		\[ \min\ens{\frac{L_1 \Delta}{\eps^2}, \frac{L_2^{\nicefrac{1}{2}} \Delta}{\eps^{\nicefrac{3}{2}}}} \le \sqrt{\frac{L_1 \Delta}{\eps^2} \cdot \frac{L_2^{\nicefrac{1}{2}} \Delta}{\eps^{\nicefrac{3}{2}}}} = \frac{L_1^{\nicefrac{1}{2}} L_2^{\nicefrac{1}{4}} \Delta}{\eps^{\nicefrac{7}{4}}}, \]
		but achieving~\eqref{b26a6d07-2917-471c-a3c4-5c9b1b33f5cc} requires full Hessian access in order to get rid of the dependency in $L_1$~\citep{pmlr-v125-arjevani20a,Carmon2021}.
	\end{remark}
	
	\section{Optimal Method under First-Order $p$--BCM}\label{sec:optimal-method-nsgd-mvr}
	
	To establish the optimal rate, we consider the well-known Normalized \algname{SGD} algorithm with momentum variance reduction \texttt{MVR} technique~\citep{cutkosky2019momentum}, and present its pseudo-code in \Cref{algo:nsgd-mvr}.
	
	\begin{algorithm}%
		\caption{\algname{NSGD-MVR} (Normalized \algname{SGD} with \texttt{MVR})}%
		\label{algo:nsgd-mvr}%
		
		\DontPrintSemicolon%
		\SetKwProg{Init}{Initialization}{:}{}%
		\Init{}{%
			$x_0 \in \R^d$, the starting point\;
			$T > 0$, the number of iterations\;
			$g_0 \in \R^d$, an initial vector\;
			$\gamma > 0$, the stepsize\;
			$\alpha \in \intof{0}{1}$, the momentum parameter for \texttt{MVR}\;
		}%
		
		\vspace{\baselineskip}
		
		$x_1 \gets x_0 - \gamma \frac{g_0}{\norm{g_0}}$\;
		\For{$t = 1, 2, \ldots, T - 1$}{%
			\tcp*[h]{Apply \texttt{MVR}.}\;
			$g_t \gets (1 - \alpha) \left( g_{t - 1} + \nabla f\left( x_t, \xi_t \right) - \nabla f \left( x_{t - 1}, \xi_t \right) \right) + \alpha \nabla f \left( x_t, \xi_t \right)$\;
			\tcp*[h]{Do one descent step.}\;
			$x_{t + 1} \gets x_t - \gamma \frac{g_t}{\norm{g_t}}$\;
		}%
		\KwOut{$x_T$}%
	\end{algorithm}%
	
	\subsection{Convergence Analysis}
	
	\subsubsection{Case of Known Tail Indices $p$ and $q$}
	
	\begin{theorem}\label{thm:nsgd-mvr-convergence-analysis}
		Under~\Cref{ass:lower-boundedness,ass:p-bounded-central-moment-gradient,ass:mean-squared-smoothness}, let the initial gradient estimate $g_0$ be given by
		\[ g_0 = \frac{1}{B_{\textnormal{init}}} \sum_{j = 1}^{B_{\textnormal{init}} - 1} \nabla f\left( x_0, \xi_{0, j} \right), \]
		where $B_{\textnormal{init}} = \max\ens{1, \left( \frac{\sigma_1}{\eps} \right)^{\frac{p}{p - 1}}}$, let the stepsize $\gamma = \sqrt{\frac{\Delta \alpha^{\nicefrac{1}{q}}}{\bar{L} T}}$, the momentum parameter $\alpha = \min\ens{1, \alpha_{\textnormal{eff}}}$ where
		\[ \alpha_{\textnormal{eff}} = \max\ens{\left( \frac{\eps}{\sigma_1 T} \right)^{\frac{p}{2p - 1}}, \left( \frac{\bar{L} \Delta}{\sigma_1^2 T} \right)^{\frac{pq}{p (2 q + 1) - 2 q}}}. \numberthis\label{2c0dbb02-323f-4955-a4f2-8f3b73a0f89d} \]
		Then,~\Cref{algo:nsgd-mvr} guarantees to find an $\eps$--stationary point with  total sample complexity
		\[ \cO\left( \left( \frac{\sigma_1}{\eps} \right)^{\frac{p}{p - 1}} + \frac{\bar{L} \Delta}{\eps^2} + \frac{\bar{L} \Delta}{\eps^2} \left( \frac{\sigma_1}{\eps} \right)^{\frac{p}{q (p - 1)}} \right). \]
	\end{theorem}
	The detailed proof is provided in~\Cref{appdx:proof-nsgd-mvr}.
	
	Using~\Cref{ass:mean-squared-smoothness-2} combined with~\Cref{ass:L-lipschitz-gradients}, we obtain the following refined result.
	\begin{theorem}\label{thm:nsgd-mvr-convergence-analysis-2}
		Under~\Cref{ass:lower-boundedness,ass:L-lipschitz-gradients,ass:p-bounded-central-moment-gradient,ass:mean-squared-smoothness-2}, let the initial gradient estimate $g_0$ be given by
		\[ g_0 = \frac{1}{B_{\textnormal{init}}} \sum_{j = 1}^{B_{\textnormal{init}} - 1} \nabla f\left( x_0, \xi_{0, j} \right), \]
		with $B_{\textnormal{init}} = \max\ens{1, \left( \frac{\sigma_1}{\eps} \right)^{\frac{p}{p - 1}}}$,   stepsize $\gamma = \min\ens{\sqrt{\frac{\Delta}{L_1 T}}, \sqrt{\frac{\Delta \alpha^{\nicefrac{1}{q}}}{\delta T}}}$,  momentum parameter $\alpha = \min\ens{1, \alpha_{\textnormal{eff}}}$, where
		\[ \alpha_{\textnormal{eff}} = \max\ens{\left( \frac{\eps}{\sigma_1 T} \right)^{\frac{p}{2p - 1}}, \left( \frac{\delta \Delta}{\sigma_1^2 T} \right)^{\frac{pq}{p (2 q + 1) - 2 q}}}. \numberthis\label{2c0dbb02-323f-4955-a4f2-8f3b73a0f89d-2} \]
		Then \Cref{algo:nsgd-mvr} is guaranteed to find an $\eps$-stationary point with total sample complexity
		\[ \cO\left( \left( \frac{\sigma_1}{\eps} \right)^{\frac{p}{p - 1}} + \frac{(L_1 + \delta) \Delta}{\eps^2} + \frac{\delta \Delta}{\eps^2} \left( \frac{\sigma_1}{\eps} \right)^{\frac{p}{q (p - 1)}} \right). \]
	\end{theorem}
	The detailed proof is provided in~\Cref{appdx:proof-nsgd-mvr}. Essentially, the constant $L_1 + \delta$ is of  same order as the constant $\bar{L}$ from~\Cref{ass:mean-squared-smoothness}.
	
	This shows that the lower bound from~\Cref{thm:lower-bound-sofo-mean-squared-smoothness-2} is matched, up to some multiplicative constant, by the combination of \algname{NSGD-Mom} (normalized \algname{SGD} with momentum) with \algname{NSGD-MVR}, which is~\Cref{algo:nsgd-mvr}.

	\subsubsection{Case of Unknown Tail Indices $p$ and $q$}
	
	\begin{theorem}\label{thm:nsgd-mvr-convergence-analysis-unknown-p-q}
		Under~\Cref{ass:lower-boundedness,ass:p-bounded-central-moment-gradient,ass:mean-squared-smoothness}, assume 
		$g_0 = \nabla f(x_0, \xi_0)$, 
		$\gamma = \sqrt{\frac{\Delta \alpha}{\bar{L} T}}$, 
		$\alpha = T^{-\frac{1}{2}} \in \intof{0}{1}$. Then in ~\Cref{algo:nsgd-mvr} we have 
		\[ \frac{1}{T} \sum_{t = 0}^{T - 1} \E{\norm{\nabla F(x_t)}} = \cO\left( \frac{\sigma_1}{T^{\frac{p - 1}{2 p}}} + \frac{\sqrt{\bar{L} \Delta}}{T^{\nicefrac{1}{4}}} \right). \numberthis\label{12c01754-28b6-4bc2-9ab4-c6f380187237} \]
	\end{theorem}
	
	Our~\Cref{thm:nsgd-mvr-convergence-analysis-unknown-p-q} makes the pessimistic assumption $q = 1$, under which we recover the usual $\cO(\smash{T^{-\frac{p - 1}{2 p}}})$ rate with unknown tail index $p$~\citep{hübler2025gradientclippingnormalizationheavy,liu2025nonconvex}. The bound~\eqref{12c01754-28b6-4bc2-9ab4-c6f380187237} can be slightly refined under~\Cref{ass:L-lipschitz-gradients,ass:mean-squared-smoothness-2}.
	
	\begin{theorem}\label{thm:nsgd-mvr-convergence-analysis-unknown-p-q-2}
		Under~\Cref{ass:lower-boundedness,ass:L-lipschitz-gradients,ass:p-bounded-central-moment-gradient,ass:mean-squared-smoothness-2}, 
		assume $g_0 = \nabla f(x_0, \xi_0)$, 
		$\gamma = \min\ens{\sqrt{\frac{\Delta}{L_1 T}}, \sqrt{\frac{\Delta \alpha}{\delta T}}}$,  
		$\alpha = T^{-\frac{1}{2}} \in \intof{0}{1}$. Then in \Cref{algo:nsgd-mvr} we have
		\[ \frac{1}{T} \sum_{t = 0}^{T - 1} \E{\norm{\nabla F(x_t)}} = \cO\left( \frac{\sigma_1}{T^{\frac{p - 1}{2 p}}} + \frac{\sqrt{\delta \Delta}}{T^{\nicefrac{1}{4}}} + \frac{\sqrt{L_1 \Delta}}{T^{\nicefrac{1}{2}}} \right). \]
	\end{theorem}
	The detailed proof is provided in~\Cref{appdx:proof-nsgd-mvr}.
	
	\subsection{Discussion of the Obtained Results}
	
	As we can see, \algname{NSGD-MVR} achieves the optimal rate: our upper bounds in \Cref{thm:nsgd-mvr-convergence-analysis,thm:nsgd-mvr-convergence-analysis-2} match the lower bounds in \Cref{thm:lower-bound-sofo-mean-squared-smoothness,thm:lower-bound-sofo-mean-squared-smoothness-2}. For the standard case $q=2$, the stochastic terms are
	$\tfrac{\bar L \Delta}{\varepsilon^2}\!\left(\tfrac{\sigma_1}{\varepsilon}\right)^{\nicefrac{p}{2(p-1)}}
	\quad \text{and} \quad
	\tfrac{\delta \Delta}{\varepsilon^2}\!\left(\tfrac{\sigma_1}{\varepsilon}\right)^{\nicefrac{p}{2(p-1)}},$
	under \Cref{ass:mean-squared-smoothness} ($q$-WAS) and \Cref{ass:mean-squared-smoothness-2} ($(q,\delta)$-S), respectively. Compared to \cite{pmlr-v195-liu23c}, this yields faster rates under weaker assumptions.
	
	In the case $p=q$, these stochastic terms coincide with those for \algname{NSGDHess} in \cite{sadiev2025second} under \Cref{ass:q-bounded-central-moment-hessian}:
	$\tfrac{\sigma_2 \Delta}{\varepsilon^2}\!\left(\tfrac{\sigma_1}{\varepsilon}\right)^{\nicefrac{1}{(p-1)}}.$
	This relation follows since \Cref{ass:mean-squared-smoothness-2} implies \Cref{ass:q-bounded-central-moment-hessian}, and under certain conditions $\delta \leq \sigma_2$. Moreover, for empirical risk minimization with $p=q=2$, the two assumptions are equivalent \citep{pmlr-v125-arjevani20a}.
	
	These insights motivate the study of a more general framework for second-order stochastic optimization.

	\section{Near-Optimal Method under $p$--BCM Gradients and $q$--BCM Hessians}\label{sec:near-optimal-method-nsgd-mvr-hess}
	
	Unlike Algorithm~\ref{algo:nsgd-mvr}, we consider a method, which employs Hessian–vector products instead of differences of stochastic gradients. This modification leads to the Hessian-corrected momentum (\texttt{Hess}) technique~\citep{10.5555/3600270.3600525,salehkaleybar-et-al22}.
	\begin{algorithm}%
		\caption{\algname{NSGD-\texttt{Hess}} (Normalized \algname{SGD} with Hessian-corrected Momentum)}%
		\label{algo:nsgd-mvr-hess}%
		
		\DontPrintSemicolon%
		\SetKwProg{Init}{Initialization}{:}{}%
		\Init{}{%
			$x_0 \in \R^d$, the starting point\;
			$T > 0$, the number of iterations\;
			$g_0 \in \R^d$, an initial vector\;
			$\gamma > 0$, the stepsize\;
			$\alpha \in \intof{0}{1}$, the momentum parameter for \texttt{Hess}\;
		}%
		
		\vspace{\baselineskip}
		
		$x_1 \gets x_0 - \gamma \frac{g_0}{\norm{g_0}}$\;
		\For{$t = 1, 2, \ldots, T - 1$}{%
			Sample $q_t \sim \mathcal{U}\left(\intff{0}{1}\right)$\;
			$\hat{x}_t \gets q_t x_t + (1 - q_t) x_{t - 1}$\;
			\tcp*[h]{Apply \texttt{Hess}, here $\xi_t, \hat{\xi}_t \sim \cD$ are independent.}\;
			$g_t \gets (1 - \alpha) \left( g_{t - 1} + \nabla^2 f(\hat{x}_t, \hat{\xi}_t) (x_t - x_{t - 1}) \right) + \alpha \nabla f \left( x_t, \xi_t \right)$\;
			\tcp*[h]{Do one descent step.}\;
			$x_{t + 1} \gets x_t - \gamma \frac{g_t}{\norm{g_t}}$\;
		}%
		\KwOut{$x_T$}%
	\end{algorithm}%
	
	\begin{remark}
		\Cref{algo:nsgd-mvr-hess} is not new and has been considered in~\citet{sadiev2025second}. Here, we provide a refined analysis and improve over~\citet[Theorem~2]{sadiev2025second}. Note that~\Cref{algo:nsgd-mvr-hess} is \textit{zero-respecting} (see~\Cref{appdx-def:zero-especting-algo}) since $\supp(\hat{x}_t) \subseteq \supp(x_{t - 1}) \cup \supp(x_t)$, where $\supp(x) \eqdef \enstq{i \in [d]}{x_i \neq 0}$ is the support of $x = (x_1, \ldots, x_d) \in \R^d$.
		
		Additionally, the use of $\hat{x}_t$ taken uniformly at random on the line $\intff{x_{t - 1}}{x_t}$ allows to use~\Cref{appdx-technical-lem:app-von-Bahr-and-Essen}, which provides a better overall sample complexity than the one derived in~\cite{10.5555/3600270.3600525}.
	\end{remark}
	
	\begin{theorem}\label{thm:nsgd-mvr-hess-convergence-analysis}
		Under~\Cref{ass:lower-boundedness,ass:L-lipschitz-gradients,ass:L-lipschitz-hessians,ass:p-bounded-central-moment-gradient,ass:q-bounded-central-moment-hessian}, let the initial gradient estimate $g_0$ be given by
		\[ g_0 = \frac{1}{B_{\textnormal{init}}} \sum_{j = 1}^{B_{\textnormal{init}} - 1} \nabla f\left( x_0, \xi_{0, j} \right), \]
		where $B_{\textnormal{init}} = \max\ens{1, \left( \frac{\sigma_1}{\eps} \right)^{\frac{p}{p - 1}}}$, let the stepsize 
		\[ \gamma = \min\ens{\sqrt{\frac{\Delta}{L_1 T}}, \sqrt{\frac{\Delta \alpha^{\nicefrac{1}{q}}}{\sigma_2 T}}, \sqrt[3]{\frac{\Delta \alpha^{\nicefrac{1}{2}}}{L_2 T}}}, \]
		the momentum parameter $\alpha = \min\ens{1, \alpha_{\textnormal{eff}}}$, where 
		\begin{alignat*}{2}
			\alpha_{\textnormal{eff}} = & \max\left\{\left( \frac{\eps}{\sigma_1 T} \right)^{\frac{p}{2p - 1}}\!, \left( \frac{\sigma_2 \Delta}{\sigma_1^2 T} \right)^{\frac{pq}{p (2 q + 1) - 2 q}}\!, 
			\left( \frac{L_2^{\nicefrac{1}{2}} \Delta}{\sigma_1^{\nicefrac{3}{2}} T} \right)^{\frac{4 p}{7 p - 6}}\right\}.
		\end{alignat*}
		Then,~\Cref{algo:nsgd-mvr-hess} guarantees to find an $\eps$--stationary point with the total sample complexity\footnote{For clarity, we omit the optimization terms in this lower bound. The full lower bound can be found in the proof of~\Cref{thm:nsgd-mvr-hess-convergence-analysis} in~\Cref{appdx-subsec:proof-nsgd-mvr-hess-convergence-analysis}.}
		\begin{alignat*}{2}
			&\cO\left( \left( \frac{\sigma_1}{\eps} \right)^{\frac{p}{p - 1}} + \frac{\Delta \sigma_2}{\eps^2} + \frac{\Delta \sigma_2}{\eps^2} \left( \frac{\sigma_1}{\eps} \right)^{\frac{p}{q (p - 1)}} + \frac{L_2^{\nicefrac{1}{2}} \Delta \sigma_1^{\nicefrac{1}{4}}}{\eps^{\nicefrac{7}{4}}} \left( \frac{\sigma_1}{\eps} \right)^{\frac{1}{4 (p - 1)}} \right).
		\end{alignat*}
	\end{theorem}
	The detailed proof is provided in~\Cref{appdx:proof-nsgd-mvr-hess}.

	\section{Gradient Clipping for High-Probability Convergence}\label{sec:clipped-nsgd-mvr}
	
	In this section, we conduct a high-probability analysis under heavy-tailed noise. Since we work under \Cref{ass:p-bounded-central-moment-gradient}, we incorporate gradient clipping in~\Cref{algo:nsgd-mvr}; the pseudo-code of the new method is provided in~\Cref{algo:clipped-nsgd-mvr}. Formally, a clipping operator (clipping for short) is defined as
	\[ \clip(v, \lbd) = \min\ens{1, \frac{\lbd}{\norm{v}}} v \quad \text{ for any $v \neq 0$ in $\R^d$}, \]
	where $\lbd > 0$ is called clipping level/threshold. The proofs are deferred to~\Cref{appdx:proof-clipped-nsgd-mvr}.
	
	\begin{algorithm}%
		\caption{\algname{D-\clip-NSGD-MVR} (Double Clipped Normalized \algname{SGD} with \texttt{MVR})}%
		\label{algo:clipped-nsgd-mvr}%
		
		\DontPrintSemicolon%
		\SetKwProg{Init}{Initialization}{:}{}%
		\Init{}{%
			$x_0 \in \R^d$, the starting point\;
			$T > 0$, the number of iterations\;
			$g_0 \in \R^d$, an initial vector\;
			$\gamma > 0$, the stepsize\;
			$\alpha \in \intof{0}{1}$, the momentum parameter for \texttt{MVR}\;
			$\lbd_1, \lbd_2 > 0$, the clipping thresholds\;
		}%
		
		\vspace{\baselineskip}
		
		$x_1 \gets x_0 - \gamma \frac{g_0}{\norm{g_0}}$\;
		\For{$t = 1, 2, \ldots, T - 1$}{%
			\tcp*[h]{Apply \texttt{MVR} with \clip~operator.}\;
			$ \begin{aligned}[t]
				g_t \gets&(1 - \alpha) \left( g_{t - 1} + \clip(\nabla f\left( x_t, \xi_t \right) - \nabla f \left( x_{t - 1}, \xi_t \right), \lbd_1) \right) + \alpha \, \clip(\nabla f \left( x_t, \xi_t \right), \lbd_2)
			\end{aligned}
			$\;
			\tcp*[h]{Do one descent step.}\;
			$x_{t + 1} \gets x_t - \gamma \frac{g_t}{\norm{g_t}}$\;
		}%
		\KwOut{$x_T$}%
	\end{algorithm}%

	\Cref{algo:clipped-nsgd-mvr} modifies \Cref{algo:nsgd-mvr} by applying clipping not only to the stochastic gradient — to control heavy-tailed noise — but also to the gradient difference term. This additional clipping step is the primary feature distinguishing our method from Accelerated NSGD with clipping and momentum \citep{pmlr-v195-liu23c}. Crucially, this modification enables us to establish high-probability guarantees under assumptions weaker than those in \cite{pmlr-v195-liu23c}. While their analysis relies on individual smoothness (i.e., assuming $f(\cdot,\xi)$ is $\bar L$-smooth almost surely), our analysis holds under $q$-WAS (\Cref{ass:mean-squared-smoothness}). Individual smoothness is a significantly stronger condition; notably, under that assumption alone, \cite{lei2019stochastic} showed it is possible to attain an $\cO(\varepsilon^{-4})$ rate independent of the heavy-tail index $p$.
	
	\begin{theorem}\label{thm:clipped-nsgd-mvr-convergence-analysis}
		Under~\Cref{ass:lower-boundedness,ass:p-bounded-central-moment-gradient,ass:mean-squared-smoothness}, let $T \ge 1$ and $\beta \in \intof{0}{1}$ be such that $\log \frac{8 T}{\beta} \ge 1$. Let $x_0\in\mathbb{R}^d$ and define $\Delta_1 \eqdef F(x_0) - F^{\inf}$. 
		Suppose that \Cref{algo:clipped-nsgd-mvr} is run with $g_0 = 0$, momentum parameter $\alpha = \max\{T^{-\frac{p}{2 p - 1}}, T^{-\frac{p q}{p (2 q + 1) - 2 q}}\}$, clipping thresholds $\lbd_1 = 2 \gamma \bar{L} \alpha^{-\frac{1}{q}}$ and $\lbd_2 = \max\{4 \sqrt{\bar{L} \Delta_1}, \sigma_1 \alpha^{-\frac{1}{p}}\}$, and  stepsize
		\begin{alignat*}{2}
			\gamma & = \cO\left( \min\left\{\sqrt{\frac{\Delta_1}{\bar{L} T}}, \alpha \sqrt{\frac{\Delta_1}{\bar{L}}}, \frac{1}{\alpha T \log \frac{T}{\beta}} \sqrt{\frac{\Delta_1}{\bar{L}}},\frac{\Delta_1}{\sigma_1 \alpha^{\frac{p - 1}{p}} T \log \frac{T}{\beta}}, \sqrt{\frac{\Delta_1 \alpha^{\frac{1}{q}}}{\bar{L} T \log \frac{T}{\beta}}} \right\} \right).
		\end{alignat*}
		
		Then, with probability at least $1 - \beta$, the output of~\Cref{algo:clipped-nsgd-mvr} satisfies
		\[ \frac{1}{T} \sum_{t = 0}^{T - 1} \norm{\nabla F(x_t)} \le \frac{2 \Delta_1}{\gamma T}, \]
		and, by our choice of parameters, the gradient norm converges with high probability at the rate
		\[ \frac{1}{T} \sum_{t = 0}^{T - 1} \norm{\nabla F(x_t)} = \cO\left( \left( \frac{\sqrt{\bar{L} \Delta_1} + \sigma_1}{T^{\frac{p - 1}{2 p - 1} \wedge \frac{q (p - 1)}{p (2 q + 1) - 2 q}}} \right)\log \frac{T}{\beta} \right). \]
		
	\end{theorem}
	
	The detailed proof is provided in~\Cref{appdx:proof-clipped-nsgd-mvr}.
	
	Notably, in the case where $p=q$, we establish the same rate  $\widetilde \cO\left(\varepsilon^{\frac{2p-1}{p-1}}\right)$  as \citet{pmlr-v195-liu23c} did for Algorithm~2, but under weaker assumptions. Moreover, our lower bounds (see \Cref{thm:lower-bound-sofo-mean-squared-smoothness}) indicate that this rate is optimal in terms of the dependence on $\varepsilon$. 
	
	Another interesting observation is that when $p\leq q$, the high-probability rate remains identical to the $p=q$ case. In contrast, for in-expectation guarantees, we observe a faster rate (see \Cref{thm:nsgd-mvr-convergence-analysis}). Determining whether this gap is fundamental or a limitation of the current proof technique remains an open question, and calls for a tighter high-probability analysis for these methods.
	
	Furthermore, \Cref{appdx:proof-clipped-nsgd-mvr} establishes high-probability convergence guarantees for \Cref{algo:clipped-nsgd-mvr} under \Cref{ass:mean-squared-smoothness-2}, and for \Cref{algo:clipped-nsgd-mvr-hess} under \Cref{ass:L-lipschitz-hessians,ass:q-bounded-central-moment-hessian}.
	
	\textbf{Interpretation of Clipping Thresholds and Practical Implementation.} 
	While deriving optimal high-probability bounds in heavy-tailed settings remains an open challenge, our double-clipping mechanism provides a significant theoretical advantage: it establishes convergence guarantees without assuming bounded noise---a restrictive assumption often required by prior methods lacking this specific mechanism (e.g., \cite{pmlr-v195-liu23c}). 
	
	Beyond the theoretical guarantees, our analysis offers a crucial insight for practical implementation. Under \Cref{ass:p-bounded-central-moment-gradient} (heavy-tailed noise, exponent $p$) and \Cref{ass:mean-squared-smoothness,ass:mean-squared-smoothness-2} ($q$-WAS and $(q,\delta)$-S, exponent $q$), our theoretical clipping thresholds scale as $\lambda_1 \sim \mathcal{O}(\gamma \alpha^{-1/q})$ for the gradient differences, and $\lambda_2 \sim \mathcal{O}(\alpha^{-1/p})$ for the raw stochastic gradients. In practice, we can rewrite the gradient difference clipping step by factoring out the stepsize $\gamma$:
	\begin{align*} 
		\clip\big(\nabla f(x_t, \xi_t) - \nabla f(x_{t-1}, \xi_t), \lambda_1\big) =\gamma \cdot \clip\left( \frac{\nabla f(x_t, \xi_t) - \nabla f(x_{t-1}, \xi_t)}{\gamma}, \bar{\lambda}_1 \right),
	\end{align*}
	where we define the rescaled threshold $\bar{\lambda}_1 := \gamma^{-1} \lambda_1$. This reveals that $\bar{\lambda}_1$ scales as $\mathcal{O}(\alpha^{-1/q})$. Consequently, in the standard case where $p = q$, both thresholds share the exact same scaling: $\bar{\lambda}_1 \sim \lambda_2 \sim \mathcal{O}(\alpha^{-1/p})$. This is a highly valuable insight for practitioners, as it implies that we do not need to tune two independent clipping hyperparameters. Instead, we can simply tune a single clipping parameter $\lambda := \bar{\lambda}_1 = \lambda_2$, significantly simplifying the deployment of \algname{D-\clip-NSGD-MVR}.

	\section*{Acknowledgments} This work was supported by funding from King Abdullah University of Science and Technology (KAUST): 
	
	\noindent i) KAUST Baseline Research Scheme, 
	
	\noindent ii) Center of Excellence for Generative AI, under award no.\ 5940, 
	
	\noindent iii) Competitive Research Grant (CRG) Program, under award no.\ 6460, 
	
	\noindent iv) SDAIA-KAUST Center of Excellence in Data Science and Artificial Intelligence (SDAIA-KAUST AI).

	\bibliography{bib.bib}

	
	\clearpage%
	
	\appendix

	\section{Additional Notation}\label{appdx-sec:notation}

	Following~\cite{pmlr-v125-arjevani20a}, a $q^{\textnormal{th}}$-order ($q \ge 0$) tensor $T \in \R^{d \times \cdots \times d} = \R^{\otimes^q d}$ is a $q$-dimensional array of real numbers, where $\R^{\otimes^q d}$ denotes the $q$-fold tensor product of $\R^d$. By convention, a $0^{\textnormal{th}}$-order tensor corresponds to a scalar (i.e., an element of $\R$), a $1^{\textnormal{st}}$-order tensor corresponds to a vector in $\R^d$, and a $2^{\textnormal{nd}}$-order tensor corresponds to a matrix in $\R^{d \times d}$. If $q \ge 1$, we denote $T = (T_1, \ldots, T_d)$ where each $T_i$ is the $(q-1)^{\textnormal{th}}$-order subtensor of $T$ obtained by fixing the first index to $i$. Formally, for all $i \in [d]$, we define $[T_i]_{j_1, \ldots, j{q-1}} = T_{i, j_1, \ldots, j_{q-1}}$ where $j_1, \ldots, j_{q-1} \in [d]$. This recursive definition allows us to view any tensor as an ordered collection of its subtensors along a given mode. When $q = 2$, i.e., when $T$ is a matrix, we write $[T]_{i, \cdot}$ for its $i^{\textnormal{th}}$ row to avoid confusion with the $i^{\textnormal{th}}$ coordinate of a vector. Similarly, $[T]_{\cdot, j}$ denotes its $j^{\textnormal{th}}$ column. Throughout this work, we only consider tensors of order $1$ and $2$: first-order tensors correspond to gradients, and second-order tensors correspond to Hessians.
	
	For any integer $n > 0$, we define the index set $[n] \eqdef \ens{1, 2, \ldots, n}$. Let $d \ge 1$ denote the ambient dimension. We use $\ps{\cdot}{\cdot}$ to represent the standard Euclidean inner product on $\R^d$, i.e., $\ps{x}{y} = \sum_{i=1}^d x_i y_i$, and $\norm{\cdot}$ to denote the associated $\ell^2$–norm, $\norm{x} = \sqrt{\ps{x}{x}}$. For a matrix $A \in \R^{d \times d}$, we denote by $\normop{A}$ its operator (spectral) norm, defined as $\normop{A} = \sup_{\norm{x}=1} \norm{A x}$. We write $\nabla f(\cdot, \cdot)$ and $\nabla^2 f(\cdot, \cdot)$ for the stochastic gradient and Hessian oracles, respectively. For any two real numbers $a,b$, we use $a \wedge b \eqdef \min\ens{a,b}$ and $a \vee b \eqdef \max\ens{a,b}$. We adopt the standard asymptotic notations $\cO(\cdot)$ and $\Omega(\cdot)$ to denote upper and lower bounds on growth rates. To prevent ambiguity, constants related to first-order (gradient) quantities are indexed by the subscript $1$ (e.g., $L_1$, $\sigma_1$), whereas those related to second-order (Hessian) quantities carry the subscript $2$ (e.g., $L_2$, $\sigma_2$).
	
	We denote by $\Proba{E}$ the probability of an event $E$ and by $\E{X}$ the expectation of a random variable $X$. Conditional probability and conditional expectation are written, respectively, as $\probac{E}{\mathcal{F}}$ and $\ExpCond{X}{\mathcal{F}}$, where $\mathcal{F}$ denotes a $\sigma$–algebra or a conditioning event. The notation $\ExpSub{\xi}{\cdot}$ (or equivalently $\ExpSub{\xi \sim \P}{\cdot}$) indicates that the expectation is taken with respect to the randomness of $\xi$, and explicitly that $\xi$ is distributed according to the probability law $\P$. This notation is used to clarify the source of randomness when multiple random variables are involved or when the distribution of $\xi$ is not immediately clear from the context.

	\section{Additional Definitions}\label{appdx-sec:lower-bound-def}
	We present in this section the formal setup (function, oracle and algorithm/protocol classes) we considered and under which our (dimension-free) lower and upper bounds are derived.
	
	\subsection{The Setup}
	
	\subsubsection{Function Class}
	The lower bounds developed in this work apply to algorithms that find $\eps$--stationary point of (nonconvex) functions. All functions considered here are defined from $\R^d$ to $\R$ where $d \ge 1$ is an integer. Depending on the assumptions (\Cref{ass:mean-squared-smoothness,ass:mean-squared-smoothness-2,ass:L-lipschitz-hessians}), the class of functions may vary. Formally, we define
	\[
	\mathcal{F}(\Delta) \eqdef \left\{ F \in \mathcal{C}^1(\R^d, \R) \;\middle|\; F(0) - \inf_{x \in \R^d} F(x) \le \Delta \right\} \numberthis\label{106ec3da-9b9d-4af4-acb0-1a325cdaabae}
	\]
	which is the class of continuously differentiable and $\Delta$--bounded functions. If the considered function further satisfies some (standard) regularity conditions (see~\Cref{ass:L-lipschitz-gradients,ass:L-lipschitz-hessians}) we let
	\[
	\mathcal{F}(\Delta, L_1) \eqdef \left\{ F \in \mathcal{C}^1(\R^d, \R) \;\middle|\;
	\begin{aligned}
		& F(0) - \inf_{x \in \R^d} F(x) \le \Delta, \\
		& \norm{\nabla F(x) - \nabla F(y)} \le L_1 \norm{x - y}, \quad \forall x, y \in \R^d
	\end{aligned}
	\right\}
	\]
	and
	\[
	\mathcal{F}(\Delta, L_1, L_2) \eqdef \left\{ F \in \mathcal{C}^2(\R^d, \R) \;\middle|\;
	\begin{aligned}
		& F(0) - \inf_{x \in \R^d} F(x) \le \Delta, \\
		& \norm{\nabla F(x) - \nabla F(y)} \le L_1 \norm{x - y},  \quad \forall x, y \in \R^d \\
		& \normop{\nabla^2 F(x) - \nabla^2 F(y)} \le L_2 \norm{x - y}, \quad \forall x, y \in \R^d
	\end{aligned}
	\right\}
	\]
	when access to second-order information is possible. The Lipchitz constants $L_1$ and $L_2$ appearing above are measured with respect to the canonical (Euclidean) $\ell^2$--norm over $\R^d$ for the gradients and points while, for the Hessian we use the operator norm $\normop{\cdot}$ induced by $\norm{\cdot}$. We recall in~\Cref{appdx-def:operator-norm} the definition of $\normop{\cdot}$.
	
	The dimension $d \ge 1$ appearing in the above definitions will be made explicit in the proofs (see~\Cref{appdx-sec:proofs-lower-bound}) and it may depends on some problem specific parameters, e.g., $\eps$, $\Delta$ and the smoothness constants.
	
	Note that the function class~\eqref{106ec3da-9b9d-4af4-acb0-1a325cdaabae} as stated is too broad and lacks regularity, e.g., Lipchitz continuous gradients. Nonetheless in some cases, appropriate assumptions on the stochastic gradients (for instance, \textit{mean-squared smoothness}, see~\Cref{ass:mean-squared-smoothness}) are enough to ensure the underlying function itself is smooth, i.e., belongs to $\mathcal{F}(\Delta, L)$ for some constant $L \ge 0$ (for the exact derivation of this fact, see~\eqref{94828cd2-cf02-492b-a6ec-40b3012f66c1}).
	
	\subsubsection{Oracle Class}
	
	\begin{definition}[{Stochastic $p^{\textnormal{th}}$-order Oracles~\citep{pmlr-v125-arjevani20a}}]\label{appdx-def:stochastic-pth-order-oracle}
		Given integers $d, p \ge 1$ and a function $F \in \mathcal{C}^p(\R^d, \R)$, we define  $\mathcal{O}_p(F)$ as the class stochastic $p^{\textnormal{th}}$-order oracles, i.e., the pairs $(\P_{\xi}, \texttt{O}^p_F)$ where $\P_{\xi}$ is a distribution on a measurable set $\mathcal{Z}$ and $\texttt{O}^p_F$ is an unbiased mapping defined as
		\[ \texttt{O}^p_F \colon (x, \xi) \mapsto \left( F(x), \nabla f(x, \xi), \ldots, \nabla^p f(x, \xi) \right); \]
		that is, for every $r \in [p]$ we have 
		\[ \ExpSub{\xi \sim \P_{\xi}}{\nabla^r f(x, \xi)} = \nabla^r F(x). \]
	\end{definition}
	Furthermore, if some of the derivative estimator (here, only the gradient estimator $\nabla f(x, \xi)$ or the Hessian estimator $\nabla^2 f(x, \xi)$ are used) satisfies additional properties or assumption depending on some parameters $\sigma_1$ (\Cref{ass:p-bounded-central-moment-gradient}), $\sigma_2$ (\Cref{ass:q-bounded-central-moment-hessian}), $\bar{L}$ (\Cref{ass:mean-squared-smoothness}), $\delta$ (\Cref{ass:mean-squared-smoothness-2}) and so on, the oracle class will be denoted by
	\[ \mathcal{O}_p(F, \sigma_1, \sigma_2\ldots), \numberthis\label{440058b9-2a5c-4a0d-a158-4aaff45f095b} \]
	with all parameters listed inside the parenthesis, in an arbitrary order. As each assumption defines unambiguously its own parameters, we keep the notation~\eqref{440058b9-2a5c-4a0d-a158-4aaff45f095b} for simplicity as it avoids any risk of ambiguity.
	
	\subsubsection{Optimization Protocol and Algorithm Class}
	
	First, let us introduce some important definitions.
	\begin{definition}[Support of a Vector/Tensor]\label{appdx-def:support}
		Let $d \ge 1$ be an integer, the support of a vector $x = (x_1, \ldots, x_d) \in \R^d$ is the set
		\[ \supp(x) \eqdef \enstq{i \in [d]}{x_i \neq 0}, \]
		i.e., the set of all indices $i \in [d]$ for which $x$ has a nonzero $i^{\textnormal{th}}$ coordinate.
		
		For a given $p^{\textnormal{th}}$-order tensor $T = (T_1, \ldots, T_d) \in \R^{d \times \cdots \times d} = \R^{\otimes^p d}$, its support is defined as 
		\[ \supp(T) \eqdef \enstq{i \in [d]}{T_i \neq 0}, \]
		where $T_i$ is the $i^{\textnormal{th}}$ sub-tensor (which is a $(p - 1)^{\textnormal{th}}$-order tensor), e.g., the $i^{\textnormal{th}}$ row of the matrix $T$ if $p = 2$.
	\end{definition}
	
	\begin{definition}[Progress, \protect\say{$\prog$}]
		Let $d \ge 1$ be an integer, for any $x = (x_1, \ldots, x_d) \in \R^d$ and any $\alpha \in \intfo{0}{+\infty}$ we define
		\[ \prog_{\alpha}(x) \eqdef \max\enstq{i \in [d]}{\abs{x_i} > \alpha}, \]
		and $\prog_{\alpha}(x) = 0$ if $\abs{x_i} \le \alpha$ for all $i \in [d]$.
		
		If $\alpha = 0$, for a given $p^{\textnormal{th}}$-order tensor $T = (T_1, \ldots, T_d) \in \R^{d \times \cdots \times d} = \R^{\otimes^p d}$, its \say{$\prog$} is defined as 
		\[ \prog_0(T) \eqdef \max\enstq{i \in [d]}{T_i \neq 0}, \]
		or $0$ if no there is not such index $i \in [d]$ for which $T_i \neq 0$.
	\end{definition}
	Notably, if $\alpha = 0$ then $\prog_0(x)$ is the largest index at which $x \in \R^d$ has a nonzero coordinate. $\prog_0$ will be used to capture the rate at which new coordinates are \say{discovered}. Initially, all coordinates are set to $0$ and, as we progressively acquire information from the queries to the oracle, the union of the support of the oracle responses grows. The growth rate is quantified using $\prog$ and controlled thanks to the notion os \emph{zero-chain}, which we recall formally below in~\Cref{appdx-def:deterministic-zero-chain,appdx-def:probability-zero-chain}.
	
	We recall below some technical notions from the paper of~\citet{arjevani2022lower}.
	\paragraph{Optimization Protocol:} the lower bound guarantees obtained in this work apply to algorithms interacting with an oracle over several rounds, where in each round they may issue a batch of $K \ge 1$ queries (\say{\textit{multi-point}} queries). More formally, at round $i \ge 1$, the algorithm queries the oracle at a batch $\smash{\mathbf{x}^{(i)} = (x^{(i)}_1, \ldots, x^{(i)}_K)}$ of $K$ points in $\R^d$ and for each point $\smash{x^{(i)}_j}$, $j \in [K]$, the oracle performs an independent draw $\xi^{(i)} \sim \P_{\xi}$ and replies with
	\[ \texttt{O}_F^{p, K}\left( \mathbf{x}^{(i)}, \xi^{(i)} \right) \eqdef \left( \texttt{O}^p_F\left( x^{(i)}_1, \xi^{(i)} \right), \ldots, \texttt{O}^p_F\left( x^{(i)}_K, \xi^{(i)} \right) \right), \]
	where the randomness is shared across all queries of the batch: the same seed $\smash{\xi^{(i)}}$ is used. For instance, in~\Cref{algo:nsgd-mvr} we use a batch a $K = 2$ queries.
	
	\paragraph{Optimization Algorithms:} a (randomized) algorithm $\texttt{A}$ consists of a distribution $\P_r$ (over a measurable set $\mathcal{R}$), a random seed $r \sim \P_r$ drawn at the very beginning of the protocol, and a sequence of measurable mappings $\{\texttt{A}^{(i)}\}_{i \ge 1}$ such that $\smash{\texttt{A}^{(i)}}$ takes all the previous $i - 1$ oracle responses and use the random seed $r \in \mathcal{R}$ to produce the next $i^{\textnormal{th}}$ query. Formally, a randomized algorithm $\texttt{A}$ produces the sequence of iterates $\{ \mathbf{x}_{\texttt{A}[\texttt{O}^{p, K}_F]}^{(i)} \}_{i \ge 1}$ where
	\[ \mathbf{x}^{(i)}_{\texttt{A}[\texttt{O}^{p, K}_F]} \eqdef \texttt{A}^{(i)} \left( \left[ \texttt{O}_F^{p, K}\left( \mathbf{x}_{\texttt{A}[\texttt{O}^{p, K}_F]}^{(1)}, \xi^{(1)} \right), \ldots, \texttt{O}_F^{p, K}\left( \mathbf{x}_{\texttt{A}[\texttt{O}^{p, K}_F]}^{(i - 1)}, \xi^{(i - 1)} \right) \right] , r\right). \]
	
	We define $\mathcal{A}_{\texttt{rand}}(K)$ as the class of all algorithms that follow the aforementioned protocol with a batch size of $K$ queries per round.
	
	\begin{definition}[{Zero-Respecting Algorithm~\citep[Definition~1]{arjevani2022lower,pmlr-v125-arjevani20a}}]\label{appdx-def:zero-especting-algo}
		A stochastic $p^{\textnormal{th}}$-order algorithm $\texttt{A}$ is \emph{zero-respecting} if for any function $F \in \mathcal{C}^p(\R^d, \R)$ and any $p^{\textnormal{th}}$-order oracle $\texttt{O}^{p, K}_F$, the iterates $\{ \mathbf{x}_{\texttt{A}[\texttt{O}^{p, K}_F]}^{(i)} \}_{i \ge 1}$ satisfies, for any 
		\[ \supp\left( \mathbf{x}_{\texttt{A}[\texttt{O}^{p, K}_F]}^{(i)} \right) \subseteq \bigcup_{j = 1}^{i - 1} \supp\left(  \texttt{O}_F^{p, K}\left( \mathbf{x}_{\texttt{A}[\texttt{O}^{p, K}_F]}^{(j)}, \xi^{(j)} \right) \right), \]
		
		We define $\mathcal{A}_{\texttt{zr}}(K) \subseteq \mathcal{A}_{\texttt{rand}}(K)$ as the class of all \emph{zero-respecting} algorithms.
	\end{definition}
	In other words, a zero-respecting algorithm cannot modify coordinates where no information is known; that is, its queries at each round have support in the union of the supports of all previous oracle responses.
	
	In the rest of the paper, we drop the $K$ in all our notation for simplicity and because it does not appear in any of the lower bound complexity we derive.
	
	\paragraph{Proof Strategy: Zero-Chains}
	The main motivations for building the hard instance~\eqref{f9d4cc59-a955-4e12-997e-7870c5314204} is the notion of zero-chain, which in some sense, are functions for which it is \say{hard} to uncover new, nonzero, coordinates.
	
	\begin{definition}[{Deterministic $p^{\textnormal{th}}$-order Zero-Chain~\citep[Definition~3]{carmon2020lower}}]\label{appdx-def:deterministic-zero-chain}
		Given an integer $p \ge 1$ and a function $F \in \mathcal{C}^p(\R^d, \R)$, we say that $F$ is a \emph{$p^{\textnormal{th}}$-order zero-chain} if, for every $x \in \R^d$ and every $i \in [d]$ we have
		\[ \supp(x) \subseteq \ens{1, 2, \ldots, i - 1} \,\, \text{ implies } \,\, \bigcup_{r = 1}^p \supp(\nabla^r F(x)) \subseteq \ens{1, 2, \ldots, i}. \numberthis\label{15294226-4f58-484e-b80c-22c08bec83d5-0} \]
		
		We say that $F$ is a \emph{zero-chain} if it is a $p^{\textnormal{th}}$-order zero-chain for every integer $p \ge 1$.
	\end{definition}
	In other word,~\Cref{appdx-def:deterministic-zero-chain} tells us that given $x \in \R^d$, we can \say{discover} at most one new coordinate when accessing the gradient of $F$ at $x$ (or any high-order derivatives of $F$). When dealing with stochastic estimators instead, we can extend the previous~\Cref{appdx-def:deterministic-zero-chain} as follows:
	\begin{definition}[{Probability-$\theta$ $p^{\textnormal{th}}$-order Zero-Chain~\citep[Definition~2]{pmlr-v125-arjevani20a}}]\label{appdx-def:probability-zero-chain}
		Given an integer $p \ge 1$, $\theta \in \intof{0}{1}$, a function $F \in \mathcal{C}^p(\R^d, \R)$ and derivative estimators $\nabla f(x, \xi)$, $\ldots$, $\nabla^p f(x, \xi)$ of $F$, we say that $F$ is a probability-$\theta$ \emph{$p^{\textnormal{th}}$-order zero-chain} if we have
		\[\Proba{\exists x_0 \mid \prog_0\left( \nabla f(x_0, \xi), \ldots, \nabla^p f(x_0, \xi) \right) > \prog_{\alpha}(x_0) + 1} = 0, \numberthis\label{15294226-4f58-484e-b80c-22c08bec83d5-1} \]
		and
		\[ \Proba{\exists x_0 \mid \prog_0\left( \nabla f(x_0, \xi), \ldots, \nabla^p f(x_0, \xi) \right) = \prog_{\alpha}(x_0) + 1} \le \theta, \numberthis\label{15294226-4f58-484e-b80c-22c08bec83d5-2} \]
		
		We say that $F$ is a \emph{probability-$\theta$ zero-chain} if it is a probability-$\theta$ $p^{\textnormal{th}}$-order zero-chain for every integer $p \ge 1$.
	\end{definition}
	In the above~\Cref{appdx-def:probability-zero-chain}, condition~\eqref{15294226-4f58-484e-b80c-22c08bec83d5-1} is the analogue of condition~\eqref{15294226-4f58-484e-b80c-22c08bec83d5-0} when we have only access to noisy derivatives estimators of $F$. In addition, condition~\eqref{15294226-4f58-484e-b80c-22c08bec83d5-2} tells us that we have a \say{small} chance to discover a new coordinate upon querying the oracle, i.e., the added noise behaves \say{adversarially} and can slow down the discovery of new coordinates (in expectation).

	An important properties of probability-$\theta$ zero-chain is the following lemma.
	\begin{lemma}[{\citet[Lemma~16]{pmlr-v125-arjevani20a}}]\label{appdx-lem:zero-chain-progress}
		Let $p \ge 1$, $\theta \in \intof{0}{1}$, a function $F \in \mathcal{C}^p(\R^d, \R)$ and unbiased derivative estimators $\nabla f(x, \xi)$, $\ldots$, $\nabla^p f(x, \xi)$ of $F$ which form a probability-$\theta$ zero-chain and let ${\normalfont\texttt{O}^p_F}$ be an oracle such that ${\normalfont\texttt{O}^p_F}(x, \xi) = \left( F(x), \nabla f(x, \xi), \ldots, \nabla^p f(x, \xi) \right)$.
		
		Let $\{ \mathbf{x}_{{\normalfont\texttt{A}}[{\normalfont\texttt{O}^p_F}]}^{(i)} \}_{i \ge 1}$ be the queries produced by any zero-respecting algorithm ${\normalfont\texttt{A}} \in \mathcal{A}_{\normalfont\texttt{zr}}$ interacting with ${\normalfont\texttt{O}^p_F}$. Then, with probability at least $1 - \delta$, we have
		\[ \prog_0\left( \mathbf{x}_{{\normalfont\texttt{A}}[{\normalfont\texttt{O}^p_F}]}^{(t)} \right) < T \,\, \text{ for all } \,\, t \le \frac{T - \log(\frac{1}{\delta})}{2 \theta}. \]
	\end{lemma}
	The proof can be found in~\citet[Lemma~16]{pmlr-v125-arjevani20a} and~\citet[Lemma~1]{arjevani2020tight}.
	
	\subsubsection{Complexity Measures}
	
	As in~\citet{arjevani2022lower,pmlr-v125-arjevani20a}, we develop lower bounds on the \emph{distributional complexity} for finding an $\eps$--stationary point, which in turn, implies lower bounds on the \emph{minimax complexity} for finding such stationary points. Formally, following~\cite{arjevani2022lower}
	\[ \mathfrak{m}^{\texttt{zr}}_{\eps}(K, \Delta\ldots) \eqdef \sup_{F \in \mathcal{F}(\Delta\ldots)} \sup_{(\P_{\xi}, \texttt{O}_F) \in \mathcal{O}(F\ldots)} \inf_{\texttt{A} \in \mathcal{A}_{\texttt{zr}}} \inf \enstq{T \ge 1}{\ExpSub{\texttt{O}_F, \texttt{A}}{\norm{\nabla F\left( \mathbf{x}_{{\normalfont\texttt{A}}[{\normalfont\texttt{O}_F}]}^{(T)} \right)}} \le \eps}, \]
	where the expectation is taken over the randomness in the oracle $\texttt{O}_F$ and in the algorithm $\texttt{A} \in \mathcal{A}_{\texttt{zr}}$, if any.
	
	As in the definition of the oracle class, complexity measures may depend on various parameters depending on the assumptions considered. The convention is to list all involved parameters (in an arbitrary order), and there is no ambiguity in doing so.
	
	\subsection{The Worst-Case Function}\label{appdx-subsec:hard-instance}
	
	In this section, we recall the \say{worst-case} function introduced in~\citet{carmon2020lower,arjevani2022lower} and which is used to prove our lower bounds. This function (or some variations, depending on the targeted class of functions) is at the core of many lower bounds in stochastic nonvonvex optimization~\citep{arjevani2022lower,NEURIPS2024_dd850be1,10.5555/3737916.3739635,islamov2025safeef,sun2025improved,sadiev2025second}. Formally, given an integer $T \ge 1$ which denotes the dimension in which the hard instance $F_T \colon \R^T \to \R$ lies, we define
	\[ F_T \colon x \mapsto -\Psi(1) \Phi(x_1) + \sum_{i = 2}^T \left[ \Psi(-x_{i - 1}) \Phi(-x_i) - \Psi(x_{i - 1}) \Phi(x_i) \right], \numberthis\label{f9d4cc59-a955-4e12-997e-7870c5314204} \]
	where the inner components $\Psi$ and $\Phi$ are defined, for any $t \in \R$, as
	\[ \Psi(t) \eqdef \begin{cases}
		0, & \text{if $t \le \frac{1}{2}$;} \\
		\exp\left( 1 - \frac{1}{(2 t - 1)^2} \right), & \text{if $t > \frac{1}{2}$;}
	\end{cases} \quad \text{ and } \quad \Phi(t) \eqdef \sqrt{e} \int_{-\infty}^t e^{-\frac{s^2}{2}} \odif{s}. \]
	The particular design~\eqref{f9d4cc59-a955-4e12-997e-7870c5314204} of $F_T$ is such that it is a (deterministic) zero-chain (see~\Cref{appdx-def:deterministic-zero-chain}) and the gradients of $F_T$ are large unless all coordinates are large (i.e., $\prog_1(x) \ge T$); these two keys properties are recalled below ($5.$ and $6.$). Several other properties of the function $F_T$ are stated in the next lemma.
	\begin{lemma}[Properties of the Hard Instance~\citep{carmon2020lower,arjevani2022lower}]\label{appdx-lem:properties-hard-instance}
		For any integer $T \ge 1$, the function $F_T$ satisfies:
		\begin{enumerate}
			\item $F_T(0) - \inf_{x \in \R^T} F_T(x) \le \Delta_0 T$ where $\Delta_0 \eqdef 12$,
			
			\item The gradient of $F_T$ is $\ell_1$--Lipschitz continuous with $\ell_1 \eqdef 152$,
			
			\item For all $x \in \R^T$, we have $\norm{\nabla F_T(x)}_{\infty} \le \gamma_{\infty}$ where $\gamma_{\infty} \eqdef 23$. Notably, this shows that $\norm{\nabla F_T(x)} \le 23 \sqrt{T}$,
			
			\item There exists an universal constant $0 < c < +\infty$ such that for every integer $p \ge 1$, the $p$--th order derivatives of $F_T$ are $\ell_p$--Lipschitz continuous for some $\ell_p \le \exp\left( \frac{5}{2} p \log p + c p \right)$,
			
			\item For all $x \in \R^T$ we have $\prog_0(\nabla F_T(x)) \le \prog_{\frac{1}{2}}(x) + 1$,
			
			\item For all $x \in \R^T$, if $\prog_0(x) < T$ then $\norm{\nabla F_T(x)} > 1$.
		\end{enumerate}
	\end{lemma}
	In the proofs of our lower bounds, we mostly focus on properties $1.$, $2.$, $3.$ and $4.$ which, after a specific rescaling, are enough to ensure the hard instance $F_T$ belongs to the right class of functions. The two remaining properties $5.$ and $6.$ are not explicitly used in this paper but their role, however, should not be underestimated: they guarantee that the inner mechanics of the hard instance of~\citet{carmon2020lower,arjevani2022lower} remain the same, thereby ensuring that any zero-respecting algorithm on $F_T$ behaves as intended and that their lower-bound strategy carries over.
	
	\subsection{The Stochastic Gradient and Hessian Estimator}
	
	We recall below from~\citet{arjevani2022lower,pmlr-v125-arjevani20a} the gradient and Hessian estimators used:
	\begin{definition}[Gradient Estimator]\label{appdx-def:gradient-estimator}
		Given an integer $T \ge 1$ and $\theta \in \intof{0}{1}$, the gradient estimator $g_T \colon \R^T \times \mathcal{Z} \to \R^T$ of the hard instance $F_T$ in~\eqref{f9d4cc59-a955-4e12-997e-7870c5314204} is defined coordinate-wise as
		\[ \left[ g_T(x, \xi) \right]_i \eqdef \left[ \nabla F_T(x) \right]_i \cdot \left( 1 + \mathbb{I}\!\ens{i > \prog_{\frac{1}{4}}(x)} \left( \frac{\xi}{\theta} - 1 \right) \right), \]
		where $i \in [T]$, $x \in \R^T$ and $\xi \sim \Ber(\theta)$, i.e., we take as distribution $\P_{\xi} = \Ber(\theta)$ and measurable set $\mathcal{Z} = \ens{0, 1}$.
	\end{definition}
	
	\begin{definition}[Hessian Estimator]\label{appdx-def:hessian-estimator}
		Given an integer $T \ge 1$ and $\theta \in \intof{0}{1}$, the Hessian estimator $\nabla g_T \colon \R^T \times \mathcal{Z} \to \R^T$ of the hard instance $F_T$ in~\eqref{f9d4cc59-a955-4e12-997e-7870c5314204} is defined row-wise as
		\[ \left[ \nabla g_T(x, \xi) \right]_{i, \cdot} \eqdef \left( 1 + \mathbb{I}\!\ens{i > \prog_{\frac{1}{4}}(x)} \left( \frac{\xi}{\theta} - 1 \right) \right) \cdot \left[ \nabla^2 F_T(x) \right]_{i, \cdot}, \]
		where $i \in [T]$, $x \in \R^T$ and $\xi \sim \Ber(\theta)$ (as before).
	\end{definition}
	Note that, for all $x \in \R^T$ and all $i > \prog_{\frac{1}{4}}(x) + 1$ we have $\left[ \nabla F_T(x) \right]_i = 0$ by~\Cref{appdx-lem:properties-hard-instance} (property $5.$) so, only the specific coordinate at $i = \prog_{\frac{1}{4}}(x) + 1$ is noisy.
	
	In the case of~\Cref{ass:mean-squared-smoothness,ass:mean-squared-smoothness-2}, i.e., additional regularity conditions on the gradient estimator, we define, as in~\citet{arjevani2022lower} the \say{smoothed} indicator $\Theta_i$ of $\mathbb{I}\!\ens{i > \prog_{\frac{1}{4}}(\cdot)}$ as:
	\[ \Theta_i \colon x \mapsto \Gamma\left( 1 - \left( \sum_{k = i}^T \Gamma\left(\abs{x_k}\right)^2 \right)^{\frac{1}{2}} \right) = \Gamma\left( 1 - \norm{\left( \Gamma(\abs{x_i}), \ldots, \Gamma(\abs{x_T}) \right)} \right), \]
	where $\Gamma \colon \R \to \R$ is any $\mathcal{C}^{\infty}$, non-decreasing Lipschitz continuous function with $\Gamma(t) = 0$ for all $t \le \frac{1}{4}$ and $\Gamma(t) = 1$ for all $t \ge \frac{1}{2}$. For instance, we can take, as in~\citet{arjevani2022lower}
	\[ \Gamma(t) \eqdef \frac{\int_{\frac{1}{4}}^t \Lambda(s) \odif{s}}{\int_{\frac{1}{4}}^{\frac{1}{2}} \Lambda(s) \odif{s}}, \,\, \text{ where } \,\, \Lambda(t) \eqdef \begin{cases}
		0, & \text{if $t \le \frac{1}{4}$ or $t \ge \frac{1}{2}$;} \\
		\exp\left( -\frac{1}{100 \left( t - \frac{1}{4} \right) \left( \frac{1}{2} - t \right)} \right), & \text{if $\frac{1}{4} < t < \frac{1}{2}$;}
	\end{cases} \]
	for any $t \in \R$. Notably, the smoothed indicator $\Theta_i$ satisfies for all $x \in \R^T$
	\[ \mathbb{I}\!\ens{i > \prog_{\frac{1}{4}}(x)} \le \Theta_i(x) \le \mathbb{I}\!\ens{i > \prog_{\frac{1}{2}}(x)}, \]
	hence
	\[ \Theta_i(x) = \begin{cases}
		1, & \text{for all $i > \prog_{\frac{1}{4}}(x)$;} \\
		0, & \text{for all $i \le \prog_{\frac{1}{2}}(x)$.}
	\end{cases} \]
	
	Overall, this gives the following \say{smoothed} gradient estimator:
	\begin{definition}[\protect\say{Smoothed} Gradient Estimator]\label{appdx-def:smoothed-gradient-estimator}
		Given an integer $T \ge 1$ and $\theta \in \intof{0}{1}$, the gradient estimator $g_T \colon \R^T \times \mathcal{Z} \to \R^T$ of the hard instance $F_T$ in~\eqref{f9d4cc59-a955-4e12-997e-7870c5314204} is defined coordinate-wise as
		\[ \left[ g_T(x, \xi) \right]_i \eqdef \left[ \nabla F_T(x) \right]_i \cdot \left( 1 + \Theta_i(x) \left( \frac{\xi}{\theta} - 1 \right) \right), \]
		where $i \in [T]$, $x \in \R^T$ and $\xi \sim \Ber(\theta)$ (as before).
	\end{definition}
	
	The gradient and Hessian estimators defined above are all probability-$\theta$ zero-chain.
	
	\newpage
	\section{Proofs of the Lower Bounds}\label{appdx-sec:proofs-lower-bound}
	
	\subsection{A Lower Bound in the \textit{Global Stochastic Model} Under Heavy-Tail Noise}
	
	In this part, we extend~\citet[Lemma~11]{arjevani2022lower} in the setting where we only assume the variance of the gradient estimator to have bounded $p$--th moment for some $p \in \intof{1}{2}$. This result holds in any dimension $d \ge 1$.
	
	\begin{lemma}\label{appdx-lem:lower-bound-global-stichastic-model-mean-squared-smoothness-p-BCM}
		Under~\Cref{ass:lower-boundedness,ass:p-bounded-central-moment-gradient,ass:mean-squared-smoothness} and as long as $0 < \eps \le \frac{1}{8} \sqrt{\bar{L} \Delta}$, the number of samples required to obtain an $\eps$--stationary point in the global stochastic model defined above is
		\[ \Omega(1) \cdot \left( \frac{\sigma_1}{\eps} \right)^{\frac{p}{p - 1}}, \]
		where $p \in \intof{1}{2}$.
	\end{lemma}
	
	In particular, this lower bound does not depend on the exponent $q$ from~\Cref{ass:mean-squared-smoothness}.
	
	\begin{remark}\label{appdx-rem:improve-lower-bound-global-stichastic-model-mean-squared-smoothness-p-BCM}
		More precisely,~\Cref{appdx-lem:lower-bound-global-stichastic-model-mean-squared-smoothness-p-BCM} shows that
		\[ \mathfrak{m}^{\normalfont\texttt{zr}}_{\eps} \left( K, \bar{L}, \Delta, \sigma_1^p \right) \ge \Omega(1) \cdot \max\ens{1, \left( \frac{\sigma_1}{\eps} \right)^{\frac{p}{p - 1}}}. \numberthis\label{df3dc971-577b-4af1-818f-ae694fee0100} \]
	\end{remark}
	
	\begin{remark}\label{appdx-rem:improve-lower-bound-global-stichastic-model-mean-squared-smoothness-2-p-BCM}
		It is also worth mentioning that the same lower bound holds under~\Cref{ass:lower-boundedness,ass:L-lipschitz-gradients,ass:p-bounded-central-moment-gradient,ass:mean-squared-smoothness-2} since, by our choice of function~\eqref{af23ded8-f6c2-4361-8fc4-1e72f4395d6d} for any $s \in \ens{-1, 1}$ and any $x, y \in \R^d$ we have
		\begin{alignat*}{2}
			&\ExpSub{\xi \sim \P^s}{\norm{\left[ \nabla f_d^{\star}(x, \xi) - \nabla f_d^{\star}(y, \xi) \right] - \left[ \nabla F_{d, s}^{\star}(x) - \nabla F_{d, s}^{\star}(y) \right]}^q} \\
			&\qquad\qquad\begin{aligned}[t]
				\oversetlab{\eqref{70110a40-c2b6-4b8f-b790-e09561a73316}}&{=} \beta^q \ExpSub{\xi \sim \P^s}{\norm{\left[ \nabla f_d(\beta x, \xi) - \nabla f_d(\beta y, \xi) \right] - \left[ \nabla F_{d, s}(\beta x) - \nabla F_{d, s}(\beta y) \right]}^q} \\
				\oversetlab{\eqref{2130180a-a3a9-4040-8a04-214e99e99948}}&{=} \beta^q \ExpSub{\xi \sim \P^s}{\norm{\left[ L (\beta x - \boldsymbol{\xi}) - L(\beta y - \boldsymbol{\xi}) \right] - \left[ L (\beta x - \theta_s) - L (\beta y - \theta_s) \right]}^q} \\
				& = 0,
			\end{aligned}
		\end{alignat*}
		hence~\Cref{ass:mean-squared-smoothness-2} is satisfied. Moreover, taking $L = \frac{2 \Delta}{r^2}$ is enough to ensure that $F_{d, s}^{\star}$ has $\Delta$--bounded sub-optimality (see~\eqref{f38cccea-2eb2-49b1-910d-1eb1671cb5ff}). On the other hand, for any $x, y \in \R^d$ we have
		\begin{alignat*}{2}
			\norm{\nabla F_{d, s}^{\star}(x) - \nabla F_{d, s}^{\star}(y)} \oversetlab{\eqref{70110a40-c2b6-4b8f-b790-e09561a73316}}&{=} \beta \norm{\nabla F_{d, s}(\beta x) - \nabla F_{d, s}(\beta y)} \\
			& = \beta^2 L \norm{x - y}, \numberthis\label{1446c17c-7865-4a08-be55-9a576a51828f}
		\end{alignat*}
		and it's enough to take $0 < \beta \le \sqrt{\frac{L_1}{L}}$ to ensure the function $F_{d, s}^{\star}$ has $L_1$--Lipschitz gradients. The rest of the proof (\textbf{steps 3}, \textbf{4} and \textbf{5}) is the same, using $L_1$ instead of $\bar{L}$.
	\end{remark}
	
	\begin{remark}\label{appdx-rem:improve-lower-bound-global-stichastic-model-mean-squared-smoothness-2-hessian}
		As a last remark before proving~\Cref{appdx-lem:lower-bound-global-stichastic-model-mean-squared-smoothness-p-BCM}, it is also worth noting that the stochastic function $f_d \colon \R^d \to \R$ used to prove the lower bound~\eqref{df3dc971-577b-4af1-818f-ae694fee0100} still holds when one has access to high-order information (e.g., the Hessian), that is, under~\Cref{ass:lower-boundedness,ass:L-lipschitz-gradients,ass:p-bounded-central-moment-gradient,ass:q-bounded-central-moment-hessian}, because by construction $f_d$ is a quadratic function so, for any $s \in \ens{-1, 1}$ and any $x, y \in \R^d$, we have
		\begin{alignat*}{2}
			\ExpSub{\xi \sim \P^s}{\normop{\nabla^2 f^{\star}_d(x, \xi) - \nabla^2 F^{\star}_{d, s}(x)}^q} & = \beta^{2 q} \ExpSub{\xi \sim \P^s}{\normop{\nabla^2 f_d(\beta x, \xi) - \nabla^2 F_{d, s}(\beta x)}^q} \\
			& = \beta^{2 q} \ExpSub{\xi \sim \P^s}{\normop{L I_d - L I_d}^q} \\
			& = 0,
		\end{alignat*}
		where $I_d \in \R^{d \times d}$ is the identity matrix. Hence,~\Cref{ass:q-bounded-central-moment-hessian} is satisfied. Additionally, since $\nabla^2 F^{\star}_{d, s}(x) = L \beta^2 I_d$ for any $x \in \R^d$ then the Hessian of $F^{\star}_{d, s}$ is $L_2$--Lipschitz continuous for any constant $L_2 \ge 0$. Combining this observation with the arguments in the previous remark (see~\eqref{1446c17c-7865-4a08-be55-9a576a51828f}), we obtain that $F^{\star}_{d, s} \in \mathcal{F}(\Delta, L_1, L_2)$ for any $d \ge 1$ and any $s \in \ens{-1, 1}$, as long as we choose
		\[ L = \frac{2 \Delta}{r^2} \,\, \text{ and } \,\, 0 < \beta \le \sqrt{\frac{L_1}{L}}. \]
		The rest of the proof (\textbf{steps 3}, \textbf{4} and \textbf{5}) is the same.
	\end{remark}
	
	\begin{proof}
		Let the accuracy parameter $\eps > 0$, initial sub-optimality $\Delta \ge 0$, the mean-squared smoothness parameter $\bar{L}$, and the variance parameter $\sigma_1 \ge 0$ be fixed, and $0 < L \le \bar{L}$ to be specified late. Our proof follows the same step as in~\citet{arjevani2020tight} except that, taking inspiration from~\cite{hübler2025gradientclippingnormalizationheavy} and the gradient oracle from~\cite{arjevani2022lower} as recalled in~\Cref{appdx-def:gradient-estimator}, we take two Bernoulli distributions instead of normal distributions.
		
		Let us fix $d \ge 1$ an integer and consider the following family of functions $f_d \colon \R^d \times \ens{-1, 0, 1} \to \R$ defined as
		\[ f_d(x, \xi) \eqdef \frac{L}{2} \left( \sqnorm{x} - 2 x_1 \xi + r^2 \right), \numberthis\label{af23ded8-f6c2-4361-8fc4-1e72f4395d6d} \]
		where $r \in \intof{0}{1}$ is a fixed parameter (to be specified later), $(x, \xi) \in \R^d \times \ens{-1, 0, 1}$ and $x = (x_1, \ldots, x_d)$ are its coordinates. Then, for $\beta > 0$, we rescale the function $f_d$ as
		\[ f_d^{\star} (x, \xi) \colon x \mapsto f_d(\beta x, \xi). \numberthis\label{70110a40-c2b6-4b8f-b790-e09561a73316} \]
		
		Next, we define the two probability distributions $\P^1 = \Ber(r)$ and $\P^{-1} = -\Ber(r)$ (whose support is $\ens{-1, 0}$) and we let $\theta_s \eqdef (r s, 0, \ldots, 0) \in \R^d$ for $s \in \ens{-1, 1}$. Additionally, for any $s \in \ens{-1, 1}$, we define the function $F_{d, s}^{\star} \colon \R^d \to \R$ for all $x \in \R^d$ as
		\[ F_{d, s}^{\star}(x) \eqdef \ExpSub{\xi \sim \P^s}{f_d^{\star}\left(x, \xi \right)} \oversetrel{rel:e390b969-f5f5-40f5-bbfa-ccb12b9adbf4}{=} \frac{L}{2} \left( \sqnorm{\beta x} - 2 \beta x_1 r s + r^2 \right) = \frac{L}{2} \sqnorm{\beta x - \theta_s} \reqdef F_{d, s}( \beta x), \]
		where~\relref{rel:e390b969-f5f5-40f5-bbfa-ccb12b9adbf4} follows from the fact that, when $s = 1$ then $\xi \sim \Ber(r)$ and $\E{\xi} = r = r s$, while when $s = -1$ then $\xi \sim -\Ber(r)$ and $\E{\xi} = -r = r s$.
		
		\paragraph{Step 1:} \textit{Ensuring $F_{d, s}^{\star}$ has $\Delta$--bounded initial sub-optimality (\Cref{ass:lower-boundedness})}.
		
		To ensure $F_{d, s}^{\star}$ satisfies~\Cref{ass:lower-boundedness}, let us compute the initial sub-optimality. Assuming the starting point is $x^0 = 0$, we have
		\[ F^{\star}_{d, s}(0) - \inf_{x \in \R^d} F^{\star}_{d, s}(x) \oversetlab{\eqref{70110a40-c2b6-4b8f-b790-e09561a73316}}{=} F_{d, s}(0) = \frac{L r^2}{2}, \numberthis\label{f38cccea-2eb2-49b1-910d-1eb1671cb5ff} \]
		thus, it's enough to take $L = \frac{2 \Delta}{r^2} < +\infty$ (since $r > 0)$ so as to ensure the function $F_{d, s}^{\star}$ has $\Delta$--bounded initial sub-optimality. In the next step, we will show that $F_{d, s}^{\star}$ also has $\bar{L}$--Lipschitz gradients, as a consequence of the mean-squared smoothness property.
		
		\paragraph{Step 2:} \textit{The Oracle Class and~\Cref{ass:p-bounded-central-moment-gradient,ass:mean-squared-smoothness}}.
		
		Now, let us compute the $p$--th central moment, for any $x \in \R^d$ and any $s \in \ens{-1, 1}$ we have
		\begin{eqnarray}{2}
			\ExpSub{\xi \sim \P^s}{\norm{\nabla f_d^{\star}(x, \xi) - \nabla F_{d, s}^{\star}(x)}^p} &\oversetlab{\eqref{70110a40-c2b6-4b8f-b790-e09561a73316}}{=}& \beta^p \, \ExpSub{\xi \sim \P^s}{\norm{\nabla f_d(\beta x, \xi) - \nabla F_{d, s}(\beta x)}^p} \notag\\
			&\oversetrel{rel:e204c2ed-159a-493d-8b6b-09add6f9cd4a}{=}& \beta^p \, \ExpSub{\xi \sim \P^s}{\norm{L \left( \beta x - \boldsymbol{\xi} \right) - L \left( \beta x - \theta_s \right)}^p} \notag\\
			& =& \left( L \beta \right)^p \ExpSub{\xi \sim \P^s}{\abs{\xi - r s}^p} \notag\\
			& =& \left( L \beta r \right)^p \ExpSub{\xi \sim \P^s}{\abs{\frac{\xi}{r} - s}^p} \notag\\
			&\oversetrel{rel:2ee486e0-265e-41db-8749-66406c614680}{=}& \left( L \beta r \right)^p \left( (1 - r) \abs{s}^p + r \abs{\frac{s}{r} - s}^p \right) \notag\\
			& =& \left( L \beta r \right)^p \left( (1 - r) + r \abs{\frac{1}{r} - 1}^p \right) \notag\\
			& =& \left( L \beta r \right)^p (1 - r) \left( 1 + \left( \frac{1 - r}{r} \right)^{p - 1} \right) \notag \\
			&\oversetrel{rel:1ab9605a-dfe8-49bb-82e2-5fbe63923a0a}{=}& \left( L \beta \right)^p r (1 - r) \left( r^{p - 1} + \left( 1 - r \right)^{p - 1} \right) \notag\\
			& \le& 2 \left( L \beta \right)^p r (1 - r), \numberthis\label{279a6448-b491-4172-8fb0-c0db4abd8e3f}
		\end{eqnarray}
		where in~\relref{rel:e204c2ed-159a-493d-8b6b-09add6f9cd4a} we let $\boldsymbol{\xi} = (\xi, 0, \ldots, 0) \in \R^d$ and the gradient of $f_d$ and $F_{d, s}$ are given, for any $x \in \R^d$ by
		\[ \nabla f_d(x, \xi) = L (x - \boldsymbol{\xi}) \,\, \text{ and } \,\, \nabla F_{d, s}(x) = L (x - \theta_s), \numberthis\label{2130180a-a3a9-4040-8a04-214e99e99948} \]
		notably, $\ExpSub{\xi \sim \P^s}{\nabla f_d^{\star}(x, \xi)} = \nabla F_{d, s}^{\star} (x)$. In~\relref{rel:2ee486e0-265e-41db-8749-66406c614680} we use the fact that $\abs{s} = 1$ while in~\relref{rel:1ab9605a-dfe8-49bb-82e2-5fbe63923a0a} we use both $p > 1$ and $r \in \intof{0}{1}$ to bound $r^{p - 1} + \left( 1 - r \right)^{p - 1}$ by $2$. Hence, following~\eqref{279a6448-b491-4172-8fb0-c0db4abd8e3f} we need to guarantee
		\[ 2 \left( L \beta \right)^p r ( 1 - r) \le \sigma_1^p \,\, \text{ so } \,\, 2 ^{\frac{1}{p}} L \beta \left( r ( 1 - r) \right)^{\frac{1}{p}} \le \sigma_1. \numberthis\label{d98b1956-8d19-4272-af9a-9a2d85b182c3} \]
		
		Now, computing the mean-squared smoothness constant, we have, for $q \in \intof{1}{2}$, $s \in \ens{-1, 1}$ and any $x, y \in \R^d$
		\begin{alignat*}{2}
			\ExpSub{\xi \sim \P^s}{\norm{\nabla f_d^{\star}(x, \xi) - \nabla f_d^{\star}(y, \xi)}^q} \oversetlab{\eqref{70110a40-c2b6-4b8f-b790-e09561a73316}}&{=} \beta^q \, \ExpSub{\xi \sim \P^s}{\norm{\nabla f_d(\beta x, \xi) - \nabla f_d(\beta y, \xi)}^q} \\
			\oversetlab{\eqref{2130180a-a3a9-4040-8a04-214e99e99948}}&{=} \left( L \beta \right)^q \, \ExpSub{\xi \sim \P^s}{\norm{\beta x - \beta y}^q} \\
			& = \left( L \beta^2 \right)^q \norm{x - y}^q,
		\end{alignat*}
		and it suffices to ensure
		\[ \left( L \beta^2 \right)^q \le \bar{L}^q, \,\, \text{ that is, } \,\, L \beta^2 \le \bar{L}, \numberthis\label{63deeb98-3d9a-4bca-9cc0-9acaba35425e} \]
		so as to satisfy~\Cref{ass:mean-squared-smoothness}.
		
		\paragraph{Step 3:} \textit{Choice of $\beta$, $r$}.
		
		It remains to choose $\beta > 0$, $r \in \intof{0}{1}$ (universal constant) to satisfy the two inequalities~\eqref{d98b1956-8d19-4272-af9a-9a2d85b182c3} and~\eqref{63deeb98-3d9a-4bca-9cc0-9acaba35425e}. Additionally, let $c' > 0$, to be fixed later, be an universal constant such that $0 < \eps \le c' \sqrt{\bar{L} \Delta}$. From inequality~\eqref{63deeb98-3d9a-4bca-9cc0-9acaba35425e} we have
		\[ \beta \le \sqrt{\frac{\bar{L}}{L}} = \frac{1}{L} \sqrt{\bar{L} L} = \frac{\sqrt{2 \bar{L} \Delta}}{L r}, \]
		hence for any constant $c' > 0$, taking $\beta = \frac{\eps \sqrt{2}}{c' L r} > 0$ gives
		\[ \beta = \frac{\eps \sqrt{2}}{c' L r} \le \frac{c' \sqrt{\bar{L} \Delta} \sqrt{2}}{c' L r} = \frac{\sqrt{2 \bar{L} \Delta}}{L r}, \]
		as desired. Plugging back this value in the inequality~\eqref{d98b1956-8d19-4272-af9a-9a2d85b182c3} we need to have
		\[ \frac{2^{\frac{1}{p}} \sqrt{2}}{c'} \eps r^{\frac{1}{p} - 1} (1 - r)^{\frac{1}{p}} \le \sigma_1, \,\, \text{ i.e., }\,\, \frac{2^{\frac{1}{p}} \sqrt{2}}{c'} \left( \frac{\eps}{\sigma_1} \right) (1 - r)^{\frac{1}{p}} \le r^{\frac{p - 1}{p}}, \]
		which is equivalent to
		\[ C^{\frac{p}{p - 1}} \left( \frac{\eps}{\sigma_1} \right)^{\frac{p}{p - 1}} (1 - r)^{\frac{1}{p - 1}} \le r, \]
		where we set $C \eqdef \frac{2^{\frac{1}{p}} \sqrt{2}}{c'}$. Thus, as $1 - r \le 1$, it is enough to take
		\[ r = \min\ens{1, \left( \frac{C \eps}{\sigma_1} \right)^{\frac{p}{p - 1}}} > 0, \]
		and now all the inequalities are satisfied so do~\Cref{ass:lower-boundedness,ass:p-bounded-central-moment-gradient,ass:mean-squared-smoothness} hold.
		
		\paragraph{Step 4:} \textit{Transforming the Optimization Problem into a Function Identification Problem}.
		
		We now continue to follow the proof of~\cite{arjevani2022lower}. First, let us randomized the selection of the instances $\{F_{d, s}^{\star}\}_{s \in \ens{-1, 1}}$ by drawing $s$ uniformly in the set $\ens{-1, 1}$. Hence, we let $S \sim \mathcal{U}\left(\ens{-1, 1}\right)$ and consider any algorithm \texttt{A} that takes as input the iid samples $\xi_1, \ldots, \xi_T \sim \P^S$, where $T$ is the number of queries, and which returns the (random) iterate $\widehat{x} \in \R^d$. We now bound the expected norm of the gradient at $\widehat{x}$. To do so, let us define the (random) quantity $\widehat{S} \eqdef \argmin_{s \in \ens{-1, 1}} \|\nabla F^{\star}_{d, s} \left( \widehat{x} \right)\!\|$, breaking ties arbitrarily. Then if $S \neq \widehat{S}$ we have, by definition of $\widehat{S}$, 
		\[ \norm{\nabla F^{\star}_{d, S}\left( \widehat{x} \right)} \ge \norm{\nabla F^{\star}_{d, \widehat{S}}\left( \widehat{x} \right)}, \]
		hence
		\begin{alignat*}{2}
			2 \E{\norm{\nabla F^{\star}_{d, S}\left( \widehat{x} \right)}} & \ge \E{\norm{\nabla F^{\star}_{d, S}\left( \widehat{x} \right)}} + \E{\norm{\nabla F^{\star}_{d, \widehat{S}}\left( \widehat{x} \right)}} \\
			& = \E{\norm{\nabla F^{\star}_{d, S}\left( \widehat{x} \right)} + \norm{\nabla F^{\star}_{d, \widehat{S}}\left( \widehat{x} \right)}} \\
			\oversetrel{rel:7707f412-0c19-4e37-98b0-62d39ea4ba31}&{\ge} \inf_{x \in \R^d} \left( \norm{\nabla F^{\star}_{d, 1} (x)} + \norm{\nabla F^{\star}_{d, -1} (x)} \right) \\
			\oversetlab{\eqref{2130180a-a3a9-4040-8a04-214e99e99948}}&{=} L \beta \inf_{x \in \R^d} \left( \norm{\beta x - \theta_1} + \norm{\beta x - \theta_{-1}} \right) \\
			\oversetrel{rel:fa79275c-d6ff-48bd-8566-d261b2edaab2}&{\ge} L \beta \norm{\theta_1 - \theta_{-1}} \\
			& = 2 L \beta r, \numberthis\label{174ddb60-9406-45e1-b077-a317743c46a2}
		\end{alignat*}
		where in~\relref{rel:7707f412-0c19-4e37-98b0-62d39ea4ba31} we use $S \neq \widehat{S}$ and that both belongs to $\ens{-1, 1}$. In~\relref{rel:fa79275c-d6ff-48bd-8566-d261b2edaab2} we use the triangle inequality. Hence,
		\begin{alignat*}{2}
			\E{\norm{\nabla F^{\star}_{d, S} \left( \widehat{x} \right)}} \oversetref{Lem.}{\ref{lem:markov-inequality}}&{\ge} L \beta r \Proba{\norm{\nabla F^{\star}_{d, S}\left( \widehat{x} \right)} \ge L \beta r} \\
			\oversetlab{\eqref{174ddb60-9406-45e1-b077-a317743c46a2}}&{\ge} L \beta r \Proba{\widehat{S} \neq S},
		\end{alignat*}
		since $L \beta r > 0$ and the last inequality follows from the lower bound in~\eqref{174ddb60-9406-45e1-b077-a317743c46a2}, implied by the event $\{ \widehat{S} \neq S \}$.
		
		\paragraph{Step 5:} \textit{Lower Bounding the Misidentification Probability $\Proba{\widehat{S} \neq S}$}.
		
		Next, for $s \in \ens{-1, 1}$ let $\P_T^s = s\Ber^{\otimes T}(r)$ be the law of $\left( \xi_1, \ldots, \xi_T \right)$ conditioned on the event $\ens{S = s}$, then
		\begin{alignat*}{2}
			\Proba{\widehat{S} \neq S} & = 1 - \Proba{\widehat{S} = S} \\
			& = 1 - \left( \Proba{S = 1} \ProbCond{\widehat{S} = S}{S = 1} + \Proba{S = -1} \ProbCond{\widehat{S} = S}{S = -1} \right) \\
			& = 1 - \frac{1}{2} \left( \P_T^1\left( \widehat{S} = 1 \right) + \P_T^{-1} \left( \widehat{S} \neq 1 \right) \right) \\
			& \le 1 - \frac{1}{2} \sup_{A \subseteq \R^T \textnormal{ mesurable}} \left( \P^1_T(A) + \P^{-1}_T(A^{\texttt{c}}) \right) \\
			& = \frac{1}{2} - \frac{1}{2} \sup_{A \subseteq \R^T \textnormal{ mesurable}} \left( \P^1_T(A) - \P^{-1}_T(A) \right) \\
			\oversetrel{rel:d3a8f793-de7b-40f7-8a9f-dd2e89f0ca7b}&{=} \frac{1}{2} \left( 1 - \TVdist{\P_T^1 - \P_T^{-1}} \right),
		\end{alignat*}
		where~\relref{rel:d3a8f793-de7b-40f7-8a9f-dd2e89f0ca7b} follows form the definition of the total variation distance. Then we compute $\TVdist{\P_T^1 - \P_T^{-1}}$, to do so observe that the support of $\P_T^1$ is $\ens{0, 1}^T$ and the support of $\P_T^{-1}$ is $\ens{-1, 0}^T$ thus, if we let $\mathbf{0} \eqdef (0, \ldots, 0)$,
		\begin{alignat*}{2}
			\TVdist{\P_T^1 - \P_T^{-1}} & = \frac{1}{2} \sum_{x \in \ens{-1, 0, 1}^T} \abs{\P_T^1(x) - \P_T^{-1} (x)} \\
			\oversetrel{rel:d24508f1-b202-4372-9da7-f989372bd339}&{=} \frac{1}{2} \sum_{x \in \ens{-1, 0}^T \setminus \ens{\mathbf{0}}} \P_T^{-1}(x) + \frac{1}{2} \sum_{x \in \ens{0, 1}^T \setminus \ens{\mathbf{0}}} \P_T^1(x) \\
			& = \frac{1}{2} \left( 1 - (1 - r)^T \right) + \frac{1}{2} \left( 1 - (1 - r)^T \right) \\
			& = 1 - (1 - r)^T,
		\end{alignat*}
		where~\relref{rel:d24508f1-b202-4372-9da7-f989372bd339} follows from $\P_T^1\left( \mathbf{0} \right) = (1 - r)^T = \P_T^{-1}\left( \mathbf{0} \right)$ along with the fact that $\ens{-1, 0}^T \cap \ens{0, 1}^T = \ens{\mathbf{0}}$. Hence, we obtain
		\[ \Proba{\widehat{S} \neq S} \ge \frac{1}{2} ( 1 - r)^T, \]
		and plugging this bound into~\eqref{174ddb60-9406-45e1-b077-a317743c46a2} gives
		\[ \E{\norm{\nabla F^{\star}_{d, S} \left( \widehat{x} \right)}} \ge \frac{L \beta r}{2} (1 - r)^T. \]
		Also, note that
		\begin{eqnarray*}
			\max\ens{\E{\norm{\nabla F^{\star}_{d, 1} \left( \widehat{x} \right)}}, \E{\norm{\nabla F^{\star}_{d, -1} \left( \widehat{x} \right)}}} &\ge& \frac{1}{2} \left( \E{\norm{\nabla F^{\star}_{d, 1} \left( \widehat{x} \right)}} + \E{\norm{\nabla F^{\star}_{d, -1} \left( \widehat{x} \right)}} \right) \\
			&=& \E{\norm{\nabla F^{\star}_{d, S} \left( \widehat{x} \right)}}.
		\end{eqnarray*}
		
		It remains to lower bound adequately the quantity $\frac{L \beta r}{2} (1 - r)^T$. For this, we consider two cases:
		\begin{itemize}
			\item if $\sigma_1 > C \eps$ then $1 > \frac{C \eps}{\sigma_1}$ so $1 > \left( \frac{C \eps}{\sigma_1} \right)^{\frac{p}{p - 1}}$ thus $r = \left( \frac{C \eps}{\sigma_1} \right)^{\frac{p}{p - 1}}$ so
			\begin{alignat*}{2}
				\frac{L \beta r}{2} (1 - r)^T & = \frac{\eps \sqrt{2}}{2 c'} (1 - r)^T,
			\end{alignat*}
			and, assume we have $T \le \frac{1}{2 r}$ then
			\[ (1 - r)^T \oversetrel{rel:b5b918cc-3d5b-4b94-b68e-094fb2d4acd5}{\ge} 1 - rT \ge \frac{1}{2}, \]
			where~\relref{rel:b5b918cc-3d5b-4b94-b68e-094fb2d4acd5} follows from the Bernoulli's inequality (see~\Cref{lem:bernoulli-inequality}) thus
			\[ \max\ens{\E{\norm{\nabla F^{\star}_{d, 1} \left( \widehat{x} \right)}}, \E{\norm{\nabla F^{\star}_{d, -1} \left( \widehat{x} \right)}}} \ge \frac{\eps \sqrt{2}}{4 c'}, \]
			and it's enough to take $c' = \frac{1}{8}$ to ensure 
			\[ \max\ens{\E{\norm{\nabla F^{\star}_{d, 1} \left( \widehat{x} \right)}}, \E{\norm{\nabla F^{\star}_{d, -1} \left( \widehat{x} \right)}}} \ge 2 \eps \sqrt{2} > 2\eps \]
			So, when $T \le \frac{1}{2 r}$ it is not possible to reach an $\eps$--stationary point on both $F^{\star}_{d, 1}$ and on $F^{\star}_{d, -1}$.
			
			Hence, we deduce that we must have 
			\[ T > \frac{1}{2 r} = \frac{1}{2} \left( \frac{\sigma_1}{C \eps} \right)^{\frac{p}{p - 1}} = \Omega(1) \cdot \left( \frac{\sigma_1}{\eps} \right)^{\frac{p}{p - 1}} \ge \Omega(1), \]
			as desired,

			\item if $\sigma_1 \le C \eps$ then $r = 1$ and
			\[ \left( \frac{\sigma_1}{C \eps} \right)^{\frac{p}{p - 1}} \le 1, \]
			and since for any $s \in \ens{-1, 1}$ we have 
			\[ \norm{\nabla F^{\star}_{d, s} (0)} \oversetlab{\eqref{2130180a-a3a9-4040-8a04-214e99e99948}}{=} L \beta r = \frac{\eps \sqrt{2}}{c'} = 8 \eps \sqrt{2} > 2 \eps, \]
			so at least one query is required to reach an $\eps$--stationary point, i.e., $T \ge 1$ from where 
			\[ T \ge 1 \ge \left( \frac{\sigma_1}{C \eps} \right)^{\frac{p}{p - 1}} = \Omega(1) \cdot \left( \frac{\sigma_1}{\eps} \right)^{\frac{p}{p - 1}}, \]
			which achieves the proof of the lemma.
		\end{itemize}
	\end{proof}
	
	
	\begin{remark}\label{appdx-rem:lower-bound-global-stochastic-model-motivations}
		Let us expand on the proof strategy of the above lemma. First, we start by defining a stochastic function $f \colon (x, \xi) \mapsto \frac{L}{2} \left( \sqnorm{x} - 2 x_1 \xi + r^2 \right)$ and we then provide two probability distributions $\P^-$ and $\P^+$ for the random variable $\xi$ which allows to define two deterministic functions
		\[ F_- \colon x \mapsto \ExpSub{\xi \sim \P_-}{f(x, \xi)} \,\, \text{ and } \,\, F_+ \colon x \mapsto \ExpSub{\xi \sim \P_+}{f(x, \xi)}. \]
		The strategy to establish a lower bound then is to randomized the initial choice of the function by choosing $F_-$ or $F_+$ with probability $\frac{1}{2}$ and lower bound the optimization error on the norm of the gradient by the misidentification error, in other word, how many samples $\xi^{(1)}, \ldots, \xi^{(T)}$ are needed in expectation to distinguish between the two distributions $\P^-$ and $\P^+$.
		
		Contrary to~\citet{arjevani2022lower}, instead of using normal distributions we use two Bernoulli distributions $\Ber(r)$ and $-\Ber(r)$ where, after tuning the parameters so as to satisfy the different assumptions, we obtain
		\[ r = \Theta\left( \min\ens{1, \left( \frac{\eps}{\sigma_1} \right)^{\frac{p}{p - 1}}} \right). \]
		The choice of Bernoulli distributions is actually inspired from the definition of the gradient estimator in the work of~\citet{arjevani2022lower} where a Bernoulli distribution is used to probabilistically hide new coordinates (see~\Cref{appdx-def:gradient-estimator}).
		
		Hence with our Bernoulli distributions, upon querying the oracle, it produces repeated $0$s until it ultimately returns $-1$ or $+1$. In the former situation, we are unable to distinguish between the distributions $\Ber(r)$ and $-\Ber(r)$ since the outcome $0$ happens with equal probability while, in the later case, we can immediately tells which distribution the oracle has chosen initially. As the outcome $\pm 1$ happens with probability $r$, we have to do in expectation $\frac{1}{r} = \Omega(1 + \left( \nicefrac{\sigma_1}{\eps} \right)^{\nicefrac{p}{p - 1}})$ queries to be able to distinguish between $\P^-$ and $\P^+$ and this leads to the claimed lower bound.
	\end{remark}
	
	\begin{remark}
		Combining~\citet[Lemma~11]{arjevani2022lower} and the above lemma, we deduce that under the general $p$--bounded central moment where $p > 1$ is any real number, we have
		\[ \mathfrak{m}^{\normalfont\texttt{zr}}_{\eps} \left( K, \bar{L}, \Delta, \sigma_1^p \right) \ge \Omega(1) \cdot \left( \frac{\sigma_1}{\eps} \right)^{\frac{p}{p - 1} \vee\, 2}, \numberthis\label{4f3b2ba5-4a06-44f5-b589-94a0605dc4ab} \]
		where $\vee$ denotes the maximum between the two exponent, and the hidden constant in $\Omega(1)$ may depends on $p$. Hence
		\begin{itemize}
			\item $p$--bounded central moment for gradient when $p = 1$ brings literally \emph{no information} and we can't find an $\eps$--stationary points in a finite number $T$ of oracle queries,
			
			\item assuming high-order bounded moments ($p > 2$) does not bring additional information than in the bounded variance case ($p = 2$).
		\end{itemize}
	\end{remark}
	
	\begin{proof}[Proof of~\eqref{4f3b2ba5-4a06-44f5-b589-94a0605dc4ab}]
		The proof here applies not only to the setting of~\Cref{appdx-lem:lower-bound-global-stichastic-model-mean-squared-smoothness-p-BCM}, i.e.,~\Cref{ass:lower-boundedness,ass:p-bounded-central-moment-gradient,ass:mean-squared-smoothness} but also to the settings discussed in~\Cref{appdx-rem:improve-lower-bound-global-stichastic-model-mean-squared-smoothness-2-p-BCM,appdx-rem:improve-lower-bound-global-stichastic-model-mean-squared-smoothness-2-hessian}.
		
		First, the case of an exponent $p \in \intof{1}{2}$ is already covered by~\Cref{appdx-lem:lower-bound-global-stichastic-model-mean-squared-smoothness-p-BCM}. For the case when $p > 2$, we reuse the exact same proof as in~\citet[Lemma~11]{arjevani2022lower}. More precisely, instead of considering Bernoulli distributions for $\P^{-1}$ and $\P^1$, we fall back to normal distributions, that is, $\P^s = \smash{\mathcal{N}(r s, \frac{\sigma_1^2}{c_p^2})}$ for $s \in \ens{-1, 1}$ and some well-chosen constant $c_p > 0$ (depending on the exponent $p$ and $\bar{L}$) that we fix later. Taking $\beta = 1$ and $L = \bar{L}$, as in~\citet{arjevani2022lower}, it remains to bound the $p^{\textnormal{th}}$ central moment of the stochastic gradient, we have for any $x \in \R^d$ and any $s \in \ens{-1, 1}$, by~\eqref{279a6448-b491-4172-8fb0-c0db4abd8e3f}
		\begin{alignat*}{2}
			\ExpSub{\xi \sim \P^s}{\norm{\nabla f_d^{\star}(x, \xi) - \nabla F_{d, s}^{\star}(x)}^p} \oversetlab{\eqref{279a6448-b491-4172-8fb0-c0db4abd8e3f}}&{=} \bar{L}^p \ExpSub{\xi \sim \P^s}{\abs{\xi - r s}^p} \\
			\oversetref{Lem.}{\ref{appdx-lem:absolute-central-moment-gaussian}}&{=} \bar{L}^p \left( \frac{\sigma_1 \sqrt{2}}{c_p} \right)^p \frac{\Gamma\left( \frac{p + 1}{2} \right)}{\sqrt{\pi}} \\
			\oversetrel{rel:0d8cfffb-413d-4c15-a6c3-6c83cc2eac36}&{\le} \sigma_1^p,
		\end{alignat*}
		where in~\relref{rel:0d8cfffb-413d-4c15-a6c3-6c83cc2eac36} we took
		\[ c_p = \bar{L} \sqrt{2} \left( \frac{\Gamma\left( \frac{p + 1}{2} \right)}{\sqrt{\pi}} \right)^{\frac{1}{p}} = C_p \bar{L} > 0, \]
		where $C_p = \sqrt{2} \left( \frac{\Gamma\left( \frac{p + 1}{2} \right)}{\sqrt{\pi}} \right)^{\frac{1}{p}}$. The rest of the proof is similar to~\citet[Lemma~10]{arjevani2022lower}.        
	\end{proof}
	
	\subsection{Some Bounds on the Gradient Estimator $\bar{g}_T$}
	
	\begin{lemma}[Properties of the Gradient Estimator $\bar{g}_T$]\label{appdx-lem:properties-gradient-estimator-1}
		The stochastic gradient estimator $\bar{g}_T$ is a probability--$\theta$ zero-chain, is unbiased with respect to $\nabla F_T$ and satisfies
		\[ \E{\norm{\bar{g}_T(x, \xi) - \nabla F_T(x)}^p} \le \frac{2 \gamma_{\infty}^p (1 - \theta)}{\theta^{p - 1}}, \,\, \text{ and } \,\, \E{\norm{\bar{g}_T(x, \xi) - \bar{g}_T(y, \xi)}^q} \le \frac{\bar{\ell}_1^q}{\theta^{q - 1}} \norm{x - y}^q, \numberthis\label{9ff1cee4-5c5c-4c5b-854f-445bd4e1539d} \]
		and
		\[ \E{\norm{\left[ \bar{g}_T(x, \xi) - \bar{g}_T(y, \xi) \right] - \left[ \nabla F_T(x) - \nabla F_T(y) \right]}^q} \le \frac{\bar{\delta_1}^q (1 - \theta)}{\theta^{q - 1}} \norm{x - q}^q, \numberthis\label{5ebfcdd7-34dd-4433-bb48-69a29f3c7dab} \]
		for all $x, y \in \R^T$, where $p, q \in \intof{1}{2}$, $\gamma_{\infty}$ is defined in~\Cref{appdx-lem:properties-hard-instance}, $\bar{\ell}_1 \eqdef 2 \left( 2^q \left( 6^{2 q} \gamma_{\infty}^q + \ell_1^q \right) + \ell_1^q  \right)^{\frac{1}{q}} \ge \ell_1$ and $\bar{\delta}_1 \eqdef 4 \left( 6^{2 q} \gamma_{\infty}^q + \ell_1^q \right)^{\frac{1}{q}}$.
	\end{lemma}
	
	The proof of~\Cref{appdx-lem:properties-gradient-estimator-1} is largely inspired from~\citep[Lemma~4]{arjevani2022lower}. Notably, we rely on~\Cref{lem:norm-power-alpha-inequality} which states the following  inequality
	\[ \norm{a + b}^{\alpha} \le 2^{\alpha - 1} \left( \norm{a}^{\alpha} + \norm{b}^{\alpha} \right), \numberthis\label{12e46e5c-0bf2-4de8-ac73-920ea999660b} \]
	holds, for any vectors $a, b \in \R^d$ and any exponent $\alpha \ge 1$. This generalizes the well-known inequality $\sqnorm{a + b} \le 2( \sqnorm{a} + \sqnorm{b})$. We use inequality~\eqref{12e46e5c-0bf2-4de8-ac73-920ea999660b} in~\eqref{af4fdc92-39df-4ba0-b39e-528947561b51} as a \say{substitute} for the squared norm expansion (as used in~\citet{arjevani2022lower}) which does not hold anymore with exponent $q \in \intof{1}{2}$ instead of $2$.
	
	\begin{proof}
		Following the proof of~\citet[Lemma~4]{arjevani2022lower}, for any $\xi$, the vector $\delta(x, \xi) \eqdef \bar{g}_T(x, \xi) - \nabla F_T(x)$ has at most one nonzero entry at coordinate $i_x = \prog_{\frac{1}{4}}(x) + 1$. Moreover, for any $i \in \Int{1}{T}$, the $i^{\textnormal{th}}$ entry $\delta_i(x, \xi)$ of $\delta(x, \xi)$ reads
		\[ \delta_i(x, \xi) = \left[ \nabla F_T(x) \right]_i \Theta_i(x) \left( \frac{\xi}{\theta} - 1 \right), \]
		where the function $\Theta_i$ is defined in~\Cref{appdx-def:smoothed-gradient-estimator}. Hence, we have
		\begin{alignat*}{2}
			\E{\norm{\bar{g}_T(x, \xi) - \nabla F_T(x)}^p} & = \E{\abs{\delta_{i_x}(x, \xi)}^p} \\
			& = \abs{ \left[ \nabla F_T(x) \right]_{i_x}}^p \abs{\Theta_{i_x}(x)}^p \E{\abs{\frac{\xi}{\theta} - 1}^p} \\
			\oversetref{Lem.}{\ref{appdx-lem:properties-hard-instance}}&{\le} \gamma_{\infty}^p \abs{\Theta_{i_x}(x)}^p \E{\abs{\frac{\xi}{\theta} - 1}^p} \\
			& \le \gamma_{\infty}^p \, \E{\abs{\frac{\xi}{\theta} - 1}^p} \\
			\oversetrel{rel:91d78f4b-c07c-4b5e-974c-cea7451e140d}&{=} \gamma_{\infty}^p \left( (1 - \theta) + \theta \left( \frac{1}{\theta} - 1 \right)^p \right) \\
			& = \gamma_{\infty}^p (1 - \theta) \left( 1 + \left( \frac{1 - \theta}{\theta} \right)^{p - 1} \right) \\
			\oversetrel{rel:51b68383-0e41-40f2-a4b7-c07c7d81d151}&{\le} \frac{2 \gamma_{\infty}^p (1 - \theta)}{\theta^{p - 1}}, \numberthis\label{717d22a1-a3cf-4c72-9b37-af698eef534b}
		\end{alignat*}
		where in~\relref{rel:91d78f4b-c07c-4b5e-974c-cea7451e140d} we use the fact that $\xi \sim \Ber(\theta)$. In~\relref{rel:51b68383-0e41-40f2-a4b7-c07c7d81d151} we use $\theta \in \intff{0}{1}$ and $p \ge 1$ to bound $(1 - \theta)^{p - 1} \le 1$ and $1 \le \theta^{- (p - 1)}$. This establishes the first inequality from~\eqref{9ff1cee4-5c5c-4c5b-854f-445bd4e1539d}.
		
		Now, for the second inequality in~\eqref{9ff1cee4-5c5c-4c5b-854f-445bd4e1539d}, we use the fact that $\delta(\cdot, \xi) \eqdef \bar{g}_T(\cdot, \xi) - \nabla F_T(\cdot)$ has at most one nonzero coordinate, then if we let $i_x = \prog_{\frac{1}{4}}(x) + 1$ and $i_y = \prog_{\frac{1}{4}}(y) + 1$ we obtain
		\begin{alignat*}{2}
			\E{\norm{\bar{g}_T(x, \xi) - \bar{g}_T(y, \xi)}^q} & = \E{\norm{\left[ \delta(x, \xi) - \delta(y, \xi) \right] + \left[ \nabla F_T(x) - \nabla F_T(y) \right]}^q} \\
			\oversetref{Lem.}{\ref{lem:norm-power-alpha-inequality}}&{\le} 2^{q - 1} \left( \E{\norm{\delta(x, \xi) - \delta(y, \xi)}^q} + \norm{\nabla F_T(x) - \nabla F_T(y)}^q \right) \\
			& \le 2^{q - 1} \left( \E{\abs{\delta_{i_x}(x, \xi) - \delta_{i_x}(y, \xi)}^q} + \E{\abs{\delta_{i_y}(x, \xi) - \delta_{i_y}(y, \xi)}^q}\right)\\
			& + 2^{q - 1}\norm{\nabla F_T(x) - \nabla F_T(y)}^q , \numberthis\label{af4fdc92-39df-4ba0-b39e-528947561b51}
		\end{alignat*}
		and, for any integer $i \in \Int{1}{T}$ we have
		\begin{alignat*}{2}
			\E{\abs{\delta_i(x, \xi) - \delta_i(y, \xi)}^q} & = \abs{\left[ \nabla F_T(x) \right]_i \Theta_i(x) - \left[ \nabla F_T(y) \right]_i \Theta_i(y)}^q \E{\abs{\frac{\xi}{\theta} - 1}^q} \\
			\oversetrel{rel:e38ff9cf-aa6b-4698-a9b3-41486d609379}&{\le} \abs{\left[ \nabla F_T(x) \right]_i \Theta_i(x) - \left[ \nabla F_T(y) \right]_i \Theta_i(y)}^q \left( \frac{2}{\theta^{q - 1}} \right) \\
			& = \abs{\left[ \nabla F_T(x) \right]_i \left( \Theta_i(x) - \Theta_i(y) \right) + \left( \left[ \nabla F_T(x) \right]_i - \left[ \nabla F_T(y) \right]_i \right) \Theta_i(y)}^q \left( \frac{2}{\theta^{q - 1}} \right) \\
			\oversetref{Lem.}{\ref{lem:norm-power-alpha-inequality}}&{\le} \left( \frac{2^q}{\theta^{q - 1}} \right) \left( \abs{\left[ \nabla F_T(x) \right]_i}^q \abs{\Theta_i(x) - \Theta_i(y)}^q + \abs{\left[ \nabla F_T(x) \right]_i - \left[ \nabla F_T(y) \right]_i}^q \abs{\Theta_i(y)}^q \right)  \\
			\oversetrel{rel:396b3332-3a8a-41c4-97bf-31a23104be17}&{\le} \left( \frac{2^q}{\theta^{q - 1}} \right) \left( 6^{2 q} \abs{\left[ \nabla F_T(x) \right]_i}^q \norm{x - y}^q + \abs{\left[ \nabla F_T(x) \right]_i - \left[ \nabla F_T(y) \right]_i}^q \right) \\
			\oversetref{Lem.}{\ref{appdx-lem:properties-hard-instance}}&{\le} \left( \frac{2^q}{\theta^{q - 1}} \right) \left( 6^{2 q} \gamma_{\infty}^q \norm{x - y}^q + \norm{\nabla F_T(x) - \nabla F_T(y)}^q \right) , \numberthis\label{78a8d3c0-b886-4e27-965b-695544cbd732}
		\end{alignat*}
		where in~\relref{rel:e38ff9cf-aa6b-4698-a9b3-41486d609379} we use our previous bound from~\eqref{717d22a1-a3cf-4c72-9b37-af698eef534b} (which we derived with exponent $p$ instead of $q$). In~\relref{rel:396b3332-3a8a-41c4-97bf-31a23104be17} we use the fact that $\Theta_i$ is $6^2$--Lipschitz (see~\cite{arjevani2022lower}) and $\abs{\Theta_i(\cdot)} \le 1$. Now, plugging back the bound~\eqref{78a8d3c0-b886-4e27-965b-695544cbd732} in~\eqref{af4fdc92-39df-4ba0-b39e-528947561b51} we obtain
		\begin{alignat*}{2}
			\E{\norm{\bar{g}_T(x, \xi) - \bar{g}_T(y, \xi)}^q} \oversetlab{\eqref{78a8d3c0-b886-4e27-965b-695544cbd732}+\eqref{af4fdc92-39df-4ba0-b39e-528947561b51}}&{\le} 2^{q - 1}  \left( \frac{2^{q}}{\theta^{q - 1}} \right)\left( 6^{2 q} \gamma_{\infty}^q \norm{x - y}^q + \norm{\nabla F_T(x) - \nabla F_T(y)}^q \right) \\
			&\qquad + 2^{q - 1} \norm{\nabla F_T(x) - \nabla F_T(y) }^q  \\
			\oversetref{Lem.}{\ref{appdx-lem:properties-hard-instance}}&{\le} 2^{q - 1} \left( 2^q \left( 6^{2 q} \gamma_{\infty}^q + \ell_1^q \right) \left( \frac{2}{\theta^{q - 1}} \right) + \ell_1^q \right) \norm{x - y}^q \\
			\oversetrel{rel:e804c734-01d8-4ff4-bcb2-476336204562}&{\le} 2^q \left( 2^q \left( 6^{2 q} \gamma_{\infty}^q + \ell_1^q \right) + \ell_1^q \right) \frac{\norm{x - y}^q}{\theta^{q - 1}},
		\end{alignat*}
		where in~\relref{rel:e804c734-01d8-4ff4-bcb2-476336204562} we use $1 \le \nicefrac{2}{\theta^{q - 1}}$ to factor it out. If we let $\bar{\ell}_1 \eqdef 2 \left( 2^q \left( 6^{2 q} \gamma_{\infty}^q + \ell_1^q \right) + \ell_1^q  \right)^{\frac{1}{q}}$ we obtain
		\[ \E{\norm{\bar{g}_T(x, \xi) - \bar{g}_T(y, \xi)}^q} \le \frac{\bar{\ell}_1^q}{\theta^{q - 1}} \norm{x - y}^q, \]
		as desired.
		
		It remains to establish the third inequality~\eqref{5ebfcdd7-34dd-4433-bb48-69a29f3c7dab}.Using that 
		\begin{equation*}
			\E{\norm{\left[ \bar{g}_T(x, \xi) - \bar{g}_T(y, \xi) \right] - \left[ \nabla F_T(x) - \nabla F_T(y) \right]}^q}  = \E{\norm{\delta(x, \xi) - \delta(y, \xi)}^q},
		\end{equation*}
		and combining the bounds~\eqref{af4fdc92-39df-4ba0-b39e-528947561b51} and~\eqref{78a8d3c0-b886-4e27-965b-695544cbd732} we have
		\begin{alignat*}{2}
			\E{\norm{\delta(x, \xi) - \delta(y, \xi)}^q} \oversetref{Lem.}{\ref{lem:norm-power-alpha-inequality}}&{\le} 2^{q - 1} \left( \E{\abs{\delta_{i_x}(x, \xi) - \delta_{i_x}(y, \xi)}^q} + \E{\abs{\delta_{i_y}(x, \xi) - \delta_{i_y}(y, \xi)}^q} \right) \\
			\oversetlab{\eqref{717d22a1-a3cf-4c72-9b37-af698eef534b}+\eqref{78a8d3c0-b886-4e27-965b-695544cbd732}}&{\le} 4^q \left( 6^{2 q} \gamma_{\infty}^q + \ell_1^q \right) \frac{1 - \theta}{\theta^{q - 1}} \norm{x - y}^q \\
			& = \frac{\bar{\delta_1}^q (1 - \theta)}{\theta^{q - 1}} \norm{x - q}^q,
		\end{alignat*}
		where we define $\bar{\delta}_1 \eqdef 4 \left( 6^{2 q} \gamma_{\infty}^q + \ell_1^q \right)^{\frac{1}{q}}$. This achieves the proof of the lemma.
	\end{proof}
	
	\subsection{Proof of~\Cref{thm:lower-bound-sofo-mean-squared-smoothness}}\label{appdx-subsec:proof-lower-bound-sofo-mean-squared-smoothness}
	
	\begin{restate-theorem}{\ref{thm:lower-bound-sofo-mean-squared-smoothness}}
		Given $\Delta, \bar{L} > 0$ and $0 < \eps \le c_1 \sqrt{\bar{L} \Delta}$ then, for any algorithm $A \in \mathcal{A}_{\texttt{zr}}$, there exists a function $f \in \mathcal{F}\left( \Delta \right)$, an oracle and a distribution $(O, \cD) \in \mathcal{O}\left( f, \bar{L}^q, \sigma_1^p \right)$ such that
		
		\begin{alignat*}{2}
			\mathfrak{m}^{\normalfont\texttt{zr}}_{\eps}\left( K, \bar{L}, \Delta, \sigma_1^p \right) \ge \Omega(1) \cdot \left( \left( \frac{\sigma_1}{\eps} \right)^{\frac{p}{p - 1}} + \frac{\bar{L} \Delta}{\eps^2} + \frac{\bar{L} \Delta}{\eps^2} \left( \frac{\sigma_1}{\eps} \right)^{\frac{p}{q (p - 1)}} \right).
		\end{alignat*}
	\end{restate-theorem}
	
	\begin{proof}
		Let $\Delta_0$, $\ell_1$, $\gamma_{\infty}$ and $\bar{\ell}_1$ be the numerical constants in~\Cref{appdx-lem:properties-hard-instance,appdx-lem:properties-gradient-estimator-1} respectively. Additionally, we let the accuracy parameter $\eps > 0$, initial sub-optimality $\Delta \ge 0$, the $q$--weak average smoothness parameter $\bar{L}$, and the variance parameter $\sigma_1 \ge 0$ be fixed, and $0 < L \le \bar{L}$ to be specified late. Then, for $\alpha, \beta > 0$ two positive real numbers, following~\citet{arjevani2022lower}, we rescale the function $F_T$ as
		\[ F^{\star}_T \colon x \mapsto \alpha F_T(\beta x). \numberthis\label{6f55ce68-cc03-4d22-a262-4aa3367f948d} \]

		\paragraph{Step 1:} \textit{Ensuring $F^{\star}_T \in \mathcal{F} \left(\Delta, L \right)$}.
		
		To guarantee the rescaled function $F^{\star}_T$ belongs to the function class $\mathcal{F}(\Delta, L)$, let us compute the initial sub-optimality $\Delta$ and the smoothness constant $L$. Assuming the algorithm \texttt{A} starts at $x^0 = 0$ we have
		\[ F^{\star}_T(0) - \inf_{x \in \R^T} F^{\star}_T(x) \oversetlab{\eqref{6f55ce68-cc03-4d22-a262-4aa3367f948d}}{=} \alpha \left( F_T(0) - \inf_{x \in \R^T} F_T(x) \right) \oversetref{Lem.}{\ref{appdx-lem:properties-hard-instance}}{\le} \alpha \Delta_0 T, \]
		thus, it's enough to take $T = \Floor{\frac{\Delta}{\alpha \Delta_0}}$ so as to ensure $F^{\star}_T(0) - \inf_{x \in \R^T} F^{\star}_T(x) \le \Delta$. Moreover, for any $x, y \in \R^T$,
		\begin{alignat*}{2}
			\norm{\nabla F^{\star}_T(x) - \nabla F^{\star}_T(y)} & = \alpha \beta \norm{\nabla F_T(\beta x) - \nabla F_T(\beta y)} \\
			\oversetref{Lem.}{\ref{appdx-lem:properties-hard-instance}}&{\le} \alpha \beta \ell_1 \norm{\beta x - \beta y} \\
			& = \alpha \beta^2 \ell_1 \norm{x - y},
		\end{alignat*}
		and it suffices to take $\alpha = \frac{L}{\beta^2 \ell_1} > 0$ to ensure the function $F^{\star}_T$ has $L$--Lipschitz gradients. Consequently, we have $F^{\star}_T \in \mathcal{F}(\Delta, L)$, as desired.
		
		\paragraph{Step 2:} \textit{Analysis of the Protocol and Choice for $\beta$}.
		
		Following the proof of~\citet[Theorem~1]{arjevani2022lower}, according to~\Cref{appdx-lem:properties-gradient-estimator-1}, for all points $x \in \R^T$ such that $\prog_0(x) < T$ we have $\prog_0(\beta x) = \prog_0(x) < T$ so
		\[ \norm{\nabla F^{\star}_T(x)} \oversetlab{\eqref{6f55ce68-cc03-4d22-a262-4aa3367f948d}}{=} \frac{L}{\ell_1 \beta} \norm{\nabla F_T(\beta x)} \oversetref{Lem.}{\ref{appdx-lem:properties-gradient-estimator-1}}{>} \frac{L}{\ell_1 \beta}, \numberthis\label{191cd1ea-8e5f-4b85-8e63-f31fd25f3b17} \]
		and we need to guarantee that
		\[ \norm{\nabla F^{\star}_T(x)} > 2\eps, \]
		for all $x \in \R^T$ with $\prog_0(x) < T$ which, given~\eqref{191cd1ea-8e5f-4b85-8e63-f31fd25f3b17}, can be done if we set $\beta = \frac{L}{2 \ell_1 \eps}$.

		\paragraph{Step 3:} \textit{The Oracle Class and~\Cref{ass:p-bounded-central-moment-gradient,ass:mean-squared-smoothness}}.
		
		It remains to choose the parameter $\theta \in \intof{0}{1}$ and constant $L$ such that the gradient estimator $\bar{g}^{\star}_T$ of $\nabla F^{\star}_T$ satisfies~\Cref{ass:p-bounded-central-moment-gradient,ass:mean-squared-smoothness}. Computing the $p$--th central moment of $\bar{g}^{\star}_T$ gives, for all $x \in \R^T$
		\begin{alignat*}{2}
			\E{\norm{\bar{g}^{\star}_T(x, \xi) - \nabla F^{\star}_T(x)}^p} \oversetlab{\eqref{6f55ce68-cc03-4d22-a262-4aa3367f948d}}&{=} \left( \alpha \beta \right)^p \E{\norm{\bar{g}_T(\beta x, \xi) - \nabla F_T(\beta x)}^p} \\
			\oversetref{Lem.}{\ref{appdx-lem:properties-gradient-estimator-1}}&{\le} \frac{2 \left( \gamma_{\infty} \alpha \beta \right)^p (1 - \theta)}{\theta^{p - 1}} \\
			\oversetrel{rel:0fc4ad15-661f-4519-87ce-6459c8d1c74b}&{\le} \frac{\left( 2 \gamma_{\infty} \alpha \beta \right)^p}{\theta^{p - 1}}, \numberthis\label{29bb8bb5-e6e5-482d-ab46-fab6357955ec}
		\end{alignat*}
		where in~\relref{rel:0fc4ad15-661f-4519-87ce-6459c8d1c74b} we use the fact that $0 < \theta \le 1$ and $p \ge 1$ so that $2 \le 2^p$. From~\eqref{29bb8bb5-e6e5-482d-ab46-fab6357955ec}, so as to satisfy~\Cref{ass:p-bounded-central-moment-gradient} it's enough to take $\theta = \min\ens{1, \bar{\theta}}$ where
		
		\[ \bar{\theta}^{p - 1} \ge \left( \frac{2 \gamma_{\infty} \alpha \beta}{\sigma_1} \right)^p \oversetrel{rel:46a820b5-b4a2-4c73-b77f-41e02836f2e0}{=} \left( \frac{4 \gamma_{\infty} \eps}{\sigma_1} \right)^p, \,\, \text{ so } \,\, \bar{\theta} \ge \left( \frac{4 \gamma_{\infty} \eps}{\sigma_1} \right)^{\frac{p}{p - 1}} \]
		where in~\relref{rel:46a820b5-b4a2-4c73-b77f-41e02836f2e0} we use the value of $\alpha$ and $\beta$ fixed earlier. Hence
		\[ \theta = \min\ens{1, \left( \frac{4 \gamma_{\infty} \eps}{\sigma_1} \right)^{\frac{p}{p - 1}}}. \numberthis\label{b3242270-31d2-4d6a-a567-27370b26f605-2} \]
		
		Next, concerning the mean-squared smoothness assumption, we have
		\begin{alignat*}{2}
			\E{\norm{\bar{g}^{\star}_T(x, \xi) - \bar{g}^{\star}_T(y, \xi)}^q} \oversetlab{\eqref{6f55ce68-cc03-4d22-a262-4aa3367f948d}}&{=} \left( \alpha \beta \right)^q \E{\norm{\bar{g}_T(\beta x, \xi) - \bar{g}_T(\beta y, \xi)}^q} \\
			\oversetref{Lem.}{\ref{appdx-lem:properties-gradient-estimator-1}}&{\le} \frac{\left( \alpha \beta \bar{\ell}_1 \right)^q}{\theta^{q - 1}} \norm{\beta x - \beta y}^q \\
			& = \left( \frac{\alpha \beta^2 \bar{\ell}_1}{\theta^{\frac{q - 1}{q}}} \right)^q \norm{x - y}^q \\
			\oversetrel{rel:b4ce77e4-45fb-40a3-9459-1426ca4e5db4}&{=} \left( \frac{L \bar{\ell}_1}{\ell_1 \theta^{\frac{q - 1}{q}}} \right)^q \norm{x - y}^q, \numberthis\label{b4ce77e4-45fb-40a3-9459-1426ca4e5db4-2}
		\end{alignat*}
		where in~\relref{rel:b4ce77e4-45fb-40a3-9459-1426ca4e5db4} we use $\alpha = \frac{L}{\beta^2 \ell_1}$. Hence, from the upper bound~\eqref{b4ce77e4-45fb-40a3-9459-1426ca4e5db4-2} it suffices to take
		\[ L = \frac{\ell_1 \bar{L}}{\bar{\ell}_1} \theta^{\frac{q - 1}{q}} \le \bar{L} \min\ens{1, \left( \frac{4 \gamma_{\infty} \eps}{\sigma_1} \right)^{\frac{p (q - 1)}{q (p - 1)}}} \le \bar{L}, \]
		since $\ell_1 = 152 \le \bar{\ell}_1$ (see~\Cref{appdx-lem:properties-gradient-estimator-1}). This proves~\Cref{ass:mean-squared-smoothness} is satisfied by the gradient estimator $\bar{g}^{\star}_T$.
		
		\paragraph{Step 4:} \textit{Lower Bounding $\mathfrak{m}^{\normalfont\texttt{zr}}_{\eps}\left( K, \Delta, \bar{L}, \sigma^p_1 \right)$}.
		
		Continuing on \textbf{step 2}, by~\Cref{appdx-lem:zero-chain-progress} we know that with probability at least $\frac{1}{2}$ it holds that for all integer $0 \le t \le \frac{T - 1}{2 \theta}$ and all $k \in [K]$ we have
		\[ \norm{\nabla F^{\star}_T\left( x^{(t, k)}_{\texttt{A}[\texttt{O}_F]} \right)} > 2\eps, \,\, \text{ hence } \,\, \E{\norm{\nabla F^{\star}_T\left( x^{(t, k)}_{\texttt{A}[\texttt{O}_F]} \right)}} > \eps, \]
		from where it follows that
		\begin{alignat*}{2}
			\mathfrak{m}^{\normalfont\texttt{zr}}_{\eps}\left( K, \Delta, \bar{L}, \sigma^p_1 \right) > \frac{T - 1}{2 \theta} = \frac{1}{2 \theta} \left( \Floor{\frac{\Delta}{\alpha \Delta_0}} - 1 \right) & = \frac{1}{2 \theta} \left( \Floor{\frac{\beta^2 \ell_1 \Delta}{L \Delta_0}} - 1 \right) \\
			& = \frac{1}{2 \theta} \left( \Floor{\frac{L \Delta}{4 \Delta_0 \ell_1 \eps^2}} - 1 \right) \\
			& = \frac{1}{2 \theta} \left( \Floor{\frac{\bar{L} \Delta \theta^{\frac{q - 1}{q}}}{4 \Delta_0 \bar{\ell}_1 \eps^2}} - 1 \right), \numberthis\label{b3242270-31d2-4d6a-a567-27370b26f605}
		\end{alignat*}
		we then distinguish two cases:
		\begin{itemize}
			\item if $\frac{\bar{L} \Delta \theta^{\frac{q - 1}{q}}}{4 \Delta_0 \bar{\ell}_1 \eps^2} \ge 3$ then, using the inequality $\Floor{x} - 1 \ge \frac{x}{2}$, valid for all real number $x \ge 3$ we obtain
			\begin{alignat*}{2}
				\mathfrak{m}^{\normalfont\texttt{zr}}_{\eps}\left( K, \Delta, \bar{L}, \sigma^p_1 \right) \oversetlab{\eqref{b3242270-31d2-4d6a-a567-27370b26f605}}&{>} \frac{1}{4 \theta} \cdot \frac{\bar{L} \Delta \theta^{\frac{q - 1}{q}}}{4 \Delta_0 \bar{\ell}_1 \eps^2} \\
				& = \frac{1}{16 \Delta_0 \bar{\ell}_1} \cdot \frac{\bar{L} \Delta}{\eps^2 \theta^{\frac{1}{q}}} \\
				\oversetlab{\eqref{b3242270-31d2-4d6a-a567-27370b26f605-2}}&{\ge} \frac{1}{32 \Delta_0 \bar{\ell}_1} \cdot \frac{\bar{L} \Delta}{\eps^2} \left[ \left( \frac{\sigma_1}{4 \gamma_{\infty} \eps} \right)^{\frac{p}{q (p - 1)}} + 1 \right] \\
				& = \Omega(1) \cdot \left( \frac{\bar{L} \Delta}{\eps^2} + \frac{\bar{L} \Delta}{\eps^2} \left( \frac{\sigma_1}{\eps} \right)^{\frac{p}{q (p - 1)}} \right),
			\end{alignat*}
			and combining the lower bound above with~\Cref{appdx-lem:lower-bound-global-stichastic-model-mean-squared-smoothness-p-BCM}, with the same choice for $c'$ as provided below, gives the desired results.
			
			\item otherwise, if $\frac{\bar{L} \Delta \theta^{\frac{q - 1}{q}}}{4 \Delta_0 \bar{\ell}_1 \eps^2} < 3$, choosing the universal constant $0 < c' = \left( 12 \Delta_0 \bar{\ell}_1 \right)^{-\frac{1}{2}} \le \left( 12 \Delta_0 \ell_1 \right)^{-\frac{1}{2}} \approx 0.0067592 < \frac{1}{8}$ then the assumption 
			\[ 0 < \eps \le c' \sqrt{\bar{L} \Delta} = \sqrt{\frac{\bar{L} \Delta}{12 \Delta_0 \bar{\ell}_1}}, \]
			precludes the possibility that $\theta = 1$ so $\sigma_1 \ge 4 \gamma_{\infty} \eps$ hence $\nicefrac{\sigma_1}{\eps} \gtrsim 1$. Moreover, we have
			\begin{alignat*}{2}
				3 > \frac{\bar{L} \Delta}{4 \Delta_0 \bar{\ell}_1 \eps^2} \left( \frac{4 \gamma_{\infty} \eps}{\sigma_1} \right)^{\frac{p (q - 1)}{q (p - 1)}},
			\end{alignat*}
			and multiplying both sides by $\left( \frac{\sigma_1}{\eps} \right)^{\frac{p}{p - 1}} > 0$ leads to
			\[ \left( \frac{\sigma_1}{\eps} \right)^{\frac{p}{p - 1}} > \frac{\bar{L} \Delta \left( 4 \gamma_{\infty} \right)^{\frac{p (q - 1)}{q (p - 1)}}}{12 \Delta_0 \bar{\ell}_1 \eps^2} \left( \frac{\sigma_1}{\eps} \right)^{\frac{p}{p - 1} - \frac{p (q - 1)}{q (p - 1)}} = \frac{\bar{L} \Delta \left( 4 \gamma_{\infty} \right)^{\frac{p (q - 1)}{q (p - 1)}}}{12 \Delta_0 \bar{\ell}_1 \eps^2} \left( \frac{\sigma_1}{\eps} \right)^{\frac{p}{q (p - 1)}}, \]
			and using~\Cref{appdx-lem:lower-bound-global-stichastic-model-mean-squared-smoothness-p-BCM}, since $0 < \eps \le \frac{1}{8} \sqrt{\bar{L} \Delta}$ then there exists an universal constant $C_p > 0$ (depending only on $p$) such that
			\[ \mathfrak{m}^{\normalfont\texttt{zr}}_{\eps}\left( K, \bar{L}, \Delta, \sigma_1^p \right) \ge C_p \left( \frac{\sigma_1}{\eps} \right)^{\frac{p}{p - 1}} \ge \Omega(1) \cdot \frac{\bar{L} \Delta}{\eps^2} \left( \frac{\sigma_1}{\eps} \right)^{\frac{p}{q (p - 1)}} \gtrsim \frac{\bar{L} \Delta}{\eps^2}, \numberthis\label{e278b503-7290-42c8-a9cd-e396da8844d8} \]
			hence, 
			\[ \mathfrak{m}^{\normalfont\texttt{zr}}_{\eps}\left( K, \bar{L}, \Delta, \sigma_1^p \right) \ge \Omega(1) \cdot \left( \left( \frac{\sigma_1}{\eps} \right)^{\frac{p}{p - 1}} + \frac{\bar{L} \Delta}{\eps^2} + \frac{\bar{L} \Delta}{\eps^2} \left( \frac{\sigma_1}{\eps} \right)^{\frac{p}{q (p - 1)}} \right), \]
			which holds in both cases and concludes the proof of the theorem.
		\end{itemize}
		
	\end{proof}
	
	\subsection{Proof of~\Cref{thm:lower-bound-sofo-mean-squared-smoothness-2}}\label{appdx-subsec:proof-lower-bound-sofo-mean-squared-smoothness-2}
	
	\begin{restate-theorem}{\ref{thm:lower-bound-sofo-mean-squared-smoothness-2}}
		Given $\Delta, L_1, \delta > 0$ and $0 < \eps \le c_1 \sqrt{L_1 \Delta}$ then, for any algorithm $A \in \mathcal{A}_{\texttt{zr}}$, there exists a function $f \in \mathcal{F}\left( \Delta, L_1 \right)$, an oracle and a distribution $(O, \cD) \in \mathcal{O}\left( f, \delta^q, \sigma_1^p \right)$ such that
		
		\begin{alignat*}{2}
			\mathfrak{m}^{\normalfont\texttt{zr}}_{\eps}\left( K, L_1, \Delta, \delta, \sigma_1^p \right) \ge \Omega(1) \cdot \min\left\{ \frac{L_1 \Delta}{\eps^2} + \frac{L_1 \Delta}{\eps^2} \left( \frac{\sigma_1}{\eps} \right)^{\frac{p}{p - 1}},  \left( \frac{\sigma_1}{\eps} \right)^{\frac{p}{p - 1}} + \frac{(L_1 + \delta) \Delta}{\eps^2} + \frac{\delta \Delta}{\eps^2} \left( \frac{\sigma_1}{\eps} \right)^{\frac{p}{q (p - 1)}} \right\}.
		\end{alignat*}
	\end{restate-theorem}
	
	\begin{proof}
		The proof follows the same lines as the proof of~\Cref{thm:lower-bound-sofo-mean-squared-smoothness}.  Let $\Delta_0$, $\ell_1$, $\gamma_{\infty}$ and $\bar{\delta}_1$ be the numerical constants in~\Cref{appdx-lem:properties-hard-instance,appdx-lem:properties-gradient-estimator-1} respectively. Additionally, we let the accuracy parameter $\eps > 0$, initial sub-optimality $\Delta \ge 0$, the smoothness constant $L_1 \ge 0$, the $q$--weak average smoothness parameter $\delta$ (\Cref{ass:mean-squared-smoothness-2}), and the variance parameter $\sigma_1 \ge 0$ be fixed, and $0 < L \le L_1$ to be specified late. Then, for $\alpha, \beta > 0$ two positive real numbers, following~\citet{arjevani2022lower}, we rescale the function $F_T$ as
		\[ F^{\star}_T \colon x \mapsto \alpha F_T(\beta x). \numberthis\label{6f55ce68-cc03-4d22-a262-4aa3367f948d-bis} \]
		
		\paragraph{Step 1:} \textit{Ensuring $F^{\star}_T \in \mathcal{F} \left(\Delta, L \right)$}.
		
		To guarantee the rescaled function $F^{\star}_T$ belongs to the function class $\mathcal{F}(\Delta, L)$, let us compute the initial sub-optimality $\Delta$ and the smoothness constant $L$. Assuming the algorithm \texttt{A} starts at $x^0 = 0$ we have
		\[ F^{\star}_T(0) - \inf_{x \in \R^T} F^{\star}_T(x) \oversetlab{\eqref{6f55ce68-cc03-4d22-a262-4aa3367f948d}}{=} \alpha \left( F_T(0) - \inf_{x \in \R^T} F_T(x) \right) \oversetref{Lem.}{\ref{appdx-lem:properties-hard-instance}}{\le} \alpha \Delta_0 T, \]
		thus, it's enough to take $T = \Floor{\frac{\Delta}{\alpha \Delta_0}}$ so as to ensure $F^{\star}_T(0) - \inf_{x \in \R^T} F^{\star}_T(x) \le \Delta$. Moreover, as done previously, for any $x, y \in \R^T$,
		\begin{alignat*}{2}
			\norm{\nabla F^{\star}_T(x) - \nabla F^{\star}_T(y)} & = \alpha \beta \norm{\nabla F_T(\beta x) - \nabla F_T(\beta y)} \\
			\oversetref{Lem.}{\ref{appdx-lem:properties-hard-instance}}&{\le} \alpha \beta \ell_1 \norm{\beta x - \beta y} \\
			& = \alpha \beta^2 \ell_1 \norm{x - y},
		\end{alignat*}
		and it suffices to take $\alpha = \frac{L}{\beta^2 \ell_1} > 0$ to ensure the function $F^{\star}_T$ has $L$--Lipschitz gradients. Consequently, we have $F^{\star}_T \in \mathcal{F}(\Delta, L)$, as desired.
		
		\paragraph{Step 2:} \textit{Analysis of the Protocol and Choice for $\beta$}.
		
		Following the proof of~\citet[Theorem~1]{arjevani2022lower}, according to~\Cref{appdx-lem:properties-gradient-estimator-1}, for all points $x \in \R^T$ such that $\prog_0(x) < T$ we have $\prog_0(\beta x) = \prog_0(x) < T$ so
		\[ \norm{\nabla F^{\star}_T(x)} \oversetlab{\eqref{6f55ce68-cc03-4d22-a262-4aa3367f948d}}{=} \frac{L}{\ell_1 \beta} \norm{\nabla F_T(\beta x)} \oversetref{Lem.}{\ref{appdx-lem:properties-gradient-estimator-1}}{>} \frac{L}{\ell_1 \beta}, \numberthis\label{191cd1ea-8e5f-4b85-8e63-f31fd25f3b17-bis} \]
		and we need to guarantee that
		\[ \norm{\nabla F^{\star}_T(x)} > 2\eps, \numberthis\label{e8299abb-d1db-4ee4-9919-3b49ed05df54} \]
		for all $x \in \R^T$ with $\prog_0(x) < T$ which, given~\eqref{191cd1ea-8e5f-4b85-8e63-f31fd25f3b17}, can be done if we set $\beta = \frac{L}{2 \ell_1 \eps}$.

		\paragraph{Step 3:} \textit{The Oracle Class and~\Cref{ass:p-bounded-central-moment-gradient,ass:mean-squared-smoothness}}.
		
		It remains to choose the parameter $\theta \in \intof{0}{1}$ and constant $L$ such that the gradient estimator $\bar{g}^{\star}_T$ of $\nabla F^{\star}_T$ satisfies~\Cref{ass:p-bounded-central-moment-gradient,ass:mean-squared-smoothness-2}. Computing the $p$--th central moment of $\bar{g}^{\star}_T$ gives, for all $x \in \R^T$
		\begin{alignat*}{2}
			\E{\norm{\bar{g}^{\star}_T(x, \xi) - \nabla F^{\star}_T(x)}^p} \oversetlab{\eqref{6f55ce68-cc03-4d22-a262-4aa3367f948d}}&{=} \left( \alpha \beta \right)^p \E{\norm{\bar{g}_T(\beta x, \xi) - \nabla F_T(\beta x)}^p} \\
			\oversetref{Lem.}{\ref{appdx-lem:properties-gradient-estimator-1}}&{\le} \frac{2 \left( \gamma_{\infty} \alpha \beta \right)^p (1 - \theta)}{\theta^{p - 1}} \\
			\oversetrel{rel:0fc4ad15-661f-4519-87ce-6459c8d1c74b}&{\le} \frac{\left( 2 \gamma_{\infty} \alpha \beta \right)^p}{\theta^{p - 1}}, \numberthis\label{29bb8bb5-e6e5-482d-ab46-fab6357955ec-bis}
		\end{alignat*}
		where in~\relref{rel:0fc4ad15-661f-4519-87ce-6459c8d1c74b} we use the fact that $0 < \theta \le 1$ and $p \ge 1$ so that $2 \le 2^p$. From~\eqref{29bb8bb5-e6e5-482d-ab46-fab6357955ec}, so as to satisfy~\Cref{ass:p-bounded-central-moment-gradient} it's enough to take $\theta = \min\ens{1, \bar{\theta}}$ where
		\[ \bar{\theta}^{p - 1} \ge \left( \frac{2 \gamma_{\infty} \alpha \beta}{\sigma_1} \right)^p \oversetrel{rel:46a820b5-b4a2-4c73-b77f-41e02836f2e0}{=} \left( \frac{4 \gamma_{\infty} \eps}{\sigma_1} \right)^p, \,\, \text{ so } \,\, \bar{\theta} \ge \left( \frac{4 \gamma_{\infty} \eps}{\sigma_1} \right)^{\frac{p}{p - 1}} \]
		where in~\relref{rel:46a820b5-b4a2-4c73-b77f-41e02836f2e0} we use the value of $\alpha$ and $\beta$ fixed earlier. Hence
		\[ \theta = \min\ens{1, \left( \frac{4 \gamma_{\infty} \eps}{\sigma_1} \right)^{\frac{p}{p - 1}}}. \numberthis\label{b3242270-31d2-4d6a-a567-27370b26f605-2-bis} \]
		
		Next, concerning the mean-squared smoothness assumption (\Cref{ass:mean-squared-smoothness-2}), we have
		\begin{alignat*}{2}
			&\E{\norm{\left[ \bar{g}^{\star}_T(x, \xi) - \bar{g}^{\star}_T(y, \xi) \right] - \left[ \nabla F_T^{\star}(x) - \nabla F_T^{\star}(y) \right]}^q} \\
			&\qquad\qquad\qquad\begin{aligned}[t]
				\oversetlab{\eqref{6f55ce68-cc03-4d22-a262-4aa3367f948d}}&{=} \left( \alpha \beta \right)^q \E{\norm{\left[ \bar{g}^{\star}_T(\beta x, \xi) - \bar{g}^{\star}_T(\beta y, \xi) \right] - \left[ \nabla F_T^{\star}(\beta x) - \nabla F_T^{\star}(\beta y) \right]}^q} \\
				\oversetref{Lem.}{\ref{appdx-lem:properties-gradient-estimator-1}}&{\le} \frac{\left( \alpha \beta \bar{\delta}_1 \right)^q (1 - \theta)}{\theta^{q - 1}} \norm{\beta x - \beta y}^q \\
				& = \left( \frac{\alpha \beta^2 \bar{\delta}_1}{\theta^{\frac{q - 1}{q}}} \right)^q (1 - \theta) \norm{x - y}^q \\
				\oversetrel{rel:b4ce77e4-45fb-40a3-9459-1426ca4e5db4}&{=} \left( \frac{L \bar{\delta}_1}{\ell_1 \theta^{\frac{q - 1}{q}}} \right)^q (1 - \theta) \norm{x - y}^q,
			\end{aligned} \numberthis\label{b4ce77e4-45fb-40a3-9459-1426ca4e5db4-2-bis}
		\end{alignat*}
		where in~\relref{rel:b4ce77e4-45fb-40a3-9459-1426ca4e5db4} we use $\alpha = \frac{L}{\beta^2 \ell_1}$. Hence, from the upper bound~\eqref{b4ce77e4-45fb-40a3-9459-1426ca4e5db4-2} it suffices to take
		\[ L \le \frac{\ell_1 \delta}{\bar{\delta}_1} \theta^{\frac{q - 1}{q}}, \]
		and since we must have $L \le L_1$, we set
		\[ L = \min\ens{L_1, \frac{\ell_1 \delta}{\bar{\delta}_1} \theta^{\frac{q - 1}{q}}}. \numberthis\label{bf6dc5ba-5913-4acb-b17c-aa57b62b93b3} \]
		This proves~\Cref{ass:mean-squared-smoothness-2} is satisfied by the gradient estimator $\bar{g}^{\star}_T$.
		
		\paragraph{Step 4:} \textit{Lower Bounding $\mathfrak{m}^{\normalfont\texttt{zr}}_{\eps}\left( K, L_1, \Delta, \delta, \sigma_1^p \right)$}.
		
		Continuing on \textbf{step 2}, by~\Cref{appdx-lem:zero-chain-progress} we know that with probability at least $\frac{1}{2}$ it holds that for all integer $0 \le t \le \frac{T - 1}{2 \theta}$ and all $k \in [K]$ we have
		\[ \norm{\nabla F^{\star}_T\left( x^{(t, k)}_{\texttt{A}[\texttt{O}_F]} \right)} > 2\eps, \,\, \text{ hence } \,\, \E{\norm{\nabla F^{\star}_T\left( x^{(t, k)}_{\texttt{A}[\texttt{O}_F]} \right)}} > \eps, \]
		from where it follows that
		\begin{alignat*}{2}
			\mathfrak{m}^{\normalfont\texttt{zr}}_{\eps}\left( K, L_1, \Delta, \delta, \sigma_1^p \right) > \frac{T - 1}{2 \theta} = \frac{1}{2 \theta} \left( \Floor{\frac{\Delta}{\alpha \Delta_0}} - 1 \right) = \frac{1}{2 \theta} \left( \Floor{\frac{\beta^2 \ell_1 \Delta}{L \Delta_0}} - 1 \right) = \frac{1}{2 \theta} \left( \Floor{\frac{L \Delta}{4 \Delta_0 \ell_1 \eps^2}} - 1 \right) \numberthis\label{b3242270-31d2-4d6a-a567-27370b26f605-bis}
		\end{alignat*}
		we then distinguish two cases:
		\begin{itemize}
			\item if $\frac{L \Delta}{4 \Delta_0 \ell_1 \eps^2} \ge 3$ then, using the inequality $\Floor{x} - 1 \ge \frac{x}{2}$, valid for all real number $x \ge 3$ we obtain
			\begin{alignat*}{2}
				\mathfrak{m}^{\normalfont\texttt{zr}}_{\eps}\left( K, L_1, \Delta, \delta, \sigma_1^p \right) \oversetlab{\eqref{b3242270-31d2-4d6a-a567-27370b26f605-bis}}&{>} \frac{1}{4 \theta} \cdot \frac{L \Delta}{4 \Delta_0 \ell_1 \eps^2} \\
				\oversetlab{\eqref{bf6dc5ba-5913-4acb-b17c-aa57b62b93b3}}&{=} \frac{1}{16 \Delta_0 \ell_1} \cdot \frac{\Delta}{\theta \eps^2} \min\ens{L_1, \frac{\ell_1 \delta}{\bar{\delta}_1} \theta^{\frac{q - 1}{q}}} \\
				& = \frac{1}{16 \Delta_0 \ell_1} \cdot \min\ens{\frac{L_1 \Delta}{\theta \eps^2}, \frac{\ell_1 \delta \Delta}{\bar{\delta}_1 \eps^2} \theta^{-\frac{1}{q}}} \\
				\oversetlab{\eqref{b3242270-31d2-4d6a-a567-27370b26f605-2-bis}}&{\ge} \frac{1}{32 \Delta_0 \ell_1} \cdot \min\ens{\frac{L_1 \Delta}{\eps^2} \left( 1 + \left( \frac{\sigma_1}{4 \gamma_{\infty} \eps} \right)^{\frac{p}{p - 1}} \right), \frac{\ell_1 \delta \Delta}{\bar{\delta}_1 \eps^2} \left( 1 + \left( \frac{\sigma_1}{4 \gamma_{\infty} \eps} \right)^{\frac{p}{q (p - 1)}} \right)} \\
				& = \Omega(1) \cdot \min\ens{\frac{L_1 \Delta}{\eps^2} + \frac{L_1 \Delta}{\eps^2} \left( \frac{\sigma_1}{\eps} \right)^{\frac{p}{p - 1}}, \frac{\delta \Delta}{\eps^2} + \frac{\delta \Delta}{\eps^2} \left( \frac{\sigma_1}{\eps} \right)^{\frac{p}{q (p - 1)}}}, \numberthis\label{3b256059-e30a-4307-b1b7-a8e09886c0f1}
			\end{alignat*}
			and combining the lower bound~\eqref{3b256059-e30a-4307-b1b7-a8e09886c0f1} with~\Cref{appdx-lem:lower-bound-global-stichastic-model-mean-squared-smoothness-p-BCM}, using the choice for $c'$ provided below, gives:
			\[ \mathfrak{m}^{\normalfont\texttt{zr}}_{\eps}\left( K, L_1, \Delta, \delta, \sigma_1^p \right) \ge \Omega(1) \cdot \min\ens{\frac{L_1 \Delta}{\eps^2} + \frac{L_1 \Delta}{\eps^2} \left( \frac{\sigma_1}{\eps} \right)^{\frac{p}{p - 1}}, \left( \frac{\sigma_1}{\eps} \right)^{\frac{p}{p - 1}} + \frac{\delta \Delta}{\eps^2} + \frac{\delta \Delta}{\eps^2} \left( \frac{\sigma_1}{\eps} \right)^{\frac{p}{q (p - 1)}}}, \numberthis\label{5aea445e-d78f-404f-ac66-dbff30f0a00c} \]
			since from the assumption $0 < \eps \le c' \sqrt{L_1 \Delta}$ we have
			\[ \frac{L_1 \Delta}{\eps^2} \left( \frac{\sigma_1}{\eps} \right)^{\frac{p}{p - 1}} \gtrsim \left( \frac{\sigma_1}{\eps} \right)^{\frac{p}{p - 1}}. \]
			
			\item otherwise, if $\frac{L \Delta}{4 \Delta_0 \ell_1 \eps^2} < 3$, choosing the universal constant $0 < c' = \left( 12 \Delta_0 \ell_1 \right)^{-\frac{1}{2}} \approx 0.0067592 < \frac{1}{8}$ then the assumption 
			\[ 0 < \eps \le c' \sqrt{L_1 \Delta} = \sqrt{\frac{L_1 \Delta}{12 \Delta_0 \ell_1}}, \]
			precludes the possibility that $L = L_1$. While we can still argue the same way as previously done in the proof of~\Cref{thm:lower-bound-sofo-mean-squared-smoothness}, we follows here a different strategy, thanks to~\Cref{appdx-rem:improve-lower-bound-global-stichastic-model-mean-squared-smoothness-p-BCM}. So, let us divide both sides of $\frac{L \Delta}{4 \Delta_0 \ell_1 \eps^2} < 3$ by our choice of $\theta > 0$ to obtain
			\[ \frac{1}{12 \Delta_0 \ell_1} \cdot \frac{L \Delta}{\eps^2} \theta^{-1} \le \theta^{-1} = \max\ens{1, \left( \frac{\sigma_1}{4 \gamma_{\infty} \eps} \right)^{\frac{p}{p - 1}}} \le \max\ens{1, \left( \frac{\sigma_1}{\eps} \right)^{\frac{p}{p - 1}}}, \numberthis\label{2b14b2b5-0770-4ca0-b177-61bcd5125db3} \]
			where the last inequality follows from the fact that $\frac{p}{p - 1} > 0$ and $4 \gamma_{\infty} = 92 \ge 1$. Using the definition of $L$ we obtain
			\begin{alignat*}{2}
				\max\ens{1, \left( \frac{\sigma_1}{\eps} \right)^{\frac{p}{p - 1}}} \oversetlab{\eqref{2b14b2b5-0770-4ca0-b177-61bcd5125db3}+\eqref{bf6dc5ba-5913-4acb-b17c-aa57b62b93b3}}&{\ge} \frac{1}{12 \Delta_0 \ell_1} \cdot \min\ens{\frac{L_1 \Delta}{\theta \eps^2}, \frac{\ell_1 \delta \Delta}{\bar{\delta}_1 \eps^2} \theta^{-\frac{1}{q}}} \\
				\oversetlab{\eqref{b3242270-31d2-4d6a-a567-27370b26f605-2-bis}}&{\ge} \frac{1}{24 \Delta_0 \ell_1} \cdot \min\ens{\frac{L_1 \Delta}{\eps^2} + \frac{L_1 \Delta}{\eps^2} \left( \frac{\sigma_1}{4 \gamma_{\infty} \eps} \right)^{\frac{p}{p - 1}}, \frac{\ell_1 \delta \Delta}{\bar{\delta}_1 \eps^2} + \frac{\ell_1 \delta \Delta}{\bar{\delta}_1 \eps^2} \left( \frac{\sigma_1}{4 \gamma_{\infty} \eps} \right)^{\frac{p}{q (p - 1)}}} \\
				& \ge C_{p, q} \min\ens{\frac{L_1 \Delta}{\eps^2} + \frac{L_1 \Delta}{\eps^2} \left( \frac{\sigma_1}{\eps} \right)^{\frac{p}{p - 1}}, \frac{\delta \Delta}{\eps^2} + \frac{\delta \Delta}{\eps^2} \left( \frac{\sigma_1}{\eps} \right)^{\frac{p}{q (p - 1)}}},
			\end{alignat*}
			where $C_{p, q} > 0$ is an universal constant depending only on $p$ and $q$. Moreover, using~\Cref{appdx-lem:lower-bound-global-stichastic-model-mean-squared-smoothness-p-BCM} (and more precisely~\Cref{appdx-rem:improve-lower-bound-global-stichastic-model-mean-squared-smoothness-2-p-BCM}), since $0 < \eps \le \frac{1}{8} \sqrt{L_1 \Delta}$ then there exists an universal constant $C_p > 0$ (which depends only on $p$) such that
			\[ \mathfrak{m}^{\normalfont\texttt{zr}}_{\eps}\left( K, L_1, \Delta, \delta, \sigma_1^p \right) \ge C_p \max\ens{1, \left( \frac{\sigma_1}{\eps} \right)^{\frac{p}{p - 1}}}, \]
			hence
			\begin{alignat*}{2}
				\mathfrak{m}^{\normalfont\texttt{zr}}_{\eps}\left( K, L_1, \Delta, \delta, \sigma_1^p \right) & \ge \Omega(1) \cdot \left( \left( \frac{\sigma_1}{\eps} \right)^{\frac{p}{p - 1}} + \min\ens{\frac{L_1 \Delta}{\eps^2} + \frac{L_1 \Delta}{\eps^2} \left( \frac{\sigma_1}{\eps} \right)^{\frac{p}{p - 1}}, \frac{\delta \Delta}{\eps^2} + \frac{\delta \Delta}{\eps^2} \left( \frac{\sigma_1}{\eps} \right)^{\frac{p}{q (p - 1)}}} \right) \\
				& \ge \Omega(1) \cdot \min\ens{\frac{L_1 \Delta}{\eps^2} + \frac{L_1 \Delta}{\eps^2} \left( \frac{\sigma_1}{\eps} \right)^{\frac{p}{p - 1}}, \left( \frac{\sigma_1}{\eps} \right)^{\frac{p}{p - 1}} + \frac{\delta \Delta}{\eps^2} + \frac{\delta \Delta}{\eps^2} \left( \frac{\sigma_1}{\eps} \right)^{\frac{p}{q (p - 1)}}}, \numberthis\label{5aea445e-d78f-404f-ac66-dbff30f0a00c-2}
			\end{alignat*}
			as $0 < \eps \le c' \sqrt{L_1 \Delta}$ implies
			\[ \frac{L_1 \Delta}{\eps^2} \left( \frac{\sigma_1}{\eps} \right)^{\frac{p}{p - 1}} \gtrsim \left( \frac{\sigma_1}{\eps} \right)^{\frac{p}{p - 1}}, \]
			and we can forget $\left( \frac{\sigma_1}{\eps} \right)^{\frac{p}{p - 1}}$ in the first term of the $\min$.
		\end{itemize}
		
		\paragraph{Step 5:} \textit{A Last Bound: the Case $\theta = 1$}.
		
		Observe that, if instead of taking $\theta$ as in~\eqref{b3242270-31d2-4d6a-a567-27370b26f605-2-bis}, we choose directly $\theta = 1$ then, thanks to~\Cref{appdx-lem:properties-gradient-estimator-1} we immediately have
		\[ \E{\norm{\bar{g}^{\star}_T(x, \xi) - \nabla F^{\star}_T(x)}^p} = 0, \]
		and
		\[ \E{\norm{\left[ \bar{g}^{\star}_T(x, \xi) - \bar{g}^{\star}_T(y, \xi) \right] - \left[ \nabla F_T^{\star}(x) - \nabla F_T^{\star}(y) \right]}^q} = 0, \]
		so~\Cref{ass:p-bounded-central-moment-gradient} and~\Cref{ass:mean-squared-smoothness-2} are satisfied. Hence, if we set 
		\[ T = \Floor{\frac{\Delta}{\alpha \Delta_0}}, \quad \alpha = \frac{L_1}{\beta^2 \ell_1}, \quad \text{ and } \quad \beta = \frac{L_1}{2 \ell_1 \eps}, \]
		then $F_T^{\star} \in \mathcal{F}(\Delta, L_1)$ and the inequality~\eqref{e8299abb-d1db-4ee4-9919-3b49ed05df54} is satisfied. Hence, with probability at least $\frac{1}{2}$ it holds that for all integer $0 \le t \le \frac{T - 1}{2 \theta}$ and all $k \in [K]$ we have
		\[ \norm{\nabla F^{\star}_T\left( x^{(t, k)}_{\texttt{A}[\texttt{O}_F]} \right)} > 2\eps, \,\, \text{ hence } \,\, \E{\norm{\nabla F^{\star}_T\left( x^{(t, k)}_{\texttt{A}[\texttt{O}_F]} \right)}} > \eps, \]
		from where we obtain also
		\begin{alignat*}{2}
			\mathfrak{m}^{\normalfont\texttt{zr}}_{\eps}\left( K, L_1, \Delta, \delta, \sigma_1^p \right) > \frac{T - 1}{2 \theta} = \frac{1}{2 \theta} \left( \Floor{\frac{\Delta}{\alpha \Delta_0}} - 1 \right) = \frac{1}{2 \theta} \left( \Floor{\frac{\beta^2 \ell_1 \Delta}{L_1 \Delta_0}} - 1 \right) = \frac{1}{2 \theta} \left( \Floor{\frac{L_1 \Delta}{4 \Delta_0 \ell_1 \eps^2}} - 1 \right) \numberthis\label{b3242270-31d2-4d6a-a567-27370b26f605-bis-bis}
		\end{alignat*}
		and, since we assume $0 < \eps < c' \sqrt{L_1 \Delta} = \sqrt{\frac{L_1 \Delta}{12 \Delta_0 \ell_1}}$ this implies $\frac{L_1 \Delta}{4 \Delta_0 \ell_1 \eps^2} \ge 3$ thus
		\begin{alignat*}{2}
			\mathfrak{m}^{\normalfont\texttt{zr}}_{\eps}\left( K, L_1, \Delta, \delta, \sigma_1^p \right) \oversetlab{\eqref{b3242270-31d2-4d6a-a567-27370b26f605-bis-bis}}&{>} \frac{1}{4 \theta} \cdot \frac{L_1 \Delta}{4 \Delta_0 \ell_1 \eps^2} \\
			& = \Omega(1) \cdot \frac{L_1 \Delta}{\eps^2},
		\end{alignat*}
		and, combining this bound with~\eqref{5aea445e-d78f-404f-ac66-dbff30f0a00c} and~\eqref{5aea445e-d78f-404f-ac66-dbff30f0a00c-2} respectively leads to the desired result, i.e.,
		\begin{alignat*}{2}
			\mathfrak{m}^{\normalfont\texttt{zr}}_{\eps}\left( K, L_1, \Delta, \delta, \sigma_1^p \right) \ge \Omega(1) \cdot \min\left\{ \frac{L_1 \Delta}{\eps^2} + \frac{L_1 \Delta}{\eps^2} \left( \frac{\sigma_1}{\eps} \right)^{\frac{p}{p - 1}},  \left( \frac{\sigma_1}{\eps} \right)^{\frac{p}{p - 1}} + \frac{(L_1 + \delta) \Delta}{\eps^2} + \frac{\delta \Delta}{\eps^2} \left( \frac{\sigma_1}{\eps} \right)^{\frac{p}{q (p - 1)}} \right\}.
		\end{alignat*}
		
	\end{proof}

	\subsection{Proof of~\Cref{thm:lower-bound-sofo-bounded-central-moments}}\label{appdx-subsec:proof-lower-bound-sofo-bounded-central-moments}
	
	\begin{lemma}[Properties of the Gradient and Hessian Estimators $g_T$ and $\nabla g_T$]\label{appdx-lem:properties-gradient-estimator-2}
		The stochastic gradient estimator $g_T$ is a probability--$\theta$ zero-chain, is unbiased with respect to $\nabla F_T$ and satisfies
		\[ \E{\norm{\nabla f_T (x, \xi) - \nabla F_T(x)}^p} \le \frac{2 \gamma_{\infty}^p (1 - \theta)}{\theta^{p - 1}}, \,\, \text{ and } \,\, \E{\normop{\nabla^2 f_T (x, \xi) - \nabla^2 F_T (x)}^q} \le \frac{2 \ell_1^q (1 - \theta)}{\theta^{q - 1}}, \]
		for all $x \in \R^T$, where $p \in \intof{1}{2}$, $q \in \intff{1}{2}$, $\gamma_{\infty}$ and $\ell_1$ are defined in~\Cref{appdx-lem:properties-hard-instance}.
	\end{lemma}
	
	The proof of~\Cref{appdx-lem:properties-gradient-estimator-2} is very similar to~\Cref{appdx-lem:properties-gradient-estimator-1} (and is simpler since we can directly bound $\mathbb{I}\!\{i > \prog_{\frac{1}{4}}(x)\}$ by $1$). For the bound on the difference $\nabla^2 f_T (x, \xi) - \nabla^2 F_T (x)$ we use~\Cref{appdx-lem:properties-hard-instance}, that is, $F_T$ has $\ell_1$--Lipschitz gradients along with the fact that $F_T$ is twice continuously differentiable which allows to bound the Hessian appropriately (see~\Cref{appdx-lem:lipschitz-gradients-implies-bounded-hessian}). Moreover, by definition of the operator norm $\normop{\cdot}$ (\Cref{appdx-def:operator-norm}) we have
	\begin{alignat*}{2}
		\normop{\left( \left[ \nabla^2 F_T(x) \right]_{i, \cdot} \, \mathbb{I}\!\ens{i > \prog_{\frac{1}{4}}(x)} \right)_{i \in [T]}} \oversetref{Def.}{\ref{appdx-def:operator-norm}}&{=} \sup_{y \in \R^d, \norm{y} = 1} \left( \sum_{i = 1}^T \mathbb{I}\!\ens{i > \prog_{\frac{1}{4}}(x)} \abs{\T{\left[ \nabla^2 F_T(x) \right]_{i, \cdot}} y}^2 \right)^\frac{1}{2} \\
		\oversetrel{rel:46d9f020-9e97-442f-9ab0-af748448e2b1}&{\le} \sup_{y \in \R^d, \norm{y} = 1} \left( \sum_{i = 1}^T \abs{\T{\left[ \nabla^2 F_T(x) \right]_{i, \cdot}} y}^2 \right)^\frac{1}{2} \\
		& = \normop{\nabla^2 F_T(x)} \\
		\oversetref{Lem.}{\ref{appdx-lem:lipschitz-gradients-implies-bounded-hessian}}&{\le} \ell_1,
	\end{alignat*}
	where $\left[ \nabla^2 F_T(x) \right]_{i, \cdot}$ denotes the $i$--th row of the Hessian of $F_T$. In~\relref{rel:46d9f020-9e97-442f-9ab0-af748448e2b1} we use $\mathbb{I}\!\ens{i > \prog_{\frac{1}{4}}(x)} \le 1$.
	
	
	\begin{restate-theorem}{\ref{thm:lower-bound-sofo-bounded-central-moments}}
		Given $\Delta, L_1, L_2 > 0$, $\sigma_1, \sigma_2 \ge 0$ and $0 < \eps \le c_1 \min\{\sqrt{L_1 \Delta}, L_2^{\nicefrac{1}{3}} \Delta^{\nicefrac{2}{3}} \}$ then, for any algorithm $A \in \mathcal{A}_{\texttt{zr}}$, there exists a function $f \in \mathcal{F}\left( \Delta, L_1, L_2 \right)$, an oracle and a distribution $(O, \cD) \in \mathcal{O}\left( f, \sigma_1^p, \sigma_2^q \right)$ such that
		\begin{alignat*}{2}
			\mathfrak{m}^{\normalfont\texttt{zr}}_{\eps}\left( K, \Delta, L_1, L_2, \sigma^p_1, \sigma_2^q \right) \ge \Omega(1) \cdot &\min\left\{\frac{L_1 \Delta}{\eps^2} + \frac{L_1 \Delta}{\eps^2} \left( \frac{\sigma_1}{\eps} \right)^{\frac{p}{p - 1}}, \frac{L_2^{\nicefrac{1}{2}} \Delta}{\eps^{\nicefrac{3}{2}}} + \frac{L_2^{\nicefrac{1}{2}} \Delta}{\eps^{\nicefrac{3}{2}}} \left( \frac{\sigma_1}{\eps} \right)^{\frac{p}{p - 1}}, \right. \\
			&\qquad \left.\min\ens{\frac{L_1 \Delta}{\eps^2}, \frac{L_2^{\nicefrac{1}{2}} \Delta}{\eps^{\nicefrac{3}{2}}}} + \frac{\Delta \sigma_2}{\eps^2} + \frac{\Delta \sigma_2}{\eps^2} \left( \frac{\sigma_1}{\eps} \right)^{\frac{p}{q (p - 1)}} + \left( \frac{\sigma_1}{\eps} \right)^{\frac{p}{p - 1}}\right\}.
		\end{alignat*}
	\end{restate-theorem}
	
	\begin{proof}
		Let $\Delta_0$, $\ell_1$ and $\gamma_{\infty}$ be the numerical constants in~\Cref{appdx-lem:properties-hard-instance} respectively. Additionally, we let the accuracy parameter $\eps > 0$, initial sub-optimality $\Delta \ge 0$, the Lipschitz constants $L_1, L_2 \ge 0$ of the gradients and Hessians of $F$ respectively, and the variance parameters $\sigma_1, \sigma_2 \ge 0$ of the stochastic gradients and Hessians be fixed. Then, for $\alpha, \beta > 0$ two positive real numbers, as in~\citet{arjevani2022lower}, we rescale the function $F_T$ as
		\[ F^{\star}_T \colon x \mapsto \alpha F_T(\beta x). \numberthis\label{6f55ce68-cc03-4d22-a262-4aa3367f948d-2} \]
		
		\begin{itemize}
			\item \textbf{Step 1:} \textit{Ensuring $F^{\star}_T \in \mathcal{F} \left(\Delta, L_1, L_2 \right)$}.
			
			To guarantee the rescaled function $F^{\star}_T$ belongs to the function class $\mathcal{F}(\Delta, L_1, L_2)$, let us compute the initial sub-optimality $\Delta$ and check if $F_T$ has $L_1$--Lipschitz gradients and $L_2$--Lipschitz Hessians. Assuming the algorithm \texttt{A} starts at $x^0 = 0$ we have
			\[ F^{\star}_T(0) - \inf_{x \in \R^T} F^{\star}_T(x) \oversetlab{\eqref{6f55ce68-cc03-4d22-a262-4aa3367f948d-2}}{=} \alpha \left( F_T(0) - \inf_{x \in \R^T} F_T(x) \right) \oversetref{Lem.}{\ref{appdx-lem:properties-hard-instance}}{\le} \alpha \Delta_0 T, \]
			thus, it's enough to take $T = \Floor{\frac{\Delta}{\alpha \Delta_0}}$ so as to ensure $F^{\star}_T(0) - \inf_{x \in \R^T} F^{\star}_T(x) \le \Delta$. Moreover, for any $x, y \in \R^T$,
			\begin{alignat*}{2}
				\norm{\nabla F^{\star}_T(x) - \nabla F^{\star}_T(y)} & = \alpha \beta \norm{\nabla F_T(\beta x) - \nabla F_T(\beta y)} \\
				\oversetref{Lem.}{\ref{appdx-lem:properties-hard-instance}}&{\le} \alpha \beta \ell_1 \norm{\beta x - \beta y} \\
				& = \alpha \beta^2 \ell_1 \norm{x - y},
			\end{alignat*}
			and
			\begin{alignat*}{2}
				\normop{\nabla^2 F^{\star}_T(x) - \nabla^2 F^{\star}_T(y)} & = \alpha \beta^2 \norm{\nabla^2 F_T(\beta x) - \nabla^2 F_T(\beta y)} \\
				\oversetref{Lem.}{\ref{appdx-lem:properties-hard-instance}}&{\le} \alpha \beta^2 \ell_2 \norm{\beta x - \beta y} \\
				& = \alpha \beta^3 \ell_2 \norm{x - y},
			\end{alignat*}
			so it suffices to take $0 < \alpha \beta^2 \le \frac{L_1}{\ell_1}$ and $0 < \alpha \beta^3 \le \frac{L_2}{\ell_2}$ to ensure the function $F^{\star}_T$ has $L_1$--Lipschitz gradients and $L_2$--Lipschitz Hessians. Consequently, we have $F^{\star}_T \in \mathcal{F}(\Delta, L_1, L_2)$, as desired.
			
			\item \textbf{Step 2:} \textit{Analysis of the Protocol and Choice for $\beta$}.
			
			Following the proof of~\citet[Theorem~1]{arjevani2022lower}, according to~\Cref{appdx-lem:properties-gradient-estimator-2}, for all points $x \in \R^T$ such that $\prog_0(x) < T$ we have $\prog_0(\beta x) = \prog_0(x) < T$ so
			\[ \norm{\nabla F^{\star}_T(x)} \oversetlab{\eqref{6f55ce68-cc03-4d22-a262-4aa3367f948d-2}}{=} \alpha \beta \norm{\nabla F_T(\beta x)} \oversetref{Lem.}{\ref{appdx-lem:properties-gradient-estimator-2}}{>} \alpha \beta, \numberthis\label{191cd1ea-8e5f-4b85-8e63-f31fd25f3b17-2} \]
			and we need to guarantee that
			\[ \norm{\nabla F^{\star}_T(x)} > 2\eps, \numberthis\label{f01cde9c-e4f8-4c8b-b9b4-7e86b18da9c1} \]
			for all $x \in \R^T$ with $\prog_0(x) < T$ which, given~\eqref{191cd1ea-8e5f-4b85-8e63-f31fd25f3b17-2}, can be done if we set $\alpha = \frac{2 \eps}{\beta}$.
			
			
			
			\item \textbf{Step 3:} \textit{The Oracle Class and~\Cref{ass:p-bounded-central-moment-gradient,ass:q-bounded-central-moment-hessian}}.
			
			It remains to choose the parameter $\theta \in \intof{0}{1}$ and $\beta > 0$ such that $0 < \alpha \beta^2 = 2 \eps \beta \le \frac{L_1}{\ell_1}$, $0 < \alpha \beta^3 = 2 \eps \beta^2 \le \frac{L_2}{\ell_2}$ and the gradient estimator $g^{\star}_T$ of $\nabla F^{\star}_T$ and Hessian estimator $\nabla g^{\star}_T$ of $\nabla^2 F^{\star}_T$ satisfies~\Cref{ass:p-bounded-central-moment-gradient,ass:q-bounded-central-moment-hessian}. Computing the $p$--th central moment of $g^{\star}_T$ gives, for all $x \in \R^T$
			\begin{alignat*}{2}
				\E{\norm{g^{\star}_T(x, \xi) - \nabla F^{\star}_T(x)}^p} \oversetlab{\eqref{6f55ce68-cc03-4d22-a262-4aa3367f948d}}&{=} \left( \alpha \beta \right)^p \E{\norm{g_T(\beta x, \xi) - \nabla F_T(\beta x)}^p} \\
				\oversetref{Lem.}{\ref{appdx-lem:properties-gradient-estimator-2}}&{\le} \frac{2 \left( \gamma_{\infty} \alpha \beta \right)^p (1 - \theta)}{\theta^{p - 1}} \\
				\oversetrel{rel:0fc4ad15-661f-4519-87ce-6459c8d1c74b-2}&{\le} \frac{\left( 2 \gamma_{\infty} \alpha \beta \right)^p}{\theta^{p - 1}} \\
				\oversetrel{rel:0fc4ad15-661f-4519-87ce-6459c8d1c74c-2}&{=} \frac{\left( 4 \gamma_{\infty} \eps \right)^p}{\theta^{p - 1}}, \numberthis\label{29bb8bb5-e6e5-482d-ab46-fab6357955ec-2}
			\end{alignat*}
			where in~\relref{rel:0fc4ad15-661f-4519-87ce-6459c8d1c74b-2} we use the fact that $0 < \theta \le 1$ and $p \ge 1$ so that $2 \le 2^p$ while in~\relref{rel:0fc4ad15-661f-4519-87ce-6459c8d1c74c-2} we use $\alpha = \frac{2 \eps}{\beta}$. Moreover, following the sames lines as in~\eqref{29bb8bb5-e6e5-482d-ab46-fab6357955ec-2} we have
			\begin{alignat*}{2}
				\E{\norm{\nabla g^{\star}_T(x, \xi) - \nabla^2 F^{\star}_T(x)}^q} \oversetlab{\eqref{6f55ce68-cc03-4d22-a262-4aa3367f948d}}&{=} \left( \alpha \beta^2 \right)^q \E{\norm{\nabla g_T(\beta x, \xi) - \nabla^2 F_T(\beta x)}^q} \\
				\oversetref{Lem.}{\ref{appdx-lem:properties-gradient-estimator-2}}&{\le} \frac{2 \left( \ell_1 \alpha \beta^2 \right)^q (1 - \theta)}{\theta^{q - 1}} \\
				& \le \frac{\left( 2 \ell_1 \alpha \beta^2 \right)^q}{\theta^{q - 1}}. \numberthis\label{29bb8bb5-e6e5-482d-ab46-fab6357955ec-3}
			\end{alignat*}
			
			From~\eqref{29bb8bb5-e6e5-482d-ab46-fab6357955ec-2}, so as to satisfy~\Cref{ass:p-bounded-central-moment-gradient,ass:q-bounded-central-moment-hessian} it's enough to take $\theta = \min\ens{1, \theta_1}$ such that
			\[ \theta_1^{p - 1} \ge \left( \frac{4 \gamma_{\infty} \eps}{\sigma_1} \right)^p, \,\, \text{ so } \,\, \theta_1 \ge \left( \frac{4 \gamma_{\infty} \eps}{\sigma_1} \right)^{\frac{p}{p - 1}}, \numberthis\label{46a820b5-b4a2-4c73-b77f-41e02836f2e0-2} \]
			hence
			\[ \theta = \min\ens{1, \left( \frac{4 \gamma_{\infty} \eps}{\sigma_1} \right)^{\frac{p}{p - 1}}}, \numberthis\label{b3242270-31d2-4d6a-a567-27370b26f605-4} \]
			while for~\eqref{29bb8bb5-e6e5-482d-ab46-fab6357955ec-3} we need to have $0 < \alpha \beta^2 \le \frac{\sigma_2 \theta^{\frac{q - 1}{q}}}{2 \ell_1}$ thus we fix $\beta$ such that
			\[ \alpha \beta^2 = 2 \eps \beta \le \min\ens{\frac{L_1}{\ell_1}, \frac{\sigma_2 \theta^{\frac{q - 1}{q}}}{2 \ell_1}} \,\, \text{ and } \,\, \alpha \beta^3 = 2 \eps \beta^2 \le \frac{L_2}{\ell_2}, \]
			that is to say
			\[ \beta = \min\ens{\frac{L_1}{2 \ell_1 \eps}, \sqrt{\frac{L_2}{2 \ell_2 \eps}}, \frac{\sigma_2 \theta^{\frac{q - 1}{q}}}{4 \ell_1 \eps}}. \numberthis\label{979cbf5c-3220-43c2-aaca-043c521969e6} \]
			
			\item \textbf{Step 4:} \textit{Lower Bounding $\mathfrak{m}^{\normalfont\texttt{zr}}_{\eps}\left( K, \Delta, L_1, L_2, \sigma^p_1, \sigma_2^q \right)$}.
			
			Continuing on \textbf{step 2}, by~\Cref{appdx-lem:zero-chain-progress} we know that with probability at least $\frac{1}{2}$ it holds that for all integer $0 \le t \le \frac{T - 1}{2 \theta}$ and all $k \in [K]$ we have
			\[ \norm{\nabla F^{\star}_T\left( x^{(t, k)}_{\texttt{A}[\texttt{O}_F]} \right)} > 2\eps, \,\, \text{ hence } \,\, \E{\norm{\nabla F^{\star}_T\left( x^{(t, k)}_{\texttt{A}[\texttt{O}_F]} \right)}} > \eps, \]
			from where it follows that
			\begin{alignat*}{2}
				\mathfrak{m}^{\normalfont\texttt{zr}}_{\eps}\left( K, \Delta, L_1, L_2, \sigma^p_1, \sigma_2^q \right) > \frac{T - 1}{2 \theta} = \frac{1}{2 \theta} \left( \Floor{\frac{\Delta}{\alpha \Delta_0}} - 1 \right) & = \frac{1}{2 \theta} \left( \Floor{\frac{\Delta \beta}{2 \Delta_0 \eps}} - 1 \right), \numberthis\label{b3242270-31d2-4d6a-a567-27370b26f605-3}
			\end{alignat*}
			we then distinguish two cases:
			
			\paragraph{Case 1:} if $\frac{\Delta \beta}{2 \Delta_0 \eps} \ge 3$ then, using the inequality $\Floor{x} - 1 \ge \frac{x}{2}$, valid for all real number $x \ge 3$ we obtain
			\begin{alignat*}{2}
				&\mathfrak{m}^{\normalfont\texttt{zr}}_{\eps}\left( K, \Delta, L_1, L_2, \sigma^p_1, \sigma_2^q \right) 
				\oversetlab{\eqref{b3242270-31d2-4d6a-a567-27370b26f605-3}}{>} \frac{1}{4 \theta} \cdot \frac{\Delta \beta}{2 \Delta_0 \eps}  = \frac{1}{16 \Delta_0 \ell_1} \cdot \frac{\Delta}{\eps} \min\ens{\frac{L_1}{\eps \theta}, \frac{1}{\theta} \sqrt{\frac{2 L_2}{\ell_2 \eps}}, \frac{\sigma_2 \theta^{-\frac{1}{q}}}{2 \eps}} \\
				&\qquad\oversetlab{\eqref{b3242270-31d2-4d6a-a567-27370b26f605-4}}{\ge} \frac{1}{32 \Delta_0 \ell_1} \cdot \frac{\Delta}{\eps} \min\Bigg\{\frac{L_1}{\eps} \left( 1 + \left( \frac{\sigma_1}{4 \gamma_{\infty} \eps} \right)^{\frac{p}{p - 1}} \right),\\
				&\qquad\qquad\qquad\qquad  \sqrt{\frac{2 L_2}{\ell_2 \eps}} \left( 1 + \left( \frac{\sigma_1}{4 \gamma_{\infty} \eps} \right)^{\frac{p}{p - 1}} \right),\frac{\sigma_2}{2 \eps} \left( 1 + \left( \frac{\sigma_1}{4 \gamma_{\infty} \eps} \right)^{\frac{p}{q (p - 1)}} \right)\Bigg\} \\
				& = \Omega(1) \cdot \min\Bigg\{\frac{L_1 \Delta}{\eps^2} + \frac{L_1 \Delta}{\eps^2} \left( \frac{\sigma_1}{\eps} \right)^{\frac{p}{p - 1}}, \frac{L_2^{\nicefrac{1}{2}} \Delta}{\eps^{\nicefrac{3}{2}}} + \frac{L_2^{\nicefrac{1}{2}} \Delta}{\eps^{\nicefrac{3}{2}}} \left( \frac{\sigma_1}{\eps} \right)^{\frac{p}{p - 1}},\\
				&\qquad\qquad\qquad\qquad\frac{\Delta \sigma_2}{\eps^2} + \frac{\Delta \sigma_2}{\eps^2} \left( \frac{\sigma_1}{\eps} \right)^{\frac{p}{q (p - 1)}}\Bigg\},
				\numberthis\label{3b256059-e30a-4307-b1b7-a8e09886c0f1-2}
			\end{alignat*}
			and combining the lower bound~\eqref{3b256059-e30a-4307-b1b7-a8e09886c0f1-2} with~\Cref{appdx-lem:lower-bound-global-stichastic-model-mean-squared-smoothness-p-BCM}, using the choice for $c'$ provided below, gives:
			\begin{alignat*}{2}
				&\mathfrak{m}^{\normalfont\texttt{zr}}_{\eps}\left( K, L_1, L_2, \Delta, \sigma_1^p, \sigma_2^q \right) \ge \Omega(1) \cdot \min\Bigg\{\frac{L_1 \Delta}{\eps^2} + \frac{L_1 \Delta}{\eps^2} \left( \frac{\sigma_1}{\eps} \right)^{\frac{p}{p - 1}}, \frac{L_2^{\nicefrac{1}{2}} \Delta}{\eps^{\nicefrac{3}{2}}} + \frac{L_2^{\nicefrac{1}{2}} \Delta}{\eps^{\nicefrac{3}{2}}} \left( \frac{\sigma_1}{\eps} \right)^{\frac{p}{p - 1}},\\
				&\qquad\qquad\qquad\qquad\qquad\qquad\qquad \qquad\qquad\left( \frac{\sigma_1}{\eps} \right)^{\frac{p}{p - 1}} + \frac{\Delta \sigma_2}{\eps^2} + \frac{\Delta \sigma_2}{\eps^2} \left( \frac{\sigma_1}{\eps} \right)^{\frac{p}{q (p - 1)}}\Bigg\}, \numberthis\label{5aea445e-d78f-404f-ac66-dbff30f0a00c-3}
			\end{alignat*}
			since from the assumption $0 < \eps \le c' \min\ens{\sqrt{L_1 \Delta}, L_2^{\nicefrac{1}{3}} \Delta^{\nicefrac{2}{3}}}$ we have
			\[ \frac{L_1 \Delta}{\eps^2} \left( \frac{\sigma_1}{\eps} \right)^{\frac{p}{p - 1}} \gtrsim \left( \frac{\sigma_1}{\eps} \right)^{\frac{p}{p - 1}} \,\, \text{ and } \,\, \frac{L_2^{\nicefrac{1}{2}} \Delta}{\eps^{\nicefrac{3}{2}}} \left( \frac{\sigma_1}{\eps} \right)^{\frac{p}{p - 1}} \gtrsim \left( \frac{\sigma_1}{\eps} \right)^{\frac{p}{p - 1}}, \]
			and we can forget $\left( \frac{\sigma_1}{\eps} \right)^{\frac{p}{p - 1}}$ in the first two terms of the $\min$.
			
			\paragraph{Case 2:} if $\frac{\Delta \beta}{2 \Delta_0 \eps} < 3$, we choose the universal constant $c'$ as
			\[ 0 < c' = \min\ens{\left( 12 \Delta_0 \ell_1 \right)^{-\frac{1}{2}}, \left( 6 \sqrt{2} \Delta_0 \ell_2^{\nicefrac{1}{2}} \right)^{-\frac{2}{3}}} < \frac{1}{8}, \numberthis\label{e4896e63-043b-4b68-b23e-bb9b172efed4} \]
			then, using~\Cref{appdx-rem:improve-lower-bound-global-stichastic-model-mean-squared-smoothness-p-BCM}, dividing both sides of $\frac{\Delta \beta}{2 \Delta_0 \eps} < 3$ by our choice of $\theta > 0$ yields
			\[ \frac{1}{6 \Delta_0} \cdot \frac{\Delta \beta}{\eps} \theta^{-1} \le \theta^{-1} = \max\ens{1, \left( \frac{\sigma_1}{4 \gamma_{\infty} \eps} \right)^{\frac{p}{p - 1}}} \le \max\ens{1, \left( \frac{\sigma_1}{\eps} \right)^{\frac{p}{p - 1}}}, \numberthis\label{2b14b2b5-0770-4ca0-b177-61bcd5125db3-2} \]
			where the last inequality follows from the fact that $\frac{p}{p - 1} > 0$ and $4 \gamma_{\infty} = 92 \ge 1$. Using the definition of $\beta$ we obtain
			\begin{alignat*}{2}
				&\max\ens{1, \left( \frac{\sigma_1}{\eps} \right)^{\frac{p}{p - 1}}} \oversetlab{\eqref{2b14b2b5-0770-4ca0-b177-61bcd5125db3-2}+\eqref{979cbf5c-3220-43c2-aaca-043c521969e6}}{\ge} \frac{1}{6 \Delta_0} \cdot \min\ens{\frac{L_1 \Delta}{2 \ell_1 \theta \eps^2}, \frac{\Delta}{\theta \eps}\sqrt{\frac{L_2}{2 \ell_2 \eps}}, \frac{\Delta \sigma_2 \theta^{-\frac{1}{q}}}{4 \ell_1 \eps^2}} \\
				\oversetlab{\eqref{b3242270-31d2-4d6a-a567-27370b26f605-4}}&{\ge} \frac{1}{24 \Delta_0 \max\ens{\ell_1, \ell_2^{\nicefrac{1}{2}}}} \cdot \min\left\{\frac{L_1 \Delta}{\eps^2} + \frac{L_1 \Delta}{\eps^2} \left( \frac{\sigma_1}{4 \gamma_{\infty} \eps} \right)^{\frac{p}{p - 1}}, \right. \\
				&\qquad\qquad\qquad\qquad\qquad\qquad\qquad\quad\left. \frac{L_2^{\nicefrac{1}{2}} \Delta}{\eps^{\nicefrac{3}{2}}} + \frac{L_2^{\nicefrac{1}{2}} \Delta}{\eps^{\nicefrac{3}{2}}} \left( \frac{\sigma_1}{4 \gamma_{\infty} \eps} \right)^{\frac{p}{p - 1}}, \frac{\Delta \sigma_2}{\eps^2} + \frac{\Delta \sigma_2}{\eps^2} \left( \frac{\sigma_1}{4 \gamma_{\infty} \eps} \right)^{\frac{p}{q (p - 1)}}\right\} \\
				& \ge C_{p, q} \min\ens{\frac{L_1 \Delta}{\eps^2} + \frac{L_1 \Delta}{\eps^2} \left( \frac{\sigma_1}{\eps} \right)^{\frac{p}{p - 1}}, \frac{L_2^{\nicefrac{1}{2}} \Delta}{\eps^{\nicefrac{3}{2}}} + \frac{L_2^{\nicefrac{1}{2}} \Delta}{\eps^{\nicefrac{3}{2}}} \left( \frac{\sigma_1}{\eps} \right)^{\frac{p}{p - 1}}, \frac{\Delta \sigma_2}{\eps^2} + \frac{\Delta \sigma_2}{\eps^2} \left( \frac{\sigma_1}{\eps} \right)^{\frac{p}{q (p - 1)}}},
			\end{alignat*}
			where $C_{p, q} > 0$ is an universal constant depending only on $p$ and $q$. Moreover, using~\Cref{appdx-lem:lower-bound-global-stichastic-model-mean-squared-smoothness-p-BCM} (and more precisely~\Cref{appdx-rem:improve-lower-bound-global-stichastic-model-mean-squared-smoothness-2-hessian}), since $0 < \eps \le \frac{1}{8} \sqrt{L_1 \Delta}$ then there exists an universal constant $C_p > 0$ (which depends only on $p$) such that
			\[ \mathfrak{m}^{\normalfont\texttt{zr}}_{\eps}\left( K, L_1, L_2, \Delta, \sigma_1^p, \sigma_2^q \right) \ge C_p \max\ens{1, \left( \frac{\sigma_1}{\eps} \right)^{\frac{p}{p - 1}}}, \]
			hence
			\begin{alignat*}{2}
				&\mathfrak{m}^{\normalfont\texttt{zr}}_{\eps}\left( K, L_1, L_2, \Delta, \sigma_1^p, \sigma_2^q \right) \\
				&\ge \Omega(1) \cdot \Bigg( \left( \frac{\sigma_1}{\eps} \right)^{\frac{p}{p - 1}} + \min\Bigg\{\frac{L_1 \Delta}{\eps^2} + \frac{L_1 \Delta}{\eps^2} \left( \frac{\sigma_1}{\eps} \right)^{\frac{p}{p - 1}}, \\
				&\qquad\qquad\qquad\qquad\qquad\qquad\qquad \frac{L_2^{\nicefrac{1}{2}} \Delta}{\eps^{\nicefrac{3}{2}}} + \frac{L_2^{\nicefrac{1}{2}} \Delta}{\eps^{\nicefrac{3}{2}}} \left( \frac{\sigma_1}{\eps} \right)^{\frac{p}{p - 1}}, \frac{\Delta \sigma_2}{\eps^2} + \frac{\Delta \sigma_2}{\eps^2} \left( \frac{\sigma_1}{\eps} \right)^{\frac{p}{q (p - 1)}}\Bigg\} \Bigg) \\
				& \ge \Omega(1) \cdot \min\Bigg\{\frac{L_1 \Delta}{\eps^2} + \frac{L_1 \Delta}{\eps^2} \left( \frac{\sigma_1}{\eps} \right)^{\frac{p}{p - 1}}, \frac{L_2^{\nicefrac{1}{2}} \Delta}{\eps^{\nicefrac{3}{2}}} + \frac{L_2^{\nicefrac{1}{2}} \Delta}{\eps^{\nicefrac{3}{2}}} \left( \frac{\sigma_1}{\eps} \right)^{\frac{p}{p - 1}}, \\
				&\qquad\qquad\qquad\qquad\qquad\qquad\qquad
				\left( \frac{\sigma_1}{\eps} \right)^{\frac{p}{p - 1}} + \frac{\Delta \sigma_2}{\eps^2} + \frac{\Delta \sigma_2}{\eps^2} \left( \frac{\sigma_1}{\eps} \right)^{\frac{p}{q (p - 1)}}\Bigg\}, \numberthis\label{5aea445e-d78f-404f-ac66-dbff30f0a00c-4}
			\end{alignat*}
			as $0 < \eps \le c' \min\ens{\sqrt{L_1 \Delta}, L_2^{\nicefrac{1}{3}} \Delta^{\nicefrac{2}{3}}}$ implies
			\[ \frac{L_1 \Delta}{\eps^2} \left( \frac{\sigma_1}{\eps} \right)^{\frac{p}{p - 1}} \gtrsim \left( \frac{\sigma_1}{\eps} \right)^{\frac{p}{p - 1}} \,\, \text{ and } \,\, \frac{L_2^{\nicefrac{1}{2}} \Delta}{\eps^{\nicefrac{3}{2}}} \left( \frac{\sigma_1}{\eps} \right)^{\frac{p}{p - 1}} \gtrsim \left( \frac{\sigma_1}{\eps} \right)^{\frac{p}{p - 1}}, \]
			and we can forget $\left( \frac{\sigma_1}{\eps} \right)^{\frac{p}{p - 1}}$ in the first two terms of the $\min$.

			\item \textbf{Step 5:} \textit{A Last Bound: the Case $\theta = 1$}.
			
			Observe that, if instead of taking $\theta$ as in~\eqref{b3242270-31d2-4d6a-a567-27370b26f605-4}, we choose directly $\theta = 1$ then, thanks to~\Cref{appdx-lem:properties-gradient-estimator-2} we immediately have
			\[ \E{\norm{g^{\star}_T(x, \xi) - \nabla F^{\star}_T(x)}^p} = 0, \]
			and
			\[ \E{\norm{\nabla g^{\star}_T(x, \xi) - \nabla^2 F^{\star}_T(x)}^q} = 0, \]
			so~\Cref{ass:p-bounded-central-moment-gradient} and~\Cref{ass:q-bounded-central-moment-hessian} are satisfied. Hence, if we set 
			\[ T = \Floor{\frac{\Delta}{\alpha \Delta_0}}, \quad \alpha = \frac{2 \eps}{\beta}, \quad \text{ and } \quad \beta = \min\ens{\frac{L_1}{2 \ell_1 \eps}, \sqrt{\frac{L_2}{2 \ell_2 \eps}}}, \]
			then $F_T^{\star} \in \mathcal{F}(\Delta, L_1, L_2)$ and the inequality~\eqref{f01cde9c-e4f8-4c8b-b9b4-7e86b18da9c1} is satisfied. Hence, with probability at least $\frac{1}{2}$ it holds that for all integer $0 \le t \le \frac{T - 1}{2 \theta}$ and all $k \in [K]$ we have
			\[ \norm{\nabla F^{\star}_T\left( x^{(t, k)}_{\texttt{A}[\texttt{O}_F]} \right)} > 2\eps, \,\, \text{ hence } \,\, \E{\norm{\nabla F^{\star}_T\left( x^{(t, k)}_{\texttt{A}[\texttt{O}_F]} \right)}} > \eps, \]
			from where we obtain also
			\begin{alignat*}{2}
				\mathfrak{m}^{\normalfont\texttt{zr}}_{\eps}\left( K, \Delta, L_1, L_2, \sigma^p_1, \sigma_2^q \right) > \frac{T - 1}{2 \theta} = \frac{1}{2 \theta} \left( \Floor{\frac{\Delta}{\alpha \Delta_0}} - 1 \right) & = \frac{1}{2 \theta} \left( \Floor{\frac{\Delta \beta}{2 \Delta_0 \eps}} - 1 \right), \numberthis\label{b3242270-31d2-4d6a-a567-27370b26f605-bis-bis-bis}
			\end{alignat*}
			and by our assumption on $\eps$ (see~\eqref{e4896e63-043b-4b68-b23e-bb9b172efed4}), we assumed
			\[ 0 < \eps < c' \min\ens{\sqrt{L_1 \Delta}, L_2^{\nicefrac{1}{3}} \Delta^{\nicefrac{2}{3}}} \le \min\ens{\sqrt{\frac{L_1 \Delta}{12 \Delta_0 \ell_1}}, \left( \frac{L_2^{\nicefrac{1}{2}} \Delta}{6 \sqrt{2} \Delta_0 \ell_2^{\nicefrac{1}{2}}} \right)^{\frac{2}{3}}}, \]
			which is enough to imply the inequality
			\[ \min\ens{\frac{L_1 \Delta}{4 \Delta_0 \ell_1 \eps^2}, \frac{L_2^{\nicefrac{1}{2}} \Delta}{2 \sqrt{2} \Delta_0 \ell_2^{\nicefrac{1}{2}} \eps^{\nicefrac{3}{2}}}} = \frac{\Delta \beta}{2 \Delta_0 \eps} \ge 3, \]
			thus, we have
			\begin{alignat*}{2}
				\mathfrak{m}^{\normalfont\texttt{zr}}_{\eps}\left( K, L_1, L_2, \Delta, \sigma_1^p, \sigma_2^q \right) \oversetlab{\eqref{b3242270-31d2-4d6a-a567-27370b26f605-bis-bis-bis}}&{>} \frac{1}{4 \theta} \cdot \frac{\Delta \beta}{2 \Delta_0 \eps^2} \\
				& = \frac{1}{8 \Delta_0} \cdot \frac{\Delta}{\eps}\min\ens{\frac{L_1}{2 \ell_1 \eps}, \sqrt{\frac{L_2}{2 \ell_2 \eps}}} \\
				& = \Omega(1) \cdot \min\ens{\frac{L_1 \Delta}{\eps^2}, \frac{L_2^{\nicefrac{1}{2}} \Delta}{\eps^{\nicefrac{3}{2}}}},
			\end{alignat*}
			and, combining this bound with~\eqref{5aea445e-d78f-404f-ac66-dbff30f0a00c-3} and~\eqref{5aea445e-d78f-404f-ac66-dbff30f0a00c-4} respectively leads to the desired result, i.e.,
			\begin{alignat*}{2}
				&\mathfrak{m}^{\normalfont\texttt{zr}}_{\eps}\left( K, \Delta, L_1, L_2, \sigma^p_1, \sigma_2^q \right)\\
				&\quad\ge \Omega(1) \cdot \min\left\{\frac{L_1 \Delta}{\eps^2} + \frac{L_1 \Delta}{\eps^2} \left( \frac{\sigma_1}{\eps} \right)^{\frac{p}{p - 1}}, \frac{L_2^{\nicefrac{1}{2}} \Delta}{\eps^{\nicefrac{3}{2}}} + \frac{L_2^{\nicefrac{1}{2}} \Delta}{\eps^{\nicefrac{3}{2}}} \left( \frac{\sigma_1}{\eps} \right)^{\frac{p}{p - 1}}, \right. \\
				&\qquad\qquad\qquad \left.\min\ens{\frac{L_1 \Delta}{\eps^2}, \frac{L_2^{\nicefrac{1}{2}} \Delta}{\eps^{\nicefrac{3}{2}}}} + \frac{\Delta \sigma_2}{\eps^2} + \frac{\Delta \sigma_2}{\eps^2} \left( \frac{\sigma_1}{\eps} \right)^{\frac{p}{q (p - 1)}} + \left( \frac{\sigma_1}{\eps} \right)^{\frac{p}{p - 1}}\right\}.
			\end{alignat*}
		\end{itemize}
	\end{proof}
	
	\newpage
	\section{Missing Proofs in Section~\ref{sec:optimal-method-nsgd-mvr}}\label{appdx:proof-nsgd-mvr}
	
	\subsection{Auxiliary Lemmas}
	
	\begin{lemma}[A Descent Lemma]\label{appdx-lem:descent-lemma-nsgd-mvr}
		Under~\Cref{ass:mean-squared-smoothness,ass:p-bounded-central-moment-gradient}, for any choice of stepsize $\gamma > 0$, and for any $t \in \Int{0}{T - 1}$ we have
		\[ F(x_{t + 1}) \le F(x_t) + 2 \gamma \norm{\hat{e_t}} - \gamma \norm{\nabla F(x_t)} + \frac{\gamma^2 \bar{L}}{2}, \numberthis\label{2331e9e8-3c3e-4a5b-b3b6-50d71097a688} \]
		where $\hat{e}_t \eqdef g_t - \nabla F(x_t)$ is the error term.
	\end{lemma}
	
	\begin{proof}
		According to~\Cref{ass:mean-squared-smoothness} we know that the function $f$ is $\bar{L}$--smooth~\citep{nesterov2018lectures} since by Jensen's inequality (\Cref{lem:jensen-inequality}) applied on the convex function $x \mapsto \norm{x}^q$ (since $q \ge 1$)
		\begin{eqnarray}
			\norm{\nabla F(x) - \nabla F(y)}^q &\oversetref{Ass.}{\ref{ass:p-bounded-central-moment-gradient}}{=}& \norm{\ExpSub{\xi \sim \cD}{\nabla f(x, \xi) - \nabla f(y, \xi)}}^q \notag\\
			&\oversetref{Lem.}{\ref{lem:jensen-inequality}}{\le}& \ExpSub{\xi \sim \cD}{\norm{\nabla f(x, \xi) - \nabla f(y, \xi)}^q}\notag\\ 
			&\oversetref{Ass.}{\ref{ass:mean-squared-smoothness}}{\le}& \bar{L}^q \norm{x - y}^q, \numberthis\label{94828cd2-cf02-492b-a6ec-40b3012f66c1}
		\end{eqnarray}
		thus it holds
		\begin{alignat*}{2}
			F(x_{t + 1}) & \le F(x_t) + \ps{\nabla F(x_t)}{x_{t + 1} - x_t} + \frac{\bar{L}}{2}\sqnorm{x_{t + 1} - x_t} \\
			\oversetrel{rel:f31537d4-b478-4ead-af5c-ada0fbf2d600}&{=} F(x_t) - \gamma \ps{\nabla F(x_t)}{\frac{g_t}{\norm{g_t}}} + \frac{\gamma^2 \bar{L}}{2} \\
			& = F(x_t) - \gamma \ps{\nabla F(x_t) - g_t}{\frac{g_t}{\norm{g_t}}} - \gamma \norm{g_t} + \frac{\gamma^2 \bar{L}}{2} \\
			\oversetref{Lem.}{\ref{lem:cauchy-schwarz}}&{\le} F(x_t) + \gamma \norm{\nabla F(x_t) - g_t} - \gamma \norm{g_t} + \frac{\gamma^2 \bar{L}}{2} \\
			\oversetrel{rel:93f220b1-6765-4926-bdd5-2a465a2ff721}&{\le} F(x_t) + 2 \gamma \norm{\nabla F(x_t) - g_t} - \gamma \norm{\nabla F(x_t)} + \frac{\gamma^2 \bar{L}}{2} \\
			& = F(x_t) + 2 \gamma \norm{\hat{e_t}} - \gamma \norm{\nabla F(x_t)} + \frac{\gamma^2 \bar{L}}{2},
		\end{alignat*}
		where in~\relref{rel:f31537d4-b478-4ead-af5c-ada0fbf2d600} we use the update rule $x_{t + 1} = x_t - \gamma \frac{g_t}{\norm{g_t}}$ while in~\relref{rel:93f220b1-6765-4926-bdd5-2a465a2ff721} we use the triangle inequality. This establishes the desired claim.
	\end{proof}
	
	\begin{lemma}[Another Descent Lemma]\label{appdx-lem:descent-lemma-nsgd-mvr-hess}
		Under~\Cref{ass:L-lipschitz-gradients,ass:p-bounded-central-moment-gradient}, for any choice of stepsize $\gamma > 0$, and for any $t \in \Int{0}{T - 1}$ we have
		\[ F(x_{t + 1}) \le F(x_t) + 2 \gamma \norm{\hat{e_t}} - \gamma \norm{\nabla F(x_t)} + \frac{\gamma^2 L_1}{2}, \numberthis\label{2331e9e8-3c3e-4a5b-b3b6-50d71097a688-2} \]
		where $\hat{e}_t \eqdef g_t - \nabla F(x_t)$ is the error term.
	\end{lemma}
	
	\begin{proof}
		The proof is the same as in the previous descent lemma (\Cref{appdx-lem:descent-lemma-nsgd-mvr}) where now $F$ has $L_1$--Lipschitz continuous gradients (instead of $\bar{L}$).
	\end{proof}
	
	\begin{lemma}[Unrolling the Descent Lemma]\label{appdx-lem:unrolling-descent-lemma-nsgd-mvr}
		Under~\Cref{ass:lower-boundedness,ass:p-bounded-central-moment-gradient,ass:mean-squared-smoothness}, for any choice of stepsize $\gamma > 0$ the iterates $\{x_t\}_{t \in \Int{0}{T}}$ produced by~\Cref{algo:nsgd-mvr} satisfy
		\[ \frac{1}{T} \sum_{t = 0}^{T - 1} \norm{\nabla F(x_t)} \le \frac{\Delta}{\gamma T} + \frac{2}{T} \sum_{t = 0}^{T - 1} \norm{\hat{e}_t} + \frac{\gamma \bar{L}}{2}. \]
	\end{lemma}
	
	\begin{proof}
		From the previous descent lemma (\Cref{appdx-lem:descent-lemma-nsgd-mvr}), summing inequality~\eqref{2331e9e8-3c3e-4a5b-b3b6-50d71097a688} over $t \in \Int{0}{T - 1}$ gives
		\[ \gamma \sum_{t = 0}^{T - 1} \norm{\nabla F(x_t)} \le F(x_0) - F(x_T) + 2 \gamma \sum_{t = 0}^{T - 1} \norm{\hat{e}_t} + \frac{\gamma^2 \bar{L} T}{2}, \numberthis\label{eaec4407-2c6f-461b-805b-b59a34c6eee9} \]
		where we telescope the terms $F(x_t) - F(x_{t + 1})$. Multiplying both sides of~\eqref{eaec4407-2c6f-461b-805b-b59a34c6eee9} by $\nicefrac{1}{\gamma T}$ leads to
		\[ \frac{1}{T} \sum_{t = 0}^{T - 1} \norm{\nabla F(x_t)} \le \frac{1}{\gamma T} \left( F(x_0) - F(x_T)  \right) + \frac{2}{T} \sum_{t = 0}^{T - 1} \norm{\hat{e}_t} + \frac{\gamma \bar{L}}{2}, \]
		and using~\Cref{ass:lower-boundedness} we obtain
		\[ \frac{1}{T} \sum_{t = 0}^{T - 1} \norm{\nabla F(x_t)} \le \frac{\Delta}{\gamma T} + \frac{2}{T} \sum_{t = 0}^{T - 1} \norm{\hat{e}_t} + \frac{\gamma \bar{L}}{2}, \]
		as desired.
	\end{proof}
	
	If we assume the function $F$ has $L_1$--Lipschitz continuous gradients (\Cref{ass:L-lipschitz-gradients}) then~\Cref{appdx-lem:descent-lemma-nsgd-mvr-hess} holds and we can unroll it in the same way as we did above. For that reason, we only state the result and we omit the proof.
	\begin{lemma}[Unrolling the Descent Lemma]\label{appdx-lem:unrolling-descent-lemma-nsgd-mvr-hess}
		Under~\Cref{ass:lower-boundedness,ass:L-lipschitz-gradients,ass:p-bounded-central-moment-gradient}, for any choice of stepsize $\gamma > 0$ the iterates $\{x_t\}_{t \in \Int{0}{T}}$ produced by~\Cref{algo:nsgd-mvr} satisfy
		\[ \frac{1}{T} \sum_{t = 0}^{T - 1} \norm{\nabla F(x_t)} \le \frac{\Delta}{\gamma T} + \frac{2}{T} \sum_{t = 0}^{T - 1} \norm{\hat{e}_t} + \frac{\gamma L_1}{2}. \]
	\end{lemma}
	
	\begin{lemma}[Bounding the Error Term]\label{appdx-lem:bounding-error-term-nsgd-mvr}
		Under~\Cref{ass:p-bounded-central-moment-gradient,ass:mean-squared-smoothness}, for all $t \in \Int{0}{T - 1}$ we have
		\[ \E{\norm{\hat{e}_t}} \le (1 - \alpha)^t \E{\norm{\hat{e}_0}} + 2 \sigma_1 \alpha^{\frac{p - 1}{p}} + 4 \gamma \bar{L} \alpha^{-\frac{1}{q}}, \]
		where $\hat{e}_t \eqdef g_t - \nabla F(x_t)$.
	\end{lemma}
	
	\begin{proof}
		By the update rule of the gradient estimator in~\Cref{algo:nsgd-mvr} (line 9) we have, for all $t \in \Int{1}{T - 1}$
		\begin{alignat*}{2}
			\hat{e}_t \eqdef* g_t - \nabla F(x_t) \\
			& = (1 - \alpha) \left(g_{t - 1} + \nabla f(x_t, \xi_t) - \nabla f(x_{t - 1}, \xi_t)\right) + \alpha \nabla f(x_t, \xi_t) - \nabla F(x_t) \\
			& = (1 - \alpha) (g_{t - 1} - \nabla F(x_{t - 1})) + \alpha \left( \nabla f(x_t, \xi_t) - \nabla F(x_t) \right) \\
			&\qquad- (1 - \alpha) \left( \left[ \nabla F(x_t) - \nabla f(x_t, \xi_t) \right] - \left[ \nabla F(x_{t - 1}) - \nabla f(x_{t - 1}, \xi_t) \right] \right) \\
			& = (1 - \alpha) \hat{e}_{t - 1} + \alpha e_t - (1 - \alpha) \hat{S}_t \numberthis\label{2f56c833-5913-40e8-a00c-20c7b1a4b32b},
		\end{alignat*}
		where we let $e_t \eqdef \nabla f(x_t, \xi_t) - \nabla F(x_t)$ and $\hat{S}_t \eqdef \left[ \nabla F(x_t) - \nabla f(x_t, \xi_t) \right] - \left[ \nabla F(x_{t - 1}) - \nabla f(x_{t - 1}, \xi_t) \right]$. It is worth noting that both $(e_t)_{t \ge 0}$ and $(\hat{S}_t)_{t \ge 0}$ are martingale difference sequence with respect to the filtration $(\mathcal{F}_t)_{t \ge 0}$ where $\mathcal{F}_t \eqdef \sigma(g_0, \xi_1, \ldots, \xi_t)$.
		
		Then, unrolling the recursion~\eqref{2f56c833-5913-40e8-a00c-20c7b1a4b32b} gives
		\[ \hat{e}_t = (1 - \alpha)^t \hat{e}_0 + \alpha \sum_{j = 0}^{t - 1} (1 - \alpha)^{t - j - 1} e_{j + 1} - \sum_{j = 0}^{t - 1} (1 - \alpha)^{t  - j} \hat{S}_{j + 1}, \]
		and taking the norm followed by the total expectation yields
		\begin{alignat*}{2}
			\E{\norm{\hat{e}_t}} \oversetrel{rel:}&{\le} (1 - \alpha)^t \E{\norm{\hat{e}_0}} + \alpha \E{\norm{ \sum_{j = 0}^{t - 1} (1 - \alpha)^{t - j - 1} e_{j + 1}}} + (1 - \alpha) \E{\norm{\sum_{j = 0}^{t - 1} (1 - \alpha)^{t  - j - 1} \hat{S}_{j + 1}}}.
		\end{alignat*}
		We now need to upper bound the last two terms of the previous inequality. For the first term, using Jensen's inequality (\Cref{lem:jensen-inequality}) we have
		\begin{alignat*}{2}
			\E{\norm{ \sum_{j = 0}^{t - 1} (1 - \alpha)^{t - j - 1} e_{j + 1}}} \oversetref{Lem.}{\ref{lem:jensen-inequality}}&{\le} \left( \E{\norm{ \sum_{j = 0}^{t - 1} (1 - \alpha)^{t - j - 1} e_{j + 1}}^p} \right)^{\frac{1}{p}} \\
			\oversetref{Lem.}{\ref{appdx-technical-lem:app-von-Bahr-and-Essen}}&{\le} \left( 2 \sum_{j = 0}^{t - 1} (1 - \alpha)^{p (t - j -1)} \E{\norm{e_{j + 1}}^p} \right)^{\frac{1}{p}} \\
			\oversetref{Ass.}{\ref{ass:p-bounded-central-moment-gradient}}&{\le} 2 \left( \sum_{j = 0}^{t - 1} (1 - \alpha)^{p (t - j -1)} \sigma_1^p \right)^{\frac{1}{p}} \\
			& \le 2 \sigma_1 \left( \sum_{j = 0}^{t - 1} (1 - \alpha)^{t - j - 1} \right)^{\frac{1}{p}} \\
			& \le 2 \sigma_1 \left( \sum_{j \ge 0} (1 - \alpha)^j \right)^{\frac{1}{p}} \\
			& = 2 \sigma_1 \alpha^{-\frac{1}{p}}, \numberthis\label{cea5881d-8688-41f3-8083-df58a5981198}
		\end{alignat*}
		while, for the last term we have
		\begin{alignat*}{2}
			\E{\norm{\sum_{j = 0}^{t - 1} (1 - \alpha)^{t  - j - 1} \hat{S}_{j + 1}}} \oversetref{Lem.}{\ref{lem:jensen-inequality}}&{\le} \left( \E{\norm{\sum_{j = 0}^{t - 1} (1 - \alpha)^{t  - j - 1} \hat{S}_{j + 1}}^q} \right)^{\frac{1}{q}} \\
			\oversetref{Lem.}{\ref{appdx-technical-lem:app-von-Bahr-and-Essen}}&{\le} \left( 2 \sum_{j = 0}^{t - 1} (1 - \alpha)^{q (t - j - 1)} \E{\norm{\hat{S}_{j + 1}}^q} \right)^{\frac{1}{q}}, \numberthis\label{d98da1f4-086e-4004-a809-2eb471da42a6}
		\end{alignat*}
		and, using~\Cref{ass:mean-squared-smoothness} we obtain
		\begin{alignat*}{2}
			\E{\norm{\hat{S}_{j + 1}}^q} & = \E{\norm{\left[ \nabla F(x_{j + 1}) - \nabla F(x_j) \right] - \left[ \nabla f(x_{j + 1}, \xi_{j + 1}) - \nabla f(x_j, \xi_{j + 1}) \right]}^q} \\
			\oversetref{Lem.}{\ref{lem:norm-power-alpha-inequality}}&{\le} 2^{q - 1} \E{\norm{\nabla F(x_{j + 1}) - \nabla F(x_j)}^q + \norm{\nabla f(x_{j + 1}, \xi_ {j + 1}) - \nabla f(x_j, \xi_{j + 1})}^q} \\
			\oversetref{Ass.}{\ref{ass:mean-squared-smoothness}}&{\le} 2^q \bar{L}^q \E{\norm{x_{j + 1} - x_j}^q} \\
			& = 2^q \gamma^q \bar{L}^q, \numberthis\label{f79e21c1-04c0-497e-81e9-c10afd7f04f9}
		\end{alignat*}
		thus, using~\eqref{f79e21c1-04c0-497e-81e9-c10afd7f04f9} and $2^{\frac{1}{q}} \le 2$ we obtain
		\begin{alignat*}{2}
			\E{\norm{\sum_{j = 0}^{t - 1} (1 - \alpha)^{t  - j - 1} \hat{S}_{j + 1}}} \oversetlab{\eqref{d98da1f4-086e-4004-a809-2eb471da42a6}}&{\le} 4 \gamma \bar{L} \left( \sum_{j = 0}^{t - 1} (1 - \alpha)^{q (t - j - 1)} \right)^{\frac{1}{q}} \\
			& \le 4 \gamma \bar{L} \alpha^{-\frac{1}{q}}. \numberthis\label{91e3e64b-54a6-46f3-96a5-a9b4f24dc6ee}
		\end{alignat*}
		
		Then, combining the bounds~\eqref{cea5881d-8688-41f3-8083-df58a5981198} and~\eqref{91e3e64b-54a6-46f3-96a5-a9b4f24dc6ee} we have
		\[ \E{\norm{\hat{e}_t}} \le (1 - \alpha)^t \E{\norm{\hat{e}_0}} + 2 \sigma_1 \alpha^{\frac{p - 1}{p}} + 4 \gamma \bar{L} \alpha^{-\frac{1}{q}}, \]
		which achieves the proof of the lemma.
	\end{proof}
	
	\begin{remark}
		In~\eqref{f79e21c1-04c0-497e-81e9-c10afd7f04f9} we only use the $q$--weak average smoothness assumption (\Cref{ass:mean-squared-smoothness}) to achieve the bound $2^q \gamma^q \bar{L}^q$. It is worth mentioning that this assumption can be replaced by the combinations of both $\bar{L}$--Lipschitz continuous gradients of $F$ and 
		\[ \ExpSub{\xi \sim \cD}{\norm{\left[ \nabla f(x, \xi) - \nabla f(y, \xi) \right] - \left[ \nabla F(x) - \nabla F(y) \right]}^q} \le \delta^q \norm{x - y}^q, \]
		for all $x, y \in \R^d$, where $\delta \ge 0$ is some fixed constant which can be much smaller than $\delta$ (see~\Cref{ass:mean-squared-smoothness-2}).
	\end{remark}
	
	\begin{corollary}[Bounding the Error Term: a Refined Version]\label{appdx-lem:bounding-error-term-nsgd-mvr-refined}
		Under~\Cref{ass:p-bounded-central-moment-gradient,ass:mean-squared-smoothness-2}, for all $t \in \Int{0}{T - 1}$ we have
		\[ \E{\norm{\hat{e}_t}} \le (1 - \alpha)^t \E{\norm{\hat{e}_0}} + 2 \sigma_1 \alpha^{\frac{p - 1}{p}} + 2 \gamma \delta \alpha^{-\frac{1}{q}}, \numberthis\label{b9474e8b-ac2e-49ce-9636-37a177622ba3} \]
		where $\hat{e}_t \eqdef g_t - \nabla F(x_t)$.
	\end{corollary}
	
	\begin{proof}
		The first two term in the upper bound~\eqref{b9474e8b-ac2e-49ce-9636-37a177622ba3} are obtained exactly the same way as in~\Cref{appdx-lem:bounding-error-term-nsgd-mvr}. For the last term, we start exactly as in~\eqref{d98da1f4-086e-4004-a809-2eb471da42a6} and, using~\Cref{ass:mean-squared-smoothness-2} we have
		\begin{eqnarray*}
			\E{\norm{\hat{S}_{j + 1}}^q} &=& \E{\norm{\left[ \nabla F(x_t) - \nabla F(x_{t - 1}) \right] - \left[ \nabla f(x_t, \xi_t) - \nabla f(x_{t - 1}, \xi_t) \right]}^q} \\
			&\oversetref{Ass.}{\ref{ass:mean-squared-smoothness-2}}{\le}& \delta^q \norm{x_t - x_{t - 1}}^q \le \gamma^q \delta^q,
		\end{eqnarray*}
		and plugging this new bound into~\eqref{91e3e64b-54a6-46f3-96a5-a9b4f24dc6ee} gives the inequality~\eqref{b9474e8b-ac2e-49ce-9636-37a177622ba3}, as claimed.
	\end{proof}
	
	\subsection{Proof of~\Cref{thm:nsgd-mvr-convergence-analysis}}
	
	Thanks to~\Cref{appdx-lem:unrolling-descent-lemma-nsgd-mvr,appdx-lem:bounding-error-term-nsgd-mvr} we can now establish the convergence analysis (in expectation) of~\Cref{algo:nsgd-mvr}.
	
	\begin{restate-theorem}{\ref{thm:nsgd-mvr-convergence-analysis}}
		Under~\Cref{ass:lower-boundedness,ass:p-bounded-central-moment-gradient,ass:mean-squared-smoothness}, let the initial gradient estimate $g_0$ be given by
		\[ g_0 = \frac{1}{B_{\textnormal{init}}} \sum_{j = 1}^{B_{\textnormal{init}} - 1} \nabla f\left( x_0, \xi_{0, j} \right), \]
		where $B_{\textnormal{init}} = \max\ens{1, \left( \frac{\sigma_1}{\eps} \right)^{\frac{p}{p - 1}}}$, let the stepsize $\gamma = \sqrt{\frac{\Delta \alpha^{\nicefrac{1}{q}}}{\bar{L} T}}$, the momentum parameter $\alpha = \min\ens{1, \alpha_{\textnormal{eff}}}$ where
		\[ \alpha_{\textnormal{eff}} = \max\ens{\left( \frac{\eps}{\sigma_1 T} \right)^{\frac{p}{2p - 1}}, \left( \frac{\bar{L} \Delta}{\sigma_1^2 T} \right)^{\frac{pq}{p (2 q + 1) - 2 q}}}. \numberthis\label{2c0dbb02-323f-4955-a4f2-8f3b73a0f89d-3} \]
		Then,~\Cref{algo:nsgd-mvr} guarantees to find an $\eps$--stationary point with the total sample complexity
		\[ \cO\left( \left( \frac{\sigma_1}{\eps} \right)^{\frac{p}{p - 1}} + \frac{\bar{L} \Delta}{\eps^2} + \frac{\bar{L} \Delta}{\eps^2} \left( \frac{\sigma_1}{\eps} \right)^{\frac{p}{q (p - 1)}} \right). \]
	\end{restate-theorem}
	
	\begin{proof}
		According to~\Cref{appdx-lem:unrolling-descent-lemma-nsgd-mvr} we have
		\[ \frac{1}{T} \sum_{t = 0}^{T - 1} \norm{\nabla F(x_t)} \le \frac{\Delta}{\gamma T} + \frac{2}{T} \sum_{t = 0}^{T - 1} \norm{\hat{e}_t} + \frac{\gamma \bar{L}}{2}, \]
		and using~\Cref{appdx-lem:bounding-error-term-nsgd-mvr} this yields
		\begin{eqnarray*}
			\frac{1}{T} \sum_{t = 0}^{T - 1} \E{\norm{\hat{e}_t}} &\le& \frac{1}{T} \sum_{t = 0}^{T - 1} (1 - \alpha)^t \E{\norm{\hat{e}_0}} + 2 \sigma_1 \alpha^{\frac{p - 1}{p}} + 4 \gamma \bar{L} \alpha^{-\frac{1}{q}} \\
			&\le& \frac{\E{\norm{\hat{e}_0}}}{\alpha T} + 2 \sigma_1 \alpha^{\frac{p - 1}{p}} + 4 \gamma \bar{L} \alpha^{-\frac{1}{q}},
		\end{eqnarray*}
		hence 
		\begin{eqnarray}
			\frac{1}{T} \sum_{t = 0}^{T - 1} \E{\norm{\nabla F(x_t)}} &\le& \frac{\Delta}{\gamma T} + \frac{2 \E{\norm{\hat{e}_0}}}{\alpha T} + 4 \sigma_1 \alpha^{\frac{p - 1}{p}} + 8 \gamma \bar{L} \alpha^{-\frac{1}{q}} + \frac{\gamma \bar{L}}{2}\notag\\
			&\le& \frac{\Delta}{\gamma T} + \frac{2 \E{\norm{\hat{e}_0}}}{\alpha T} + 4 \sigma_1 \alpha^{\frac{p - 1}{p}} + 9 \gamma \bar{L} \alpha^{-\frac{1}{q}}, \numberthis\label{d53dc0d7-cefc-4bc1-88bf-08c8201be495}
		\end{eqnarray}
		since $0 < \alpha \le 1$. Now, using $\gamma = \sqrt{\frac{\Delta \alpha^{\nicefrac{1}{q}}}{\bar{L} T}}$ we have
		\[ \frac{1}{T} \sum_{t = 0}^{T - 1} \E{\norm{\nabla F(x_t)}} = \cO\left( \alpha^{-\frac{1}{2 q}} \sqrt{\frac{\bar{L} \Delta}{T}} + \frac{\E{\norm{\hat{e}_0}}}{\alpha T} + \sigma_1 \alpha^{\frac{p - 1}{p}} \right). \numberthis\label{36a6b109-a2ff-445f-937c-7040717cb89d} \]
		Now, by our choice of $g_0$ we have
		\begin{alignat*}{2}
			\E{\norm{\hat{e}_0}} \oversetref{Lem.}{\ref{lem:jensen-inequality}}&{\le} \left( \E{\norm{\hat{e}_0}^p} \right)^{\frac{1}{p}}  = \left( \E{\norm{g_0 - \nabla F(x_0)}^p} \right)^{\frac{1}{p}} \\
			\oversetref{Lem.}{\ref{appdx-technical-lem:app-von-Bahr-and-Essen}}&{\le} \frac{2}{B_{\textnormal{init}}} \left( \sum_{j = 0}^{B_{\textnormal{init}} - 1} \E{\norm{\nabla f(x_0, \xi_{0, j}) - \nabla F(x_0)}^p} \right)^{\frac{1}{p}} \\
			\oversetref{Ass.}{\ref{ass:p-bounded-central-moment-gradient}}&{\le} \frac{2}{B_{\textnormal{init}}} \left( \sum_{j = 0}^{B_{\textnormal{init}} - 1} \sigma_1^p \right)^{\frac{1}{p}}  = \frac{2 \sigma_1}{B_{\textnormal{init}}^{\frac{p - 1}{p}}}, \numberthis\label{6a9db682-cb1e-43bb-878f-e4b60ebd3183}
		\end{alignat*}
		and since $B_{\textnormal{init}} = \max\ens{1, \left( \frac{\sigma_1}{\eps} \right)^{\frac{p}{p - 1}}}$ we have $\E{\norm{\hat{e}_0}} \le \sigma_1 \times \left( \frac{\eps}{\sigma_1} \right) = \eps$. This gives
		\begin{alignat*}{2}
			\frac{1}{T} \sum_{t = 0}^{T - 1} \E{\norm{\nabla F(x_t)}} \oversetlab{\eqref{36a6b109-a2ff-445f-937c-7040717cb89d}}&{=} \cO\left( \alpha^{-\frac{1}{2 q}} \sqrt{\frac{\bar{L} \Delta}{T}} + \frac{\E{\norm{\hat{e}_0}}}{\alpha T} + \sigma_1 \alpha^{\frac{p - 1}{p}} \right) \\
			& = \cO\left( \alpha^{-\frac{1}{2 q}} \sqrt{\frac{\bar{L} \Delta}{T}} + \frac{\eps}{\alpha T} + \sigma_1 \alpha^{\frac{p - 1}{p}} \right) \\
			& = \cO\left( \left[ \sqrt{\frac{\bar{L} \Delta}{T}} + \alpha_{\textnormal{eff}}^{-\frac{1}{2 q}} \sqrt{\frac{\bar{L} \Delta}{T}} \right] + \frac{\eps}{T} \alpha_{\textnormal{eff}}^{-1} + \sigma_1 \alpha_{\textnormal{eff}}^{\frac{p - 1}{p}} \right) \\
			\oversetrel{rel:629eafc0-82b4-4035-b358-3aa96a206e1a}&{=} \cO\left( \sqrt{\frac{\bar{L} \Delta}{T}} + \sigma_1 \left( \frac{\bar{L} \Delta}{\sigma_1^2 T} \right)^{\frac{q (p - 1)}{p (2 q + 1) - 2 q}} + \sigma_1 \left( \frac{\eps}{\sigma_1 T}  \right)^{\frac{p - 1}{2 p - 1}} \right), \numberthis\label{0f0a2e3d-d04e-4e64-948f-d38d2619b1fb}
		\end{alignat*}
		where in~\relref{rel:629eafc0-82b4-4035-b358-3aa96a206e1a} we use the choice of $\alpha_{\textnormal{eff}}$ from~\eqref{2c0dbb02-323f-4955-a4f2-8f3b73a0f89d-3} since
		\[ \frac{\eps}{T} \alpha_{\textnormal{eff}}^{-1} \le \frac{\eps}{T} \left( \frac{\sigma_1 T}{\eps} \right)^{\frac{p}{2p - 1}} = \sigma_1 \left( \frac{\eps}{\sigma_1 T} \right)^{\frac{p - 1}{2 p - 1}}, \]
		and
		\[ \sigma_1 \alpha_{\textnormal{eff}}^{\frac{p - 1}{p}} \le \sigma_1 \left( \frac{\eps}{\sigma_1 T} \right)^{\frac{p - 1}{2p - 1}} + \sigma_1 \left( \frac{\bar{L} \Delta}{\sigma_1^2 T} \right)^{\frac{q (p - 1)}{p (2 q + 1) - 2 q}}. \]
		
		Finally, from the bound~\eqref{0f0a2e3d-d04e-4e64-948f-d38d2619b1fb} we deduce that the sample complexity of~\Cref{algo:nsgd-mvr} is exactly
		\[ \cO\left( \left( \frac{\sigma_1}{\eps} \right)^{\frac{p}{p - 1}} + \frac{\bar{L} \Delta}{\eps^2} + \frac{\bar{L} \Delta}{\eps^2} \left( \frac{\sigma_1}{\eps} \right)^{\frac{p}{q (p - 1)}} \right), \]
		as claimed, and it matches our lower bound from~\Cref{thm:lower-bound-sofo-mean-squared-smoothness}.
	\end{proof}
	
	\subsection{Proof of~\Cref{thm:nsgd-mvr-convergence-analysis-2}}
	
	\begin{restate-theorem}{\ref{thm:nsgd-mvr-convergence-analysis-2}}
		Under~\Cref{ass:lower-boundedness,ass:L-lipschitz-gradients,ass:p-bounded-central-moment-gradient,ass:mean-squared-smoothness-2}, let the initial gradient estimate $g_0$ be given by
		\[ g_0 = \frac{1}{B_{\textnormal{init}}} \sum_{j = 1}^{B_{\textnormal{init}} - 1} \nabla f\left( x_0, \xi_{0, j} \right), \]
		where $B_{\textnormal{init}} = \max\ens{1, \left( \frac{\sigma_1}{\eps} \right)^{\frac{p}{p - 1}}}$, let the stepsize $\gamma = \min\ens{\sqrt{\frac{\Delta}{L_1 T}}, \sqrt{\frac{\Delta \alpha^{\nicefrac{1}{q}}}{\delta T}}}$, the momentum parameter $\alpha = \min\ens{1, \alpha_{\textnormal{eff}}}$ where
		\[ \alpha_{\textnormal{eff}} = \max\ens{\left( \frac{\eps}{\sigma_1 T} \right)^{\frac{p}{2p - 1}}, \left( \frac{\delta \Delta}{\sigma_1^2 T} \right)^{\frac{pq}{p (2 q + 1) - 2 q}}}. \numberthis\label{2c0dbb02-323f-4955-a4f2-8f3b73a0f89d-2-bis} \]
		Then,~\Cref{algo:nsgd-mvr} guarantees to find an $\eps$--stationary point with the total sample complexity
		\[ \cO\left( \left( \frac{\sigma_1}{\eps} \right)^{\frac{p}{p - 1}} + \frac{(L_1 + \delta) \Delta}{\eps^2} + \frac{\delta \Delta}{\eps^2} \left( \frac{\sigma_1}{\eps} \right)^{\frac{p}{q (p - 1)}} \right). \]
	\end{restate-theorem}
	
	\begin{proof}
		As the function $F$ has $L_1$--Lipschitz gradients by~\Cref{ass:L-lipschitz-gradients} then applying~\Cref{appdx-lem:unrolling-descent-lemma-nsgd-mvr-hess} gives
		\[ \frac{1}{T} \sum_{t = 0}^{T - 1} \norm{\nabla F(x_t)} \le \frac{\Delta}{\gamma T} + \frac{2}{T} \sum_{t = 0}^{T - 1} \norm{\hat{e}_t} + \frac{\gamma L_1}{2}, \]
		and using~\Cref{appdx-lem:bounding-error-term-nsgd-mvr-refined} yields
		\[ \frac{1}{T} \sum_{t = 0}^{T - 1} \E{\norm{\hat{e}_t}} \le \frac{1}{T} \sum_{t = 0}^{T - 1} (1 - \alpha)^t \E{\norm{\hat{e}_0}} + 2 \sigma_1 \alpha^{\frac{p - 1}{p}} + 2 \gamma \delta \alpha^{-\frac{1}{q}} \le \frac{\E{\norm{\hat{e}_0}}}{\alpha T} + 2 \sigma_1 \alpha^{\frac{p - 1}{p}} + 2 \gamma \delta \alpha^{-\frac{1}{q}}, \]
		hence 
		\[ \frac{1}{T} \sum_{t = 0}^{T - 1} \E{\norm{\nabla F(x_t)}} \le \frac{\Delta}{\gamma T} + \frac{2 \E{\norm{\hat{e}_0}}}{\alpha T} + 4 \sigma_1 \alpha^{\frac{p - 1}{p}} + 4 \gamma \delta \alpha^{-\frac{1}{q}} + \frac{\gamma L_1}{2}, \numberthis\label{c6f0fec1-2964-442a-b2c3-7c25dfd2e653} \]
		Now, using our choice of stepsize $\gamma = \min\ens{\sqrt{\frac{\Delta}{L_1 T}}, \sqrt{\frac{\Delta \alpha^{\nicefrac{1}{q}}}{\delta T}}}$ we have
		\[ \frac{1}{T} \sum_{t = 0}^{T - 1} \E{\norm{\nabla F(x_t)}} = \cO\left( \sqrt{\frac{L_1 \Delta}{T}} + \alpha^{-\frac{1}{2 q}} \sqrt{\frac{\delta \Delta}{T}} + \frac{\E{\norm{\hat{e}_0}}}{\alpha T} + \sigma_1 \alpha^{\frac{p - 1}{p}} \right). \numberthis\label{36a6b109-a2ff-445f-937c-7040717cb89d-bis} \]
		Next, by our choice of $g_0$ we have, as in we did in the previous proof
		\begin{alignat*}{2}
			\E{\norm{\hat{e}_0}} \oversetlab{\eqref{6a9db682-cb1e-43bb-878f-e4b60ebd3183}}&{\le}  \frac{2 \sigma_1}{B_{\textnormal{init}}^{\frac{p - 1}{p}}},
		\end{alignat*}
		and since $B_{\textnormal{init}} = \max\ens{1, \left( \frac{\sigma_1}{\eps} \right)^{\frac{p}{p - 1}}}$ we have $\E{\norm{\hat{e}_0}} \le \sigma_1 \times \left( \frac{\eps}{\sigma_1} \right) = \eps$. This gives
		\begin{alignat*}{2}
			\frac{1}{T} \sum_{t = 0}^{T - 1} \E{\norm{\nabla F(x_t)}} \oversetlab{\eqref{36a6b109-a2ff-445f-937c-7040717cb89d-bis}}&{=} \cO\left( \sqrt{\frac{L_1 \Delta}{T}} + \alpha^{-\frac{1}{2 q}} \sqrt{\frac{\delta \Delta}{T}} + \frac{\E{\norm{\hat{e}_0}}}{\alpha T} + \sigma_1 \alpha^{\frac{p - 1}{p}} \right) \\
			& = \cO\left( \sqrt{\frac{L_1 \Delta}{T}} + \alpha^{-\frac{1}{2 q}} \sqrt{\frac{\delta \Delta}{T}} + \frac{\eps}{\alpha T} + \sigma_1 \alpha^{\frac{p - 1}{p}} \right) \\
			& = \cO\left( \sqrt{\frac{L_1 \Delta}{T}} + \left[ \sqrt{\frac{\delta \Delta}{T}} + \alpha_{\textnormal{eff}}^{-\frac{1}{2 q}} \sqrt{\frac{\delta \Delta}{T}} \right] + \frac{\eps}{T} \alpha_{\textnormal{eff}}^{-1} + \sigma_1 \alpha_{\textnormal{eff}}^{\frac{p - 1}{p}} \right) \\
			\oversetrel{rel:629eafc0-82b4-4035-b358-3aa96a206e1ab}&{=} \cO\left( \sqrt{\frac{L_1 \Delta}{T}} + \sqrt{\frac{\delta \Delta}{T}} + \sigma_1 \left( \frac{\delta \Delta}{\sigma_1^2 T} \right)^{\frac{q (p - 1)}{p (2 q + 1) - 2 q}} + \sigma_1 \left( \frac{\eps}{\sigma_1 T}  \right)^{\frac{p - 1}{2 p - 1}} \right), \numberthis\label{0f0a2e3d-d04e-4e64-948f-d38d2619b1fbc}
		\end{alignat*}
		where in~\relref{rel:629eafc0-82b4-4035-b358-3aa96a206e1ab} we use the choice of $\alpha_{\textnormal{eff}}$ from~\eqref{2c0dbb02-323f-4955-a4f2-8f3b73a0f89d-2-bis} since
		\[ \frac{\eps}{T} \alpha_{\textnormal{eff}}^{-1} \le \frac{\eps}{T} \left( \frac{\sigma_1 T}{\eps} \right)^{\frac{p}{2p - 1}} = \sigma_1 \left( \frac{\eps}{\sigma_1 T} \right)^{\frac{p - 1}{2 p - 1}}, \]
		and
		\[ \sigma_1 \alpha_{\textnormal{eff}}^{\frac{p - 1}{p}} \le \sigma_1 \left( \frac{\eps}{\sigma_1 T} \right)^{\frac{p - 1}{2p - 1}} + \sigma_1 \left( \frac{\delta \Delta}{\sigma_1^2 T} \right)^{\frac{q (p - 1)}{p (2 q + 1) - 2 q}}. \]
		
		Finally, from the bound~\eqref{0f0a2e3d-d04e-4e64-948f-d38d2619b1fbc} we deduce that the sample complexity of~\Cref{algo:nsgd-mvr} is exactly
		\[ \cO\left( \left( \frac{\sigma_1}{\eps} \right)^{\frac{p}{p - 1}} + \frac{(L_1 + \delta) \Delta}{\eps^2} + \frac{\delta \Delta}{\eps^2} \left( \frac{\sigma_1}{\eps} \right)^{\frac{p}{q (p - 1)}} \right), \]
		as claimed, and, combining~\Cref{algo:nsgd-mvr} with \algname{NSGD-Mom} is enough to match our lower bound from~\Cref{thm:lower-bound-sofo-mean-squared-smoothness-2}.
	\end{proof}
	
	\subsection{Proof of~\Cref{thm:nsgd-mvr-convergence-analysis-unknown-p-q}}
	
	\begin{restate-theorem}{\ref{thm:nsgd-mvr-convergence-analysis-unknown-p-q}}
		Under~\Cref{ass:lower-boundedness,ass:p-bounded-central-moment-gradient,ass:mean-squared-smoothness}, let the initial gradient estimate $g_0 = \nabla f(x_0, \xi_0)$, let the stepsize $\gamma = \sqrt{\frac{\Delta \alpha}{\bar{L} T}}$, the momentum parameter $\alpha = T^{-\frac{1}{2}} \in \intof{0}{1}$. Then,~\Cref{algo:nsgd-mvr} guarantees the bound
		\[ \frac{1}{T} \sum_{t = 0}^{T - 1} \E{\norm{\nabla F(x_t)}} = \cO\left( \frac{\sigma_1}{T^{\frac{p - 1}{2 p}}} + \frac{\sqrt{\bar{L} \Delta}}{T^{\nicefrac{1}{4}}} \right). \]
	\end{restate-theorem}
	
	\begin{proof}
		Notice that the bound~\eqref{d53dc0d7-cefc-4bc1-88bf-08c8201be495} holds for any choice of the parameters $\gamma > 0$, $\alpha \in \intof{0}{1}$ and $g_0 \in \R^d$. Hence, for our particular choice $g_0 = \nabla f(x_0, \xi_0)$, $\gamma = \sqrt{\frac{\Delta \alpha}{\bar{L} T}}$ and $\alpha = T^{-\frac{1}{2}}$ we obtain
		\begin{alignat*}{2}
			\frac{1}{T} \sum_{t = 0}^{T - 1} \E{\norm{\nabla F(x_t)}} \oversetlab{\eqref{d53dc0d7-cefc-4bc1-88bf-08c8201be495}}&{=} \cO\left( \frac{\Delta}{\gamma T} + \frac{\E{\norm{\hat{e}_0}}}{\alpha T} + \sigma_1 \alpha^{\frac{p - 1}{p}} + \gamma \bar{L} \alpha^{-\frac{1}{q}} \right) \\
			& = \cO\left( \alpha^{-\frac{1}{2}} \sqrt{\frac{\bar{L} \Delta}{T}} + \frac{\E{\norm{\hat{e}_0}}}{T^{\nicefrac{1}{2}}} + \frac{\sigma_1}{T^{\frac{p - 1}{2 p}}} + \alpha^{- \frac{2 - q}{2 q}} \sqrt{\frac{\bar{L} \Delta}{T}} \right) \\
			\oversetrel{rel:0ca871ff-05d5-4d08-93d6-a55411fd6a2c}&{=} \cO\left( \frac{\sqrt{\bar{L} \Delta}}{T^{\nicefrac{1}{4}}} + \frac{\E{\norm{\hat{e}_0}}}{T^{\nicefrac{1}{2}}} + \frac{\sigma_1}{T^{\frac{p - 1}{2 p}}} \right) \\
			\oversetrel{rel:460b3480-98c7-49a2-9e0d-1cdd0fb1d065}&{=} \cO\left( \frac{\sigma_1}{T^{\frac{p - 1}{2 p}}} + \frac{\sigma_1}{T^{\nicefrac{1}{2}}} + \frac{\sqrt{\bar{L} \Delta}}{T^{\nicefrac{1}{4}}} \right) \\
			& = \cO\left( \frac{\sigma_1}{T^{\frac{p - 1}{2 p}}} + \frac{\sqrt{\bar{L} \Delta}}{T^{\nicefrac{1}{4}}} \right),
		\end{alignat*}
		where in~\relref{rel:0ca871ff-05d5-4d08-93d6-a55411fd6a2c} we use $\alpha^{-\frac{1}{2}} \ge \alpha^{-\frac{2 - q}{2 q}}$ since $\alpha \in \intof{0}{1}$ and $\frac{2 - q}{2 q} \le \frac{1}{2}$. In~\relref{rel:460b3480-98c7-49a2-9e0d-1cdd0fb1d065} we use the definition of $\hat{e}_0$ and $g_0$, i.e.,
		\[ \E{\norm{\hat{e}_0}} = \E{\norm{\nabla f(x_0, \xi_0) - \nabla F(x_0)}} \oversetref{Lem.}{\ref{lem:jensen-inequality}}{\le} \left( \E{\norm{\nabla f(x_0, \xi_0) - \nabla F(x_0)}^p} \right)^{\frac{1}{p}} \oversetref{Ass.}{\ref{ass:p-bounded-central-moment-gradient}} \le \sigma_1. \]
		This achieves the proof of the theorem.
	\end{proof}
	
	\subsection{Proof of~\Cref{thm:nsgd-mvr-convergence-analysis-unknown-p-q-2}}
	
	\begin{restate-theorem}{\ref{thm:nsgd-mvr-convergence-analysis-unknown-p-q-2}}
		Under~\Cref{ass:lower-boundedness,ass:L-lipschitz-gradients,ass:p-bounded-central-moment-gradient,ass:mean-squared-smoothness-2}, let the initial gradient estimate $g_0 = \nabla f(x_0, \xi_0)$, let the stepsize $\gamma = \min\ens{\sqrt{\frac{\Delta}{L_1 T}}, \sqrt{\frac{\Delta \alpha}{\delta T}}}$, the momentum parameter $\alpha = T^{-\frac{1}{2}} \in \intof{0}{1}$. Then,~\Cref{algo:nsgd-mvr} guarantees guarantees the bound
		\[ \frac{1}{T} \sum_{t = 0}^{T - 1} \E{\norm{\nabla F(x_t)}} = \cO\left( \frac{\sigma_1}{T^{\frac{p - 1}{2 p}}} + \frac{\sqrt{\delta \Delta}}{T^{\nicefrac{1}{4}}} + \frac{\sqrt{L_1 \Delta}}{T^{\nicefrac{1}{2}}} \right). \numberthis\label{004535e0-7ad4-4957-8cc1-223ee18990af} \]
	\end{restate-theorem}
	
	\begin{proof}
		As in the previous theorem, note that the bound~\eqref{c6f0fec1-2964-442a-b2c3-7c25dfd2e653} holds for any choice of the parameters $\gamma > 0$, $\alpha \in \intof{0}{1}$ and $g_0 \in \R^d$. Hence, for our particular choice $g_0 = \nabla f(x_0, \xi_0)$, $\gamma = \min\ens{\sqrt{\frac{\Delta}{L_1 T}}, \sqrt{\frac{\Delta \alpha}{\delta T}}}$ and $\alpha = T^{-\frac{1}{2}}$ we obtain, as in the previous proof
		\begin{alignat*}{2}
			\frac{1}{T} \sum_{t = 0}^{T - 1} \E{\norm{\nabla F(x_t)}} \oversetlab{\eqref{c6f0fec1-2964-442a-b2c3-7c25dfd2e653}}&{=} \cO\left( \frac{\Delta}{\gamma T} + \frac{\E{\norm{\hat{e}_0}}}{\alpha T} + \sigma_1 \alpha^{\frac{p - 1}{p}} + \gamma \delta \alpha^{-\frac{1}{q}} + \gamma L_1 \right) \\
			& = \cO\left( \sqrt{\frac{L_1 \Delta}{T}} + \alpha^{-\frac{1}{2}} \sqrt{\frac{\delta \Delta}{T}} + \frac{\E{\norm{\hat{e}_0}}}{T^{\nicefrac{1}{2}}} + \frac{\sigma_1}{T^{\frac{p - 1}{2 p}}} + \alpha^{- \frac{2 - q}{2 q}} \sqrt{\frac{\delta \Delta}{T}} \right) \\
			\oversetrel{rel:0ca871ff-05d5-4d08-93d6-a55411fd6a2c}&{=} \cO\left( \sqrt{\frac{L_1 \Delta}{T}} + \frac{\sqrt{\delta \Delta}}{T^{\nicefrac{1}{4}}} + \frac{\E{\norm{\hat{e}_0}}}{T^{\nicefrac{1}{2}}} + \frac{\sigma_1}{T^{\frac{p - 1}{2 p}}} \right) \\
			\oversetrel{rel:460b3480-98c7-49a2-9e0d-1cdd0fb1d065}&{=} \cO\left( \frac{\sigma_1}{T^{\frac{p - 1}{2 p}}} + \frac{\sigma_1}{T^{\nicefrac{1}{2}}} + \frac{\sqrt{\delta \Delta}}{T^{\nicefrac{1}{4}}} + \sqrt{\frac{L_1 \Delta}{T}}\right) \\
			& = \cO\left( \frac{\sigma_1}{T^{\frac{p - 1}{2 p}}} + \frac{\sqrt{\delta \Delta}}{T^{\nicefrac{1}{4}}} + \sqrt{\frac{L_1 \Delta}{T}} \right),
		\end{alignat*}
		where in~\relref{rel:0ca871ff-05d5-4d08-93d6-a55411fd6a2c} we use $\alpha^{-\frac{1}{2}} \ge \alpha^{-\frac{2 - q}{2 q}}$ since $\alpha \in \intof{0}{1}$ and $\frac{2 - q}{2 q} \le \frac{1}{2}$. In~\relref{rel:460b3480-98c7-49a2-9e0d-1cdd0fb1d065} we use, as before, the bound~\eqref{004535e0-7ad4-4957-8cc1-223ee18990af}.
		
		This achieves the proof of the theorem.
	\end{proof}

	\newpage%
	\section{Missing Proofs in Section~\ref{sec:near-optimal-method-nsgd-mvr-hess}}\label{appdx:proof-nsgd-mvr-hess}
	
	\subsection{Auxiliary Lemmas}
	
	Some of the auxiliary lemmas needed in this section (\Cref{appdx-lem:descent-lemma-nsgd-mvr-hess,appdx-lem:unrolling-descent-lemma-nsgd-mvr-hess}) have already been established in the previous section.
	\begin{lemma}[Bounding the Error Term]\label{appdx-lem:bounding-error-term-nsgd-mvr-hess}
		Under~\Cref{ass:p-bounded-central-moment-gradient,ass:q-bounded-central-moment-hessian,ass:L-lipschitz-hessians}, for all $t \in \Int{0}{T - 1}$ we have
		\[ \E{\norm{\hat{e}_t}} \le (1 - \alpha)^t \E{\norm{\hat{e}_0}} + 2 \sigma_1 \alpha^{\frac{p - 1}{p}} + 2 \gamma \sigma_2 \alpha^{-\frac{1}{q}} + 4 \gamma \min\ens{L_1, \gamma L_2} \alpha^{-\frac{1}{2}}, \]
		where $\hat{e}_t \eqdef g_t - \nabla F(x_t)$.
	\end{lemma}
	
	\begin{proof}
		By the update rule of the gradient estimator in~\Cref{algo:nsgd-mvr-hess} (line 9) we have, for all $t \in \Int{1}{T - 1}$
		\begin{alignat*}{2}
			\hat{e}_t \eqdef* g_t - \nabla F(x_t) \\
			& = (1 - \alpha) (g_{t - 1} + \nabla^2 f(\hat{x}_t, \hat{\xi}_t)(x_t - x_{t - 1}) + \alpha \nabla f(x_t, \xi_t) - \nabla F(x_t) \\
			& = (1 - \alpha) (g_{t - 1} - \nabla F(x_{t - 1})) 
			\begin{aligned}[t]
				&+ \alpha \left( \nabla f(x_t, \xi_t) - \nabla F(x_t) \right) \\
				&+ (1 - \alpha) \left( \nabla F(x_{t - 1}) - \nabla F(x_t) - \nabla^2 f(\hat{x}_t, \hat{\xi}_t)(x_{t - 1} - x_t) \right)
			\end{aligned} \\
			& = (1 - \alpha) (g_{t - 1} - \nabla F(x_{t - 1})) 
			\begin{aligned}[t]
				&+ \alpha \left( \nabla f(x_t, \xi_t) - \nabla F(x_t) \right) \\
				&+ (1 - \alpha) \left( \nabla F(x_{t - 1}) - \nabla F(x_t) - \nabla^2 F(\hat{x}_t)(x_{t - 1} - x_t) \right) \\
				&+ (1 - \alpha) \left( \nabla^2 F(\hat{x}_t)(x_{t - 1} - x_t) - \nabla^2 f(\hat{x}_t, \hat{\xi}_t)(x_{t - 1} - x_t) \right)
			\end{aligned} \\
			& = (1 - \alpha) \hat{e}_{t - 1} + \alpha e_t + (1 - \alpha) \hat{R}_t + (1 - \alpha) \hat{S}_t \numberthis\label{2f56c833-5913-40e8-a00c-20c7b1a4b32b-2},
		\end{alignat*}
		where we let $e_t \eqdef \nabla f(x_t, \xi_t) - \nabla F(x_t)$ and
		\begin{alignat*}{2}
			\hat{R}_t \eqdef* \nabla F(x_{t - 1}) - \nabla F(x_t) - \nabla^2 F(\hat{x}_t)(x_{t - 1} - x_t), \numberthis\label{14872232-13b0-4801-9fa1-bccc0be89a0b} \\
			\hat{S}_t \eqdef* \nabla^2 F(\hat{x}_t)(x_{t - 1} - x_t) - \nabla^2 f(\hat{x}_t, \hat{\xi}_t)(x_{t - 1} - x_t). \numberthis\label{891801fc-22ef-4562-81bb-3b5040b0f520}
		\end{alignat*}
		Notably, it is worth mentioning that $(e_t)_{t \ge 0}$, $(\hat{R}_t)_{t \ge 0}$ and $(\hat{S}_t)_{t \ge 0}$ are all martingale difference sequence with respect to the filtration $(\mathcal{F}_t)_{t \ge 0}$ where $\mathcal{F}_t \eqdef \sigma(g_0, (\xi_1, \hat{\xi}_1, q_1), \ldots, (\xi_t, \hat{\xi}_t, q_t))$. Effectively, for $\hat{R}_t$ we have
		\[ \ExpSub{q_t}{F(\hat{x}_t)(x_t - x_{t - 1})} = \int_0^1 \nabla^2 F(s x_t + (1 - s) x_{t - 1}) (x_t - x_{t - 1}) \odif{s} = \nabla F(x_t) - \nabla F(x_{t - 1}). \numberthis\label{ff082357-0807-4b61-b6c0-12bd29b8b854} \]
		
		Then, unrolling the recursion~\eqref{2f56c833-5913-40e8-a00c-20c7b1a4b32b-2} gives
		\[ \hat{e}_t = (1 - \alpha)^t \hat{e}_0 + \alpha \sum_{j = 0}^{t - 1} (1 - \alpha)^{t - j - 1} e_{j + 1} + \sum_{j = 0}^{t - 1} (1 - \alpha)^{t  - j} \hat{R}_{j + 1} + \sum_{j = 0}^{t - 1} (1 - \alpha)^{t  - j} \hat{S}_{j + 1}, \]
		and taking the norm followed by the total expectation yields
		\begin{alignat*}{2}
			\E{\norm{\hat{e}_t}} &\le (1 - \alpha)^t \E{\norm{\hat{e}_0}} + \alpha \E{\norm{ \sum_{j = 0}^{t - 1} (1 - \alpha)^{t - j - 1} e_{j + 1}}} \\
			&+ (1 - \alpha) \E{\norm{\sum_{j = 0}^{t - 1} (1 - \alpha)^{t  - j - 1} \hat{R}_{j + 1}}} + (1 - \alpha) \E{\norm{\sum_{j = 0}^{t - 1} (1 - \alpha)^{t  - j - 1} \hat{S}_{j + 1}}}.
		\end{alignat*}
		We now need to upper bound the last three terms of the previous inequality. For the first term, using Jensen's inequality (\Cref{lem:jensen-inequality}) as we did before in~\Cref{appdx-lem:bounding-error-term-nsgd-mvr} (see~\eqref{d98da1f4-086e-4004-a809-2eb471da42a6}) we have
		\begin{alignat*}{2}
			\E{\norm{ \sum_{j = 0}^{t - 1} (1 - \alpha)^{t - j - 1} e_{j + 1}}} \oversetlab{\eqref{d98da1f4-086e-4004-a809-2eb471da42a6}}&{\le} 2 \sigma_1 \alpha^{-\frac{1}{p}}. \numberthis\label{cea5881d-8688-41f3-8083-df58a5981198-2}
		\end{alignat*}
		Then, for the second term we have
		\begin{alignat*}{2}
			\E{\norm{\sum_{j = 0}^{t - 1} (1 - \alpha)^{t  - j - 1} \hat{R}_{j + 1}}} \oversetref{Lem.}{\ref{lem:jensen-inequality}}&{\le} \left( \E{\norm{\sum_{j = 0}^{t - 1} (1 - \alpha)^{t  - j - 1} \hat{R}_{j + 1}}^2} \right)^{\frac{1}{2}} \\
			\oversetref{Lem.}{\ref{appdx-technical-lem:app-von-Bahr-and-Essen}}&{\le} \left( 2 \sum_{j = 0}^{t - 1} (1 - \alpha)^{2 (t - j - 1)} \E{\norm{\hat{R}_{j + 1}}^2} \right)^{\frac{1}{2}}, \numberthis\label{d98da1f4-086e-4004-a809-2eb471da42a6-3}
		\end{alignat*}
		and, the variance of $\nabla^2 F(\hat{x}_t) (x_t - x_{t - 1})$ can be bounded in two different ways. First, by using~\Cref{ass:L-lipschitz-gradients} and Jensen's inequality we have
		\begin{alignat*}{2}
			\E{\sqnorm{\hat{R}_t}} \oversetlab{\eqref{14872232-13b0-4801-9fa1-bccc0be89a0b}}&{=} \E{\sqnorm{\nabla F(x_{t - 1}) - \nabla F(x_t) - \nabla^2 F(\hat{x}_t)(x_{t - 1} - x_t)}} \\
			\oversetref{Lem.}{\ref{lem:jensen-form-1}}&{\le} 2 \E{\sqnorm{\nabla F(x_t) - \nabla F(x_{t - 1})}} + 2 \E{\sqnorm{\nabla^2 F(\hat{x}_t)(x_t - x_{t - 1})}} \\
			\oversetref{Ass.}{\ref{ass:L-lipschitz-gradients}}&{\le} 2 L_1^2 \E{\sqnorm{x_t - x_{t - 1}}} + 2 \E{\normop{\nabla^2 F(\hat{x}_t)}^2 \cdot \sqnorm{x_t - x_{t - 1}}} \\
			\oversetref{Lem.}{\ref{appdx-lem:lipschitz-gradients-implies-bounded-hessian}}&{\le} 4 L_1^2 \E{\sqnorm{x_t - x_{t - 1}}} \\
			& = 4 \gamma^2 L_1^2, \numberthis\label{46d7ba91-ef00-4ec2-9b1e-d222e9968289a}
		\end{alignat*}
		and, using the connection between $\nabla^2 F(\hat{x}_t)(x_t - x_{t - 1})$ and $\nabla F(x_t) - \nabla F(x_{t - 1})$ as displayed in~\eqref{ff082357-0807-4b61-b6c0-12bd29b8b854} we also have
		\begin{alignat*}{2}
			\E{\sqnorm{\hat{R}_t}} \oversetlab{\eqref{14872232-13b0-4801-9fa1-bccc0be89a0b}}&{=} \E{\sqnorm{\nabla F(x_{t - 1}) - \nabla F(x_t) - \nabla^2 F(\hat{x}_t)(x_{t - 1} - x_t)}} \\
			& = \E{\sqnorm{\int_0^1 \left( \nabla^2 F(x_{t - 1} + s (x_t - x_{t - 1})) - \nabla^2 F(\hat{x}_t) \right) (x_t - x_{t - 1}) \odif{s}}} \\
			\oversetrel{rel:658b5028-bdf6-48c6-aa1c-db16d2f3a272}&{\le} \E{\left( \int_0^1 \norm{\left( \nabla^2 F(x_{t - 1} + s (x_t - x_{t - 1})) - \nabla^2 F(\hat{x}_t) \right) (x_t - x_{t - 1})} \odif{s} \right)^2} \\
			& \le \E{\left( \int_0^1 \normop{\nabla^2 F(x_{t - 1} + s (x_t - x_{t - 1})) - \nabla^2 F(\hat{x}_t)} \cdot \norm{x_t - x_{t - 1}} \odif{s} \right)^2} \\
			& = \gamma^2 \, \E{\left( \int_0^1 \normop{\nabla^2 F(x_{t - 1} + s (x_t - x_{t - 1})) - \nabla^2 F(\hat{x}_t)} \odif{s} \right)^2} \\
			\oversetref{Ass.}{\ref{ass:L-lipschitz-hessians}}&{\le} \gamma^2 L_2^2 \, \E{\left( \int_0^1 \norm{s (x_t - \hat{x}_t) + (1 - s) (x_{t - 1} - \hat{x}_t)} \odif{s} \right)^2} \\
			\oversetrel{rel:658b5028-bdf6-48c6-aa1c-db16d2f3a272a}&{\le} \gamma^2 L_2^2 \, \E{\left( \int_0^1 s \norm{x_t - \hat{x}_t} \odif{s} + \int_0^1 (1 - s) \norm{x_{t - 1} - \hat{x}_t} \odif{s} \right)^2} \\
			& = \frac{\gamma^2 L_2^2}{4} \, \E{\left( \norm{x_t - \hat{x}_t} + \norm{x_{t - 1} - \hat{x}_t} \right)^2} \\
			\oversetrel{rel:658b5028-bdf6-48c6-aa1c-db16d2f3a272b}&{=} \frac{\gamma^2 L_2^2}{4} \, \E{\left( (1 - q_t) \norm{x_t - x_{t - 1}} + q_t \norm{x_t - x_{t - 1}} \right)^2} \\
			& = \frac{\gamma^4 L_2^2}{4}, \numberthis\label{46d7ba91-ef00-4ec2-9b1e-d222e9968289b}
		\end{alignat*}
		where in~\relref{rel:658b5028-bdf6-48c6-aa1c-db16d2f3a272} and~\relref{rel:658b5028-bdf6-48c6-aa1c-db16d2f3a272a} we use the triangle inequality. In~\relref{rel:658b5028-bdf6-48c6-aa1c-db16d2f3a272b} we use the definition of $\hat{x}_t$, i.e., $\hat{x}_t \eqdef q_t x_t + (1 - q_t) x_{t - 1}$. It is worth mentioning that the bounds~\eqref{46d7ba91-ef00-4ec2-9b1e-d222e9968289a} and~\eqref{46d7ba91-ef00-4ec2-9b1e-d222e9968289b} holds without the expectation $\E{\cdot}$.
		
		Thus, using~\eqref{f79e21c1-04c0-497e-81e9-c10afd7f04f9a},~\eqref{46d7ba91-ef00-4ec2-9b1e-d222e9968289b} we obtain
		\begin{alignat*}{2}
			\E{\norm{\sum_{j = 0}^{t - 1} (1 - \alpha)^{t  - j - 1} \hat{R}_{j + 1}}} \oversetlab{\eqref{d98da1f4-086e-4004-a809-2eb471da42a6-2}}&{\le} 4 \gamma \min\ens{L_1, \gamma L_2} \left( \sum_{j = 0}^{t - 1} (1 - \alpha)^{2 (t - j - 1)} \right)^{\frac{1}{2}} \\
			& \le 4 \gamma \min\ens{L_1, \gamma L_2} \alpha^{-\frac{1}{2}}. \numberthis\label{91e3e64b-54a6-46f3-96a5-a9b4f24dc6ee-3}
		\end{alignat*}
		
		Finally, for the last term, we can write
		\begin{alignat*}{2}
			\E{\norm{\sum_{j = 0}^{t - 1} (1 - \alpha)^{t  - j - 1} \hat{S}_{j + 1}}} \oversetref{Lem.}{\ref{lem:jensen-inequality}}&{\le} \left( \E{\norm{\sum_{j = 0}^{t - 1} (1 - \alpha)^{t  - j - 1} \hat{S}_{j + 1}}^q} \right)^{\frac{1}{q}} \\
			\oversetref{Lem.}{\ref{appdx-technical-lem:app-von-Bahr-and-Essen}}&{\le} \left( 2 \sum_{j = 0}^{t - 1} (1 - \alpha)^{q (t - j - 1)} \E{\norm{\hat{S}_{j + 1}}^q} \right)^{\frac{1}{q}}, \numberthis\label{d98da1f4-086e-4004-a809-2eb471da42a6-2}
		\end{alignat*}
		and, using~\Cref{ass:q-bounded-central-moment-hessian} we obtain
		\begin{alignat*}{2}
			\E{\norm{\hat{S}_{j + 1}}^q} \oversetlab{\eqref{891801fc-22ef-4562-81bb-3b5040b0f520}}&{=} \E{\norm{\nabla^2 F(\hat{x}_t)(x_{t - 1} - x_t) - \nabla^2 f(\hat{x}_t, \hat{\xi}_t)(x_{t - 1} - x_t)}^q} \\
			& \le \E{\normop{\nabla^2 F(\hat{x}_t) - \nabla^2 f(\hat{x}_t, \hat{\xi}_t)}^q \cdot \norm{x_t - x_{t - 1}}^q} \\
			\oversetref{Ass.}{\ref{ass:q-bounded-central-moment-hessian}}&{\le} \gamma^q \sigma_2^q, \numberthis\label{f79e21c1-04c0-497e-81e9-c10afd7f04f9a}
		\end{alignat*}
		thus, using~\eqref{f79e21c1-04c0-497e-81e9-c10afd7f04f9a} and the fact that $2^{\frac{1}{q}} \le 2$ 
		we obtain
		\begin{alignat*}{2}
			\E{\norm{\sum_{j = 0}^{t - 1} (1 - \alpha)^{t  - j - 1} \hat{S}_{j + 1}}} \oversetlab{\eqref{d98da1f4-086e-4004-a809-2eb471da42a6-2}}&{\le} 2 \gamma \sigma_2 \left( \sum_{j = 0}^{t - 1} (1 - \alpha)^{q (t - j - 1)} \right)^{\frac{1}{q}} \\
			& \le 2 \gamma\sigma_2 \alpha^{-\frac{1}{q}}. \numberthis\label{91e3e64b-54a6-46f3-96a5-a9b4f24dc6ee-2}
		\end{alignat*}
		
		Then, combining the bounds~\eqref{cea5881d-8688-41f3-8083-df58a5981198-2},~\eqref{91e3e64b-54a6-46f3-96a5-a9b4f24dc6ee-3} and~\eqref{91e3e64b-54a6-46f3-96a5-a9b4f24dc6ee-2} we have
		\[ \E{\norm{\hat{e}_t}} \le (1 - \alpha)^t \E{\norm{\hat{e}_0}} + 2 \sigma_1 \alpha^{\frac{p - 1}{p}} + 2 \gamma \sigma_2 \alpha^{-\frac{1}{q}} + 4 \gamma \min\ens{L_1, \gamma L_2} \alpha^{-\frac{1}{2}}, \]
		which achieves the proof of the lemma.
	\end{proof}
	
	\subsection{Proof of~\Cref{thm:nsgd-mvr-hess-convergence-analysis}}\label{appdx-subsec:proof-nsgd-mvr-hess-convergence-analysis}
	
	Thanks to~\Cref{appdx-lem:unrolling-descent-lemma-nsgd-mvr-hess,appdx-lem:bounding-error-term-nsgd-mvr-hess} we can now establish the convergence analysis (in expectation) of~\Cref{algo:nsgd-mvr-hess}.
	
	\begin{restate-theorem}{\ref{thm:nsgd-mvr-hess-convergence-analysis}}
		Under~\Cref{ass:lower-boundedness,ass:L-lipschitz-gradients,ass:L-lipschitz-hessians,ass:p-bounded-central-moment-gradient,ass:q-bounded-central-moment-hessian}, let the initial gradient estimate $g_0$ be given by
		\[ g_0 = \frac{1}{B_{\textnormal{init}}} \sum_{j = 1}^{B_{\textnormal{init}} - 1} \nabla f\left( x_0, \xi_{0, j} \right), \]
		where $B_{\textnormal{init}} = \max\ens{1, \left( \frac{\sigma_1}{\eps} \right)^{\frac{p}{p - 1}}}$, let the stepsize $\gamma = \min\ens{\sqrt{\frac{\Delta}{L_1 T}}, \sqrt{\frac{\Delta \alpha^{\nicefrac{1}{q}}}{\sigma_2 T}}, \sqrt[3]{\frac{\Delta \alpha^{\nicefrac{1}{2}}}{L_2 T}}}$, the momentum parameter $\alpha = \min\ens{1, \alpha_{\textnormal{eff}}}$ where
		\[ \alpha_{\textnormal{eff}} = \max\ens{\left( \frac{\eps}{\sigma_1 T} \right)^{\frac{p}{2p - 1}}, \left( \frac{\Delta \sigma_2}{\sigma_1^2 T} \right)^{\frac{pq}{p (2 q + 1) - 2 q}}, \left( \frac{L_2^{\nicefrac{1}{2}} \Delta}{\sigma_1^{\nicefrac{3}{2}} T} \right)^{\frac{4 p}{7 p - 6}}}. \numberthis\label{9fe7782d-3762-4699-bd42-02bdf270abe0} \]
		Then,~\Cref{algo:nsgd-mvr-hess} guarantees to find an $\eps$--stationary point with the total sample complexity
		\[ \cO\left( \left( \frac{\sigma_1}{\eps} \right)^{\frac{p}{p - 1}} + \frac{\Delta}{\eps} \left( \frac{L_1 + \sigma_2}{\eps} + \sqrt{\frac{L_2}{\eps}} \right) + \frac{\Delta \sigma_2}{\eps^2} \left( \frac{\sigma_1}{\eps} \right)^{\frac{p}{q (p - 1)}} + \frac{L_2^{\nicefrac{1}{2}} \Delta \sigma_1^{\nicefrac{1}{4}}}{\eps^{\nicefrac{7}{4}}} \left( \frac{\sigma_1}{\eps} \right)^{\frac{p}{4 (p - 1)}} \right). \]
	\end{restate-theorem}
	
	\begin{proof}
		According to~\Cref{appdx-lem:unrolling-descent-lemma-nsgd-mvr-hess} we have
		\[ \frac{1}{T} \sum_{t = 0}^{T - 1} \norm{\nabla F(x_t)} \le \frac{\Delta}{\gamma T} + \frac{2}{T} \sum_{t = 0}^{T - 1} \norm{\hat{e}_t} + \frac{\gamma L_1}{2}, \]
		and using~\Cref{appdx-lem:bounding-error-term-nsgd-mvr-hess} this yields
		\begin{alignat*}{2}
			\frac{1}{T} \sum_{t = 0}^{T - 1} \E{\norm{\hat{e}_t}} & \le \frac{1}{T} \sum_{t = 0}^{T - 1} (1 - \alpha)^t \E{\norm{\hat{e}_0}} + 2 \sigma_1 \alpha^{\frac{p - 1}{p}} + 2 \gamma \sigma_2 \alpha^{-\frac{1}{q}} + 4 \gamma^2 L_2 \alpha^{-\frac{1}{2}} \\
			& \le \frac{\E{\norm{\hat{e}_0}}}{\alpha T} + 2 \sigma_1 \alpha^{\frac{p - 1}{p}} + 2 \gamma \sigma_2 \alpha^{-\frac{1}{q}} + 4 \gamma^2 L_2 \alpha^{-\frac{1}{2}},
		\end{alignat*}
		where in the last term we drop the $\min\ens{\cdots}$ to only keep the $\gamma^2 L_2$ term. Hence we obtain the bound
		\[ \frac{1}{T} \sum_{t = 0}^{T - 1} \E{\norm{\nabla F(x_t)}} \le \frac{\Delta}{\gamma T} + \frac{2 \E{\norm{\hat{e}_0}}}{\alpha T} + 4 \sigma_1 \alpha^{\frac{p - 1}{p}} + 4 \gamma \sigma_2 \alpha^{-\frac{1}{q}} + 8 \gamma^2 L_2 \alpha^{-\frac{1}{2}} + \frac{\gamma L_1}{2}. \]
		Now, using $\gamma = \min\ens{\sqrt{\frac{\Delta}{L_1 T}}, \sqrt{\frac{\Delta \alpha^{\nicefrac{1}{q}}}{\sigma_2 T}}, \sqrt[3]{\frac{\Delta \alpha^{\nicefrac{1}{2}}}{L_2 T}}}$ we have
		\begin{alignat*}{2} 
			&\frac{1}{T} \sum_{t = 0}^{T - 1} \E{\norm{\nabla F(x_t)}} = \\
			&\qquad\cO\left( \sqrt{\frac{L_1 \Delta}{T}} + \alpha^{-\frac{1}{2 q}} \sqrt{\frac{\Delta \sigma_2}{T}} + \alpha^{-\frac{1}{6}} \left( \frac{L_2^{\nicefrac{1}{2}} \Delta}{T} \right)^{\frac{2}{3}} + \frac{\E{\norm{\hat{e}_0}}}{\alpha T} + \sigma_1 \alpha^{\frac{p - 1}{p}} \right). \numberthis\label{36a6b109-a2ff-445f-937c-7040717cb89d-2} 
		\end{alignat*}
		Now, by our choice of $g_0$ we have, as in we did in the previous section
		\begin{alignat*}{2}
			\E{\norm{\hat{e}_0}} \oversetlab{\eqref{6a9db682-cb1e-43bb-878f-e4b60ebd3183}}&{\le}  \frac{2 \sigma_1}{B_{\textnormal{init}}^{\frac{p - 1}{p}}}, \numberthis\label{c83db8de-7f7d-4c84-82b2-a4e56a709c5f}
		\end{alignat*}
		and since $B_{\textnormal{init}} = \max\ens{1, \left( \frac{\sigma_1}{\eps} \right)^{\frac{p}{p - 1}}}$ we have $\E{\norm{\hat{e}_0}} \le \sigma_1 \times \left( \frac{\eps}{\sigma_1} \right) = \eps$. This gives
		\begin{alignat*}{2}
			&\frac{1}{T} \sum_{t = 0}^{T - 1} \E{\norm{\nabla F(x_t)}} \\
			&\qquad\oversetlab{\eqref{36a6b109-a2ff-445f-937c-7040717cb89d-2}}{=} \cO\left( \sqrt{\frac{L_1 \Delta}{T}} + \alpha^{-\frac{1}{2 q}} \sqrt{\frac{\Delta \sigma_2}{T}} + \alpha^{-\frac{1}{6}} \left( \frac{L_2^{\nicefrac{1}{2}} \Delta}{T} \right)^{\frac{2}{3}} + \frac{\E{\norm{\hat{e}_0}}}{\alpha T} + \sigma_1 \alpha^{\frac{p - 1}{p}} \right) \\
			&\qquad\oversetlab{\eqref{c83db8de-7f7d-4c84-82b2-a4e56a709c5f}}{=} \cO\left( \sqrt{\frac{L_1 \Delta}{T}} + \alpha^{-\frac{1}{2 q}} \sqrt{\frac{\Delta \sigma_2}{T}} + \alpha^{-\frac{1}{6}} \left( \frac{L_2^{\nicefrac{1}{2}} \Delta}{T} \right)^{\frac{2}{3}} + \frac{\eps}{\alpha T} + \sigma_1 \alpha^{\frac{p - 1}{p}} \right) \\
			&\qquad = \cO\Bigg( \sqrt{\frac{L_1 \Delta}{T}} + \left[ \sqrt{\frac{\Delta \sigma_2}{T}} + \alpha_{\textnormal{eff}}^{-\frac{1}{2 q}} \sqrt{\frac{\Delta \sigma_2}{T}} \right] + \frac{\eps}{T} \alpha_{\textnormal{eff}}^{-1} \\
			&\qquad\qquad\qquad+ \sigma_1 \alpha_{\textnormal{eff}}^{\frac{p - 1}{p}} + \left[ \left( \frac{L_2^{\nicefrac{1}{2}} \Delta}{T} \right)^{\frac{2}{3}} + \alpha_{\textnormal{eff}}^{-\frac{1}{6}} \left( \frac{L_2^{\nicefrac{1}{2}} \Delta}{T} \right)^{\frac{2}{3}} \right] \Bigg) \\
			&\qquad\oversetrel{rel:629eafc0-82b4-4035-b358-3aa96a206e1a-2}{=} \cO\Bigg( \sqrt{\frac{L_1 \Delta}{T}} + \sqrt{\frac{\Delta \sigma_2}{T}} + \left( \frac{L_2^{\nicefrac{1}{2}} \Delta}{T} \right)^{\frac{2}{3}} + \sigma_1 \left( \frac{\Delta \sigma_2}{\sigma_1^2 T} \right)^{\frac{q (p - 1)}{p (2 q + 1) - 2 q}}\\
			&\qquad\qquad\qquad+ \sigma_1 \left( \frac{\eps}{\sigma_1 T}  \right)^{\frac{p - 1}{2 p - 1}} + \sigma_1 \left( \frac{L_2^{\nicefrac{1}{2}} \Delta}{\sigma_1^{\nicefrac{3}{2}} T} \right)^{\frac{4 (p - 1)}{7 p - 6}} \Bigg), \numberthis\label{0f0a2e3d-d04e-4e64-948f-d38d2619b1fb-2}
		\end{alignat*}
		where in~\relref{rel:629eafc0-82b4-4035-b358-3aa96a206e1a-2} we use the choice of $\alpha_{\textnormal{eff}}$ from~\eqref{9fe7782d-3762-4699-bd42-02bdf270abe0}, especially, we have
		\[ \sigma_1 \alpha_{\textnormal{eff}}^{\frac{p - 1}{p}} \le \sigma_1 \left( \frac{\Delta \sigma_2}{\sigma_1^2 T} \right)^{\frac{q (p - 1)}{p (2 q + 1) - 2 q}} + \sigma_1 \left( \frac{\eps}{\sigma_1 T} \right)^{\frac{p - 1}{2 p - 1}} + \sigma_1 \left( \frac{L_2^{\nicefrac{1}{2}} \Delta}{\sigma_1^{\nicefrac{3}{2}} T} \right)^{\frac{4 (p - 1)}{7 p - 6}}. \]
		
		Finally, from the bound~\eqref{0f0a2e3d-d04e-4e64-948f-d38d2619b1fb-2} we deduce that the sample complexity of~\Cref{algo:nsgd-mvr-hess} is exactly
		\[ \cO\left( \left( \frac{\sigma_1}{\eps} \right)^{\frac{p}{p - 1}} + \frac{\Delta}{\eps} \left( \frac{L_1 + \sigma_2}{\eps} + \sqrt{\frac{L_2}{\eps}} \right) + \frac{\Delta \sigma_2}{\eps^2} \left( \frac{\sigma_1}{\eps} \right)^{\frac{p}{q (p - 1)}} + \frac{L_2^{\nicefrac{1}{2}} \Delta \sigma_1^{\nicefrac{1}{4}}}{\eps^{\nicefrac{7}{4}}} \left( \frac{\sigma_1}{\eps} \right)^{\frac{p}{4 (p - 1)}} \right), \numberthis\label{b5be16f7-9e74-48d0-af57-e31254eaedf7} \]
		as claimed.
	\end{proof}
	
	\begin{remark}
		In particular, when $\sigma_2 = 0$ we observe that the sample complexity~\eqref{b5be16f7-9e74-48d0-af57-e31254eaedf7} does not depend on the exponent $q$ anymore, as in the lower bound we derived in~\Cref{thm:lower-bound-sofo-bounded-central-moments}.
	\end{remark}
	
	\newpage
	\section{Missing Proofs in Section~\ref{sec:clipped-nsgd-mvr}}\label{appdx:proof-clipped-nsgd-mvr}
	
	\subsection{Preliminary Lemmas}
	
	Now we start the high-probability convergence analysis of~\Cref{algo:clipped-nsgd-mvr}. The proof heavily follows the one of~\cite{sadiev2025second}.
	
	\subsubsection{Some Descent Lemma}
	
	\begin{lemma}[A Descent Lemma (for~\Cref{algo:nsgd-mvr})]\label{appdx-lem:high-probability-analysis-descent-lemma}
		Under~\Cref{ass:lower-boundedness,ass:mean-squared-smoothness}, for any choice of stepsize $\gamma > 0$ and any choice of momentum parameter $\alpha \in \intof{0}{1}$,~\Cref{algo:clipped-nsgd-mvr} generates iterates $\{x_t\}_{t \in \Int{0}{T}}$ which satisfy almost surely (a.s.) the inequality
		\begin{alignat*}{2}
			\gamma \sum_{t = 0}^{T - 1} \norm{\nabla F(x_t)} + \Delta_T \le \Delta_0 & + 2 \gamma \alpha \sum_{t = 1}^T \norm{\sum_{j = 1}^t (1 - \alpha)^{t - j} \theta_j} + 2 \gamma (1 - \alpha) \sum_{t = 1}^T \norm{\sum_{j = 1}^t (1 - \alpha)^{t - j} \omega_j} \\
			& + \frac{2 \gamma \norm{\hat{e}_0}}{\alpha} + \frac{\gamma^2 \bar{L} T}{2},
		\end{alignat*}
		where for any $j \in \Int{1}{T - 1}$, the vectors $\theta_j$ and $\omega_j$ are defined as
		\begin{alignat*}{2}
			\theta_j \eqdef* \clip\left( \nabla f(x_j, \xi_j), \lbd_2 \right) - \nabla F(x_j), \numberthis\label{appdx-lem-proof:def-theta} \\
			\omega_j \eqdef* \clip\left( \nabla f(x_j, \xi_j) - \nabla f(x_{j - 1}, \xi_j) , \lbd_1 \right) - \left( \nabla F(x_j) - \nabla F(x_{j - 1}) \right), \numberthis\label{appdx-lem-proof:def-omega}
		\end{alignat*}
		and for any $t \in \Int{0}{T - 1}$ we let $\Delta_t \eqdef F(x_t) - F^{\inf}$.
	\end{lemma}
	
	\begin{proof}
		First of all, let us observe that~\Cref{appdx-lem:descent-lemma-nsgd-mvr} still holds for~\Cref{algo:clipped-nsgd-mvr} (and can be proved analogously since~\Cref{algo:nsgd-mvr,algo:clipped-nsgd-mvr} have the same gradient update rule). Moreover, as the momentum term does not play any role in the proof of~\Cref{appdx-lem:descent-lemma-nsgd-mvr}, the iterates $\{x_t\}_{t \in \Int{0}{T}}$ of~\Cref{algo:clipped-nsgd-mvr} satisfy almost surely (a.s.)
		\[ \gamma \sum_{t = 0}^{T - 1} \norm{\nabla F(x_t)} + \Delta_T \le \Delta_0 + 2 \gamma \sum_{t = 0}^{T - 1} \norm{\hat{e}_t} + \frac{\gamma^2 \bar{L} T}{2}. \numberthis\label{ef85feaf-c75f-4a8a-9afc-c33134870e38} \]

		Next, we bound $\norm{\hat{e}_t}$ in an analogous way as we did in~\Cref{appdx-lem:bounding-error-term-nsgd-mvr}. According to the update rule for the momentum term in line $10$ of~\Cref{algo:clipped-nsgd-mvr}, we have
		\begin{alignat*}{2}
			\hat{e}_t & = g_t - \nabla F(x_t) \\
			& = (1 - \alpha) \left( g_{t - 1} + \clip(\nabla f\left( x_t, \xi_t \right) - \nabla f \left( x_{t - 1}, \xi_t \right), \lbd_1) \right) + \alpha \, \clip(\nabla f \left( x_t, \xi_t \right), \lbd_2) - \nabla F(x_t) \\
			& = (1 - \alpha) \left( g_{t - 1} - \nabla F(x_{t - 1}) \right) \\
			&\qquad+ (1 - \alpha) \left( \clip(\nabla f\left( x_t, \xi_t \right) - \nabla f \left( x_{t - 1}, \xi_t \right), \lbd_1) - \left[ \nabla F(x_t) - \nabla F(x_{t - 1}) \right] \right) \\
			&\qquad+ \alpha \left( \clip(\nabla f \left( x_t, \xi_t \right), \lbd_2) - \nabla F(x_t) \right)\\
			\oversetlab{\eqref{appdx-lem-proof:def-theta}+\eqref{appdx-lem-proof:def-omega}}&{=} (1 - \alpha) \hat{e}_{t - 1} + \alpha \theta_t + (1 - \alpha) \omega_t \\
			\oversetrel{rel:5ab6c03d-e700-4912-b084-a9509ca86e01}&{=} (1 - \alpha)^t \hat{e}_0 + \alpha \sum_{j = 1}^t (1 - \alpha)^{t - j} \theta_j + (1 - \alpha) \sum_{j = 1}^t (1 - \alpha)^{t - j} \omega_j, \numberthis\label{6f49529c-cbca-40f2-af32-aea3df8c50f2}
		\end{alignat*}
		where in~\relref{rel:5ab6c03d-e700-4912-b084-a9509ca86e01} we unroll the recursion. Then, taking the norm and applying the triangle inequality in this series of equalities, we obtain
		\begin{alignat*}{2}
			\norm{\hat{e}_t} & \le (1 - \alpha)^t \norm{\hat{e}_0} + \alpha \norm{\sum_{j = 1}^t (1 - \alpha)^{t - j} \theta_j} + (1 - \alpha) \norm{\sum_{j = 1}^t (1 - \alpha)^{t - j} \omega_j}, \numberthis\label{cc2cc25c-1b54-471f-ab22-55d547ad29a6}
		\end{alignat*}
		and plugging~\eqref{cc2cc25c-1b54-471f-ab22-55d547ad29a6} in~\eqref{ef85feaf-c75f-4a8a-9afc-c33134870e38} leads to
		\begin{alignat*}{2}
			\gamma \sum_{t = 0}^{T - 1} \norm{\nabla F(x_t)} + \Delta_T & \le \Delta_0 + 2 \gamma \sum_{t = 0}^{T - 1} \norm{\hat{e}_t} + \frac{\gamma^2 \bar{L} T}{2} \\
			& \le \Delta_0
			\begin{aligned}[t]
				&+ 2 \gamma \sum_{t = 0}^{T - 1} \left( (1 - \alpha)^t \norm{\hat{e}_0} + \alpha \norm{\sum_{j = 1}^t (1 - \alpha)^{t - j} \theta_j} + (1 - \alpha) \norm{\sum_{j = 1}^t (1 - \alpha)^{t - j} \omega_j} \right) \\
				&+ \frac{\gamma^2 \bar{L} T}{2}
			\end{aligned} \\
			\oversetrel{rel:3961661b-135d-4eca-b667-e5df3c5451e5}&{\le} \Delta_0 
			\begin{aligned}[t]
				&+ \frac{2 \gamma \norm{\hat{e}_0}}{\alpha} + \frac{\gamma^2 \bar{L} T}{2} \\
				&+ 2 \gamma \alpha \sum_{t = 0}^{T - 1} \norm{\sum_{j = 1}^t (1 - \alpha)^{t - j} \theta_j} + 2 \gamma (1 - \alpha) \sum_{t = 0}^{T - 1} \norm{\sum_{j = 1}^t (1 - \alpha)^{t - j} \omega_j},
			\end{aligned}
		\end{alignat*}
		where in~\relref{rel:3961661b-135d-4eca-b667-e5df3c5451e5} we use the inequality
		\[ 2 \gamma \sum_{t = 0}^{T - 1} (1 - \alpha)^t \norm{\hat{e}_0} \le 2 \gamma \norm{\hat{e}_0} \sum_{t \ge 0} (1 - \alpha)^t = \frac{2 \gamma \norm{\hat{e}_0}}{\alpha}. \]
		
		This proves the desired result.
	\end{proof}
	
	\begin{remark}
		In particular, if we assume $g_0 = 0$ in~\Cref{algo:clipped-nsgd-mvr} then $x_0 = x_1$ and $\Delta_0 = \Delta_1$. Moreover, we have
		\[ \norm{\hat{e}_0} = \norm{g_0 - \nabla F(x_0)} = \norm{\nabla F(x_1)} \le \sqrt{2 \bar{L} \Delta_1}, \]
		as by~\Cref{ass:mean-squared-smoothness} we know that $F$ has $\bar{L}$--Lipschitz continuous gradients.
	\end{remark}
	
	\begin{corollary}[Another Descent Lemma]\label{appdx-lem:high-probability-analysis-descent-lemma-2}
		Under~\Cref{ass:lower-boundedness,ass:L-lipschitz-gradients}, for any choice of stepsize $\gamma > 0$ and any choice of momentum parameter $\alpha \in \intof{0}{1}$,~\Cref{algo:clipped-nsgd-mvr} generates iterates $\{x_t\}_{t \in \Int{0}{T}}$ which satisfy almost surely (a.s.) the inequality
		\begin{alignat*}{2}
			\gamma \sum_{t = 0}^{T - 1} \norm{\nabla F(x_t)} + \Delta_T \le \Delta_0 & + 2 \gamma \alpha \sum_{t = 1}^T \norm{\sum_{j = 1}^t (1 - \alpha)^{t - j} \theta_j} + 2 \gamma (1 - \alpha) \sum_{t = 1}^T \norm{\sum_{j = 1}^t (1 - \alpha)^{t - j} \omega_j} \\
			& + \frac{2 \gamma \norm{\hat{e}_0}}{\alpha} + \frac{\gamma^2 L_1 T}{2},
		\end{alignat*}
		where for any $j \in \Int{1}{T - 1}$, the vectors $\theta_j$ and $\omega_j$ are defined in~\eqref{appdx-lem-proof:def-theta} and~\eqref{appdx-lem-proof:def-omega} and for any $t \in \Int{0}{T - 1}$ we let $\Delta_t \eqdef F(x_t) - F^{\inf}$.    
	\end{corollary}
	
	\begin{proof}
		The proof follows the exact same steps as in the previous lemma (\Cref{appdx-lem:high-probability-analysis-descent-lemma}), with the exception that the function $F$ has $L_1$--Lipschitz continuous gradients (instead of $\bar{L}$).
	\end{proof}
	
	\begin{lemma}[A Descent Lemma (for~\Cref{algo:nsgd-mvr-hess})]\label{appdx-lem:high-probability-analysis-descent-lemma-hess}
		Under~\Cref{ass:lower-boundedness,ass:L-lipschitz-gradients}, for any choice of stepsize $\gamma > 0$ and any choice of momentum parameter $\alpha \in \intof{0}{1}$,~\Cref{algo:clipped-nsgd-mvr} generates iterates $\{x_t\}_{t \in \Int{0}{T}}$ which satisfy almost surely (a.s.) the inequality
		\begin{alignat*}{2}
			\gamma \sum_{t = 0}^{T - 1} \norm{\nabla F(x_t)} + \Delta_T \le \Delta_0 & + 2 \gamma \alpha \sum_{t = 1}^T \norm{\sum_{j = 1}^t (1 - \alpha)^{t - j} \theta_j} + 2 \gamma (1 - \alpha) \sum_{t = 1}^T \norm{\sum_{j = 1}^t (1 - \alpha)^{t - j} \omega_j} \\
			& + \frac{2 \gamma \norm{\hat{e}_0}}{\alpha} + \frac{\gamma^2 L_1 T}{2}, \numberthis\label{5a8226aa-756e-43a3-92e4-1096c9fe1a58}
		\end{alignat*}
		where for any $j \in \Int{1}{T - 1}$, the vectors $\theta_j$ and $\omega_j$ are defined as
		\begin{alignat*}{2}
			\theta_j \eqdef* \clip\left( \nabla f(x_j, \xi_j), \lbd_2 \right) - \nabla F(x_j), \numberthis\label{appdx-lem-proof:def-theta-hess} \\
			\omega_j \eqdef* \clip\left( \nabla^2 f(\hat{x}_j, \hat{\xi}_j)(x_j - x_{j - 1}), \lbd_1 \right) - \left( \nabla F(x_j) - \nabla F(x_{j - 1}) \right), \numberthis\label{appdx-lem-proof:def-omega-hess}
		\end{alignat*}
		and for any $t \in \Int{0}{T - 1}$ we let $\Delta_t \eqdef F(x_t) - F^{\inf}$.   
	\end{lemma}
	
	\begin{proof}
		The proof is very similar to the proof of~\Cref{appdx-lem:high-probability-analysis-descent-lemma} with the exception that the here function $F$ has $L_1$--Lipschitz continuous gradients (instead of $\bar{L}$) and the computations~\eqref{6f49529c-cbca-40f2-af32-aea3df8c50f2} varies slightly due to the use of the Hessian term $\nabla^2 f(\hat{x}_j, \hat{\xi}_j)(x_j - x_{j - 1})$ instead of the difference $\nabla f(x_j, \xi_j) - \nabla f(x_{j - 1}, \xi_{j - 1})$. We thus have,
		\begin{alignat*}{2}
			\hat{e}_t & = g_t - \nabla F(x_t) \\
			& = (1 - \alpha) \left( g_{t - 1} + \clip(\nabla^2 f(\hat{x}_t, \hat{\xi}_t) (x_t - x_{t - 1}), \lbd_1) \right) + \alpha \, \clip(\nabla f \left( x_t, \xi_t \right), \lbd_2) - \nabla F(x_t) \\
			& = (1 - \alpha) \left( g_{t - 1} - \nabla F(x_{t - 1}) \right) \\
			&\qquad+ (1 - \alpha) \left( \clip(\nabla^2 f(\hat{x}_t, \hat{\xi}_t) (x_t - x_{t - 1}), \lbd_1) - \left[ \nabla F(x_t) - \nabla F(x_{t - 1}) \right] \right) \\
			&\qquad+ \alpha \left( \clip(\nabla f \left( x_t, \xi_t \right), \lbd_2) - \nabla F(x_t) \right) \\
			\oversetlab{\eqref{appdx-lem-proof:def-theta}+\eqref{appdx-lem-proof:def-omega}}&{=} (1 - \alpha) \hat{e}_{t - 1} + \alpha \theta_t + (1 - \alpha) \omega_t \\
			\oversetrel{rel:5ab6c03d-e700-4912-b084-a9509ca86e01}&{=} (1 - \alpha)^t \hat{e}_0 + \alpha \sum_{j = 1}^t (1 - \alpha)^{t - j} \theta_j + (1 - \alpha) \sum_{j = 1}^t (1 - \alpha)^{t - j} \omega_j,
		\end{alignat*}
		where in~\relref{rel:5ab6c03d-e700-4912-b084-a9509ca86e01} we unroll the recursion. The next steps are exactly as in the proof of~\Cref{appdx-lem:high-probability-analysis-descent-lemma} and the claimed result~\eqref{5a8226aa-756e-43a3-92e4-1096c9fe1a58} follows.
	\end{proof}

	\subsubsection{High-Probability Analysis}
	
	From the previous~\Cref{appdx-lem:high-probability-analysis-descent-lemma}, we need to bound the two new terms:
	\[ \norm{\sum_{j = 1}^t (1 - \alpha)^{t - j} \theta_j} \,\, \text{ and } \,\, \norm{\sum_{j = 1}^t (1 - \alpha)^{t - j} \omega_j}, \]
	where $t \in \Int{0}{T - 1}$. To do so, we use the same strategy as in previous works~\citep{10.5555/3495724.3496985,pmlr-v202-sadiev23a,pmlr-v195-liu23c,sadiev2025second} and we introduce the biased and unbiased parts of $\theta_j$ and $\omega_j$, that is, for any $j \in [T - 1]$ we let $\theta_j = \theta_j^b + \theta_j^u$ with
	\begin{alignat*}{2}
		\theta_j^b \eqdef* \ExpSub{\xi_t}{\clip\left( \nabla f(x_t, \xi_t), \lbd_1 \right)} - \nabla F(x_t), \numberthis\label{647ac97a-ff0a-48a4-a4ed-13d555d50876} \\
		\theta_j^u \eqdef* \clip\left( \nabla f(x_t, \xi_t), \lbd_2 \right) - \ExpSub{\xi_t}{\clip\left( \nabla f(x_t, \xi_t), \lbd_2 \right)}, \numberthis\label{0ea28dd7-2e3c-42c1-af0a-754531b8eac2}
	\end{alignat*}
	and $\omega_j = \omega_j^b + \omega_j^u$ where
	\begin{alignat*}{2}
		\omega_j^b &\eqdef \ExpSub{\xi_t}{\clip\left( \nabla f(x_j, \xi_j) - \nabla f(x_{j - 1}, \xi_j) , \lbd_1 \right)} - \left( \nabla F(x_j) - \nabla F(x_{j - 1}) \right), \numberthis\label{9715675c-c075-4d13-a6e5-978530f178fd} \\
		\omega_j^u &\eqdef \clip\left( \nabla f(x_j, \xi_j) - \nabla f(x_{j - 1}, \xi_j) , \lbd_1 \right)\notag\\
		&\qquad\qquad\qquad\qquad\qquad\qquad\qquad - \ExpSub{\xi_t}{\clip\left( \nabla f(x_j, \xi_j) - \nabla f(x_{j - 1}, \xi_j) , \lbd_1 \right)} \numberthis\label{5bf5f4fb-a644-41c4-a212-9c4930a0075d}
	\end{alignat*}
	for~\Cref{appdx-lem:high-probability-analysis-descent-lemma} and~\Cref{appdx-lem:high-probability-analysis-descent-lemma-2}. For~\Cref{appdx-lem:high-probability-analysis-descent-lemma-hess} (hessian clipping) we let $\omega_j = \omega_j^r + \omega_j^b + \omega_j^u$ where
	\begin{alignat*}{2}
		\omega_j^r \eqdef* \nabla^2 F(\hat{x}_j) (x_j - x_{j - 1}) - \left( \nabla F(x_j) - \nabla F(x_{j - 1}) \right), \numberthis\label{9715675c-c075-4d13-a6e5-978530f178fe-hess} \\
		\omega_j^b \eqdef* \ExpSub{\hat{\xi}_t}{\clip\left( \nabla^2 f(\hat{x}_j, \hat{\xi}_j)(x_j - x_{j - 1}) , \lbd_1 \right)} - \nabla^2 F(\hat{x}_j) (x_j - x_{j - 1}), \numberthis\label{9715675c-c075-4d13-a6e5-978530f178fd-hess} \\
		\omega_j^u \eqdef* \clip\left( \nabla^2 f(\hat{x}_j, \hat{\xi}_j)(x_j - x_{j - 1}), \lbd_1 \right) - \ExpSub{\hat{\xi}_t}{\clip\left( \nabla^2 f(\hat{x}_j, \hat{\xi}_j)(x_j - x_{j - 1}), \lbd_1 \right)} \numberthis\label{5bf5f4fb-a644-41c4-a212-9c4930a0075d-hess}
	\end{alignat*}
	
	\begin{lemma}\label{appdx-lem:high-probability-analysis-bound-1}
		Under~\Cref{ass:p-bounded-central-moment-gradient}, for any $\delta' \in \intof{0}{\frac{1}{2}}$ and any $t \in \Int{0}{T - 1}$, if the clipping threshold satisfies
		\[ \lbd_2 \ge \max\ens{2 \norm{\nabla F(x_j)}, \sigma_1 \alpha^{-\frac{1}{p}}}, \]
		for all $j \in [t]$ then, with probability at least $1 - 2 \delta'$, we have
		\[ \norm{\sum_{j = 1}^t (1 - \alpha)^{t - j} \theta_j} \le 22 \lbd_2 \log \frac{2}{\delta'}. \]
	\end{lemma}
	
	\begin{proof}
		First of all, using~\eqref{647ac97a-ff0a-48a4-a4ed-13d555d50876} and~\eqref{0ea28dd7-2e3c-42c1-af0a-754531b8eac2} it follows
		\[ \norm{\sum_{j = 1}^t (1 - \alpha)^{t - j} \theta_j} \le \underbrace{\norm{\sum_{j = 1}^t (1 - \alpha)^{t - j} \theta_j^b}}_{\Circled{4}} + \norm{\sum_{j = 1}^t (1 - \alpha)^{t - j} \theta_j^u}, \]
		and now, we need to bound both terms above. For the second term, we use~\Cref{appdx-technical-lem:cutkosky-mehta} with exponent $2$ to obtain
		\begin{alignat*}{2}
			\norm{\sum_{j = 1}^t (1 - \alpha)^{t - j} \theta_j^u} \oversetrel{rel:960251ce-6167-478b-b779-d3020d913cdb}&{\le} \abs{\sum_{j = 1}^t V_j^t} + \sqrt{\max_{j \in [t]} \sqnorm{(1 - \alpha)^{t - j} \theta_j^u} + \sum_{j = 1}^t \sqnorm{(1 - \alpha)^{t - j} \theta_j^u}} \\
			& \le \abs{\sum_{j = 1}^t V_j^t} + \sqrt{2 \sum_{j = 1}^t \sqnorm{(1 - \alpha)^{t - j} \theta_j^u}} \\
			\oversetrel{rel:3b70b84d-fa4b-45fa-889f-2aad4458b16f}&{=} \underbrace{\abs{\sum_{j = 1}^t V_j^t}}_{\Circled{1}} + \sqrt{2 \underbrace{\sum_{j = 1}^t Y_j^t}_{\Circled{2}} + \underbrace{\sum_{j = 1}^t \ExpSub{\xi_j}{\sqnorm{(1 - \alpha)^{t - j} \theta_j^u}}}_{\Circled{3}}}, \numberthis\label{21eaf977-44d0-4b4b-8b63-5af59e5cc089}
		\end{alignat*}
		where in~\relref{rel:960251ce-6167-478b-b779-d3020d913cdb} we define the sequence $V_1^t, \ldots, V_t^t$ as in~\Cref{appdx-technical-lem:cutkosky-mehta}, that is,
		\begin{alignat*}{2}
			V_j^t \eqdef 
			\begin{cases}
				0, & \text{if $j = 0$}; \\
				\sign\left( \sum\limits_{i = 1}^{j - 1} V_i^t \right) \frac{\ps{\sum\limits_{i = 1}^{j - 1} (1 - \alpha)^{t - i} \theta_i^u}{(1 - \alpha)^{t - j} \theta_j^u}}{\norm{\sum\limits_{i = 1}^{j - 1} (1 - \alpha)^{t - i} \theta_i^u}}, & \text{if $j \neq 0$ and $\sum\limits_{i = 1}^{j - 1} (1 - \alpha)^{t - i} \theta_i^u \neq 0$}; \\
				0, & \text{if $j \neq 0$ and $\sum\limits_{i = 1}^{j - 1} (1 - \alpha)^{t - i} \theta_i^u = 0$}.
			\end{cases}
		\end{alignat*}
		while in~\relref{rel:3b70b84d-fa4b-45fa-889f-2aad4458b16f}, for any $j \in [t]$ we let
		\[ Y_j^t \eqdef \sqnorm{(1 - \alpha)^{t - j} \theta_j^u} - \ExpSub{\xi_j}{\sqnorm{(1 - \alpha)^{t - j} \theta_j^u}}. \]
		
		We now bound all terms \Circled{1}, \Circled{2}, \Circled{3} and \Circled{4}.
		
		\paragraph{Bounding \Circled{1}.} The sequence $V_1^t, \ldots, V_t^t$ is a martingale difference sequence since, by definition of $\theta_j^u$ and $V_j^t$, for all $j \in [t]$ we have $\ExpCond{V_j^t}{V_{j - 1}^t, \ldots, V_1^t} = 0$ since $\ExpSub{\xi_j}{\theta_j^u} = 0$. Moreover, by~\Cref{appdx-technical-lem:cutkosky-mehta} we also have
		\[ \abs{V_j^t} \le \norm{(1 - \alpha)^{t - j} \theta_j^u} \le \norm{\theta_j^u} \oversetref{Lem.}{\ref{appdx-technical-lem:sadiev-generalization}}{\le} 2 \lbd_2, \numberthis\label{77cc40ad-0220-4787-ab09-7c07a2d6dfa7} \]
		and let $c_2 \eqdef 2 \lbd_2$ be the upper bound on the random variables $V_1^t, \ldots, V_t^t$. Additionally, if we denote by $\sigma_j^2 \eqdef \ExpCond{\left( V_j^t \right)^2}{V_{j - 1}^t, \ldots, V_1^t}$ then by~\Cref{appdx-technical-lem:freedman-inequality} for any real number $b_2 > 0$ and any $G_2> 0$, as $0 < \delta' \le \frac{1}{2}$ then $\log \frac{2}{\delta'} \ge 1$ and we have
		\[ \Proba{\abs{\sum_{j = 1}^t V_j^t} > b_2 \, \text{ and } \, \sum_{j = 1}^t \sigma_j^2 \le G_2 \log \frac{2}{\delta'}} \le 2 \exp\left( -\frac{b_2^2}{2 G_2 \log \frac{2}{\delta'} + \frac{2 b_2 c_2}{3}} \right) = \delta', \numberthis\label{158c1570-485a-49ca-9ec3-32f47e8d8e8b} \]
		where the last equality holds provided $b_2 \eqdef \left( \frac{c_2}{3} + \sqrt{\frac{c_2^2}{9} + 2 G_2} \right) \log \frac{2}{\delta'} > 0$. We now need to define the constant $G_2$. To do so, we need to bound the sum $\sigma_1^2 + \cdots + \sigma_t^2$, this gives
		\begin{alignat*}{2}
			\sum_{j = 1}^t \sigma_j^2 & = \sum_{j = 1}^t \ExpCond{\left( V_j^t \right)^2}{V_{j - 1}^t, \ldots, V_1^t} \\
			\oversetlab{\eqref{77cc40ad-0220-4787-ab09-7c07a2d6dfa7}}&{\le} \sum_{j = 1}^n \ExpSub{\xi_j}{\sqnorm{(1 - \alpha)^{t - j} \theta_j^u}} \\
			& = \sum_{j = 1}^t ( 1 - \alpha)^{2 (t - j)} \, \ExpSub{\xi_j}{\sqnorm{\theta_j^u}} \\
			\oversetref{Lem.}{\ref{appdx-technical-lem:sadiev-generalization}}&{\le} 18 \lbd_2^{2 - p} \sigma_1^p \sum_{j = 1}^t (1 - \alpha)^{2 (t - j)} \\
			& \le \frac{18 \lbd_2^{2 - p} \sigma_1^p}{1 - (1 - \alpha)^2} \\
			& \le \frac{18 \lbd_2^{2 - p} \sigma_1^p}{\alpha}, \numberthis\label{74b4e2bb-4754-4dd3-ae90-37ef40362f93}
		\end{alignat*}
		where in the application of~\Cref{appdx-technical-lem:sadiev-generalization} we require $\lbd_2 \ge 2 \max_{j \in [t]} \norm{\nabla F(x_j)}$. Setting $G_2 \eqdef \frac{18 \lbd_2^{2 - p} \sigma_1^p}{\alpha} > 0$ gives the desired bound~\eqref{158c1570-485a-49ca-9ec3-32f47e8d8e8b}.
		
		\paragraph{Bounding \Circled{2}.} As in the previous paragraph, the sequence $Y_1^t, \ldots, Y_t^t$ is also a martingale difference sequence as the definition of $(Y_j^t)_{j \in [t]}$ implies $\ExpCond{Y_j^t}{Y_{j - 1}^t \ldots, Y_1^t} = \ExpSub{\xi_j}{Y_j^t} = 0$ for all $j \in [t]$. Moreover, according to~\Cref{appdx-technical-lem:sadiev-generalization}, we also have, as we did in~\eqref{77cc40ad-0220-4787-ab09-7c07a2d6dfa7}, for any $j \in [t]$
		\[ \abs{Y_j^t} \le \sqnorm{(1 - \alpha)^{t - j} \theta_j^u} + \ExpSub{\xi_j}{\sqnorm{(1 - \alpha)^{t - j} \theta_j^u}} \le 4 \lbd_2^2 + 4 \lbd_2^2 = 8 \lbd_2^2. \numberthis\label{4923f3a4-39e5-43f3-b260-7deeb46a8c67} \]
		Hence we define $\widetilde{c}_2 \eqdef 8 \lbd_2^2$. Now, denoting the conditional variance of $Y_j^t$ as $$\widetilde{\sigma}_j^2 \eqdef \ExpCond{\left( Y_j^t \right)^2}{Y_{j - 1}^t, \ldots, Y_1^t} = \ExpSub{\xi_j}{\left( Y_j^t \right)^2}$$ we have the bound
		\begin{alignat*}{2}
			\widetilde{\sigma}_j^2 \oversetlab{\eqref{4923f3a4-39e5-43f3-b260-7deeb46a8c67}}&{\le} 8 \lbd_2^2 \, \ExpSub{\xi_j}{\abs{Y_j^t}} \le 16 \lbd_2^2 \, \ExpSub{\xi_j}{\sqnorm{(1 - \alpha)^{t - j} \theta_j^u}}, \numberthis\label{a26983bb-60b7-49ab-ac0e-a1d026f1c51e}
		\end{alignat*}
		for all $j \in [t]$. Hence, 
		\[ \sum_{j = 1}^t \widetilde{\sigma}_j^2 \oversetlab{\eqref{a26983bb-60b7-49ab-ac0e-a1d026f1c51e}}{\le} 16 \lbd_2^2 \sum_{j = 1}^t \ExpSub{\xi_j}{\sqnorm{(1 - \alpha)^{t - j} \theta_j^u}} \oversetlab{\eqref{74b4e2bb-4754-4dd3-ae90-37ef40362f93}}{\le} 16 \lbd_2^2 \cdot \frac{18 \lbd_2^{2 - p} \sigma_1^p}{\alpha} = \frac{288 \lbd_2^{4 - p} \sigma_1^p}{\alpha}, \]
		(where we require $\lbd_2 \ge 2 \max_{j \in [t]} \norm{\nabla F(x_j)}$). Next, if we define $\widetilde{G}_2 \eqdef \frac{288 \lbd_2^{4 - p} \sigma_1^p}{\alpha}$ then, applying~\Cref{appdx-technical-lem:freedman-inequality} we obtain for any real number $\widetilde{b}_2$
		\[ \Proba{\abs{\sum_{j = 1}^t Y_j^t} > \widetilde{b}_2 \, \text{ and } \, \sum_{j = 1}^t \widetilde{\sigma}_j^2 \le \widetilde{G}_2 \log \frac{2}{\delta'}} \le 2 \exp\left( -\frac{\widetilde{b}_2^2}{2 \widetilde{G}_2 \log \frac{2}{\delta'} + \frac{2 \widetilde{b}_2 \widetilde{c}_2}{3}} \right) = \delta', \numberthis\label{04635ad8-7080-4c15-a1de-f0ada3a80f78} \]
		where the last identity holds if we set $\widetilde{b}_2 \eqdef \left( \frac{\widetilde{c}_2}{3} + \sqrt{\frac{\widetilde{c}_2^2}{9} + 2 \widetilde{G}_2} \right) \log \frac{2}{\delta'}$. This establishes the desired bound.
		
		\paragraph{Bounding \Circled{3}.} As we already did in the two last paragraphs, we have
		\[ \Circled{3} \eqdef \sum_{j = 1}^t \ExpSub{\xi_j}{\sqnorm{(1 - \alpha)^{t - j} \theta_j^u}} \oversetlab{\eqref{74b4e2bb-4754-4dd3-ae90-37ef40362f93}}{\le} \frac{18 \lbd_2^{2 - p} \sigma_1^p}{\alpha}, \]
		as desired.
		
		\paragraph{Bounding \Circled{4}.} If we assume that $\lbd_2 \ge 2 \norm{\nabla F(x_j)}$ for all $j \in [t]$ then, with probability one we have
		\begin{alignat*}{2}
			\Circled{4} \eqdef* \norm{\sum_{j = 1}^t (1 - \alpha)^{t - j} \theta_j^b} \\
			& \le \sum_{j = 1}^t (1 - \alpha)^{t - j} \norm{\theta_j^b} \\
			\oversetrel{rel:6bcc4f03-036d-4593-857a-983716b63243}&{\le} 4 \lbd_2^{1 - p} \sigma_1^p \sum_{j = 1}^t (1 - \alpha)^{t - j} \\
			& \le \frac{4 \lbd_2^{1 - p} \sigma_1^p}{\alpha},
		\end{alignat*}
		where in~\relref{rel:6bcc4f03-036d-4593-857a-983716b63243} we use~\Cref{appdx-technical-lem:sadiev-generalization}, more precisely, for any $j \in [t]$,
		\[ \norm{\theta_j^b} \le 2^p \lbd_2^{1 - p} \sigma_1^p \le 4 \lbd_2^{1 - p} \sigma_1^p. \]
		
		\paragraph{Summing up all bounds \Circled{1}, \Circled{2}, \Circled{3} and \Circled{4}.} To sum up, we introduce the event $E_{\Circled{1}, t}$ as follows
		\[ E_{\Circled{1}, t} \eqdef \ens{\abs{\sum_{j = 1}^t V_j^t} \le b_2 \, \text{ or } \, \sum_{j = 1}^t \sigma_j^2 > G_2 \log \frac{2}{\delta'}}, \]
		where we defined $c_2 \eqdef 2 \lbd_2$, $G_2 \eqdef \frac{18 \lbd_2^{2 - p} \sigma_1^p}{\alpha}$ and $b_2 \eqdef \left( \frac{c_2}{3} + \sqrt{\frac{c_2^2}{9} + 2 G_2} \right) \log \frac{2}{\delta'}$ and we can bound $b_2$ as
		\begin{alignat*}{2}
			b_2 \eqdef \left( \frac{c_2}{3} + \sqrt{\frac{c_2^2}{9} + 2 G_2} \right) \log \frac{2}{\delta'} & \le \left( \frac{2 c_2}{3} + \sqrt{2 G_2} \right) \log \frac{2}{\delta'} \\
			& = \left( \frac{4 \lbd_2}{3} + \sqrt{\frac{36 \lbd_2^{2 - p} \sigma_1^p}{\alpha}} \right) \log \frac{2}{\delta'} \\
			& = \lbd_2 \left( \frac{4}{3} + 6 \sqrt{\frac{1}{\alpha} \left(\frac{\sigma_1}{\lbd_2} \right)^p} \right) \log \frac{2}{\delta'} \\
			\oversetrel{rel:122b8683-5b62-44f2-b3aa-52164298844a}&{\le} \frac{22 \lbd_2}{3} \log \frac{2}{\delta'}, \numberthis\label{35c58053-3ad9-4b62-ae23-2a6f658c08da}
		\end{alignat*}
		where~\relref{rel:122b8683-5b62-44f2-b3aa-52164298844a} holds provided $\lbd_2 \ge \sigma_1 \alpha^{-\frac{1}{p}}$. On the other hand, we also define the event $E_{\Circled{2}, t}$ as follows
		\[ E_{\Circled{2}, t} \eqdef \ens{\abs{\sum_{j = 1}^t Y_j^t} \le \widetilde{b}_2 \, \text{ or } \, \sum_{j = 1}^t \widetilde{\sigma}_j^2 > \widetilde{G}_2 \log \frac{2}{\delta'}}, \]
		where we defined $\widetilde{c}_2 \eqdef 8 \lbd_2^2$, $\widetilde{G}_2 \eqdef \frac{288 \lbd_2^{4 - p} \sigma_1^p}{\alpha}$ and $\widetilde{b}_2 \eqdef \left( \frac{\widetilde{c}_2}{3} + \sqrt{\frac{\widetilde{c}_2^2}{9} + 2 \widetilde{G}_2} \right) \log \frac{2}{\delta'}$ and we can bound $\widetilde{b}_2$ as
		\begin{alignat*}{2}
			\widetilde{b}_2 \eqdef \left( \frac{\widetilde{c}_2}{3} + \sqrt{\frac{\widetilde{c}_2^2}{9} + 2 \widetilde{G}_2} \right) \log \frac{2}{\delta'} & \le \left( \frac{2 \widetilde{c}_2}{3} + \sqrt{2 \widetilde{G}_2} \right) \log \frac{2}{\delta'} \\
			& = \left( \frac{16 \lbd_2^2}{3} + \sqrt{\frac{576 \lbd_2^{4 - p} \sigma_1^p}{\alpha}} \right) \log \frac{2}{\delta'} \\
			& = \lbd_2^2 \left( \frac{16}{3} + 24 \sqrt{\frac{1}{\alpha} \left(\frac{\sigma_1}{\lbd_2} \right)^p} \right) \log \frac{2}{\delta'} \\
			& \le \frac{88 \lbd_2^2}{3} \log \frac{2}{\delta'}, \numberthis\label{c33e8593-ede2-4c3e-b76d-7e7ba9b43443}
		\end{alignat*}
		where the last inequality holds given $\lbd_2 \ge \sigma_1 \alpha^{-\frac{1}{p}}$.
		
		Next, given $\lbd_2 \ge \max\ens{2 \norm{\nabla F(x_j)}, \sigma_1 \alpha^{-\frac{1}{p}}}$ for all $j \in [t]$ we proved in~\eqref{158c1570-485a-49ca-9ec3-32f47e8d8e8b} and~\eqref{04635ad8-7080-4c15-a1de-f0ada3a80f78} that $\Proba{E_{\Circled{1}, t}} \ge 1 - \delta'$ and $\Proba{E_{\Circled{2}, t}} \ge 1 - \delta'$ hence, by the union bound inequality we have
		\[ \Proba{E_{\Circled{1}, t} \cap E_{\Circled{2}, t}} \ge 1 - 2 \delta', \]
		and on the event $E_{\Circled{1}, t} \cap E_{\Circled{2}, t}$ we obtain the inequality
		\begin{alignat*}{2}
			\norm{\sum_{j = 1}^t (1 - \alpha)^{t - j} \theta_j} & \le \norm{\sum_{j = 1}^t (1 - \alpha)^{t - j} \theta_j^b} + \norm{\sum_{j = 1}^t (1 - \alpha)^{t - j} \theta_j^u} \\
			\oversetlab{\eqref{21eaf977-44d0-4b4b-8b63-5af59e5cc089}}&{\le} \Circled{1} + \sqrt{2 \cdot \Circled{2} + 2 \cdot \Circled{3}} + \Circled{4} \\
			\oversetlab{\eqref{35c58053-3ad9-4b62-ae23-2a6f658c08da}+\eqref{c33e8593-ede2-4c3e-b76d-7e7ba9b43443}}&{\le} \frac{22 \lbd_2}{3} \log \frac{2}{\delta'} + \sqrt{2 \cdot \frac{88 \lbd_2^2}{3} \log \frac{2}{\delta'} + 2 \cdot \frac{18 \lbd_2^{2 - p} \sigma_1^p}{\alpha}} + \frac{4 \lbd_2^{1 - p} \sigma_1^p}{\alpha} \\
			\oversetrel{rel:65519c2e-05e9-4a6c-87ef-8631a38365a}&{\le} \lbd_2 \left( \frac{22}{3} + \sqrt{\frac{176}{3} + \frac{36}{\alpha} \left( \frac{\sigma_1}{\lbd_2} \right)^p} + \frac{4}{\alpha} \left( \frac{\sigma_1}{\lbd_2} \right)^p \right) \log \frac{2}{\delta'} \\
			\oversetrel{rel:95ba6431-6302-47f5-be7b-98aa809fc5a5}&{\le} \lbd_2 \left( \frac{22}{3} + \sqrt{\frac{176}{3} + 36} + 4 \right) \log \frac{2}{\delta'} \\
			& \le 22 \lbd_2 \log \frac{2}{\delta'},
		\end{alignat*}
		where in~\relref{rel:65519c2e-05e9-4a6c-87ef-8631a38365a} we use the fact that $\log \frac{2}{\delta'} \ge 1$ and in~\relref{rel:95ba6431-6302-47f5-be7b-98aa809fc5a5} we use $\lbd_2 \ge \sigma_1 \alpha^{-\frac{1}{p}}$. This proves the lemma.
	\end{proof}
	
	\begin{lemma}\label{appdx-lem:high-probability-analysis-bound-2}
		Under~\Cref{ass:mean-squared-smoothness}, for any $\delta'' \in \intof{0}{\frac{1}{2}}$ and any $t \in \Int{0}{T - 1}$, if the clipping threshold satisfies
		\[ \lbd_1 \ge \max\ens{2 \gamma \bar{L}, \gamma \bar{L} \alpha^{-\frac{1}{q}}}, \]
		for all $j \in [t]$ then, with probability at least $1 - 2 \delta''$, we have
		\[ \norm{\sum_{j = 1}^t (1 - \alpha)^{t - j} \omega_j} \le 46 \lbd_1 \log \frac{2}{\delta''}. \]
	\end{lemma}
	
	\begin{proof}
		First, using~\eqref{9715675c-c075-4d13-a6e5-978530f178fd} and~\eqref{5bf5f4fb-a644-41c4-a212-9c4930a0075d} we have
		\[ \norm{\sum_{j = 1}^t (1 - \alpha)^{t - j} \omega_j} \le \underbrace{\norm{\sum_{j = 1}^t (1 - \alpha)^{t - j} \omega_j^b}}_{\Circled{8}} + \norm{\sum_{j = 1}^t (1 - \alpha)^{t - j} \omega_j^u}, \numberthis\label{9221fb23-db5b-497f-9062-1c71b7f12882} \]
		and, as before, we need to bound both terms above. For the second term, we use~\Cref{appdx-technical-lem:cutkosky-mehta} with exponent $2$ to obtain
		\begin{alignat*}{2}
			\norm{\sum_{j = 1}^t (1 - \alpha)^{t - j} \omega_j^u} \oversetrel{rel:960251ce-6167-478b-b779-d3020d913cdb-2}&{\le} \abs{\sum_{j = 1}^t W_j^t} + \sqrt{\max_{j \in [t]} \sqnorm{(1 - \alpha)^{t - j} \omega_j^u} + \sum_{j = 1}^t \sqnorm{(1 - \alpha)^{t - j} \omega_j^u}} \\
			& \le \abs{\sum_{j = 1}^t W_j^t} + \sqrt{2 \sum_{j = 1}^t \sqnorm{(1 - \alpha)^{t - j} \omega_j^u}} \\
			\oversetrel{rel:3b70b84d-fa4b-45fa-889f-2aad4458b16f-2}&{=} \underbrace{\abs{\sum_{j = 1}^t W_j^t}}_{\Circled{5}} + \sqrt{2 \underbrace{\sum_{j = 1}^t Z_j^t}_{\Circled{6}} + \underbrace{\sum_{j = 1}^t \ExpSub{\xi_j}{\sqnorm{(1 - \alpha)^{t - j} \omega_j^u}}}_{\Circled{7}}}, \numberthis\label{21eaf977-44d0-4b4b-8b63-5af59e5cc089-2}
		\end{alignat*}
		where in~\relref{rel:960251ce-6167-478b-b779-d3020d913cdb-2} we define the sequence $V_1^t, \ldots, V_t^t$ as in~\Cref{appdx-technical-lem:cutkosky-mehta}, that is,
		\begin{alignat*}{2}
			W_j^t \eqdef 
			\begin{cases}
				0, & \text{if $j = 0$}; \\
				\sign\left( \sum\limits_{i = 1}^{j - 1} W_i^t \right) \frac{\ps{\sum\limits_{i = 1}^{j - 1} (1 - \alpha)^{t - i} \omega_i^u}{(1 - \alpha)^{t - j} \omega_j^u}}{\norm{\sum\limits_{i = 1}^{j - 1} (1 - \alpha)^{t - i} \omega_i^u}}, & \text{if $j \neq 0$ and $\sum\limits_{i = 1}^{j - 1} (1 - \alpha)^{t - i} \omega_i^u \neq 0$}; \\
				0, & \text{if $j \neq 0$ and $\sum\limits_{i = 1}^{j - 1} (1 - \alpha)^{t - i} \omega_i^u = 0$}.
			\end{cases}
		\end{alignat*}
		while in~\relref{rel:3b70b84d-fa4b-45fa-889f-2aad4458b16f-2}, for any $j \in [t]$ we let
		\[ Z_j^t \eqdef \sqnorm{(1 - \alpha)^{t - j} \omega_j^u} - \ExpSub{\xi_j}{\sqnorm{(1 - \alpha)^{t - j} \omega_j^u}}. \]
		
		We now bound all terms \Circled{5}, \Circled{6}, \Circled{7} and \Circled{8}.
		
		\paragraph{Bounding \Circled{5}.} The sequence $W_1^t, \ldots, W_t^t$ is a martingale difference sequence since, by definition of $\omega_j^u$ and $W_j^t$, for all $j \in [t]$ we have $\ExpCond{W_j^t}{W_{j - 1}^t, \ldots, W_1^t} = 0$ since $\ExpSub{\xi_j}{\omega_j^u} = 0$. Moreover, by~\Cref{appdx-technical-lem:cutkosky-mehta} we also have
		\[ \abs{W_j^t} \le \norm{(1 - \alpha)^{t - j} \omega_j^u} \le \norm{\omega_j^u} \oversetref{Lem.}{\ref{appdx-technical-lem:sadiev-generalization}}{\le} 2 \lbd_1, \numberthis\label{77cc40ad-0220-4787-ab09-7c07a2d6dfa7-2} \]
		and let $c_1 \eqdef 2 \lbd_1$ be the upper bound on the random variables $W_1^t, \ldots, W_t^t$. Additionally, if we denote by $\sigma_j^2 \eqdef \ExpCond{\left( W_j^t \right)^2}{W_{j - 1}^t, \ldots, W_1^t}$ then by~\Cref{appdx-technical-lem:freedman-inequality} for any real number $b_1 > 0$ and any $G_1 > 0$, as $0 < \delta'' \le \frac{1}{2}$ then $\log \frac{2}{\delta''} \ge 1$ and we have
		\[ \Proba{\abs{\sum_{j = 1}^t W_j^t} > b_1 \, \text{ and } \, \sum_{j = 1}^t \sigma_j^2 \le G_1 \log \frac{2}{\delta''}} \le 2 \exp\left( -\frac{b_1^2}{2 G_1 \log \frac{2}{\delta''} + \frac{2 b_1 c_1}{3}} \right) = \delta'', \numberthis\label{158c1570-485a-49ca-9ec3-32f47e8d8e8b-2} \]
		where the last equality holds provided $b_1 \eqdef \left( \frac{c_1}{3} + \sqrt{\frac{c_1^2}{9} + 2 G_1} \right) \log \frac{2}{\delta''} > 0$. We now need to define the constant $G_1$. To do so, we need to bound the sum $\sigma_1^2 + \cdots + \sigma_t^2$, this gives
		\begin{alignat*}{2}
			\sum_{j = 1}^t \sigma_j^2 & = \sum_{j = 1}^t \ExpCond{\left( W_j^t \right)^2}{W_{j - 1}^t, \ldots, W_1^t} \\
			\oversetlab{\eqref{77cc40ad-0220-4787-ab09-7c07a2d6dfa7-2}}&{\le} \sum_{j = 1}^n \ExpSub{\xi_j}{\sqnorm{(1 - \alpha)^{t - j} \omega_j^u}} \\
			& = \sum_{j = 1}^t ( 1 - \alpha)^{2 (t - j)} \, \ExpSub{\xi_j}{\sqnorm{\omega_j^u}} \\
			\oversetref{Lem.}{\ref{appdx-technical-lem:sadiev-generalization}}&{\le} 72 \lbd_1^{2 - q} \gamma^q \bar{L}^q \sum_{j = 1}^t (1 - \alpha)^{2 (t - j)} \\
			& \le \frac{72 \lbd_1^{2 - q} \gamma^q \bar{L}^q}{1 - (1 - \alpha)^2} \\
			& \le \frac{72 \lbd_1^{2 - q} \gamma^q \bar{L}^q}{\alpha}, \numberthis\label{74b4e2bb-4754-4dd3-ae90-37ef40362f93-2}
		\end{alignat*}
		where in the application of~\Cref{appdx-technical-lem:sadiev-generalization} we use the bound we proved earlier in~\Cref{appdx-lem:bounding-error-term-nsgd-mvr} (see more precisely at~\eqref{f79e21c1-04c0-497e-81e9-c10afd7f04f9}), that is,
		\[ \ExpSub{\xi_j}{\norm{\left[ \nabla F(x_j) - \nabla F(x_{j - 1}) \right] - \left[ \nabla f(x_j, \xi_j) - \nabla f(x_{j - 1}, \xi_j) \right]}^q} \le 2^q \gamma^q \bar{L}^q. \]
		Moreover, applying~\Cref{appdx-technical-lem:sadiev-generalization} requires to take $\lbd_1 \ge 2 \max_{j \in [t]} \norm{\nabla F(x_j) - \nabla F(x_{j - 1})}$ and, since 
		\[ \norm{\nabla F(x_j) - \nabla F(x_{j - 1})} \oversetref{Ass.}{\ref{ass:mean-squared-smoothness}}{\le} \bar{L} \norm{x_j - x_{j - 1}} = \gamma \bar{L}, \]
		then it is enough to have $\lbd_1 \ge 2 \gamma \bar{L}$. Setting $G_1 \eqdef \frac{72 \lbd_1^{2 - q} \gamma^q \bar{L}^q}{\alpha} > 0$ gives the desired bound~\eqref{158c1570-485a-49ca-9ec3-32f47e8d8e8b-2}.
		
		\paragraph{Bounding \Circled{6}.} As in the previous paragraph, the sequence $Z_1^t, \ldots, Z_t^t$ is also a martingale difference sequence as the definition of $(Z_j^t)_{j \in [t]}$ implies $\ExpCond{Z_j^t}{Z_{j - 1}^t \ldots, Z_1^t} = \ExpSub{\xi_j}{Z_j^t} = 0$ for all $j \in [t]$. Moreover, according to~\Cref{appdx-technical-lem:sadiev-generalization}, we also have, as we did in~\eqref{77cc40ad-0220-4787-ab09-7c07a2d6dfa7-2}, for any $j \in [t]$
		\[ \abs{Z_j^t} \le \sqnorm{(1 - \alpha)^{t - j} \omega_j^u} + \ExpSub{\xi_j}{\sqnorm{(1 - \alpha)^{t - j} \omega_j^u}} \le 4 \lbd_1^2 + 4 \lbd_1^2 = 8 \lbd_1^2. \numberthis\label{4923f3a4-39e5-43f3-b260-7deeb46a8c67-2} \]
		Hence we define $\widetilde{c}_1 \eqdef 8 \lbd_1^2$. Now, denoting the conditional variance of $Z_j^t$ as $$\widetilde{\sigma}_j^2 \eqdef \ExpCond{\left( Z_j^t \right)^2}{Z_{j - 1}^t, \ldots, Z_1^t} = \ExpSub{\xi_j}{\left( Z_j^t \right)^2}$$ we have the bound
		\begin{alignat*}{2}
			\widetilde{\sigma}_j^2 \oversetlab{\eqref{4923f3a4-39e5-43f3-b260-7deeb46a8c67-2}}&{\le} 8 \lbd_1^2 \, \ExpSub{\xi_j}{\abs{Z_j^t}} \le 16 \lbd_1^2 \, \ExpSub{\xi_j}{\sqnorm{(1 - \alpha)^{t - j} \omega_j^u}}, \numberthis\label{a26983bb-60b7-49ab-ac0e-a1d026f1c51e-2}
		\end{alignat*}
		for all $j \in [t]$. Hence, 
		\[ \sum_{j = 1}^t \widetilde{\sigma}_j^2 \oversetlab{\eqref{a26983bb-60b7-49ab-ac0e-a1d026f1c51e-2}}{\le} 16 \lbd_1^2 \sum_{j = 1}^t \ExpSub{\xi_j}{\sqnorm{(1 - \alpha)^{t - j} \omega_j^u}} \oversetlab{\eqref{74b4e2bb-4754-4dd3-ae90-37ef40362f93-2}}{\le} 16 \lbd_1^2 \cdot \frac{72 \lbd_1^{2 - q} \gamma^q \bar{L}^q}{\alpha} = \frac{1152 \lbd_1^{4 - q} \gamma^q \bar{L}^q}{\alpha}, \]
		(where we require $\lbd_1 \ge 2 \gamma \bar{L}$). Next, if we define $\widetilde{G}_1 \eqdef \frac{1152 \lbd_1^{4 - q} \gamma^q \bar{L}^q}{\alpha}$ then, applying~\Cref{appdx-technical-lem:freedman-inequality} we obtain for any real number $\widetilde{b}_2$
		\[ \Proba{\abs{\sum_{j = 1}^t Z_j^t} > \widetilde{b}_1 \, \text{ and } \, \sum_{j = 1}^t \widetilde{\sigma}_j^2 \le \widetilde{G}_1 \log \frac{2}{\delta''}} \le 2 \exp\left( -\frac{\widetilde{b}_1^2}{2 \widetilde{G}_1 \log \frac{2}{\delta''} + \frac{2 \widetilde{b}_1 \widetilde{c}_1}{3}} \right) = \delta'', \numberthis\label{04635ad8-7080-4c15-a1de-f0ada3a80f78-2} \]
		where the last identity holds if we set $\widetilde{b}_1 \eqdef \left( \frac{\widetilde{c}_1}{3} + \sqrt{\frac{\widetilde{c}_1^2}{9} + 2 \widetilde{G}_1} \right) \log \frac{2}{\delta''}$. This establishes the desired bound.
		
		\paragraph{Bounding \Circled{7}.} As we already did in the two last paragraphs, we have
		\[ \Circled{7} \eqdef \sum_{j = 1}^t \ExpSub{\xi_j}{\sqnorm{(1 - \alpha)^{t - j} \omega_j^u}} \oversetlab{\eqref{74b4e2bb-4754-4dd3-ae90-37ef40362f93-2}}{\le} \frac{72 \lbd_1^{2 - q} \gamma^q \bar{L}^q}{\alpha}, \]
		as desired.
		
		\paragraph{Bounding \Circled{8}.} If we assume that $\lbd_1 \ge 2 \gamma \bar{L}$ then, with probability one we have
		\begin{alignat*}{2}
			\Circled{8} \eqdef* \norm{\sum_{j = 1}^t (1 - \alpha)^{t - j} \omega_j^b} \\
			& \le \sum_{j = 1}^t (1 - \alpha)^{t - j} \norm{\omega_j^b} \\
			\oversetrel{rel:6bcc4f03-036d-4593-857a-983716b63243-2}&{\le} 16 \lbd_1^{1 - q} \gamma^q \bar{L}^q \sum_{j = 1}^t (1 - \alpha)^{t - j} \\
			& \le \frac{16 \lbd_1^{1 - q} \gamma^q \bar{L}^q}{\alpha},
		\end{alignat*}
		where in~\relref{rel:6bcc4f03-036d-4593-857a-983716b63243-2} we use~\Cref{appdx-technical-lem:sadiev-generalization}, more precisely, for any $j \in [t]$,
		\[ \norm{\omega_j^b} \le 4^q \lbd_1^{1 - q} \gamma^q \bar{L}^q \le 16 \lbd_1^{1 - q} \gamma^q \bar{L}^q, \]
		since $\E{\norm{\left[ \nabla F(x_t) - \nabla F(x_{t - 1}) \right] - \left[ \nabla f(x_t, \xi_t) - \nabla f(x_{t - 1}, \xi_t) \right]}^q} \le 2^q \gamma^q \bar{L}^q$ by~\Cref{appdx-lem:bounding-error-term-nsgd-mvr} (and more precisely by~\eqref{f79e21c1-04c0-497e-81e9-c10afd7f04f9}).
		
		\paragraph{Summing up all bounds \Circled{5}, \Circled{6}, \Circled{7} and \Circled{8}.} To sum up, we introduce the event $E_{\Circled{5}, t}$ as follows
		\[ E_{\Circled{5}, t} \eqdef \ens{\abs{\sum_{j = 1}^t W_j^t} \le b_1 \, \text{ or } \, \sum_{j = 1}^t \sigma_j^2 > G_1 \log \frac{2}{\delta''}}, \numberthis\label{1ded530f-71dc-46eb-9fdb-5308c908cb80} \]
		where we defined $c_1 \eqdef 2 \lbd_1$, $G_1 \eqdef \frac{72 \lbd_1^{2 - q} \gamma^q \bar{L}^q}{\alpha}$ and $b_1 \eqdef \left( \frac{c_1}{3} + \sqrt{\frac{c_1^2}{9} + 2 G_1} \right) \log \frac{2}{\delta''}$ and we can bound $b_1$ as
		\begin{alignat*}{2}
			b_1 \eqdef \left( \frac{c_1}{3} + \sqrt{\frac{c_1^2}{9} + 2 G_1} \right) \log \frac{2}{\delta''} & \le \left( \frac{2 c_1}{3} + \sqrt{2 G_1} \right) \log \frac{2}{\delta''} \\
			& = \left( \frac{4 \lbd_1}{3} + \sqrt{\frac{144 \lbd_1^{2 - q} \gamma^q \bar{L}^q}{\alpha}} \right) \log \frac{2}{\delta''} \\
			& = \lbd_1 \left( \frac{4}{3} + 12 \sqrt{\frac{1}{\alpha} \left(\frac{\gamma \bar{L}}{\lbd_1} \right)^q} \right) \log \frac{2}{\delta''} \\
			\oversetrel{rel:122b8683-5b62-44f2-b3aa-52164298844a}&{\le} \frac{40 \lbd_1}{3} \log \frac{2}{\delta''}, \numberthis\label{35c58053-3ad9-4b62-ae23-2a6f658c08da-2}
		\end{alignat*}
		where~\relref{rel:122b8683-5b62-44f2-b3aa-52164298844a} holds provided $\lbd_1 \ge \gamma \bar{L} \alpha^{-\frac{1}{q}}$. On the other hand, we also define the event $E_{\Circled{6}, t}$ as follows
		\[ E_{\Circled{6}, t} \eqdef \ens{\abs{\sum_{j = 1}^t Z_j^t} \le \widetilde{b}_1 \, \text{ or } \, \sum_{j = 1}^t \widetilde{\sigma}_j^2 > \widetilde{G}_1 \log \frac{2}{\delta''}}, \numberthis\label{92e2f28f-a246-45f3-bdea-ba28563d3825} \]
		where we defined $\widetilde{c}_1 \eqdef 8 \lbd_1^2$, $\widetilde{G}_1 \eqdef \frac{1152 \lbd_1^{4 - q} \gamma^q \bar{L}^q}{\alpha}$ and $\widetilde{b}_1 \eqdef \left( \frac{\widetilde{c}_1}{3} + \sqrt{\frac{\widetilde{c}_1^2}{9} + 2 \widetilde{G}_1} \right) \log \frac{2}{\delta''}$ and we can bound $\widetilde{b}_1$ as
		\begin{alignat*}{2}
			\widetilde{b}_1 \eqdef \left( \frac{\widetilde{c}_1}{3} + \sqrt{\frac{\widetilde{c}_1^2}{9} + 2 \widetilde{G}_1} \right) \log \frac{2}{\delta''} & \le \left( \frac{2 \widetilde{c}_1}{3} + \sqrt{2 \widetilde{G}_1} \right) \log \frac{2}{\delta''} \\
			& = \left( \frac{16 \lbd_1^2}{3} + \sqrt{\frac{2304 \lbd_1^{4 - q} \gamma^q \bar{L}^q}{\alpha}} \right) \log \frac{2}{\delta''} \\
			& = \lbd_1^2 \left( \frac{16}{3} + 48 \sqrt{\frac{1}{\alpha} \left(\frac{\gamma \bar{L}}{\lbd_1} \right)^q} \right) \log \frac{2}{\delta''} \\
			& \le \frac{160 \lbd_1^2}{3} \log \frac{2}{\delta''}, \numberthis\label{c33e8593-ede2-4c3e-b76d-7e7ba9b43443-2}
		\end{alignat*}
		where the last inequality holds given $\lbd_1 \ge \gamma \bar{L} \alpha^{-\frac{1}{q}}$.
		
		Next, given $\lbd_1 \ge \max\ens{2 \gamma \bar{L}, \gamma \bar{L} \alpha^{-\frac{1}{q}}}$, we proved in~\eqref{158c1570-485a-49ca-9ec3-32f47e8d8e8b-2} and~\eqref{04635ad8-7080-4c15-a1de-f0ada3a80f78-2} that $\Proba{E_{\Circled{5}, t}} \ge 1 - \delta''$ and $\Proba{E_{\Circled{6}, t}} \ge 1 - \delta''$ hence, by the union bound inequality we have
		\[ \Proba{E_{\Circled{5}, t} \cap E_{\Circled{6}, t}} \ge 1 - 2 \delta'', \]
		and on the event $E_{\Circled{5}, t} \cap E_{\Circled{6}, t}$ we obtain the inequality
		\begin{alignat*}{2}
			\norm{\sum_{j = 1}^t (1 - \alpha)^{t - j} \omega_j} & \le \norm{\sum_{j = 1}^t (1 - \alpha)^{t - j} \omega_j^b} + \norm{\sum_{j = 1}^t (1 - \alpha)^{t - j} \omega_j^u} \\
			\oversetlab{\eqref{21eaf977-44d0-4b4b-8b63-5af59e5cc089-2}}&{\le} \Circled{5} + \sqrt{2 \cdot \Circled{6} + 2 \cdot \Circled{7}} + \Circled{8} \\
			\oversetlab{\eqref{35c58053-3ad9-4b62-ae23-2a6f658c08da-2}+\eqref{c33e8593-ede2-4c3e-b76d-7e7ba9b43443-2}}&{\le} \frac{40 \lbd_1}{3} \log \frac{2}{\delta''} + \sqrt{2 \cdot \frac{160 \lbd_1^2}{3} \log \frac{2}{\delta''} + 2 \cdot \frac{72 \lbd_1^{2 - q} \gamma^q \bar{L}^q}{\alpha}} + \frac{16 \lbd_1^{1 - q} \gamma^q \bar{L}^q}{\alpha} \\
			\oversetrel{rel:65519c2e-05e9-4a6c-87ef-8631a38365a-2}&{\le} \lbd_1 \left( \frac{40}{3} + \sqrt{\frac{320}{3} + \frac{144}{\alpha} \left( \frac{\gamma \bar{L}}{\lbd_1} \right)^q} + \frac{16}{\alpha} \left( \frac{\gamma \bar{L}}{\lbd_1} \right)^q \right) \log \frac{2}{\delta''} \\
			\oversetrel{rel:95ba6431-6302-47f5-be7b-98aa809fc5a5-2}&{\le} \lbd_1 \left( \frac{40}{3} + \sqrt{\frac{320}{3} + 144} + 16 \right) \log \frac{2}{\delta''} \\
			& \le 46 \lbd_1 \log \frac{2}{\delta''},
		\end{alignat*}
		where in~\relref{rel:65519c2e-05e9-4a6c-87ef-8631a38365a-2} we use the fact that $\log \frac{2}{\delta''} \ge 1$ and in~\relref{rel:95ba6431-6302-47f5-be7b-98aa809fc5a5-2} we use $\lbd_1 \ge \gamma \bar{L} \alpha^{-\frac{1}{q}}$. This proves the lemma.
	\end{proof}
	
	\begin{lemma}\label{appdx-lem:high-probability-analysis-bound-3}
		Under~\Cref{ass:L-lipschitz-gradients,ass:mean-squared-smoothness-2}, for any $\delta'' \in \intof{0}{\frac{1}{2}}$ and any $t \in \Int{0}{T - 1}$, if the clipping threshold satisfies
		\[ \lbd_1 \ge \max\ens{2 \gamma L_1, \gamma \bar{\delta} \alpha^{-\frac{1}{q}}}, \]
		for all $j \in [t]$ then, with probability at least $1 - 2 \delta''$, we have
		\[ \norm{\sum_{j = 1}^t (1 - \alpha)^{t - j} \omega_j} \le 22 \lbd_1 \log \frac{2}{\delta''}. \]
	\end{lemma}
	
	\begin{proof}
		First, note that the computations~\eqref{9221fb23-db5b-497f-9062-1c71b7f12882} and~\eqref{21eaf977-44d0-4b4b-8b63-5af59e5cc089-2} are identical and we now need to bound all terms \Circled{5}, \Circled{6}, \Circled{7} and \Circled{8} where
		\[ \Circled{5} \eqdef \abs{\sum_{j = 1}^t W_j^t}, \quad \Circled{6} \eqdef \sum_{j = 1}^t Z_j^t, \quad \Circled{7} \eqdef \sum_{j = 1}^t \ExpSub{\xi_j}{\sqnorm{(1 - \alpha)^{t - j} \omega_j^u}} \,\, \text{ and } \,\, \Circled{8} \eqdef \norm{\sum_{j = 1}^t (1 - \alpha)^{t - j} \omega_j^b}. \]
		
		\paragraph{Bounding \Circled{5}.} As in the previous lemma, we choose $c_1 \eqdef 2 \lbd_1$ and $b_1 \eqdef \left( \frac{c_1}{3} + \sqrt{\frac{c_1^2}{9} + 2 G_1} \right) \log \frac{2}{\delta''} > 0$. Then, to select $G_1 > 0$ we need to bound the sum $\sigma_1^2 + \cdots + \sigma_t^2$, i.e.,
		\begin{alignat*}{2}
			\sum_{j = 1}^t \sigma_j^2 \oversetlab{\eqref{74b4e2bb-4754-4dd3-ae90-37ef40362f93-2}}&{\le} \sum_{j = 1}^t ( 1 - \alpha)^{2 (t - j)} \, \ExpSub{\xi_j}{\sqnorm{\omega_j^u}} \\
			\oversetref{Lem.}{\ref{appdx-technical-lem:sadiev-generalization}}&{\le} 18 \lbd_1^{2 - q} \gamma^q \bar{\delta}^q \sum_{j = 1}^t (1 - \alpha)^{2 (t - j)} \\
			& \le \frac{18 \lbd_1^{2 - q} \gamma^q \bar{\delta}^q}{\alpha}, \numberthis\label{74b4e2bb-4754-4dd3-ae90-37ef40362f93-3}
		\end{alignat*}
		where $\bar{\delta} \ge 0$ is the parameter in~\Cref{ass:mean-squared-smoothness-2}. In the application of~\Cref{appdx-technical-lem:sadiev-generalization} we use~\Cref{ass:mean-squared-smoothness-2}
		\[ \ExpSub{\xi_j}{\norm{\left[ \nabla F(x_j) - \nabla F(x_{j - 1}) \right] - \left[ \nabla f(x_j, \xi_j) - \nabla f(x_{j - 1}, \xi_j) \right]}^q} \le \bar{\delta}^q \norm{x_j - x_{j - 1}}^q \le \gamma^q \bar{\delta}^q. \]
		Moreover, applying~\Cref{appdx-technical-lem:sadiev-generalization} requires to take $\lbd_1 \ge 2 \max_{j \in [t]} \norm{\nabla F(x_j) - \nabla F(x_{j - 1})}$ and, since 
		\[ \norm{\nabla F(x_j) - \nabla F(x_{j - 1})} \oversetref{Ass.}{\ref{ass:L-lipschitz-gradients}}{\le} L_1 \norm{x_j - x_{j - 1}} = \gamma L_1, \]
		then it is enough to have $\lbd_1 \ge 2 \gamma L_1$. Setting $G_1 \eqdef \frac{18 \lbd_1^{2 - q} \gamma^q \bar{\delta}^q}{\alpha} > 0$ gives the desired bound.
		
		\paragraph{Bounding \Circled{6}.} Similarly to the previous lemma, we set $\widetilde{c}_1 \eqdef 8 \lbd_1^2$ and $\widetilde{b}_1 \eqdef \left( \frac{\widetilde{c}_1}{3} + \sqrt{\frac{\widetilde{c}_1^2}{9} + 2 \widetilde{G}_1} \right) \log \frac{2}{\delta''}$. For the choice $\widetilde{G}_1$, we have the bound 
		\[ \sum_{j = 1}^t \widetilde{\sigma}_j^2 \oversetlab{\eqref{a26983bb-60b7-49ab-ac0e-a1d026f1c51e-2}}{\le} 16 \lbd_1^2 \sum_{j = 1}^t \ExpSub{\xi_j}{\sqnorm{(1 - \alpha)^{t - j} \omega_j^u}} \oversetlab{\eqref{74b4e2bb-4754-4dd3-ae90-37ef40362f93-3}}{\le} 16 \lbd_1^2 \cdot \frac{18 \lbd_1^{2 - q} \gamma^q \bar{\delta}^q}{\alpha} = \frac{288 \lbd_1^{4 - q} \gamma^q \bar{\delta}^q}{\alpha}, \]
		(where we require $\lbd_1 \ge 2 \gamma L_1$). Hence, if we let $\widetilde{G}_1 \eqdef \frac{288 \lbd_1^{4 - q} \gamma^q \bar{\delta}^q}{\alpha}$ this establishes the desired bound.
		
		\paragraph{Bounding \Circled{7}.} As we already did in the two last paragraphs, we have
		\[ \Circled{7} \eqdef \sum_{j = 1}^t \ExpSub{\xi_j}{\sqnorm{(1 - \alpha)^{t - j} \omega_j^u}} \oversetlab{\eqref{74b4e2bb-4754-4dd3-ae90-37ef40362f93-3}}{\le} \frac{18 \lbd_1^{2 - q} \gamma^q \bar{\delta}^q}{\alpha}, \]
		as desired.
		
		\paragraph{Bounding \Circled{8}.} For the last bound, if we assume that $\lbd_1 \ge 2 \gamma L_1$ then, with probability one we have
		\begin{alignat*}{2}
			\Circled{8} \eqdef* \norm{\sum_{j = 1}^t (1 - \alpha)^{t - j} \omega_j^b} \\
			& \le \sum_{j = 1}^t (1 - \alpha)^{t - j} \norm{\omega_j^b} \\
			\oversetrel{rel:6bcc4f03-036d-4593-857a-983716b63243-2}&{\le} 4 \lbd_1^{1 - q} \gamma^q \bar{\delta}^q \sum_{j = 1}^t (1 - \alpha)^{t - j} \\
			& \le \frac{4 \lbd_1^{1 - q} \gamma^q \bar{\delta}^q}{\alpha},
		\end{alignat*}
		where in~\relref{rel:6bcc4f03-036d-4593-857a-983716b63243-2} we use~\Cref{appdx-technical-lem:sadiev-generalization}, more precisely, for any $j \in [t]$,
		\[ \norm{\omega_j^b} \le 2^q \lbd_1^{1 - q} \gamma^q \bar{\delta}^q \le 4 \lbd_1^{1 - q} \gamma^q \bar{\delta}^q, \]
		since $\E{\norm{\left[ \nabla F(x_t) - \nabla F(x_{t - 1}) \right] - \left[ \nabla f(x_t, \xi_t) - \nabla f(x_{t - 1}, \xi_t) \right]}^q} \le \gamma^q \bar{\delta}^q$ by~\Cref{ass:mean-squared-smoothness-2}.
		
		\paragraph{Summing up all bounds \Circled{5}, \Circled{6}, \Circled{7} and \Circled{8}.} We introduce the same events $E_{\Circled{5}, t}$ and $E_{\Circled{6}, t}$ as in~\eqref{1ded530f-71dc-46eb-9fdb-5308c908cb80} and~\eqref{92e2f28f-a246-45f3-bdea-ba28563d3825} respectively. Then, by our choice of $c_1$, $b_1$ and $G_1$ we have the bound
		\begin{alignat*}{2}
			b_1 \eqdef \left( \frac{c_1}{3} + \sqrt{\frac{c_1^2}{9} + 2 G_1} \right) \log \frac{2}{\delta''} & \le \left( \frac{2 c_1}{3} + \sqrt{2 G_1} \right) \log \frac{2}{\delta''} \\
			& = \left( \frac{4 \lbd_1}{3} + \sqrt{\frac{36 \lbd_1^{2 - q} \gamma^q \bar{\delta}^q}{\alpha}} \right) \log \frac{2}{\delta''} \\
			& = \lbd_1 \left( \frac{4}{3} + 6 \sqrt{\frac{1}{\alpha} \left(\frac{\gamma \bar{\delta}}{\lbd_1} \right)^q} \right) \log \frac{2}{\delta''} \\
			\oversetrel{rel:122b8683-5b62-44f2-b3aa-52164298844a}&{\le} \frac{22 \lbd_1}{3} \log \frac{2}{\delta''}, \numberthis\label{35c58053-3ad9-4b62-ae23-2a6f658c08da-3}
		\end{alignat*}
		where~\relref{rel:122b8683-5b62-44f2-b3aa-52164298844a} holds provided $\lbd_1 \ge \gamma \bar{\delta} \alpha^{-\frac{1}{q}}$. Moreover, by our choice of $\widetilde{c}_1$, $\widetilde{b}_1$ and $\widetilde{G}_1$ we also have the bound
		\begin{alignat*}{2}
			\widetilde{b}_1 \eqdef \left( \frac{\widetilde{c}_1}{3} + \sqrt{\frac{\widetilde{c}_1^2}{9} + 2 \widetilde{G}_1} \right) \log \frac{2}{\delta''} & \le \left( \frac{2 \widetilde{c}_1}{3} + \sqrt{2 \widetilde{G}_1} \right) \log \frac{2}{\delta''} \\
			& = \left( \frac{16 \lbd_1^2}{3} + \sqrt{\frac{576 \lbd_1^{4 - q} \gamma^q \bar{\delta}^q}{\alpha}} \right) \log \frac{2}{\delta''} \\
			& = \lbd_1^2 \left( \frac{16}{3} + 24 \sqrt{\frac{1}{\alpha} \left(\frac{\gamma \bar{\delta}}{\lbd_1} \right)^q} \right) \log \frac{2}{\delta''} \\
			& \le \frac{88 \lbd_1^2}{3} \log \frac{2}{\delta''}, \numberthis\label{c33e8593-ede2-4c3e-b76d-7e7ba9b43443-3}
		\end{alignat*}
		where the last inequality holds given $\lbd_1 \ge \gamma \bar{\delta} \alpha^{-\frac{1}{q}}$.
		
		Next, given $\lbd_1 \ge \max\ens{2 \gamma L_1, \gamma \bar{\delta} \alpha^{-\frac{1}{q}}}$, as in the previous lemma, we have $\Proba{E_{\Circled{5}, t}} \ge 1 - \delta''$ and $\Proba{E_{\Circled{6}, t}} \ge 1 - \delta''$ hence, by the union bound inequality we have
		\[ \Proba{E_{\Circled{5}, t} \cap E_{\Circled{6}, t}} \ge 1 - 2 \delta'', \]
		and on the event $E_{\Circled{5}, t} \cap E_{\Circled{6}, t}$ we obtain the inequality
		\begin{alignat*}{2}
			\norm{\sum_{j = 1}^t (1 - \alpha)^{t - j} \omega_j} & \le \norm{\sum_{j = 1}^t (1 - \alpha)^{t - j} \omega_j^b} + \norm{\sum_{j = 1}^t (1 - \alpha)^{t - j} \omega_j^u} \\
			\oversetlab{\eqref{21eaf977-44d0-4b4b-8b63-5af59e5cc089-2}}&{\le} \Circled{5} + \sqrt{2 \cdot \Circled{6} + 2 \cdot \Circled{7}} + \Circled{8} \\
			\oversetlab{\eqref{35c58053-3ad9-4b62-ae23-2a6f658c08da-3}+\eqref{c33e8593-ede2-4c3e-b76d-7e7ba9b43443-3}}&{\le} \frac{22 \lbd_1}{3} \log \frac{2}{\delta''} + \sqrt{2 \cdot \frac{176 \lbd_1^2}{3} \log \frac{2}{\delta''} + 2 \cdot \frac{18 \lbd_1^{2 - q} \gamma^q \bar{\delta}^q}{\alpha}} + \frac{4 \lbd_1^{1 - q} \gamma^q \bar{\delta}^q}{\alpha} \\
			\oversetrel{rel:65519c2e-05e9-4a6c-87ef-8631a38365a-2}&{\le} \lbd_1 \left( \frac{22}{3} + \sqrt{\frac{176}{3} + \frac{36}{\alpha} \left( \frac{\gamma \bar{\delta}}{\lbd_1} \right)^q} + \frac{4}{\alpha} \left( \frac{\gamma \bar{\delta}}{\lbd_1} \right)^q \right) \log \frac{2}{\delta''} \\
			\oversetrel{rel:95ba6431-6302-47f5-be7b-98aa809fc5a5-2}&{\le} \lbd_1 \left( \frac{22}{3} + \sqrt{\frac{176}{3} + 36} + 4 \right) \log \frac{2}{\delta''} \\
			& \le 22 \lbd_1 \log \frac{2}{\delta''},
		\end{alignat*}
		where in~\relref{rel:65519c2e-05e9-4a6c-87ef-8631a38365a-2} we use the fact that $\log \frac{2}{\delta''} \ge 1$ and in~\relref{rel:95ba6431-6302-47f5-be7b-98aa809fc5a5-2} we use $\lbd_1 \ge \gamma \bar{\delta} \alpha^{-\frac{1}{q}}$. This proves the lemma.
	\end{proof}
	
	\begin{lemma}\label{appdx-lem:high-probability-analysis-bound-hess}
		Under~\Cref{ass:L-lipschitz-gradients,ass:L-lipschitz-hessians,ass:q-bounded-central-moment-hessian}, for any $\delta'' \in \intof{0}{\frac{1}{2}}$ and any $t \in \Int{0}{T - 1}$, if the clipping threshold satisfies
		\[ \lbd_1 \ge \max\ens{2 \gamma L_1, \gamma \sigma_2 \alpha^{-\frac{1}{q}}}, \]
		for all $j \in [t]$ then, with probability at least $1 - 2 \delta''$, we have
		\[ \norm{\sum_{j = 1}^t (1 - \alpha)^{t - j} \omega_j} \le \frac{2 \min\ens{\gamma L_1, \gamma^2 L_2}}{\alpha} + 22 \lbd_1 \log \frac{2}{\delta''}. \]
	\end{lemma}
	
	\begin{proof}
		First, using~\eqref{9715675c-c075-4d13-a6e5-978530f178fe-hess},~\eqref{9715675c-c075-4d13-a6e5-978530f178fd-hess} and~\eqref{5bf5f4fb-a644-41c4-a212-9c4930a0075d-hess} we have
		\[ \norm{\sum_{j = 1}^t (1 - \alpha)^{t - j} \omega_j} \le \underbrace{\norm{\sum_{j = 1}^t (1 - \alpha)^{t - j} \omega_j^r}}_{\Circled{9}} + \underbrace{\norm{\sum_{j = 1}^t (1 - \alpha)^{t - j} \omega_j^b}}_{\Circled{8}} + \norm{\sum_{j = 1}^t (1 - \alpha)^{t - j} \omega_j^u}, \numberthis\label{9221fb23-db5b-497f-9062-1c71b7f12882-hess} \]
		As in the previous lemmas, the computation~\eqref{21eaf977-44d0-4b4b-8b63-5af59e5cc089-2} is identical and we now need to bound all terms \Circled{5}, \Circled{6}, \Circled{7}, \Circled{8} and \Circled{9} where
		\[ \Circled{5} \eqdef \abs{\sum_{j = 1}^t W_j^t}, \quad \Circled{6} \eqdef \sum_{j = 1}^t Z_j^t \,\, \text{ and } \,\, \Circled{7} \eqdef \sum_{j = 1}^t \ExpSub{\hat{\xi}_j}{\sqnorm{(1 - \alpha)^{t - j} \omega_j^u}}, \]
		and $W_1^t, \ldots, W_t^t$ is a martingale difference sequence since, by definition of $\omega_j^u$ and $W_j^t$, for all $j \in [t]$ we have $\ExpCond{W_j^t}{W_{j - 1}^t, \ldots, W_1^t} = 0$ since 
		\[ \ExpSub{q_j, \hat{\xi}_j}{\omega_j^u} = \ExpSub{q_j}{\ExpSubCond{\hat{\xi}_j}{\omega_j^u}{q_j}} = 0. \]
		The same argument applies to the sequence $Z_1^t, \ldots, Z_t^t$.
		
		\paragraph{Bounding \Circled{5}.} As already done before, we choose $c_1 \eqdef 2 \lbd_1$ and $b_1 \eqdef \left( \frac{c_1}{3} + \sqrt{\frac{c_1^2}{9} + 2 G_1} \right) \log \frac{2}{\delta''} > 0$. Then, to select $G_1 > 0$ we need to bound the sum $\sigma_1^2 + \cdots + \sigma_t^2$, i.e.,
		\begin{alignat*}{2}
			\sum_{j = 1}^t \sigma_j^2 \oversetlab{\eqref{74b4e2bb-4754-4dd3-ae90-37ef40362f93-2}}&{\le} \sum_{j = 1}^t ( 1 - \alpha)^{2 (t - j)} \, \ExpSub{\hat{\xi}_j}{\sqnorm{\omega_j^u}} \\
			\oversetref{Lem.}{\ref{appdx-technical-lem:sadiev-generalization}}&{\le} 18 \lbd_1^{2 - q} \gamma^q \sigma_2^q \sum_{j = 1}^t (1 - \alpha)^{2 (t - j)} \\
			& \le \frac{18 \lbd_1^{2 - q} \gamma^q \sigma_2^q}{\alpha}. \numberthis\label{74b4e2bb-4754-4dd3-ae90-37ef40362f93-hess}
		\end{alignat*}
		In the application of~\Cref{appdx-technical-lem:sadiev-generalization} we use~\Cref{ass:q-bounded-central-moment-hessian}
		\begin{alignat*}{2}
			&\ExpSub{\hat{\xi}_j}{\norm{\nabla^2 f(\hat{x}_j, \hat{\xi}_j)(x_j - x_{j - 1}) - \nabla^2 F(\hat{x}_j) (x_j - x_{j - 1})}^q}\\
			&\qquad\qquad\qquad\qquad\le \ExpSub{\hat{\xi}_j}{\normop{\nabla^2 f(\hat{x}_j, \hat{\xi}_j) - \nabla^2 F(\hat{x}_j)}^q \cdot \norm{x_j - x_{j - 1}}^q} \le \gamma^q \sigma_2^q, 
		\end{alignat*}
		as
		\[ \ExpSub{\hat{\xi}_j}{\nabla^2 f(\hat{x}_j, \hat{\xi}_j)(x_j - x_{j - 1})} = \nabla^2 F(\hat{x}_j)(x_j - x_{j - 1}). \]
		Moreover, applying~\Cref{appdx-technical-lem:sadiev-generalization} requires to take $\lbd_1 \ge 2 \max_{j \in [t]} \norm{\nabla^2 F(\hat{x}_j) (x_j - x_{j - 1})}$ and, since 
		\[ \norm{\nabla^2 F(\hat{x}_j) (x_j - x_{j - 1})} \le \normop{\nabla^2 F(\hat{x}_j)} \cdot \norm{x_j - x_{j - 1}} \oversetref{Ass.}{\ref{ass:L-lipschitz-gradients}}{\le} L_1 \norm{x_j - x_{j - 1}} = \gamma L_1, \]
		then it is enough to have $\lbd_1 \ge 2 \gamma L_1$. Setting $G_1 \eqdef \frac{18 \lbd_1^{2 - q} \gamma^q \sigma_2^q}{\alpha} > 0$ gives the desired bound.
		
		\paragraph{Bounding \Circled{6}.} Similarly to the previous lemma, we set $\widetilde{c}_1 \eqdef 8 \lbd_1^2$ and $\widetilde{b}_1 \eqdef \left( \frac{\widetilde{c}_1}{3} + \sqrt{\frac{\widetilde{c}_1^2}{9} + 2 \widetilde{G}_1} \right) \log \frac{2}{\delta''}$. For the choice $\widetilde{G}_1$, we have the bound 
		\[ \sum_{j = 1}^t \widetilde{\sigma}_j^2 \oversetlab{\eqref{a26983bb-60b7-49ab-ac0e-a1d026f1c51e-2}}{\le} 16 \lbd_1^2 \sum_{j = 1}^t \ExpSub{q_j, \hat{\xi}_j}{\sqnorm{(1 - \alpha)^{t - j} \omega_j^u}} \oversetlab{\eqref{74b4e2bb-4754-4dd3-ae90-37ef40362f93-hess}}{\le} 16 \lbd_1^2 \cdot \frac{18 \lbd_1^{2 - q} \gamma^q \sigma_2^q}{\alpha} = \frac{288 \lbd_1^{4 - q} \gamma^q \sigma_2^q}{\alpha}, \]
		(where we require $\lbd_1 \ge 2 \gamma L_1$). Hence, if we let $\widetilde{G}_1 \eqdef \frac{288 \lbd_1^{4 - q} \gamma^q \sigma_2^q}{\alpha}$ this establishes the desired bound.
		
		\paragraph{Bounding \Circled{7}.} As we already did in the two last paragraphs, we have
		\[ \Circled{7} \eqdef \sum_{j = 1}^t \ExpSub{\hat{\xi}_j}{\sqnorm{(1 - \alpha)^{t - j} \omega_j^u}} \oversetlab{\eqref{74b4e2bb-4754-4dd3-ae90-37ef40362f93-hess}}{\le} \frac{18 \lbd_1^{2 - q} \gamma^q \sigma_2^q}{\alpha}, \]
		as desired.
		
		\paragraph{Bounding \Circled{8}.} For this bound, if we assume that $\lbd_1 \ge 2 \gamma L_1$ then, with probability one we have
		\begin{alignat*}{2}
			\Circled{8} \eqdef* \norm{\sum_{j = 1}^t (1 - \alpha)^{t - j} \omega_j^b} \\
			& \le \sum_{j = 1}^t (1 - \alpha)^{t - j} \norm{\omega_j^b} \\
			\oversetrel{rel:6bcc4f03-036d-4593-857a-983716b63243-2}&{\le} 4 \lbd_1^{1 - q} \gamma^q \sigma_2^q \sum_{j = 1}^t (1 - \alpha)^{t - j} \\
			& \le \frac{4 \lbd_1^{1 - q} \gamma^q \sigma_2^q}{\alpha},
		\end{alignat*}
		where in~\relref{rel:6bcc4f03-036d-4593-857a-983716b63243-2} we use~\Cref{appdx-technical-lem:sadiev-generalization} (with $\ExpSubCond{\hat{\xi}_j}{\cdot}{q_j}$), more precisely, for any $j \in [t]$,
		\[ \norm{\omega_j^b} \le 2^q \lbd_1^{1 - q} \gamma^q \sigma_2^q \le 4 \lbd_1^{1 - q} \gamma^q \sigma_2^q, \]
		since 
		\begin{alignat*}{2}
			&\E{\norm{\nabla^2 f(\hat{x}_j, \hat{\xi}_j)(x_j - x_{j - 1}) - \nabla^2 F(\hat{x}_j) (x_j - x_{j - 1})}^q}\\
			&\qquad\qquad \le \E{\normop{\nabla^2 f(\hat{x}_j, \hat{\xi}_j) - \nabla^2 F(\hat{x}_j)}^q \cdot \norm{x_j - x_{j - 1}}^q} \\
			&\qquad\qquad\qquad\qquad\oversetref{Lem.}{\ref{lem:tower-property}}{=} \gamma^q \, \E{\ExpCond{\normop{\nabla^2 f(\hat{x}_j, \hat{\xi}_j) - \nabla^2 F(\hat{x}_j)}^q}{\hat{x}_j}} \\
			&\qquad\qquad\qquad\qquad\qquad\qquad\oversetref{Ass.}{\ref{ass:q-bounded-central-moment-hessian}}{\le} \gamma^q \sigma_2^q.
		\end{alignat*}
		
		\paragraph{Bounding \Circled{9}.} For the last bound, we use the triangle inequality, this gives
		\begin{alignat*}{2}
			\Circled{9} \eqdef* \norm{\sum_{j = 1}^^t (1 - \alpha)^{t  j} \omega_j^r} \le \sum_{j = 1}^t (1 - \alpha)^{t - j} \norm{\omega_j^r} \\
			& = \sum_{j = 1}^t (1 - \alpha)^{t - j} \norm{\nabla^2 F(\hat{x}_j)(x_j - x_{j - 1}) - \left[ \nabla F(x_j) - \nabla F(x_{j - 1}) \right]} \\
			\oversetrel{rel:e38b5d27-2b7b-4482-87a7-e89e0257e963}&{\le} \min\ens{2 \gamma L_1, \frac{\gamma^2 L_2}{2}} \sum_{j = 1}^t (1 - \alpha)^{t - j}  \le \frac{2 \min\ens{\gamma L_1, \gamma^2 L_2}}{\alpha}, \numberthis\label{2752ee44-9282-4441-80e7-2703274875e6}
		\end{alignat*}
		where in~\relref{rel:e38b5d27-2b7b-4482-87a7-e89e0257e963} we use the bounds~\eqref{46d7ba91-ef00-4ec2-9b1e-d222e9968289a} and~\eqref{46d7ba91-ef00-4ec2-9b1e-d222e9968289b}, which holds with probability one.
		
		\paragraph{Summing up all bounds \Circled{5}, \Circled{6}, \Circled{7}, \Circled{8} and \Circled{9}.} We introduce the same events $E_{\Circled{5}, t}$ and $E_{\Circled{6}, t}$ as in~\eqref{1ded530f-71dc-46eb-9fdb-5308c908cb80} and~\eqref{92e2f28f-a246-45f3-bdea-ba28563d3825} respectively. Then, by our choice of $c_1$, $b_1$ and $G_1$ we have the bound
		\begin{alignat*}{2}
			b_1 \eqdef \left( \frac{c_1}{3} + \sqrt{\frac{c_1^2}{9} + 2 G_1} \right) \log \frac{2}{\delta''} & \le \left( \frac{2 c_1}{3} + \sqrt{2 G_1} \right) \log \frac{2}{\delta''} \\
			& = \left( \frac{4 \lbd_1}{3} + \sqrt{\frac{36 \lbd_1^{2 - q} \gamma^q \sigma_2^q}{\alpha}} \right) \log \frac{2}{\delta''} \\
			& = \lbd_1 \left( \frac{4}{3} + 6 \sqrt{\frac{1}{\alpha} \left(\frac{\gamma \sigma_2}{\lbd_1} \right)^q} \right) \log \frac{2}{\delta''} \\
			\oversetrel{rel:122b8683-5b62-44f2-b3aa-52164298844a}&{\le} \frac{22 \lbd_1}{3} \log \frac{2}{\delta''}, \numberthis\label{35c58053-3ad9-4b62-ae23-2a6f658c08da-hess}
		\end{alignat*}
		where~\relref{rel:122b8683-5b62-44f2-b3aa-52164298844a} holds provided $\lbd_1 \ge \gamma \bar{\delta} \alpha^{-\frac{1}{q}}$. Moreover, by our choice of $\widetilde{c}_1$, $\widetilde{b}_1$ and $\widetilde{G}_1$ we also have the bound
		\begin{alignat*}{2}
			\widetilde{b}_1 \eqdef \left( \frac{\widetilde{c}_1}{3} + \sqrt{\frac{\widetilde{c}_1^2}{9} + 2 \widetilde{G}_1} \right) \log \frac{2}{\delta''} & \le \left( \frac{2 \widetilde{c}_1}{3} + \sqrt{2 \widetilde{G}_1} \right) \log \frac{2}{\delta''} \\
			& = \left( \frac{16 \lbd_1^2}{3} + \sqrt{\frac{576 \lbd_1^{4 - q} \gamma^q \sigma_2^q}{\alpha}} \right) \log \frac{2}{\delta''} \\
			& = \lbd_1^2 \left( \frac{16}{3} + 24 \sqrt{\frac{1}{\alpha} \left(\frac{\gamma \sigma_2}{\lbd_1} \right)^q} \right) \log \frac{2}{\delta''} \\
			& \le \frac{88 \lbd_1^2}{3} \log \frac{2}{\delta''}, \numberthis\label{c33e8593-ede2-4c3e-b76d-7e7ba9b43443-hess}
		\end{alignat*}
		where the last inequality holds given $\lbd_1 \ge \gamma \sigma_2 \alpha^{-\frac{1}{q}}$.
		
		Next, given $\lbd_1 \ge \max\ens{2 \gamma L_1, \gamma \sigma_2 \alpha^{-\frac{1}{q}}}$, as in the previous lemma, we have $\Proba{E_{\Circled{5}, t}} \ge 1 - \delta''$ and $\Proba{E_{\Circled{6}, t}} \ge 1 - \delta''$ hence, by the union bound inequality we have
		\[ \Proba{E_{\Circled{5}, t} \cap E_{\Circled{6}, t}} \ge 1 - 2 \delta'', \]
		and on the event $E_{\Circled{5}, t} \cap E_{\Circled{6}, t}$ we obtain the inequality
		\begin{alignat*}{2}
			\norm{\sum_{j = 1}^t (1 - \alpha)^{t - j} \omega_j} & \le \norm{\sum_{j = 1}^t (1 - \alpha)^{t - j} \omega_j^r} + \norm{\sum_{j = 1}^t (1 - \alpha)^{t - j} \omega_j^b} + \norm{\sum_{j = 1}^t (1 - \alpha)^{t - j} \omega_j^u} \\
			\oversetlab{\eqref{21eaf977-44d0-4b4b-8b63-5af59e5cc089-2}+\eqref{9221fb23-db5b-497f-9062-1c71b7f12882-hess}}&{\le} \Circled{9} + \Circled{5} + \sqrt{2 \cdot \Circled{6} + 2 \cdot \Circled{7}} + \Circled{8} \\
			\oversetlab{\eqref{2752ee44-9282-4441-80e7-2703274875e6}+\eqref{35c58053-3ad9-4b62-ae23-2a6f658c08da-hess}+\eqref{c33e8593-ede2-4c3e-b76d-7e7ba9b43443-hess}}&{\le} \frac{2 \min\ens{\gamma L_1, \gamma^2 L_2}}{\alpha} + \frac{22 \lbd_1}{3} \log \frac{2}{\delta''} \\
			&\qquad+ \sqrt{2 \cdot \frac{176 \lbd_1^2}{3} \log \frac{2}{\delta''} + 2 \cdot \frac{18 \lbd_1^{2 - q} \gamma^q \sigma_2^q}{\alpha}} + \frac{4 \lbd_1^{1 - q} \gamma^q \sigma_2^q}{\alpha} \\
			\oversetrel{rel:65519c2e-05e9-4a6c-87ef-8631a38365a-2}&{\le} \frac{2 \min\ens{\gamma L_1, \gamma^2 L_2}}{\alpha} + \lbd_1 \left( \frac{22}{3} + \sqrt{\frac{176}{3} + \frac{36}{\alpha} \left( \frac{\gamma \sigma_2}{\lbd_1} \right)^q} + \frac{4}{\alpha} \left( \frac{\gamma \sigma_2}{\lbd_1} \right)^q \right) \log \frac{2}{\delta''} \\
			\oversetrel{rel:95ba6431-6302-47f5-be7b-98aa809fc5a5-2}&{\le} \frac{2 \min\ens{\gamma L_1, \gamma^2 L_2}}{\alpha} + \lbd_1 \left( \frac{22}{3} + \sqrt{\frac{176}{3} + 36} + 4 \right) \log \frac{2}{\delta''} \\
			& \le \frac{2 \min\ens{\gamma L_1, \gamma^2 L_2}}{\alpha} + 22 \lbd_1 \log \frac{2}{\delta''},
		\end{alignat*}
		where in~\relref{rel:65519c2e-05e9-4a6c-87ef-8631a38365a-2} we use the fact that $\log \frac{2}{\delta''} \ge 1$ and in~\relref{rel:95ba6431-6302-47f5-be7b-98aa809fc5a5-2} we use $\lbd_1 \ge \gamma \sigma_2 \alpha^{-\frac{1}{q}}$. This proves the lemma.
	\end{proof}
	
	\subsection{Proof of~\Cref{thm:clipped-nsgd-mvr-convergence-analysis}}
	
	With~\Cref{appdx-lem:high-probability-analysis-bound-1,appdx-lem:high-probability-analysis-bound-2} in our hands, we are now ready to prove the main result of this section, i.e., high-probability convergence guarantees for~\Cref{algo:clipped-nsgd-mvr}, i.e., \algname{\clip-NSGD-MVR}.
	
	\begin{restate-theorem}{\ref{thm:clipped-nsgd-mvr-convergence-analysis}}
		Under~\Cref{ass:lower-boundedness,ass:p-bounded-central-moment-gradient,ass:mean-squared-smoothness}, let $T \ge 1$ and $\delta \in \intof{0}{1}$ such that $\log \frac{8 T}{\delta} \ge 1$ and suppose that we choose $g_0 = 0$ in~\Cref{algo:clipped-nsgd-mvr} and let $\Delta_1 \eqdef F(x_0) - F^{\inf}$ the initial sub-optimality. Suppose we run~\Cref{algo:clipped-nsgd-mvr} using momentum parameter $\alpha = \max\{T^{-\frac{p}{2 p - 1}}, T^{-\frac{p q}{p (2 q + 1) - 2 q}}\}$, clipping thresholds $\lbd_1 = 2 \gamma \bar{L} \alpha^{-\frac{1}{q}}$ and $\lbd_2 = \max\{4 \sqrt{\bar{L} \Delta_1}, \sigma_1 \alpha^{-\frac{1}{p}}\}$ and with stepsize
		\[ \gamma = \cO\left( \min\ens{\sqrt{\frac{\Delta_1}{\bar{L} T}}, \alpha \sqrt{\frac{\Delta_1}{\bar{L}}}, \frac{1}{\alpha T \log \frac{T}{\delta}} \sqrt{\frac{\Delta_1}{\bar{L}}}, \frac{\Delta_1}{\sigma_1 \alpha^{\frac{p - 1}{p}} T \log \frac{T}{\delta}}, \sqrt{\frac{\Delta_1 \alpha^{\frac{1}{q}}}{\bar{L} T \log \frac{T}{\delta}}}} \right). \]
		
		Then, with probability at least $1 - \delta$, the output of~\Cref{algo:clipped-nsgd-mvr} satisfies
		\[ \frac{1}{T} \sum_{t = 0}^{T - 1} \norm{\nabla F(x_t)} \le \frac{2 \Delta_1}{\gamma T}, \]
		and, by our choice of parameters, the norm of the gradients converges at the rate
		\[ \frac{1}{T} \sum_{t = 0}^{T - 1} \norm{\nabla F(x_t)} = \cO\left( \left( \frac{\sqrt{\bar{L} \Delta_1} + \sigma_1}{T^{\frac{p - 1}{2 p - 1} \wedge \frac{q (p - 1)}{p (2 q + 1) - 2 q}}} \right)\log \frac{T}{\delta} \right), \]
		with high probability.
	\end{restate-theorem}
	
	\begin{proof}
		Let us remind from the proofs of~\Cref{appdx-lem:high-probability-analysis-bound-1,appdx-lem:high-probability-analysis-bound-2} that we defined
		\begin{alignat*}{2}
			E_{\Circled{1}, t} \eqdef* \ens{\abs{\sum_{j = 1}^t V_j^t} \le b_2 \, \text{ or } \, \sum_{j = 1}^t \sigma_j^2 > G_2 \log \frac{2}{\delta'}}, & E_{\Circled{2}, t} \eqdef* \ens{\abs{\sum_{j = 1}^t Y_j^t} \le \widetilde{b}_2 \, \text{ or } \, \sum_{j = 1}^t \widetilde{\sigma}_j^2 > \widetilde{G}_2 \log \frac{2}{\delta'}}, \\
			E_{\Circled{5}, t} \eqdef* \ens{\abs{\sum_{j = 1}^t W_j^t} \le b_1 \, \text{ or } \, \sum_{j = 1}^t \sigma_j^2 > G_1 \log \frac{2}{\delta''}}, \quad & E_{\Circled{6}, t} \eqdef* \ens{\abs{\sum_{j = 1}^t Z_j^t} \le \widetilde{b}_1 \, \text{ or } \, \sum_{j = 1}^t \widetilde{\sigma}_j^2 > \widetilde{G}_1 \log \frac{2}{\delta''}},
		\end{alignat*}
		where
		\[ c_1 \eqdef 2 \lbd_1, \quad \widetilde{c}_1 \eqdef 8 \lbd_1^2, \quad c_2 \eqdef 2 \lbd_2, \quad \widetilde{c}_2 \eqdef 8 \lbd_2^2, \]
		\[ G_1 \eqdef \frac{72 \lbd_1^{2 - q} \gamma^q \bar{L}^q}{\alpha}, \quad \widetilde{G}_1 \eqdef \frac{1152 \lbd_1^{4 - q} \gamma^q \bar{L}^q}{\alpha}, \quad G_2 \eqdef \frac{18 \lbd_2^{2 - p} \sigma_1^p}{\alpha}, \quad \widetilde{G}_2 \eqdef \frac{288 \lbd_2^{4 - p} \sigma_1^p}{\alpha}, \]
		\[ b_1 \le \frac{40 \lbd_1}{3} \log \frac{2}{\delta''}, \quad \widetilde{b}_1 \le  \frac{160 \lbd_1^2}{3} \log \frac{2}{\delta''}, \quad b_2 \le \frac{22 \lbd_2}{3} \log \frac{2}{\delta'}, \quad \widetilde{b}_2 \le \frac{88 \lbd_2^2}{3} \log \frac{2}{\delta'}, \]
		and we have shown the following bounds:
		\[ \Proba{E_{\Circled{1}, t}} \ge 1 - \delta', \quad \Proba{E_{\Circled{2}, t}} \ge 1 - \delta', \quad \Proba{E_{\Circled{5}, t}} \ge 1 - \delta'', \quad \Proba{E_{\Circled{6}, t}} \ge 1 - \delta'', \numberthis\label{4ead95f0-2b11-4c81-9020-65644456a5a3} \]
		valid for all $t \in \Int{0}{T - 1}$.
		
		The idea of the proof is based on a technique from previous works in the literature~\citep{10.5555/3495724.3496985,pmlr-v202-sadiev23a,pmlr-v195-liu23c,sadiev2025second} and proceeds by a mathematical induction. Our goal is to prove that for any $\tau \in \Int{0}{T - 1}$, the event
		\[ G_{\tau} \eqdef E_{\tau} \cap E_{1, \tau} \cap E_{2, \tau} \cap E_{5, \tau} \cap E_{6, \tau}, \]
		holds with probability at least $1 - \frac{\tau \delta}{T}$, where the event $E_{\tau}$ is defined as
		\[ E_{\tau} \eqdef \bigcap_{t = 0}^{\tau} \ens{\gamma \sum_{j = 0}^t \norm{\nabla F(x_j)} + \Delta_{t + 1} \le 2 \Delta_1}, \]
		and
		\[ E_{1, \tau} \eqdef \bigcap_{t = 1}^{\tau} E_{\Circled{1}, t}, \quad E_{2, \tau} \eqdef \bigcap_{t = 1}^{\tau} E_{\Circled{2}, t}, \quad E_{5, \tau} \eqdef \bigcap_{t = 1}^{\tau} E_{\Circled{5}, t}, \quad E_{6, \tau} \eqdef \bigcap_{t = 1}^{\tau} E_{\Circled{6}, t}. \]
		
		\paragraph{Base case.} For the base case $\tau = 0$ we have $G_0 = E_0 = \ens{\Delta_1 \le 2 \Delta_1}$ which holds with probability $1 = 1 - \frac{\tau \delta}{T}$ since $\tau = 0$.
		
		\paragraph{Inductive case.} Now, let us assume that the induction hypothesis holds for $\tau - 1$, i.e., $\Proba{G_{\tau - 1}} \ge 1 - \frac{(\tau - 1) \delta}{T}$ and we need to prove that is holds for $\tau$ also, that is, $\Proba{G_{\tau}} \ge 1 - \frac{\tau \delta}{T}$. First, let us observe that
		\[ E_{\tau - 1} \cap E_{1, \tau} \cap E_{2, \tau} \cap E_{5, \tau} \cap E_{6, \tau} = G_{\tau - 1} \cap E_{\Circled{1}, \tau} \cap E_{\Circled{2}, \tau} \cap E_{\Circled{5}, \tau} \cap E_{\Circled{6}, \tau}. \numberthis\label{37c9a99d-499b-4b0e-ab10-001e0f246a9e} \]
		Then, we have
		\begin{alignat*}{2}
			\Proba{E_{\tau - 1} \cap E_{1, \tau} \cap E_{2, \tau} \cap E_{5, \tau} \cap E_{6, \tau}} \oversetlab{\eqref{37c9a99d-499b-4b0e-ab10-001e0f246a9e}}&{=} \Proba{G_{\tau - 1} \cap E_{\Circled{1}, \tau} \cap E_{\Circled{2}, \tau} \cap E_{\Circled{5}, \tau} \cap E_{\Circled{6}, \tau}} \\
			& = 1 - \Proba{\overline{G_{\tau - 1} \cap E_{\Circled{1}, \tau} \cap E_{\Circled{2}, \tau} \cap E_{\Circled{5}, \tau} \cap E_{\Circled{6}, \tau}}} \\
			& \ge 1 - \Proba{\overline{G_{\tau - 1}}} - \Proba{\overline{E_{\Circled{1}, \tau}}} - \Proba{\overline{E_{\Circled{2}, \tau}}} - \Proba{\overline{E_{\Circled{5}, \tau}}} - \Proba{\overline{E_{\Circled{6}, \tau}}} \\
			\oversetlab{\eqref{4ead95f0-2b11-4c81-9020-65644456a5a3}}&{\ge} 1 - \frac{(\tau - 1) \delta}{T} - 2 \delta' - 2 \delta'' \\
			& = 1 - \frac{\tau \delta}{T} + \left( \frac{\delta}{T} - 2 \delta' - 2 \delta'' \right) \\
			& = 1 - \frac{\tau \delta}{T}, \numberthis\label{092ed50d-cb6a-420c-b2dc-26ff706c2ac0}
		\end{alignat*}
		provided $\delta' = \delta'' = \frac{\delta}{4 T}$. Now, given $\tau \in \Int{0}{T - 1}$, on the event $E_{\tau - 1}$ we have
		\[ \Delta_t \le \gamma \sum_{j = 0}^{t - 1} \norm{\nabla F(x_j)} + \Delta_t \le 2 \Delta_1 \]
		for all integer $1 \le t \le \tau$ and therefore, for all $j \in \Int{0}{\tau - 1}$, we have
		\[ \norm{\nabla F(x_j)} \le \sqrt{2 \bar{L} \Delta_j} \le 2 \sqrt{\bar{L} \Delta_1}, \]
		hence, by our choice of clipping threshold $\lbd_2 = \max\ens{4 \sqrt{\bar{L} \Delta_1}, \sigma_1 \alpha^{-\frac{1}{p}}}$ we have $\lbd_2 \ge 2 \max_{j \in [t]} \norm{\nabla F(x_j)}$. Additionally, if we assume the clipping level $\lbd_1 = 2 \gamma \bar{L} \alpha^{-\frac{1}{q}}$ then, on the event $E_{\tau - 1} \cap E_{1, \tau} \cap E_{2, \tau} \cap E_{5, \tau} \cap E_{6, \tau}$, we have
		\begin{alignat*}{2}
			\gamma \sum_{t = 0}^{\tau} \norm{\nabla F(x_t)} + \Delta_{\tau + 1} \oversetref{Lem.}{\ref{appdx-lem:high-probability-analysis-descent-lemma}}&{\le} \Delta_1 
			\begin{aligned}[t]
				&+ \frac{2 \gamma \sqrt{2 \bar{L} \Delta_1}}{\alpha} + \frac{\gamma^2 \bar{L} (\tau + 1)}{2} \\
				&+ 2 \gamma \alpha \sum_{t = 1}^{\tau} \norm{\sum_{j = 1}^t (1 - \alpha)^{t - j} \theta_j} + 2 \gamma (1 - \alpha) \sum_{t = 1}^{\tau} \norm{\sum_{j = 1}^t (1 - \alpha)^{t - j} \omega_j}
			\end{aligned} \\
			\oversetref{Lem.}{\ref{appdx-lem:high-probability-analysis-bound-1}+\ref{appdx-lem:high-probability-analysis-bound-2}}&{\le} \Delta_1 
			\begin{aligned}[t]
				&+ \frac{2 \gamma \sqrt{2 \bar{L} \Delta_1}}{\alpha} + \frac{\gamma^2 \bar{L} (\tau + 1)}{2} \\
				&+ 44 \gamma \alpha \tau \lbd_2 \log \frac{8 T}{\delta} + 92 \gamma (1 - \alpha) \tau \lbd_1 \log \frac{8 T}{\delta},
			\end{aligned}
		\end{alignat*}
		and, as $0 \le \tau \le T - 1$ then, by our choice of stepsize
		\[ \gamma = \min\ens{\sqrt{\frac{\Delta_1}{2 \bar{L} T}}, \frac{\alpha}{8} \sqrt{\frac{\Delta_1}{2 \bar{L}}}, \frac{1}{704 \alpha T \log \frac{8 T}{\delta}} \sqrt{\frac{\Delta_1}{\bar{L}}}, \frac{\Delta_1}{176 \sigma_1 \alpha^{\frac{p - 1}{p}} T \log \frac{8 T}{\delta}}, \sqrt{\frac{\Delta_1 \alpha^{\frac{1}{q}}}{736 \bar{L} T \log \frac{8 T}{\delta}}}}, \numberthis\label{787efe83-7ee7-4154-be0c-7c0b97faf778} \]
		we have
		\[ \gamma \sum_{t = 0}^{\tau} \norm{\nabla F(x_t)} + \Delta_{\tau + 1} \oversetlab{\eqref{787efe83-7ee7-4154-be0c-7c0b97faf778}}{\le} \Delta_1 + \frac{\Delta_1}{4} + \frac{\Delta_1}{4} + \frac{\Delta_1}{4} + \frac{\Delta_1}{4} = 2 \Delta_1, \]
		therefore we have $E_{\tau} \cap E_{1, \tau} \cap E_{2, \tau} \cap E_{5, \tau} \cap E_{6, \tau} = E_{\tau - 1} \cap E_{1, \tau} \cap E_{2, \tau} \cap E_{5, \tau} \cap E_{6, \tau}$ which leads to
		\begin{alignat*}{2}
			\Proba{G_{\tau}} & = \Proba{E_{\tau} \cap E_{1, \tau} \cap E_{2, \tau} \cap E_{5, \tau} \cap E_{6, \tau}} \\
			& = \Proba{E_{\tau - 1} \cap E_{1, \tau} \cap E_{2, \tau} \cap E_{5, \tau} \cap E_{6, \tau}} \\
			\oversetlab{\eqref{092ed50d-cb6a-420c-b2dc-26ff706c2ac0}}&{\ge} 1 - \frac{\tau \delta}{T}, \numberthis\label{14f8ce1e-9a04-4930-b1bc-84b37957e99f}
		\end{alignat*}
		as claimed. This achieves the proof of the induction.
		
		\paragraph{Bounding $\frac{1}{T} \sum\limits_{t = 0}^{T - 1} \norm{\nabla F(x_t)}$ in high-probability.} From the previous paragraph, we have for $\tau = T$
		\[ \Proba{E_T} \ge \Proba{G_T} \ge 1 - \delta, \]
		hence, with probability at least $1 - \delta$ we have
		\[ \gamma \sum_{t = 0}^{T - 1} \norm{\nabla F(x_t)} + \Delta_T \le 2 \Delta_1, \]
		which implies the bound
		\[ \frac{1}{T} \sum_{t = 0}^{T - 1} \norm{\nabla F(x_t)} \le \frac{2 \Delta_1}{\gamma T}, \]
		with probability at least $1 - \delta$. By our choice of stepsize~\eqref{787efe83-7ee7-4154-be0c-7c0b97faf778} we obtain
		\[ \frac{1}{T} \sum_{t = 0}^{T - 1} \norm{\nabla F(x_t)} \le \frac{2 \Delta_1}{\gamma T} = \cO\left( \max\ens{\sqrt{\frac{\bar{L} \Delta_1}{T}}, \frac{\sqrt{\bar{L} \Delta_1}}{\alpha T}, \alpha \sqrt{\bar{L} \Delta_1} \log \frac{T}{\delta}, \sigma_1 \alpha^{\frac{p - 1}{p}} \log \frac{T}{\delta}, \sqrt{\frac{\bar{L} \Delta_1 \log \frac{T}{\delta}}{T \alpha^{\frac{1}{q}}}}} \right), \]
		and choosing $\alpha = \max\ens{T^{-\frac{p}{2 p - 1}}, T^{-\frac{p q}{p (2 q + 1) - 2 q}}}$ we have
		\[ \alpha^{\frac{p - 1}{p}} = \max\ens{T^{-\frac{p - 1}{2 p - 1}}, T^{-\frac{q (p - 1)}{p (2 q + 1) - 2 q}}}, \,\, \text{ and } \,\, T^{\frac{1}{2}} \alpha^{\frac{1}{2q}} \ge T^{\frac{1}{2} \left( 1 - \frac{p}{p (2 q + 1) - 2 q} \right)} = T^{\frac{q (p - 1)}{p (2 q + 1) - 2 q}}, \]
		and
		\[ \alpha T \ge T^{1 - \frac{p}{2 p - 1}} = T^{\frac{p - 1}{2 p - 1}}, \]
		thus, assuming $\log \frac{8 T}{\delta} \ge 1$ we get
		\begin{alignat*}{2}
			\frac{1}{T} \sum_{t = 0}^{T - 1} \norm{\nabla F(x_t)} & = \cO\Bigg( \max\Bigg\{\sqrt{\frac{\bar{L} \Delta_1}{T}}, \frac{\sqrt{\bar{L} \Delta_1}}{T^{\frac{p - 1}{2 p - 1}}} \log \frac{T}{\delta}, \frac{\sqrt{\bar{L} \Delta_1}}{T^{\frac{q (p - 1)}{p (2 q + 1) - 2 q}}} \log \frac{T}{\delta}, \frac{\sigma_1}{T^{\frac{p - 1}{2 p - 1}}} \log \frac{T}{\delta},\\
			&\qquad\qquad\qquad\qquad\frac{\sigma_1}{T^{\frac{q (p - 1)}{p (2 q + 1) - 2 q}}} \log \frac{T}{\delta}, \frac{\sqrt{\bar{L} \Delta_1 \log \frac{T}{\delta}}}{T^{\frac{q (p - 1)}{p (2 q + 1) - 2 q}}}\Bigg\} \Bigg) \\
			\oversetrel{rel:4736f1d8-e280-4380-aecc-558076fa1309}&{=} \cO\left( \left( \frac{\sqrt{\bar{L} \Delta_1} + \sigma_1}{T^{\frac{p - 1}{2 p - 1} \wedge \frac{q (p - 1)}{p (2 q + 1) - 2 q}}} \right)\log \frac{T}{\delta} \right)
		\end{alignat*}
		since $\frac{p - 1}{2 p - 1} \le \frac{1}{2}$. In~\relref{rel:4736f1d8-e280-4380-aecc-558076fa1309} we use $\wedge$ to denote the minimum between the two exponents $\frac{p - 1}{2 p - 1}$ and $\frac{q (p - 1)}{p (2 q + 1) - 2 q}$.
		
		This concludes the proof of the theorem.
	\end{proof}
	
	\subsection{Proof of~\Cref{thm:clipped-nsgd-mvr-convergence-analysis-2}}

	\begin{theorem}\label{thm:clipped-nsgd-mvr-convergence-analysis-2}
		Under~\Cref{ass:lower-boundedness,ass:L-lipschitz-gradients,ass:p-bounded-central-moment-gradient,ass:mean-squared-smoothness-2}, let $T \ge 1$ and $\beta \in \intof{0}{1}$ be such that $\log \frac{8 T}{\beta} \ge 1$. Let $x_0\in\mathbb{R}^d$ and define $\Delta_1 \eqdef F(x_0) - F^{\inf}$.
		Suppose that \Cref{algo:clipped-nsgd-mvr} is run with $g_0 = 0$, momentum parameter $\alpha = \max\{T^{-\frac{p}{2 p - 1}}, T^{-\frac{p q}{p (2 q + 1) - 2 q}}\}$, clipping thresholds $\lbd_1 = \max\{2 \gamma L_1, \gamma \delta \alpha^{-\frac{1}{q}}\}$ and $\lbd_2 = \max\{4 \sqrt{L_1 \Delta_1}, \sigma_1 \alpha^{-\frac{1}{p}}\}$, and stepsize
		\begin{alignat*}{2}
			\gamma & = \cO\left( \min\left\{\alpha \sqrt{\frac{\Delta_1}{L_1}}, \frac{1}{\alpha T \log \frac{T}{\beta}} \sqrt{\frac{\Delta_1}{L_1}}, \frac{\Delta_1}{\sigma_1 \alpha^{\frac{p - 1}{p}} T \log \frac{T}{\beta}}, \sqrt{\frac{\Delta_1 \alpha^{\frac{1}{q}}}{\delta T \log \frac{T}{\beta}}}, \sqrt{\frac{\Delta_1}{L_1 T \log \frac{T}{\beta}}}\right\} \right).
		\end{alignat*}
		
		Then, with probability at least $1 - \beta$, the output of~\Cref{algo:clipped-nsgd-mvr} satisfies
		\[ \frac{1}{T} \sum_{t = 0}^{T - 1} \norm{\nabla F(x_t)} \le \frac{2 \Delta_1}{\gamma T}, \]
		and, by our choice of parameters, the norm of the gradients converges with high probability at the rate
		\begin{alignat*}{2}
			\frac{1}{T} \sum_{t = 0}^{T - 1} \norm{\nabla F(x_t)} & = \cO\left( \left( \frac{\sqrt{L_1 \Delta_1} + \sigma_1}{T^{\frac{p - 1}{2 p - 1} \wedge \frac{q (p - 1)}{p (2 q + 1) - 2 q}}} \right)\log \frac{T}{\beta} + \frac{\sqrt{\delta \Delta_1}}{T^{\frac{q (p - 1)}{p (2 q + 1) - 2 q}}} \sqrt{\log \frac{T}{\beta}} \right).
		\end{alignat*}
	\end{theorem}
	
	\begin{proof}
		From the proofs of~\Cref{appdx-lem:high-probability-analysis-bound-1,appdx-lem:high-probability-analysis-bound-3} we have
		\begin{alignat*}{2}
			E_{\Circled{1}, t} \eqdef* \ens{\abs{\sum_{j = 1}^t V_j^t} \le b_2 \, \text{ or } \, \sum_{j = 1}^t \sigma_j^2 > G_2 \log \frac{2}{\delta'}}, & E_{\Circled{2}, t} \eqdef* \ens{\abs{\sum_{j = 1}^t Y_j^t} \le \widetilde{b}_2 \, \text{ or } \, \sum_{j = 1}^t \widetilde{\sigma}_j^2 > \widetilde{G}_2 \log \frac{2}{\delta'}}, \\
			E_{\Circled{5}, t} \eqdef* \ens{\abs{\sum_{j = 1}^t W_j^t} \le b_1 \, \text{ or } \, \sum_{j = 1}^t \sigma_j^2 > G_1 \log \frac{2}{\delta''}}, \quad & E_{\Circled{6}, t} \eqdef* \ens{\abs{\sum_{j = 1}^t Z_j^t} \le \widetilde{b}_1 \, \text{ or } \, \sum_{j = 1}^t \widetilde{\sigma}_j^2 > \widetilde{G}_1 \log \frac{2}{\delta''}},
		\end{alignat*}
		where
		\[ c_1 \eqdef 2 \lbd_1, \quad \widetilde{c}_1 \eqdef 8 \lbd_1^2, \quad c_2 \eqdef 2 \lbd_2, \quad \widetilde{c}_2 \eqdef 8 \lbd_2^2, \]
		\[ G_1 \eqdef \frac{18 \lbd_1^{2 - q} \gamma^q \bar{\delta}^q}{\alpha}, \quad \widetilde{G}_1 \eqdef \frac{288 \lbd_1^{4 - q} \gamma^q \bar{\delta}^q}{\alpha}, \quad G_2 \eqdef \frac{18 \lbd_2^{2 - p} \sigma_1^p}{\alpha}, \quad \widetilde{G}_2 \eqdef \frac{288 \lbd_2^{4 - p} \sigma_1^p}{\alpha}, \]
		\[ b_1 \le \frac{22 \lbd_1}{3} \log \frac{2}{\delta''}, \quad \widetilde{b}_1 \le  \frac{88 \lbd_1^2}{3} \log \frac{2}{\delta''}, \quad b_2 \le \frac{22 \lbd_2}{3} \log \frac{2}{\delta'}, \quad \widetilde{b}_2 \le \frac{88 \lbd_2^2}{3} \log \frac{2}{\delta'}, \]
		and we have shown the following bounds:
		\[ \Proba{E_{\Circled{1}, t}} \ge 1 - \delta', \quad \Proba{E_{\Circled{2}, t}} \ge 1 - \delta', \quad \Proba{E_{\Circled{5}, t}} \ge 1 - \delta'', \quad \Proba{E_{\Circled{6}, t}} \ge 1 - \delta'', \numberthis\label{4ead95f0-2b11-4c81-9020-65644456a5a3-2} \]
		valid for all $t \in \Int{0}{T - 1}$.
		
		Now, we follows the exact same steps as in the proof of the previous theorem, up to~\eqref{092ed50d-cb6a-420c-b2dc-26ff706c2ac0}. Then, given $\tau \in \Int{0}{T - 1}$, on the event $E_{\tau - 1}$ we have
		\[ \Delta_t \le \gamma \sum_{j = 0}^{t - 1} \norm{\nabla F(x_j)} + \Delta_t \le 2 \Delta_1 \]
		for all integer $1 \le t \le \tau$ and therefore, for all $j \in \Int{0}{\tau - 1}$, we have
		\[ \norm{\nabla F(x_j)} \le \sqrt{2 \bar{L} \Delta_j} \le 2 \sqrt{L_1 \Delta_1}, \]
		hence, by our choice of clipping threshold $\lbd_2 = \max\ens{4 \sqrt{L_1 \Delta_1}, \sigma_1 \alpha^{-\frac{1}{p}}}$ we have $\lbd_2 \ge 2 \max_{j \in [t]} \norm{\nabla F(x_j)}$. Additionally, if we assume the clipping level $\lbd_1 = \max\ens{2 \gamma L_1, \gamma \bar{\delta} \alpha^{-\frac{1}{q}}}$ then, on the event $E_{\tau - 1} \cap E_{1, \tau} \cap E_{2, \tau} \cap E_{5, \tau} \cap E_{6, \tau}$, we have
		\begin{alignat*}{2}
			\gamma \sum_{t = 0}^{\tau} \norm{\nabla F(x_t)} + \Delta_{\tau + 1} \oversetref{Lem.}{\ref{appdx-lem:high-probability-analysis-descent-lemma-2}}&{\le} \Delta_1 
			\begin{aligned}[t]
				&+ \frac{2 \gamma \sqrt{2 L_1 \Delta_1}}{\alpha} + \frac{\gamma^2 L_1 (\tau + 1)}{2} \\
				&+ 2 \gamma \alpha \sum_{t = 1}^{\tau} \norm{\sum_{j = 1}^t (1 - \alpha)^{t - j} \theta_j} + 2 \gamma (1 - \alpha) \sum_{t = 1}^{\tau} \norm{\sum_{j = 1}^t (1 - \alpha)^{t - j} \omega_j}
			\end{aligned} \\
			\oversetref{Lem.}{\ref{appdx-lem:high-probability-analysis-bound-1}+\ref{appdx-lem:high-probability-analysis-bound-3}}&{\le} \Delta_1 + \frac{2 \gamma \sqrt{2 L_1 \Delta_1}}{\alpha} + \frac{\gamma^2 L_1 (\tau + 1)}{2} \\
			&\qquad\qquad\qquad+ 44 \gamma \alpha \tau \lbd_2 \log \frac{8 T}{\delta} + 44 \gamma (1 - \alpha) \tau \lbd_1 \log \frac{8 T}{\delta}, \numberthis\label{0f3f34c1-906a-43c0-b66b-9472d1bce764}
		\end{alignat*}
		and, as $0 \le \tau \le T - 1$ then, by our choice of stepsize
		\begin{alignat*}{2} \gamma &= \min\Bigg\{\sqrt{\frac{\Delta_1}{2 L_1 T}}, \frac{\alpha}{8} \sqrt{\frac{\Delta_1}{2 L_1}}, \frac{1}{704 \alpha T \log \frac{8 T}{\delta}} \sqrt{\frac{\Delta_1}{L_1}}, \\
			&\qquad\qquad\qquad\frac{\Delta_1}{176 \sigma_1 \alpha^{\frac{p - 1}{p}} T \log \frac{8 T}{\delta}}, \sqrt{\frac{\Delta_1}{352 L_1 T \log \frac{8 T}{\delta}}}, \sqrt{\frac{\Delta_1 \alpha^{\frac{1}{q}}}{176 \bar{\delta} T \log \frac{8 T}{\delta}}}\Bigg\}, \numberthis\label{787efe83-7ee7-4154-be0c-7c0b97faf778-2}\end{alignat*}
		we have
		\[ \gamma \sum_{t = 0}^{\tau} \norm{\nabla F(x_t)} + \Delta_{\tau + 1} \oversetlab{\eqref{787efe83-7ee7-4154-be0c-7c0b97faf778-2}}{\le} \Delta_1 + \frac{\Delta_1}{4} + \frac{\Delta_1}{4} + \frac{\Delta_1}{4} + \frac{\Delta_1}{4} = 2 \Delta_1, \]
		and we can achieve the proof by induction as in~\eqref{14f8ce1e-9a04-4930-b1bc-84b37957e99f}.
		
		Then, as before, with probability as least $1 - \delta$, we have
		\[ \gamma \sum_{t = 0}^{T - 1} \norm{\nabla F(x_t)} + \Delta_T \le 2 \Delta_1, \]
		which implies the bound
		\[ \frac{1}{T} \sum_{t = 0}^{T - 1} \norm{\nabla F(x_t)} \le \frac{2 \Delta_1}{\gamma T}, \]
		with probability at least $1 - \delta$. By our choice of stepsize~\eqref{787efe83-7ee7-4154-be0c-7c0b97faf778-2} we obtain
		\begin{alignat*}{2}
			\frac{1}{T} \sum_{t = 0}^{T - 1} \norm{\nabla F(x_t)} &\le \frac{2 \Delta_1}{\gamma T}\\
			& = \cO\Bigg( \max\Bigg\{\sqrt{\frac{L_1 \Delta_1}{T}}, \frac{\sqrt{L_1 \Delta_1}}{\alpha T}, \alpha \sqrt{L_1 \Delta_1} \log \frac{T}{\delta},\\
			&\qquad \qquad \qquad \qquad \sigma_1 \alpha^{\frac{p - 1}{p}} \log \frac{T}{\delta}, \sqrt{\frac{L_1 \Delta_1 \log \frac{T}{\delta}}{T}},\sqrt{\frac{\bar{\delta} \Delta_1 \log \frac{T}{\delta}}{T \alpha^{\frac{1}{q}}}}\Bigg\} \Bigg) \\
			& = \cO\left( \max\ens{\frac{\sqrt{L_1 \Delta_1}}{\alpha T}, \alpha \sqrt{L_1 \Delta_1} \log \frac{T}{\delta}, \sigma_1 \alpha^{\frac{p - 1}{p}} \log \frac{T}{\delta}, \sqrt{\frac{L_1 \Delta_1 \log \frac{T}{\delta}}{T}}, \sqrt{\frac{\bar{\delta} \Delta_1 \log \frac{T}{\delta}}{T \alpha^{\frac{1}{q}}}}} \right),
		\end{alignat*}
		as we assumed $\log \frac{8 T}{\delta} \ge 1$. Now, choosing $\alpha = \max\ens{T^{-\frac{p}{2 p - 1}}, T^{-\frac{p q}{p (2 q + 1) - 2 q}}}$ we have
		\[ \alpha^{\frac{p - 1}{p}} = \max\ens{T^{-\frac{p - 1}{2 p - 1}}, T^{-\frac{q (p - 1)}{p (2 q + 1) - 2 q}}}, \,\, \text{ and } \,\, T^{\frac{1}{2}} \alpha^{\frac{1}{2q}} \ge T^{\frac{1}{2} \left( 1 - \frac{p}{p (2 q + 1) - 2 q} \right)} = T^{\frac{q (p - 1)}{p (2 q + 1) - 2 q}}, \]
		and
		\[ \alpha T \ge T^{1 - \frac{p}{2 p - 1}} = T^{\frac{p - 1}{2 p - 1}}, \]
		thus, given $\log \frac{8 T}{\delta} \ge 1$ we get
		\begin{alignat*}{2}
			\frac{1}{T} \sum_{t = 0}^{T - 1} \norm{\nabla F(x_t)} & = \cO\Bigg( \max\Bigg\{\frac{\sqrt{L_1 \Delta_1}}{T^{\frac{p - 1}{2 p - 1}}} \log \frac{T}{\delta}, \frac{\sqrt{L_1 \Delta_1}}{T^{\frac{q (p - 1)}{p (2 q + 1) - 2 q}}} \log \frac{T}{\delta}, \frac{\sigma_1}{T^{\frac{p - 1}{2 p - 1}}} \log \frac{T}{\delta}, \\
			&\qquad\qquad\qquad\qquad \frac{\sigma_1}{T^{\frac{q (p - 1)}{p (2 q + 1) - 2 q}}} \log \frac{T}{\delta}, \frac{\sqrt{\bar{\delta} \Delta_1 \log \frac{T}{\delta}}}{T^{\frac{q (p - 1)}{p (2 q + 1) - 2 q}}}\Bigg\} \Bigg) \\
			\oversetrel{rel:4736f1d8-e280-4380-aecc-558076fa1309}&{=} \cO\left( \left( \frac{\sqrt{L_1 \Delta_1} + \sigma_1}{T^{\frac{p - 1}{2 p - 1} \wedge \frac{q (p - 1)}{p (2 q + 1) - 2 q}}} \right)\log \frac{T}{\delta} + \frac{\sqrt{\bar{\delta} \Delta_1}}{T^{\frac{q (p - 1)}{p (2 q + 1) - 2 q}}} \sqrt{\log \frac{T}{\delta}} \right)
		\end{alignat*}
		since $\frac{p - 1}{2 p - 1} \le \frac{1}{2}$ and $\sqrt{\log \frac{T}{\delta}} \lesssim \log \frac{T}{\delta}$. In~\relref{rel:4736f1d8-e280-4380-aecc-558076fa1309} we use $\wedge$ to denote the minimum between the two exponents $\frac{p - 1}{2 p - 1}$ and $\frac{q (p - 1)}{p (2 q + 1) - 2 q}$.
		
		This concludes the proof of the theorem.
	\end{proof}
	
	\subsection{Proof of~\Cref{thm:clipped-nsgd-mvr-convergence-analysis-hess}}
	
	\begin{algorithm}%
		\caption{\algname{Clipped NSGD-\texttt{Hess}} (Clipped Normalized \algname{SGD} with Hessian-corrected Momentum)}%
		\label{algo:clipped-nsgd-mvr-hess}%
		
		\DontPrintSemicolon%
		\SetKwProg{Init}{Initialization}{:}{}%
		\Init{}{%
			$x_0 \in \R^d$, the starting point\;
			$T > 0$, the number of iterations\;
			$g_0 \in \R^d$, an initial vector\;
			$\gamma > 0$, the stepsize\;
			$\alpha \in \intof{0}{1}$, the momentum parameter for \texttt{Hess}\;
			$\lambda_1, \lambda_2 >0$, the clipping thresholds\;
		}%
		
		\vspace{\baselineskip}
		
		$x_1 \gets x_0 - \gamma \frac{g_0}{\norm{g_0}}$\;
		\For{$t = 1, 2, \ldots, T - 1$}{%
			Sample $q_t \sim \mathcal{U}\left(\intff{0}{1}\right)$\;
			$\hat{x}_t \gets q_t x_t + (1 - q_t) x_{t - 1}$\;
			\tcp*[h]{Apply \texttt{Hess}, here $\xi_t, \hat{\xi}_t \sim \cD$ are independent.}\;
			$g_t \gets (1 - \alpha) \left( g_{t - 1} + \texttt{clip}\left(\nabla^2 f(\hat{x}_t, \hat{\xi}_t) (x_t - x_{t - 1}), \lambda_1\right) \right) + \alpha \texttt{clip}\left(\nabla f \left( x_t, \xi_t \right), \lambda_2\right)$\;
			\tcp*[h]{Do one descent step.}\;
			$x_{t + 1} \gets x_t - \gamma \frac{g_t}{\norm{g_t}}$\;
		}%
		\KwOut{$x_T$}%
	\end{algorithm}%
	
	\begin{theorem}\label{thm:clipped-nsgd-mvr-convergence-analysis-hess}
		Under~\Cref{ass:lower-boundedness,ass:L-lipschitz-gradients,ass:p-bounded-central-moment-gradient,ass:L-lipschitz-hessians,ass:q-bounded-central-moment-hessian}, let $T \ge 1$ and $\beta \in \intof{0}{1}$ such that $\log \frac{8 T}{\beta} \ge 1$ and suppose that we choose $g_0 = 0$ in~\Cref{algo:clipped-nsgd-mvr-hess} and let $\Delta \eqdef F(x_0) - F^{\inf}$ the initial sub-optimality. Suppose we run~\Cref{algo:clipped-nsgd-mvr-hess} using momentum parameter $\alpha = \max\{T^{-\frac{p}{2 p - 1}}, T^{-\frac{p q}{p (2 q + 1) - 2 q}}\}$, clipping thresholds $\lbd_1 = \max\{2 \gamma L_1, \gamma \sigma_2 \alpha^{-\frac{1}{q}}\}$ and $\lbd_2 = \max\{4 \sqrt{L_1 \Delta_1}, \sigma_1 \alpha^{-\frac{1}{p}}\}$ and with stepsize
		\begin{alignat*}{2}
			\gamma & = \cO\left( \min\left\{\alpha \sqrt{\frac{\Delta}{L_1}}, \sqrt[3]{\frac{\Delta \alpha}{L_2 T}}, \frac{1}{\alpha T \log \frac{T}{\beta}} \sqrt{\frac{\Delta}{L_1}}, \frac{\Delta_1}{\sigma_1 \alpha^{\frac{p - 1}{p}} T \log \frac{T}{\beta}}, \sqrt{\frac{\Delta \alpha^{\frac{1}{q}}}{\sigma_2 T \log \frac{T}{\beta}}}, \sqrt{\frac{\Delta}{L_1 T \log \frac{T}{\beta}}}\right\} \right).
		\end{alignat*}
		
		Then, with probability at least $1 - \beta$, the output of~\Cref{algo:clipped-nsgd-mvr} satisfies
		\[ \frac{1}{T} \sum_{t = 0}^{T - 1} \norm{\nabla F(x_t)} \le \frac{2 \Delta_1}{\gamma T}, \]
		and, by our choice of parameters, the norm of the gradients converges at the rate
		\begin{alignat*}{2}
			\frac{1}{T} \sum_{t = 0}^{T - 1} \norm{\nabla F(x_t)}  = \cO\left( \frac{\left( L_2^{\nicefrac{1}{2}} \Delta \right)^{\frac{2}{3}}}{T^{\frac{2 (p - 1)}{4 p - 3}}} + \left( \frac{\sqrt{L_1 \Delta} + \sigma_1}{T^{\min\left\{\frac{p - 1}{2 p - 1},~ \frac{q (p - 1)}{p (2 q + 1) - 2 q}\right\}}} \right)\log \frac{T}{\beta} + \frac{\sqrt{\sigma_2 \Delta}}{T^{\frac{q (p - 1)}{p (2 q + 1) - 2 q}}} \sqrt{\log \frac{T}{\beta}} \right),
		\end{alignat*}
		with high probability.
	\end{theorem}

	\begin{proof}
		The proof readily follows from the proof of~\Cref{thm:clipped-nsgd-mvr-convergence-analysis-2} where $\bar{\delta}$ is replaced by $\sigma_2$. Additionally, by~\Cref{appdx-lem:high-probability-analysis-bound-hess} we need to bound an extra term in~\eqref{0f3f34c1-906a-43c0-b66b-9472d1bce764}, i.e., we have
		\begin{alignat*}{2}
			\gamma \sum_{t = 0}^{\tau} \norm{\nabla F(x_t)} + \Delta_{\tau + 1} \oversetref{Lem.}{\ref{appdx-lem:high-probability-analysis-descent-lemma-2}}&{\le} \Delta_1 
			\begin{aligned}[t]
				&+ \frac{2 \gamma \sqrt{2 L_1 \Delta_1}}{\alpha} + \frac{\gamma^2 L_1 (\tau + 1)}{2} \\
				&+ 2 \gamma \alpha \sum_{t = 1}^{\tau} \norm{\sum_{j = 1}^t (1 - \alpha)^{t - j} \theta_j} + 2 \gamma (1 - \alpha) \sum_{t = 1}^{\tau} \norm{\sum_{j = 1}^t (1 - \alpha)^{t - j} \omega_j}
			\end{aligned} \\
			\oversetref{Lem.}{\ref{appdx-lem:high-probability-analysis-bound-hess}+\ref{appdx-lem:high-probability-analysis-bound-hess}}&{\le} \Delta_1 
			\begin{aligned}[t]
				&+ \frac{2 \gamma \sqrt{2 L_1 \Delta_1}}{\alpha} + \frac{\gamma^2 L_1 (\tau + 1)}{2} + \frac{4 (1 - \alpha) \gamma^3 L_2 (\tau + 1)}{\alpha} \\
				&+ 44 \gamma \alpha \tau \lbd_2 \log \frac{8 T}{\delta} + 44 \gamma (1 - \alpha) \tau \lbd_1 \log \frac{8 T}{\delta},
			\end{aligned}
		\end{alignat*}
		where in the last inequality, we drop the $\min\ens{\cdots}$ term and simply keep the $\gamma^3 L_2$ term. Additionally, as we have $0 \le \tau \le T - 1$ then choosing
		\begin{alignat*}{2}
			\gamma &= \min\left\{\sqrt{\frac{2 \Delta_1}{5 L_1 T}}, \frac{\alpha}{10} \sqrt{\frac{\Delta_1}{2 L_1}}, \sqrt[3]{\frac{\Delta_1 \alpha}{20 L_2 T}}, \frac{1}{880 \alpha T \log \frac{8 T}{\delta}} \sqrt{\frac{\Delta_1}{L_1}},\right. \\
			&\qquad\qquad\qquad\left. \frac{\Delta_1}{220 \sigma_1 \alpha^{\frac{p - 1}{p}} T \log \frac{8 T}{\delta}}, \sqrt{\frac{\Delta_1}{440 L_1 T \log \frac{8 T}{\delta}}}, \sqrt{\frac{\Delta_1 \alpha^{\frac{1}{q}}}{220 \sigma_2 T \log \frac{8 T}{\delta}}}\right\}, \numberthis\label{787efe83-7ee7-4154-be0c-7c0b97faf778-hess}
		\end{alignat*}
		gives
		\[ \gamma \sum_{t = 0}^{\tau} \norm{\nabla F(x_t)} + \Delta_{\tau + 1} \oversetlab{\eqref{787efe83-7ee7-4154-be0c-7c0b97faf778-hess}}{\le} \Delta_1 + \frac{\Delta_1}{5} + \frac{\Delta_1}{5} + \frac{\Delta_1}{5} + \frac{\Delta_1}{5} + \frac{\Delta_1}{5} = 2 \Delta_1. \]
		Then, as before, with probability as least $1 - \delta$, we have
		\[ \gamma \sum_{t = 0}^{T - 1} \norm{\nabla F(x_t)} + \Delta_T \le 2 \Delta_1, \]
		which implies the bound
		\[ \frac{1}{T} \sum_{t = 0}^{T - 1} \norm{\nabla F(x_t)} \le \frac{2 \Delta_1}{\gamma T}, \]
		with probability at least $1 - \delta$. By our choice of stepsize~\eqref{787efe83-7ee7-4154-be0c-7c0b97faf778-hess} we obtain
		\begin{alignat*}{2}
			\frac{1}{T} \sum_{t = 0}^{T - 1} \norm{\nabla F(x_t)} \le \frac{2 \Delta_1}{\gamma T} & = \cO\left( \max\left\{\sqrt{\frac{L_1 \Delta_1}{T}}, \frac{\sqrt{L_1 \Delta_1}}{\alpha T}, \alpha^{-\frac{1}{3}} \left( \frac{L_2^{\nicefrac{1}{2}} \Delta_1}{T} \right)^{\frac{2}{3}}, \alpha \sqrt{L_1 \Delta_1} \log \frac{T}{\delta}, \right. \right.\\
			&\qquad\qquad\qquad\qquad\qquad\left.\left.\sigma_1 \alpha^{\frac{p - 1}{p}} \log \frac{T}{\delta}, \sqrt{\frac{L_1 \Delta_1 \log \frac{T}{\delta}}{T}}, \sqrt{\frac{\sigma_2 \Delta_1 \log \frac{T}{\delta}}{T \alpha^{\frac{1}{q}}}}\right\} \right) \\
			& = \cO\left( \max\left\{\frac{\sqrt{L_1 \Delta_1}}{\alpha T}, \alpha^{-\frac{1}{3}} \left( \frac{L_2^{\nicefrac{1}{2}} \Delta_1}{T} \right)^{\frac{2}{3}}, \alpha \sqrt{L_1 \Delta_1} \log \frac{T}{\delta}, \right. \right.\\
			&\qquad\qquad\qquad\qquad\qquad\left.\left.\sigma_1 \alpha^{\frac{p - 1}{p}} \log \frac{T}{\delta}, \sqrt{\frac{L_1 \Delta_1 \log \frac{T}{\delta}}{T}}, \sqrt{\frac{\sigma_2 \Delta_1 \log \frac{T}{\delta}}{T \alpha^{\frac{1}{q}}}}\right\} \right),
		\end{alignat*}
		as we assumed $\log \frac{8 T}{\delta} \ge 1$. Now, choosing 
		\[ \alpha = \max\ens{T^{-\frac{p}{2 p - 1}}, T^{-\frac{p q}{p (2 q + 1) - 2 q}}, T^{-\frac{2 p}{4 p - 3}}} = \max\ens{T^{-\frac{p}{2 p - 1}}, T^{-\frac{p q}{p (2 q + 1) - 2 q}}}, \]
		since $\frac{p}{2 p - 1} \le \frac{2 p}{4 p - 3}$, we have
		\[ \alpha^{\frac{p - 1}{p}} = \max\ens{T^{-\frac{p - 1}{2 p - 1}}, T^{-\frac{q (p - 1)}{p (2 q + 1) - 2 q}}}, \quad T^{\frac{1}{2}} \alpha^{\frac{1}{2q}} \ge T^{\frac{1}{2} \left( 1 - \frac{p}{p (2 q + 1) - 2 q} \right)} = T^{\frac{q (p - 1)}{p (2 q + 1) - 2 q}}, \]
		and
		\[ \alpha T \ge T^{1 - \frac{p}{2 p - 1}} = T^{\frac{p - 1}{2 p - 1}}, \quad T^{\frac{2}{3}} \alpha^{\frac{1}{3}} \ge T^{\frac{2}{3} - \frac{2 p}{3 (4 p - 3)}} = T^{-\frac{2 (p - 1)}{4 p - 3}} \]
		thus, given $\log \frac{8 T}{\delta} \ge 1$ we get
		\begin{alignat*}{2}
			\frac{1}{T} \sum_{t = 0}^{T - 1} \norm{\nabla F(x_t)} &= \cO\left( \max\left\{\frac{\left( L_2^{\nicefrac{1}{2}} \Delta_1 \right)^{\frac{2}{3}}}{T^{\frac{2 (p - 1)}{4 p - 3}}}, \frac{\sqrt{L_1 \Delta_1}}{T^{\frac{p - 1}{2 p - 1}}} \log \frac{T}{\delta}, \frac{\sqrt{L_1 \Delta_1}}{T^{\frac{q (p - 1)}{p (2 q + 1) - 2 q}}} \log \frac{T}{\delta}, \right. \right.\\
			&\qquad\qquad\qquad\qquad\qquad\left.\left. \frac{\sigma_1}{T^{\frac{p - 1}{2 p - 1}}} \log \frac{T}{\delta}, \frac{\sigma_1}{T^{\frac{q (p - 1)}{p (2 q + 1) - 2 q}}} \log \frac{T}{\delta}, \frac{\sqrt{\sigma_2 \Delta_1 \log \frac{T}{\delta}}}{T^{\frac{q (p - 1)}{p (2 q + 1) - 2 q}}}\right\} \right) \\
			\oversetrel{rel:4736f1d8-e280-4380-aecc-558076fa1309}&{=} \cO\left( \frac{\left( L_2^{\nicefrac{1}{2}} \Delta_1 \right)^{\frac{2}{3}}}{T^{\frac{2 (p - 1)}{4 p - 3}}} + \left( \frac{\sqrt{L_1 \Delta_1} + \sigma_1}{T^{\frac{p - 1}{2 p - 1} \wedge \frac{q (p - 1)}{p (2 q + 1) - 2 q}}} \right)\log \frac{T}{\delta} + \frac{\sqrt{\sigma_2 \Delta_1}}{T^{\frac{q (p - 1)}{p (2 q + 1) - 2 q}}} \sqrt{\log \frac{T}{\delta}} \right),
		\end{alignat*}
		since $\frac{p - 1}{2 p - 1} \le \frac{1}{2}$ and $\sqrt{\log \frac{T}{\delta}} \lesssim \log \frac{T}{\delta}$. In~\relref{rel:4736f1d8-e280-4380-aecc-558076fa1309} we use $\wedge$ to denote the minimum between the two exponents $\frac{p - 1}{2 p - 1}$ and $\frac{q (p - 1)}{p (2 q + 1) - 2 q}$.
		
		This concludes the proof of the theorem.
	\end{proof}
	
	\newpage
	
	\section{Technical Lemmas}
	
	We list below the technical lemmas used throughout this paper.
	
	\begin{lemma}[{\citet[Lemma~10]{hübler2025gradientclippingnormalizationheavy}}]\label{appdx-technical-lem:app-von-Bahr-and-Essen}
		Let $p \in \intff{1}{2}$, and $X_1, \ldots, X_n \in \R^d$ be a martingale difference sequence (MDS), i.e., $\ExpCond{X_j}{X_{j-1}, \ldots, X_1} = 0$ a.s. for all $j = 1, \ldots, n$ satisfying 
		\begin{align*}
			\E{\norm{X_j}^p } < \infty, \qquad \text{for all } j = 1, \ldots, n. 
		\end{align*}
		Define $S_n \eqdef \sum\limits_{j=1}^n X_j$, then 
		\begin{align*}
			\E{\norm{S_n}^p} \leq 2 \sum_{j = 1}^n \E{\norm{X_j}^p}.
		\end{align*}
	\end{lemma}
	
	The following lemma will be useful to derive high-probability convergence bounds. The constant $2^{p - 2}$ appearing in~\eqref{2425acdd-8fb2-4df1-ae5c-908f6e744926} is a consequence of~\citet[Lemma~1]{he2025complexitynormalizedstochasticfirstorder} which we recall in~\Cref{lem:euclidean-norm-expansion-generalized}.
	\begin{lemma}[{\cite[Lemma~10]{NEURIPS2021_26901deb}}]\label{appdx-technical-lem:cutkosky-mehta}
		Let $X_1, \ldots, X_n \in \R^d$ be a sequence of random vectors. Let us define the following sequence $w_1, \ldots, w_n$ recursively by
		\begin{enumerate}
			\item $w_0 = 0$,
			\item if $\sum\limits_{i = 1}^{j - 1} X_i \neq 0$, then we set
			\[ w_j = \sign\left( \sum\limits_{i = 1}^{j - 1} w_i \right) \frac{\ps{\sum\limits_{i = 1}^{j - 1} X_i}{X_j}}{\norm{\sum\limits_{i = 1}^{j - 1} X_i}}, \]
			\item if $\sum\limits_{i = 1}^{j - 1} X_i = 0$, we set $w_j = 0$.
		\end{enumerate}
		Then, $\abs{w_j} \le \norm{X_j}$ for all $j \in \Int{1}{n}$ and
		\[ \norm{\sum_{i = 1}^n X_i} \le \abs{\sum_{i = 1}^n w_i} + \left( \max_{i \in [n]} \norm{X_i}^p + 2^{2 - p} \sum_{i = 1}^n \norm{X_i}^p  \right)^{\frac{1}{p}}. \numberthis\label{2425acdd-8fb2-4df1-ae5c-908f6e744926} \]
	\end{lemma}
	
	\begin{lemma}[Freedman's Inequality]\label{appdx-technical-lem:freedman-inequality}
		Let $X_1, \ldots, X_n \in \R^d$ be a martingale difference sequence (MDS), i.e., $\ExpCond{X_j}{X_{j-1}, \ldots, X_1} = 0$ a.s. for $j = 1, \ldots, n$. Assume there exists a constant $c > 0$ such that $\abs{X_j} \le c$ a.s. for $j = 1, \ldots, n$ and define $\sigma_j^2 \eqdef \ExpCond{X_j^2}{X_{j - 1}, \ldots, X_1}$. Then, for all $b > 0$ and all $G > 0$ we have
		\[ \Proba{\abs{\sum_{j = 1}^n X_j} > b \, \text{ and } \, \sum_{j = 1}^n \sigma_j^2 \le G} \le 2 \exp\left( -\frac{b^2}{2 G + \frac{2 b c}{3}} \right). \]
	\end{lemma}
	
	The next lemma is a generalization of~\cite[Lemma~5.1]{pmlr-v202-sadiev23a}.
	\begin{lemma}\label{appdx-technical-lem:sadiev-generalization}
		Let $\lbd > 0$ be a scalar, $X \in \R^d$ be a random vector and let $\widetilde{X} = \clip(X, \lbd)$ then
		\[ \norm{\widetilde{X} - \E{\widetilde{X}}} \le 2 \lbd. \numberthis\label{3f712cbe-7d84-446c-a934-ad43a82e08dd-1} \]
		Moreover, if for some $\sigma_1 \ge 0$ and $p \in \intof{1}{2}$ we have $\E{X} = x \in \R^d$, $\E{\norm{X - x}^p} \le \sigma_1^p$, and $\norm{x} \le \frac{\lbd}{2}$ then, for any $q \in \intff{p}{2}$ we have
		\begin{alignat*}{2}
			\norm{\E{\widetilde{X}} - x} & \le \frac{2^p \sigma_1^p}{\lbd^{p - 1}}, \numberthis\label{3f712cbe-7d84-446c-a934-ad43a82e08dd-2} \\
			\E{\norm{\widetilde{X} - x}^q} & \le 2 \cdot 3^q \lbd^{q - p} \sigma_1^p, \numberthis\label{3f712cbe-7d84-446c-a934-ad43a82e08dd-3} \\
			\E{\norm{\widetilde{X} - \E{\widetilde{X}}}^q} & \le 2^{3 - q} \cdot 3^q \lbd^{q - p} \sigma_1^p. \numberthis\label{3f712cbe-7d84-446c-a934-ad43a82e08dd-4}
		\end{alignat*}
	\end{lemma}
	
	\begin{proof}
		The proof follows the same strategy as in~\cite{pmlr-v202-sadiev23a}; inequalities~\eqref{3f712cbe-7d84-446c-a934-ad43a82e08dd-1} and~\eqref{3f712cbe-7d84-446c-a934-ad43a82e08dd-2} have already been established in~\cite{pmlr-v202-sadiev23a}.
		
		As in~\cite{pmlr-v202-sadiev23a}, let us define the random variables
		\[ \chi = \mathbb{I}\,\ens{\norm{X} > \lbd} = \begin{cases}
			1, & \text{if $\norm{X} > \lbd$;} \\
			0, & \text{otherwise;}
		\end{cases} \,\, \text{ and } \,\, \eta = \mathbb{I}\,\ens{\norm{X - x} > \frac{\lbd}{2}} = \begin{cases}
			1, & \text{if $\norm{X - x} > \frac{\lbd}{2
				}$;} \\
			0, & \text{otherwise;}
		\end{cases} \]
		then, observe that
		\[ \clip(X, \lbd) = \min\ens{1, \frac{\lbd}{\norm{X}} X} = \frac{\lbd X}{\norm{X}} \chi + X (1 - \chi) = X + \chi \left( \frac{\lbd}{\norm{X}} - 1 \right) X, \]
		and $\chi \le \eta$ since $\norm{X} > \lbd$ implies 
		\[ \norm{X - x} \ge \norm{X} - \norm{x} > \lbd - \norm{x} \ge \frac{\lbd}{2}. \]
		
		Moreover, thanks to Markov's inequality, we have
		\begin{alignat*}{2}
			\E{\eta} & = \Proba{\norm{X - \lbd} > \frac{\lbd}{2}} \\
			& = \Proba{\norm{X - x}^p > \left( \frac{\lbd}{2} \right)^p} \\
			\oversetref{Lem.}{\ref{lem:markov-inequality}}&{\le} \left( \frac{2}{\lbd} \right)^p \E{\norm{X - x}^p} \\
			& \le \left( \frac{2}{\lbd} \right)^p \sigma_1^p. \numberthis\label{a8e6cfe7-3338-4dff-9a19-431faa006da8}
		\end{alignat*}
		
		\paragraph{Proof of~\eqref{3f712cbe-7d84-446c-a934-ad43a82e08dd-3}:} using $\norm{x} \le \frac{\lbd}{2}$ and $\norm{\widetilde{X} - x} \le \norm{\widetilde{X}} + \norm{x} \le \lbd + \frac{\lbd}{2} = \frac{3 \lbd}{2}$ we have for any $q \in \intff{p}{2}$
		\begin{alignat*}{2}
			\E{\norm{\widetilde{X} - x}^q} & = \E{\norm{\widetilde{X} - x}^p \cdot \norm{\widetilde{X} - x}^{q - p}} \\
			& \le \left( \frac{3 \lbd}{2} \right)^{q - p} \E{\norm{\widetilde{X} - x}^p \chi + \norm{\widetilde{X} - x}^p (1 - \chi)} \\
			\oversetrel{rel:14ba94db-df7d-4216-a292-3c9904775d80}&{=} \left( \frac{3 \lbd}{2} \right)^{q - p} \E{\norm{\frac{\lbd X}{\norm{X}} - x}^p \chi + \norm{X - x}^p (1 - \chi)} \\
			& \le \left( \frac{3 \lbd}{2} \right)^{q - p} \E{\left( \norm{\frac{\lbd X}{\norm{X}}} + \norm{x} \right)^p \chi + \norm{X - x}^p (1 - \chi)} \\
			& \le \left( \frac{3 \lbd}{2} \right)^{q - p} \E{\left( \frac{3 \lbd}{2} \right)^p \chi + \norm{X - x}^p (1 - \chi)} \\
			\oversetrel{rel:a904b7bb-be71-4e76-8972-5db3c541e63a}&{\le} \left( \frac{3 \lbd}{2} \right)^q \E{\chi} + \left( \frac{3 \lbd}{2} \right)^{q - p} \sigma_1^p \\
			\oversetlab{\eqref{a8e6cfe7-3338-4dff-9a19-431faa006da8}}&{\le} \left[ \left( \frac{3 \lbd}{2} \right)^q \left( \frac{2}{\lbd} \right)^p + \left( \frac{3 \lbd}{2} \right)^{q - p} \right] \sigma_1^p \\
			& \le 2 \cdot 3^q \lbd^{q - p} \sigma_1^p,
		\end{alignat*}
		where in~\relref{rel:14ba94db-df7d-4216-a292-3c9904775d80} we use the definition of $\clip$ and $\chi$. In~\relref{rel:a904b7bb-be71-4e76-8972-5db3c541e63a} we use $1 - \chi \le 1$ and $\E{\norm{X - x}^p} \le \sigma_1^p$.
		
		\paragraph{Proof of~\eqref{3f712cbe-7d84-446c-a934-ad43a82e08dd-4}:} according to~\Cref{lem:variance-decomposition-any-exponent} with $c = x$, we have
		\begin{alignat*}{2}
			\E{\norm{\widetilde{X} - \E{\widehat{X}}}^q} & \le 2^{2 - q} \E{\norm{\widetilde{X} - x}^q} \\
			\oversetlab{\eqref{3f712cbe-7d84-446c-a934-ad43a82e08dd-3}}&{\le} 2^{3 - q} \cdot 3^q \lbd^{q - p} \sigma_1^p,
		\end{alignat*}
		as claimed.
	\end{proof}
	
	\newpage
	
	\section{Useful Identities and Inequalities}
	
	For any vectors $x, y \in \R^d$, we have
	\begin{alignat*}{2}
		2 \ps{x}{y} & = \sqnorm{x} + \sqnorm{y} - \sqnorm{x - y}. \numberthis\label{eq:identity-1}
	\end{alignat*}
	
	\begin{lemma}[Variance Decomposition]\label{lem:variance-decomposition}
		For any random vector $X \in \R^d$ and any non-random vector $c \in \R^d$ we have
		\[ \E{\sqnorm{X - c}} = \E{\sqnorm{X - \E{X}}} + \sqnorm{\E{X} - c}. \]
	\end{lemma}
	
	\begin{remark}
		In particular, this implies that for any random vector $X \in \R^d$ and any non-random vector $c \in \R^d$ we have
		\[ \E{\sqnorm{X - \E{X}}} \le \E{\sqnorm{X - c}}. \]
		In~\Cref{lem:variance-decomposition-any-exponent} we extend this inequality to any exponent $\alpha > 1$ instead of just $2$.
	\end{remark}
	
	\begin{lemma}[Tower Property of the Expectation]\label{lem:tower-property}
		For any random variables $X \in \R^d$ and $Y_1, \ldots, Y_n$ we have
		\[ \E{\ExpCond{X}{Y_1, \ldots, Y_n}} = \E{X}. \]
	\end{lemma}
	
	\begin{lemma}[Cauchy Schwarz's Inequality]\label{lem:cauchy-schwarz}
		For any vectors $a, b \in \R^d$ we have
		\[ \ps{a}{b} \le \abs{\ps{a}{b}} \le \norm{a} \norm{b}. \]
	\end{lemma}
	
	\begin{lemma}[Young's Inequality (Norm Form)]\label{lem:youg-inequality}
		For any vectors $a, b \in \R^d$ and any scalar $\alpha > 0$ we have
		\[ \sqnorm{a + b} \le (1 + \alpha) \sqnorm{x} + \left( 1 + \frac{1}{\alpha} \right) \sqnorm{y}. \]
	\end{lemma}
	
	\begin{lemma}[Young's Inequality (Inner Product Form)]\label{lem:youg-inequality-inner-product}
		For any vectors $a, b \in \R^d$ and any scalar $\alpha > 0$ we have
		\[ 2\ps{a}{b} \le 2 \abs{\ps{a}{b}} \le \alpha \sqnorm{x} + \frac{1}{\alpha} \sqnorm{y}. \numberthis\label{55c90e49-73cb-4448-8f5b-fb2c30a7c439} \]
	\end{lemma}
	
	\begin{proof}
		It's enough to prove inequality~\eqref{55c90e49-73cb-4448-8f5b-fb2c30a7c439} when $d = 1$. Hence, consider $a, b \in \R$, we have given $\alpha > 0$
		\[ 2ab \le 2\abs{ab} = 2\abs{a} \cdot \abs{b} = 2\abs{\sqrt{\alpha} \, a} \cdot \abs{\frac{b}{\sqrt{\alpha}}} \oversetrel{rel:c481fc03-dcbc-4a9b-b12e-1b8282d05b5d}{\le} \alpha \abs{a}^2 + \frac{1}{\alpha} \abs{b}^2 = \alpha \, a^2 + \frac{b^2}{\alpha}, \]
		where in~\relref{rel:c481fc03-dcbc-4a9b-b12e-1b8282d05b5d} we use the arithmetic-geometric inequality in $n = 2$ variables $\left( \sqrt{\alpha} \abs{a}, \frac{1}{\sqrt{\alpha}} \abs{b} \right)$.
	\end{proof}
	
	
	\begin{lemma}\label{lem:norm-power-alpha-inequality}
		For any vectors $a, b \in \R^d$ and any scalar $\alpha \ge 1$ we have
		\[ \norm{a + b}^{\alpha} \le 2^{\alpha - 1} \left( \norm{a}^{\alpha} + \norm{b}^{\alpha} \right). \numberthis\label{e4d36f60-854e-4737-9f49-0ac946fb5b07} \]
	\end{lemma}
	
	\begin{proof}
		Note that the function $x \mapsto \norm{x}^{\alpha}$ is convex over $\R^d$ since $\norm{\cdot}$ is convex (as a norm) and $t \mapsto t^{\alpha}$ is convex since $\alpha \ge 1$. Hence, by the Jensen's inequality (\Cref{lem:jensen-inequality}), we have
		\[ \norm{\frac{a + b}{2}}^{\alpha} \le \frac{1}{2} \left( \norm{a}^{\alpha} + \norm{b}^{\alpha} \right), \]
		i.e.,
		\[ \norm{a + b}^{\alpha} \le 2^{\alpha - 1} \left( \norm{a}^{\alpha} + \norm{b}^{\alpha} \right), \]
		as claimed.
		
	\end{proof}
	
	\begin{remark}
		Similar inequalities to~\eqref{e4d36f60-854e-4737-9f49-0ac946fb5b07} appear in~\cite{he2025complexitynormalizedstochasticfirstorder}. In~\Cref{lem:euclidean-norm-expansion-generalized} we state one of them that we use in our proofs.
	\end{remark}
	
	\begin{lemma}[{\cite[Lemma~1]{he2025complexitynormalizedstochasticfirstorder}}]\label{lem:euclidean-norm-expansion-generalized}
		For any vectors $a, b \in \R^d$ and any scalar $\alpha\in \intof{1}{2}$ we have\footnote{Note that, when $a = 0$ we have $\norm{a}^{\alpha - 2} a = 0$ since $\alpha > 1$ thus $\alpha - 2 > -1$.}
		\[ \norm{a + b}^{\alpha} \le \norm{a}^{\alpha} + \alpha \norm{a}^{\alpha - 2} \ps{a}{b} + 2^{2 - \alpha} \norm{b}^{\alpha}. \]
	\end{lemma}
	
	\begin{lemma}\label{lem:variance-decomposition-any-exponent}
		For any scalar $\alpha \in \intof{1}{2}$, any random vector $X \in \R^d$ and any non-random vector $c \in \R^d$ we have
		\[ \E{\norm{X - \E{X}}^{\alpha}} \le 2^{2 - \alpha} \E{\norm{X - c}^{\alpha}}. \]
	\end{lemma}
	
	\begin{proof}
		By~\Cref{lem:euclidean-norm-expansion-generalized} we have
		\begin{alignat*}{2}
			\E{\norm{X - \E{X}}^{\alpha}} & = \E{\norm{\left( X - c \right) + \left( c - \E{X} \right)}^{\alpha}} \\
			\oversetref{Lem.}{\ref{lem:euclidean-norm-expansion-generalized}}&{\le} \norm{c - \E{X}}^{\alpha} + \alpha \norm{c - \E{X}}^{\alpha - 2} \E{\ps{c - \E{X}}{X - c}} + 2^{2 - \alpha} \E{\norm{X - c}^{\alpha}} \\
			& = \norm{c - \E{X}}^{\alpha} - \alpha \norm{c - \E{X}}^{\alpha} + 2^{2 - \alpha} \E{\norm{X - c}^{\alpha}} \\
			\oversetrel{rel:55614670-ed22-4bb1-b96d-621007650db7}&{\le} 2^{2 - \alpha} \E{\norm{X - c}^{\alpha}},
		\end{alignat*}
		where in~\relref{rel:55614670-ed22-4bb1-b96d-621007650db7} we use $\alpha > 1$.
	\end{proof}
	
	\begin{lemma}[Jensen's Inequality]\label{lem:jensen-inequality}
		Let $f \colon \R^d \to \R$ be a convex function then
		\begin{enumerate}
			\item (Probabilistic Form) for any random vector $X \in \R^d$ we have
			$$\E{f(X)} \ge f\left(\E{X}\right).$$
			
			\item (Deterministic Form) for any vectors $v_1, \ldots, v_n \in \R^d$ and scalars $\lbd_1, \ldots, \lbd_n \in \R_+$ we have
			$$\sum_{i = 1}^n \lbd_i f(v_i) \ge f\left( \sum_{i = 1}^n \lbd_i v_i \right),$$
			provided $\lbd_i \ge 0$ for all $i \in [n]$ and $\sum\limits_{i = 1}^n \lbd_i = 1$.
		\end{enumerate}
	\end{lemma}
	
	\begin{lemma}\label{lem:jensen-form-1}
		For any vectors $v_1, \ldots, v_n \in \R^d$ we have
		$$\sqnorm{\sum_{i = 1}^n v_i} \leq n \sum_{i = 1}^n \sqnorm{v_i}.$$
	\end{lemma}
	
	\begin{proof}
		The function $\sqnorm{\cdot} \colon \R^d \to \R$ is $\mu$-strongly convex with $\mu = 2$ so is convex thus applying Jensen's inequality~\ref{lem:jensen-inequality} with $\lbd_1 = \cdots = \lbd_n = \frac{1}{n}$ gives
		$$\sqnorm{\sum_{i = 1}^n \frac{v_i}{n}} \le \frac{1}{n} \sum_{i = 1}^n \sqnorm{v_i},$$
		and multiplying both sides by $n^2$ gives the desired inequality.
	\end{proof}

	\begin{lemma}[Markov's Inequality]\label{lem:markov-inequality}
		For any non-negative random variable $X$ and any scalar $a > 0$, we have
		\[ \Proba{X \ge a} \le \frac{\E{X}}{a}. \]
	\end{lemma}
	
	\begin{lemma}[Bernoulli's Inequality]\label{lem:bernoulli-inequality}
		For any real number $x \ge -1$ and $r \in \ens{0} \cup \intfo{1}{+\infty}$ we have
		\[ (1 + x)^r \ge 1 + r x. \]
	\end{lemma}
	
	\begin{lemma}[{Absolute $p^{\textnormal{th}}$ Central Moment of Normal Distribution~\citep[eq.~(18)]{winkelbauer2014momentsabsolutemomentsnormal}}]\label{appdx-lem:absolute-central-moment-gaussian}
		For any real number $p > -1$, $\sigma_1 > 0$ and $\mu \in \R$, if $X \sim \mathcal{N}(\mu, \sigma_1^2)$ then
		\[ \E{\abs{X - \mu}^p} = 2^{\frac{p}{2}} \sigma_1^p \cdot \frac{\Gamma\left( \frac{p + 1}{2} \right)}{\sqrt{\pi}}, \]
		where $\Gamma$ is the Euler's gamma function (see~\citet[{(5.2.1)}]{NIST:DLMF} for formal definition).
	\end{lemma}

	\begin{definition}[Operator Norm]\label{appdx-def:operator-norm}
		Given a $d \times d$ matrix $A \in \R^{d \times d}$, the operator norm of $A$ is the norm
		\[ \normop{A} \eqdef \sup_{y \in \R^d, \norm{y} = 1} \norm{A y}, \]
		where $\norm{\cdot}$ is the standard euclidean norm in $\R^d$.
	\end{definition}
	
	\begin{lemma}[Lipchitz Gradients Implies Bounded Hessian]\label{appdx-lem:lipschitz-gradients-implies-bounded-hessian}
		Given $F \colon \R^d \to \R$ a twice continuously differentiable function over $\R^d$ such that its gradients are $L_1$--Lipschitz continuous for some constant $L_1 \ge 0$ then, for all $x \in \R^d$ we have
		\[ \normop{\nabla^2 F(x)} \le L_1. \]
	\end{lemma}
	The above lemma can be extended to $p^{\textnormal{th}}$-order continuously differentiable function, $p \ge 2$ for which the $(p - 1)^{\textnormal{th}}$-order derivative of $F$ is Lipschitz continuous in the operator norm induced by $\ell^2$--norm. 
	
	\begin{lemma}\label{appdx-lem:bound-norm-gradient-by-L-and-Delta}
		Given $F \colon \R^d \to \R$ a continuously differentiable and lower bounded function over $\R^d$ such that its gradients are $L_1$--Lipschitz continuous for some constant $L_1 \ge 0$ then, for all $x \in \R^d$ we have
		\[ \norm{\nabla F(x)}^2 \le 2 L_1 \left( F(x) - F^{\inf} \right), \]
		for any $F^{\inf} \le \inf_{x \in \R^d} F(x)$.
	\end{lemma}
	
	\begin{lemma}\label{appdx-lem:bound-taylor-gradient-hessian}
		Given $F \colon \R^d \to \R$ a twice continuously differentiable and lower bounded function over $\R^d$ such that its Hessians are $L_2$--Lipschitz continuous for some constant $L_2 \ge 0$ then, for all $x, y \in \R^d$ we have
		\[ \norm{\nabla F(x) - \nabla F(y) - \nabla^2 F(y) (x - y)} \le \frac{L_2}{2} \sqnorm{x - y}. \]
	\end{lemma}

	\begin{lemma}[{Landau-Kolmogorov's Inequalities on a Finite Interval~\citep[Theorem~2]{Fabry_1987}}]\label{appdx-lem:landau-kolmogorov-inequalities}
		Let $p \ge 1$ be an integer, $a \le b$ be real numbers and $F \in \mathcal{C}^p(\intff{a}{b}, \R)$. For any $k \in \Int{0}{p}$ let $M_k \eqdef \sup_{x \in \intff{a}{b}} \abs{F^{(k)}(x)} \in \intff{0}{+\infty}$ then, if $M_0$ and $M_p$ are finite, we have for all $k \in [p]$,
		\begin{enumerate}
			\item $M_k$ is finite,
			\item the inequality
			\[ M_k \le c_{k, p} M_0^{\frac{p - k}{p}} \max\ens{\frac{M_p}{2^{p - 1} p!}, 4^p (b - a)^{-p} M_0}^{\frac{k}{p}}, \numberthis\label{ce6d9b55-0646-4949-b2b3-2311a3f3f37b} \]
			holds where $c_{k, p} > 0$ is the universal constant
			\[ c_{k, p} \eqdef \frac{p 2^k k!}{p + k} \binom{p + k}{p - k}. \]
		\end{enumerate}
	\end{lemma}
	
	\begin{remark}
		Notably, the constant $c_{k, p}$ appearing in the inequality~\eqref{ce6d9b55-0646-4949-b2b3-2311a3f3f37b} does not depend on the choice of the function $F$ while the $\ens{M_k}_{k \in \Int{0}{p}}$ do.
	\end{remark}
	
	\begin{remark}
		Other related inequalities can be found in~\citet{Mitrinović1991,Chen_1993}.
	\end{remark}
	
	\begin{lemma}\label{appdx-lem:lower-bounding-F-steepest-descent-direction}
		Let $F \colon \R^d \to \R$ be a continuously differentiable function over $\R^d$ with $L$--Lipschitz gradients for some $L > 0$ and $x_0 \in \R^d$. Assume that $\norm{\nabla F(x_0)} > 0$ then for all $t \in \intff{0}{2}$ we have
		\[ F(x_0 + t v) \le F(x_0), \]
		where $v = - \frac{1}{L} \nabla F(x_0)$.
	\end{lemma}
	
	\begin{proof}
		Using the fact that the function $F$ has $L$--Lipschitz gradients then we know $F$ is $L$--smooth (\cite{nesterov2018lectures}) hence, for all $t \in \intff{0}{2}$ we have
		\begin{alignat*}{2}
			D_F(x_0 + t v, x_0) \eqdef F(x_0 + t v) - F(x_0) - \ps{\nabla F(x_0)}{(x_0 + t v) - x_0} \le \frac{L}{2} \sqnorm{(x_0 + t v) - x_0} = \frac{L t^2}{2} \sqnorm{v},
		\end{alignat*}
		where $D_F(x, y)$ denotes the Bregman divergence of $F$ at $x, y \in \R^d$. Rewriting the above inequality using the choice of $v$ gives
		\begin{alignat*}{2}
			F(x_0 + t v) & \le F(x_0) + t \ps{\nabla F(x_0)}{v} + \frac{L t^2}{2} \sqnorm{v} \\
			& = F(x_0) - \frac{t}{L} \sqnorm{\nabla F(x_0)} + \frac{t^2}{2 L} \sqnorm{\nabla F(x_0)} \\
			& = F(x_0) - \frac{t}{L} \sqnorm{\nabla F(x_0)} \left( 1 - \frac{t}{2} \right) \\
			& \le F(x_0),
		\end{alignat*}
		since $0 \le t \le 2$ and $L > 0$. This achieves the proof of the lemma.
	\end{proof}
	
	\begin{remark}
		The above lemme still holds when $\norm{\nabla F(x_0)} = 0$, i.e., $v = 0$ but it is not of interest.
	\end{remark}
	
	\begin{lemma}
		Let $F \colon \R^d \to \R$ be $p+1$ times continuously differentiable ($p \ge 1$) over $\R^d$. Assume that
		\begin{enumerate}
			\item for all $k \in [p]$, the function $\nabla^k F$ is $L_k$--Lipschitz continuous for some $L_k > 0$,
			\item the function $F$ is lower bounded over $\R^d$ and we denote $F^{\inf} \eqdef \inf_{x \in \R^d} F(x)$,
		\end{enumerate}
		then, for all $k \in [p]$ there exists a constant $c_k \eqdef (k + 1)^2 > 0$, depending only on $k$ such that for all $x_0 \in \R^d$
		\[ \norm{\nabla F(x_0)} \le c_k \left( F(x_0) - F^{\inf} \right)^{\nicefrac{k}{k + 1}} \max\ens{\frac{L_k^{\frac{1}{k + 1}}}{2^{\frac{k}{k + 1}} [(k + 1)!]^{\frac{1}{k + 1}}}, \frac{2 L_1 \left( F(x_0) - F^{\inf} \right)^{\frac{1}{k + 1}}}{\norm{\nabla F(x_0)}}}. \numberthis\label{e7a373a5-ff82-4f5f-9ece-a932573b577d} \]
		
		In particular, if the gradient of $F$ at $x_0$ is large (say, $\norm{\nabla F(x_0)} = \Omega(\sqrt{L_1 \left( F(x_0) - F^{\inf} \right)})$, see~\Cref{appdx-lem:bound-norm-gradient-by-L-and-Delta}) and the Lipschitz constant $L_k$ dominates, then~\eqref{e7a373a5-ff82-4f5f-9ece-a932573b577d} simplifies to
		\[ \norm{\nabla F(x_0)} \le \widetilde{c}_k L_k^{\nicefrac{1}{k + 1}} \left( F(x_0) - F^{\inf} \right)^{\nicefrac{k}{k + 1}}, \numberthis\label{5ea0787a-1b21-4f89-9e4e-032978f6ec42} \]
		for some universal constant $\widetilde{c}_k$ which depends only on $k$.
	\end{lemma}
	
	\begin{remark}
		The bound~\eqref{e7a373a5-ff82-4f5f-9ece-a932573b577d} is (almost) a generalization to high-order Lipschitz constant of the well-known inequality
		\[ \norm{\nabla F(x_0)} \le \sqrt{2 L_1 \left( F(x_0) - F^{\inf} \right)}, \]
		which we recall in~\Cref{appdx-lem:bound-norm-gradient-by-L-and-Delta} and which corresponds to the case $k = 1$.
		
		While it is hopeless\footnote{Inequality~\eqref{5ea0787a-1b21-4f89-9e4e-032978f6ec42} already fails for the case $k = 2$ with quadratic functions. For instance, consider $F \colon x \mapsto \frac{1}{2} \norm{x}^2$ over $\R^d$ and let $x_0 = (1, 0, \ldots, 0)$ then $\norm{\nabla F(x_0)} = \norm{x_0} = 1$, but $\nabla^2 F(x) = \mathrm{Id}$ for all $x \in \R^d$ hence we can take $L_2 = 0$ since the Hessian of $F$ is constant, but then
			\[ L_2^{\nicefrac{1}{3}} \left( F(x_0) - F^{\inf} \right)^{\nicefrac{2}{3}} = 0, \]
			and there do not exists universal constant $c_2$ for which~\eqref{5ea0787a-1b21-4f89-9e4e-032978f6ec42} can hold.} to obtain the inequality~\eqref{5ea0787a-1b21-4f89-9e4e-032978f6ec42} in full generality, i.e., without extra assumptions, in practical settings the Lipschitz constants are very large and the gradient of the objective at the initial point is also large.
	\end{remark}
	
	\begin{proof}
		Let $F \colon \R^d \to \R$ be defined as above. Then, if $\norm{\nabla F(x_0)} = 0$ the inequality~\eqref{e7a373a5-ff82-4f5f-9ece-a932573b577d} holds since $c_k > 0$ and $F(x_0) - F^{\inf} \ge 0$. Now, assume $\norm{\nabla F(x_0)} > 0$, we define the scalar function $\varphi \colon \intff{0}{2} \to \R$ as
		\[ \varphi \colon t \mapsto F(x_0 + t v) - F^{\inf}, \]
		where $v = -\frac{1}{L_1} \nabla F(x_0) \neq 0$ and $L_1$ is the Lipschitz constant of $\nabla F$. Then, the above function $\varphi$ is $p + 1$ time continuously differentiable over $\intff{0}{2}$, and for all $t \in \intff{0}{2}$ and all $k \in \Int{2}{p + 1}$ we have
		\[ \varphi^{(k)}(t) = \nabla^k F(x_0 + t v) [v, \ldots, v] = \frac{(-1)^k}{L_1^k} \nabla^k F(x_0 + t v) [g, \ldots, g], \numberthis\label{441999ae-9106-410c-a85b-fdcd81183237} \]
		where we let $g \eqdef \nabla F(x_0)$ and $\nabla^k F(x_0 + t v) [\cdot, \ldots, \cdot]$ is the $k$--linear form induced by the $k^{\textnormal{th}}$--derivative of $F$. From equality~\eqref{441999ae-9106-410c-a85b-fdcd81183237} we obtain
		\[ \abs{\varphi^{(k)}(t)} = \frac{1}{L_1^k} \abs{\nabla^k F(x_0 + t v) [g, \ldots, g]} \oversetrel{rel:acc34fcc-c875-448a-88f1-e14977032a5f}{\le} \frac{1}{L_1^k} \normop{\nabla^k F(x_0 + t v)} \cdot \norm{g}^k \oversetrel{rel:acc34fcc-c875-448a-88f1-e14977032a5f-2}{\le} \left( \frac{\norm{\nabla F(x_0)}}{L_1} \right)^k L_{k - 1}, \numberthis\label{7d81208b-8c62-409a-af2f-97a7bcb40436} \]
		where in~\relref{rel:acc34fcc-c875-448a-88f1-e14977032a5f} we use the Cauchy-Schwarz's inequality (\Cref{lem:cauchy-schwarz}) while in~\relref{rel:acc34fcc-c875-448a-88f1-e14977032a5f-2} we use~\eqref{441999ae-9106-410c-a85b-fdcd81183237} and the fact that $\nabla^r F$ is $L_r$--Lipschitz continuous for any $r \in [p]$  which implies that $\nabla^k F$ is bounded in operator norm as long as $k \ge 2$ (see~\Cref{appdx-lem:lipschitz-gradients-implies-bounded-hessian} for the case $k = 2$). Hence, if we define $M_k \eqdef \sup_{t \in \intff{0}{2}} \abs{\varphi^{(k)} (t)}$ we have, by~\eqref{7d81208b-8c62-409a-af2f-97a7bcb40436},
		\[ M_k \le \left( \frac{\norm{\nabla F(x_0)}}{L_1} \right)^k L_{k - 1} < +\infty, \numberthis\label{e0aab2e4-82dd-4824-b2b2-df7ad67d99b1} \]
		since $L_k > 0$ for all $k \in [p]$ and
		\[ M_1 = \sup_{t \in \intff{0}{2}} \abs{\varphi'(t)} = \sup_{t \in \intff{0}{2}} \abs{\ps{\nabla F(x_0 + t v)}{v}} = \frac{1}{L_1} \sup_{t \in \intff{0}{2}} \abs{\ps{\nabla F(x_0 + t v)}{\nabla F(x_0)}} \ge \frac{1}{L_1} \sqnorm{\nabla F(x_0)}. \numberthis\label{a718a294-777a-4426-9f05-5289df231fe8} \]
		Additionally,
		\[ M_0 \eqdef \sup_{t \in \intff{0}{2}} \abs{\varphi(t)} \oversetrel{rel:093df292-d584-4453-8021-1b0270938a32}{=} \sup_{t \in \intff{0}{2}} \left( F(x_0 + t v) - F^{\inf} \right) \oversetrel{rel:109aa7fb-3c92-4a72-88f4-8219339a2311}{=} F(x_0) - F^{\inf} < +\infty, \numberthis\label{6c704741-f58f-49c3-be4f-986e40421895} \]
		where~\relref{rel:093df292-d584-4453-8021-1b0270938a32} follows from $F \ge F^{\inf}$ while~\relref{rel:109aa7fb-3c92-4a72-88f4-8219339a2311} follows from~\Cref{appdx-lem:lower-bounding-F-steepest-descent-direction} since $F$ is continuously differentiable and has $L_1$--Lipschitz gradients.
		
		Then for any $k \in [p]$, since the function $\varphi$ is $k + 1 \ge 2$ times continuously differentiable over $\R$ and because the $\ens{M_{\ell}}_{\ell \in \Int{0}{p}}$ are all finite from~\eqref{e0aab2e4-82dd-4824-b2b2-df7ad67d99b1} and~\eqref{6c704741-f58f-49c3-be4f-986e40421895} then, by the Landau-Kolmogorov's inequalities (\Cref{appdx-lem:landau-kolmogorov-inequalities}),
		\[ M_1 \le c_k \, M_0^{\frac{k}{k + 1}} \max\ens{\frac{M_{k + 1}}{2^k (k + 1)!}, 2^{k + 1} M_0}^{\frac{1}{k + 1}}, \numberthis\label{80ccea34-41b0-4eb4-9743-22232a23d1f2} \]
		where $c_k \eqdef c_{1, k + 1} = (k + 1)^2$ is an universal constant depending only on $k$ (by~\Cref{appdx-lem:landau-kolmogorov-inequalities}). From the inequalities~\eqref{e0aab2e4-82dd-4824-b2b2-df7ad67d99b1} and~\eqref{a718a294-777a-4426-9f05-5289df231fe8},~\eqref{6c704741-f58f-49c3-be4f-986e40421895} along with~\eqref{80ccea34-41b0-4eb4-9743-22232a23d1f2} we obtain
		\begin{alignat*}{2}
			\frac{1}{L_1} \sqnorm{\nabla F(x_0)} & \le c_k \, \left( F(x_0) - F^{\inf} \right)^{\frac{k}{k + 1}} \max\ens{\left( \left( \frac{\norm{\nabla F(x_0)}}{L_1} \right)^{k + 1} \frac{L_k}{2^k (k + 1)!} \right)^{\frac{1}{k + 1}}, 2 M_0^{\frac{1}{k + 1}}} \\
			& = \frac{c_k}{L_1} \left( F(x_0) - F^{\inf} \right)^{\nicefrac{k}{k + 1}} \norm{\nabla F(x_0)} \max\ens{\frac{L_k^{\frac{1}{k + 1}}}{2^{\frac{k}{k + 1}} [(k + 1)!]^{\frac{1}{k + 1}}}, \frac{2 L_1 \left( F(x_0) - F^{\inf} \right)^{\frac{1}{k + 1}}}{\norm{\nabla F(x_0)}}},
		\end{alignat*}
		which simplifies to 
		\[ \norm{\nabla F(x_0)} \le c_k \left( F(x_0) - F^{\inf} \right)^{\nicefrac{k}{k + 1}} \max\ens{\frac{L_k^{\frac{1}{k + 1}}}{2^{\frac{k}{k + 1}} [(k + 1)!]^{\frac{1}{k + 1}}}, \frac{2 L_1 \left( F(x_0) - F^{\inf} \right)^{\frac{1}{k + 1}}}{\norm{\nabla F(x_0)}}}. \]
		and leads to the desired inequality.
	\end{proof}
	\newpage
	\section{Additional Analysis in The Case $p=q=2$}
	
	In this section, which is of independent interest, we study regular, not normalized, \algname{SGD} with Momentum Variance Reduction (\algname{SGD-MVR}). This method is the original one by \citet{cutkosky2019momentum}, where it is called \algname{STORM}. The method is shown as Algorithm~\ref{algo:sgd-mvr}. We show that in the case $p=q=2$, it achieves the same optimal complexity as \algname{NSGD-MVR}. 
	
	That is, normalization is not required in this particular case.\\
	
	\begin{algorithm}%
		\caption{\algname{SGD-MVR} (\algname{SGD} with \texttt{MVR})}%
		\label{algo:sgd-mvr}%
		
		\DontPrintSemicolon%
		\SetKwProg{Init}{Initialization}{:}{}%
		\Init{}{%
			$x_0 \in \R^d$, the starting point\;
			$T > 0$, the number of iterations\;
			$g_0 \in \R^d$, an initial vector\;
			$\gamma > 0$, the stepsize\;
			$\alpha \in \intof{0}{1}$, the momentum parameter for \texttt{MVR}\;
		}
		
		\vspace{\baselineskip}
		
		$x_1 \gets x_0 - \gamma g_0$\;
		\For{$t = 1, 2, \ldots, T - 1$}{%
			\tcp*[h]{Apply \texttt{MVR}.}\;
			$g_t \gets (1 - \alpha) \left( g_{t - 1} + \nabla f\left( x_t, \xi_t \right) - \nabla f \left( x_{t - 1}, \xi_t \right) \right) + \alpha \nabla f \left( x_t, \xi_t \right)$\;
			\tcp*[h]{Do one descent step.}\;
			$x_{t + 1} \gets x_t - \gamma g_t$\;
		}%
		\KwOut{$x_T$}%
	\end{algorithm}%

	\begin{theorem}\label{thm:sgd-mvr-convergence-analysis}
		Under~\Cref{ass:lower-boundedness,ass:L-lipschitz-gradients,ass:p-bounded-central-moment-gradient,ass:mean-squared-smoothness-2}, with $p=q=2$, let the initial gradient estimate $g_0$ be given by
		\[ g_0 = \frac{1}{B_{\textnormal{init}}} \sum_{j = 1}^{B_{\textnormal{init}} - 1} \nabla f\left( x_0, \xi_{0, j} \right), \]
		with $B_{\textnormal{init}} = \max\ens{1, \frac{\sigma_1^2}{\eps^2} }$, 
		momentum parameter $\alpha = \min\ens{1, \frac{\eps^2}{\sigma_1^2}}$, stepsize
		$\gamma=\frac{1}{L_1+\frac{\delta\sqrt{2}(1-\alpha)}{\sqrt{\alpha}}}$.
		Then \Cref{algo:sgd-mvr} is guaranteed to find an $\eps$-stationary point with total sample complexity
		\[ \cO\left(  \frac{\sigma_1^2}{\eps^2}  + \frac{(L_1 + \delta) \Delta}{\eps^2} + \frac{\delta \Delta\sigma_1}{\eps^3}  \right). \]
	\end{theorem}

	\begin{proof}
		Let $t\geq 1$. 
		We have  the descent lemma 
		\begin{align*}
			F(x^{t})&\leq F(x^{t-1}) + \langle x^{t}-x^{t-1},\nabla F(x^{t-1})\rangle + \frac{L_1}{2} \sqnorm{x^{t}-x^{t-1}}\\
			&= F(x^{t-1}) -\gamma  \langle g_{t-1},\nabla F(x^{t-1})\rangle + \frac{L_1}{2} \sqnorm{x^{t}-x^{t-1}}\\
			&=F(x^{t-1}) +\frac{\gamma}{2}\sqnorm{g_{t-1}-\nabla F(x^{t-1})}-\frac{\gamma}{2}\sqnorm{\nabla F(x^{t-1})}-\frac{\gamma}{2}\sqnorm{g_{t-1}} + \frac{L_1}{2} \sqnorm{x^{t}-x^{t-1}}\\
			&=F(x^{t-1}) +\frac{\gamma}{2}\sqnorm{g_{t-1}-\nabla F(x^{t-1})}-\frac{\gamma}{2}\sqnorm{\nabla F(x^{t-1})} + \left(\frac{L_1}{2}-\frac{1}{2\gamma}\right) \sqnorm{x^{t}-x^{t-1}}.
		\end{align*}
		Therefore, denoting by $\mathcal{F}_{t-1}$ the  filtration  $\sigma(g_0, \xi_1, \ldots, \xi_{t-1})$, we have
		\begin{align*}
			\ExpCond{F(x^{t})-F^{\inf}}{\mathcal{F}_{t-1}}&\leq F(x^{t-1})-F^{\inf}  -\frac{\gamma}{2}\sqnorm{\nabla F(x^{t-1})} +\frac{\gamma}{2}\sqnorm{g_{t-1}-\nabla F(x^{t-1})} \\
			&\quad+ \left(\frac{L_1}{2}-\frac{1}{2\gamma}\right) \sqnorm{x^{t}-x^{t-1}}.
		\end{align*}
		We need to control the deviation $\sqnorm{g_{t-1}-\nabla F(x^{t-1})}$, that will appear in the Lyapunov function. So, let us study $\ExpCond{\sqnorm{g_{t}-\nabla F(x^{t})}}{\mathcal{F}_{t-1}}$. From a bias--variance decomposition, we have
		\begin{align*}
			&\ExpCond{\sqnorm{g_{t}-\nabla F(x^{t})}}{\mathcal{F}_{t-1}}\\
			&=\ExpCond{\sqnorm{\nabla f(x^{t},\xi^t) -\nabla F(x^{t}) +(1-\alpha)\big(g_{t-1} - \nabla f(x^{t-1},\xi^t) \big)}}{\mathcal{F}_{t-1}}\\
			&=(1-\alpha)^2 \sqnorm{g_{t-1}-\nabla F(x^{t-1})}\\
			&\quad+\ExpCond{\sqnorm{\nabla f(x^{t},\xi^t) -\nabla F(x^{t}) +(1-\alpha)\big(\nabla F(x^{t-1}) - \nabla f(x^{t-1},\xi^t) \big)}}{\mathcal{F}_{t-1}}\\
			&=(1-\alpha)^2 \sqnorm{g_{t-1}-\nabla F(x^{t-1})}\\
			&\quad+\ExpCond{\sqnorm{\alpha\big(\nabla f(x^{t},\xi^t) -\nabla F(x^{t})\big)+(1-\alpha) \big(\nabla f(x^{t},\xi^t) -\nabla F(x^{t})- \nabla f(x^{t-1},\xi^t) +\nabla F(x^{t-1} \big)}}{\mathcal{F}_{t-1}}\\
			&\leq (1-\alpha)^2 \sqnorm{g_{t-1}-\nabla F(x^{t-1})}+2\alpha^2 \ExpCond{\sqnorm{\nabla f(x^{t},\xi^t)  -\nabla F(x^{t})}}{\mathcal{F}_{t-1}}\\
			&\quad + 2(1-\alpha)^2 \ExpCond{\sqnorm{\nabla f(x^{t},\xi^t) -\nabla F(x^{t})- \nabla f(x^{t-1},\xi^t) +\nabla F(x^{t-1})}}{\mathcal{F}_{t-1}}\\
			&\leq (1-\alpha)^2 \sqnorm{g_{t-1}-\nabla F(x^{t-1})}+2\alpha^2 \sigma_1^2 + 2(1-\alpha)^2 \delta^2 \sqnorm{x^{t}-x^{t-1}}.
		\end{align*}
		We introduce the Lyapunov function 
		$$\lya^t\eqdef F(x^t)-F^{\inf} + \frac{\gamma}{2\alpha} \sqnorm{g_t-\nabla F(x^{t})}.$$
		We have 
		\begin{align*}
			\ExpCond{\lya^{t}}{\mathcal{F}_{t-1}}&
			\leq F(x^{t-1})-F^{\inf} -\frac{\gamma}{2}\sqnorm{\nabla F(x^{t-1})} +\frac{\gamma}{2}\sqnorm{g_{t-1}-\nabla F(x^{t-1})} \\
			&\quad+ \left(\frac{L}{2}-\frac{1}{2\gamma} \right)\ExpCond{\sqnorm{x^{t}-x^{t-1}}}{\mathcal{F}_{t-1}} \\
			&\quad+(1-\alpha)^2 \frac{\gamma}{2\alpha}\sqnorm{g_{t-1}-\nabla F(x^{t-1})}+\gamma \alpha\sigma_1^2 \\
			&\quad+ \frac{(1-\alpha)^2\gamma \delta^2}{\alpha} \ExpCond{\sqnorm{x^{t}-x^{t-1}}}{\mathcal{F}_{t-1}}\\
			&\leq F(x^{t-1})-F^{\inf} -\frac{\gamma}{2}\sqnorm{\nabla F(x^{t-1})} +\gamma \alpha\sigma_1^2\\
			&\quad+\big(\alpha+(1-\alpha)^2\big) \frac{\gamma}{2\alpha}\sqnorm{g_{t-1}-\nabla F(x^{t-1})} \\
			&\quad+ \left(\frac{L}{2}-\frac{1}{2\gamma} + \frac{(1-\alpha)^2\gamma \delta^2}{\alpha} \right)\ExpCond{\sqnorm{x^{t}-x^{t-1}}}{\mathcal{F}_{t-1}} .
		\end{align*}
		We have 
		$\alpha+(1-\alpha)^2\leq \alpha+(1-\alpha) = 1$. 
		A sufficient condition for $\frac{L_1}{2}-\frac{1}{2\gamma} + \frac{(1-\alpha)^2\gamma \delta^2}{\alpha}\leq 0$ is 
		
		$$\gamma\leq \frac{1}{L_1+\delta \frac{\sqrt{2}(1-\alpha)}{\sqrt{\alpha}}}.$$
		Assuming now that this condition holds, we have
		
		\begin{align}
			\ExpCond{\lya^{t}}{\mathcal{F}_{t-1}}&
			\leq \lya^{t-1} -\frac{\gamma}{2}\sqnorm{\nabla F(x^{t-1})} +\gamma \alpha\sigma_1^2.
		\end{align}
		Let $T\geq 1$. Unrolling the recursion, we have
		\begin{align}
			\sum_{t=0}^{T-1} \left(\frac{\gamma}{2} \Exp{\sqnorm{\nabla F(x^t)}} - \gamma \alpha\sigma^2\right) &\leq 
			\Delta+ \frac{\gamma}{2\alpha}\Exp{ \sqnorm{g^0-\nabla F(x^{0})}},
		\end{align}
		so that 
		\begin{align}
			\frac{1}{T}\sum_{t=0}^{T-1}\Exp{ \sqnorm{\nabla F(x^t)}} &\leq 
			\frac{2\Delta}{\gamma T} +2a\sigma^2+\frac{1}{\alpha T}\Exp{ \sqnorm{g^0-\nabla F(x^{0})}}.
		\end{align}
		We form $g^0$ as an unbiased estimate of $\nabla F(x^{0})$ using a minibatch of size $B_{\textnormal{init}}$. Hence $ \Exp{\sqnorm{g^0-\nabla F(x^{0})}}\leq \frac{\sigma_1^2}{B_{\textnormal{init}}}$
		and 
		\begin{align}
			\frac{1}{T}\sum_{t=0}^{T-1}\Exp{ \sqnorm{\nabla F(x^t)}} &\leq 
			\frac{2\Delta}{\gamma T} +\left(2a+\frac{1}{\alpha B_{\textnormal{init}}T}\right)\sigma_1^2.
		\end{align}
		We interpret the left-hand side as  $\Exp{\sqnorm{\nabla F(\hat{x}^T)}}$ for $\hat{x}^T$ chosen uniformly at random in $x^0,\ldots,x^{T-1}$.
		Hence, given $\eps>0$, with $\alpha=\min\ens{1,\frac{\eps^2}{\sigma_1^2}}$, $\gamma=\frac{1}{L_1+\frac{\delta\sqrt{2}(1-\alpha)}{\sqrt{\alpha}}}$, 
		$B_{\textnormal{init}} = \alpha^{-1}=\max\ens{1, \frac{\sigma_1^2}{\eps^2} }$, we have
		\begin{align}
			\Exp{\sqnorm{\nabla F(\hat{x}^T)}} &\leq 
			\frac{2\Delta}{T}\left(L_1+\delta\sigma_1 \eps^{-1}\right) +2\eps^2+\frac{\sigma_1^2}{T}.
		\end{align}
		With $T=\max\left(\frac{\sigma_1^2}{\eps^2},2\Delta \left(L_1\eps^{-2}+\delta\sigma_1 \eps^{-3}\right) \right)$, 
		we have $\frac{\sigma_1^2}{T}\leq \eps^2$ and  $\frac{2\Delta}{T}\left(L_1+\delta\sigma \eps^{-1}\right)  \leq \epsilon^2$.
		Hence, 
		\begin{align}
			\Exp{\sqnorm{\nabla F(\hat{x}^T)}} &\leq 4\eps^2,
		\end{align}
		with $B_{\textnormal{init}}+2(T-1)=\mathcal{O}\left(\frac{\sigma_1^2}{\eps^2}+\frac{\Delta (L_1+\delta) }{ \eps^2}+\frac{\Delta \delta \sigma_1}{ \epsilon^{3}}\right)$ stochastic gradient evaluations. 
	\end{proof}
	
	\newpage 
	\section{Empirical Evaluation}
	\label{sec:experiments}
	
	To validate our theoretical findings and provide valuable intuition regarding the behavior and stability of the proposed algorithms, we conduct a series of synthetic experiments. These experiments evaluate convergence trajectories, the necessity of our double-clipping mechanism, and how algorithmic complexity scales with the heavy-tail index $p$.
	
	\subsection{Experimental Setup}
	We consider the minimization of a highly ill-conditioned quadratic objective $F(x) = \frac{1}{2} x^\top A x$, where $A \in \R^{d \times d}$ is a diagonal matrix with $A_{1,1} = 0.01$ and $A_{i,i} = 1$ for $i > 1$, and dimension $d=10$. To simulate heavy-tailed noise under the $p$-BCM assumption (\Cref{ass:p-bounded-central-moment-gradient}), we inject synthetic noise into the exact gradients and Hessian-vector products. Specifically, the noise vectors are drawn from a symmetric Pareto-like distribution generated via $s \cdot (u^{-1/p} - 1)$, where $u \sim \mathcal{U}(0,1)$ and $s \sim \{-1, 1\}$ uniformly. This ensures that the noise has an infinite variance when $p < 2$, strictly adhering to our theoretical noise model.
	
	\subsection{The Necessity of Double-Clipping and Algorithmic Stability}
	
	\begin{figure}[htbp]
		\centering
		
		\includegraphics[width=0.6\linewidth]{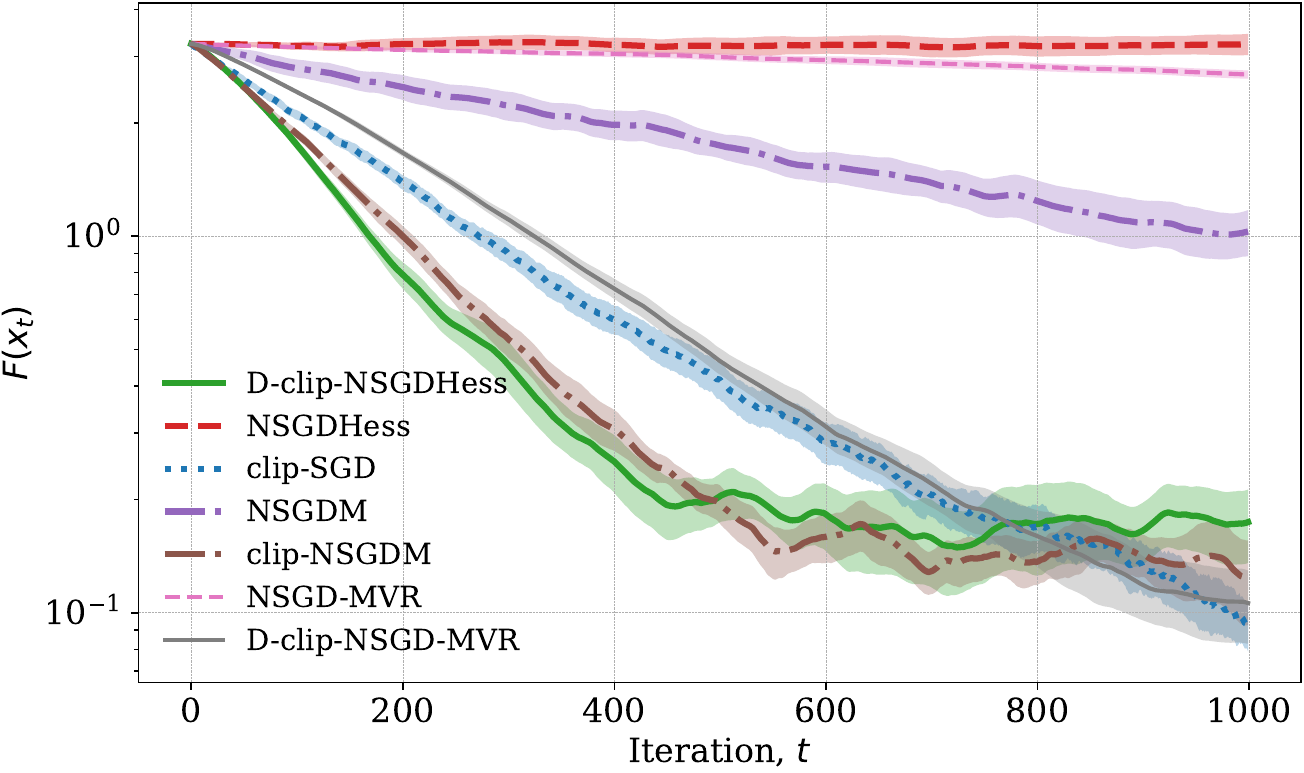} \\[1em]
		
		\includegraphics[width=\linewidth]{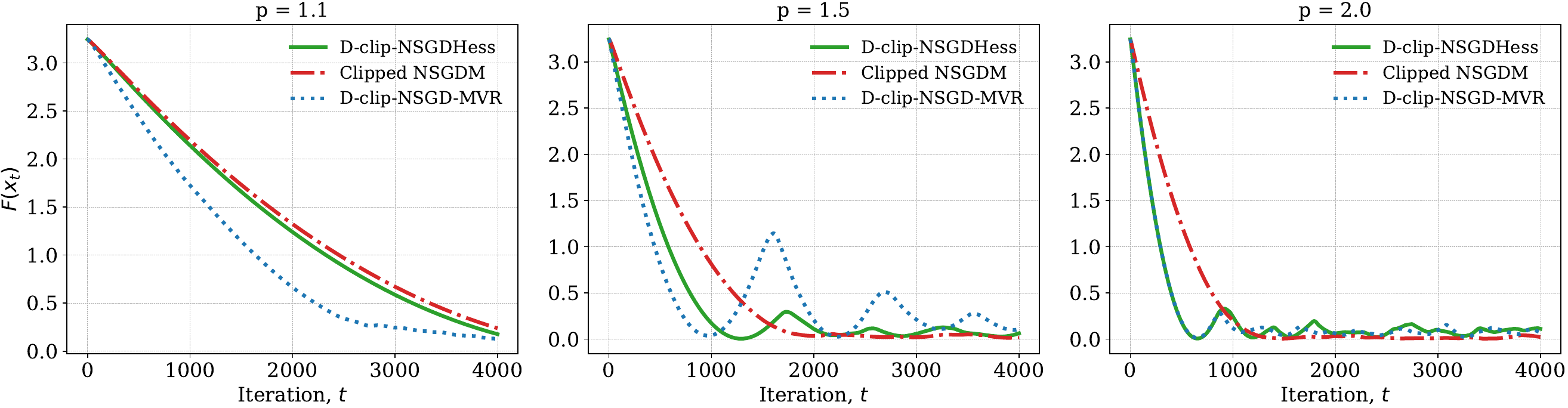}
		
		\caption{\textbf{Top:} Convergence trajectories of $F(x_t)$ over $T=1000$ iterations for $p=1.1$. Solid/dashed lines represent the mean over 20 independent runs, and shaded regions denote the standard deviation. \textbf{Bottom:} Algorithm performance across different tail indices $p \in \{1.1, 1.5, 2.0\}$. Hyperparameters are scaled strictly according to their respective theoretical optimal rates.}
		\label{fig:convergence_trajectories}
	\end{figure}

	In our first experiment (Figure~\ref{fig:convergence_trajectories}, Left), we fix the tail index to $p=1.1$ (an extreme heavy-tailed regime) and compare our proposed double-clipped methods (\textsf{D-\clip-NSGD-MVR} and \textsf{D-\clip-NSGDHess}) against standard clipped SGD \cite{10.5555/3495724.3497014}, clipped NSGD with momentum (\textsf{clip-NSGDM} \citep{NEURIPS2021_26901deb,pmlr-v195-liu23c}), and their unclipped counterparts. We run $T=1000$ iterations and average the results over 20 independent trials. 
	
	The results clearly demonstrate the necessity of clipping in this regime: unclipped methods (\algname{NSGDM} \citep{pmlr-v119-cutkosky20b,hübler2025gradientclippingnormalizationheavy} and \algname{NSGD-MVR}) fail to converge and exhibit severe instability, eventually diverging. In contrast, all clipped variants successfully decrease the objective. Notably, our proposed \algname{D-\clip-NSGD-MVR} and \algname{D-\clip-NSGDHess} exhibit the fastest and most stable convergence, validating the theoretical benefit of combining variance reduction or Hessian-correction with our proposed double-clipping mechanism.
	
	Next, we evaluate the performance of these algorithms across a spectrum of heavy-tailed noise distributions by varying the tail index $p \in \{1.1, 1.5, 2.0\}$ over $T=4000$ iterations (Figure~\ref{fig:convergence_trajectories}, Right). Crucially, to test our theoretical bounds, the stepsize $\gamma$, momentum parameter $\alpha$, and clipping thresholds for each method are scaled \textit{strictly} according to the theoretically optimal $T$-dependent exponents derived in our main theorems. The results confirm our theoretical characterizations: in the most challenging regime ($p=1.1$), \algname{D-\clip-NSGD-MVR} significantly outperforms the standard \algname{\clip-NSGDM} \citep{NEURIPS2021_26901deb,pmlr-v195-liu23c} baseline. As $p \to 2.0$ (approaching the standard bounded variance setting), the gap between the methods narrows, and all algorithms exhibit rapid, stable convergence.
	
	\subsection{Theoretical vs. Empirical Complexity Scaling}

	\begin{figure}[htbp]
		\centering
		\begin{minipage}{0.48\textwidth}
			\centering
			\includegraphics[width=\linewidth]{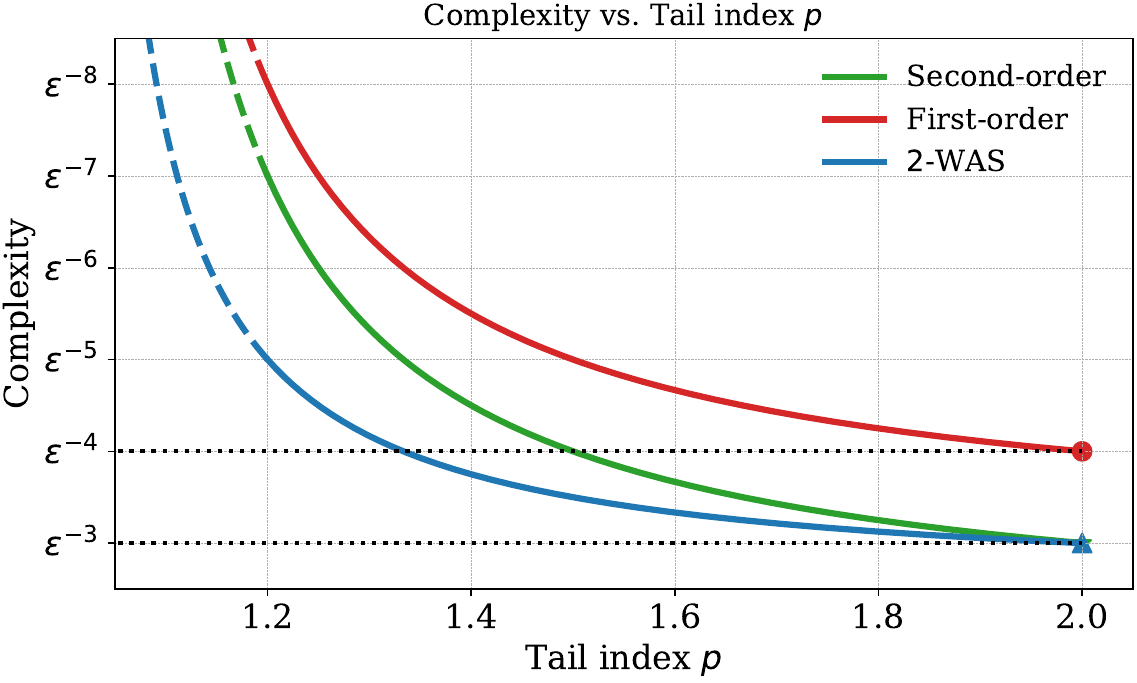} 
		\end{minipage}\hfill
		\begin{minipage}{0.48\textwidth}
			\centering
			\includegraphics[width=\linewidth]{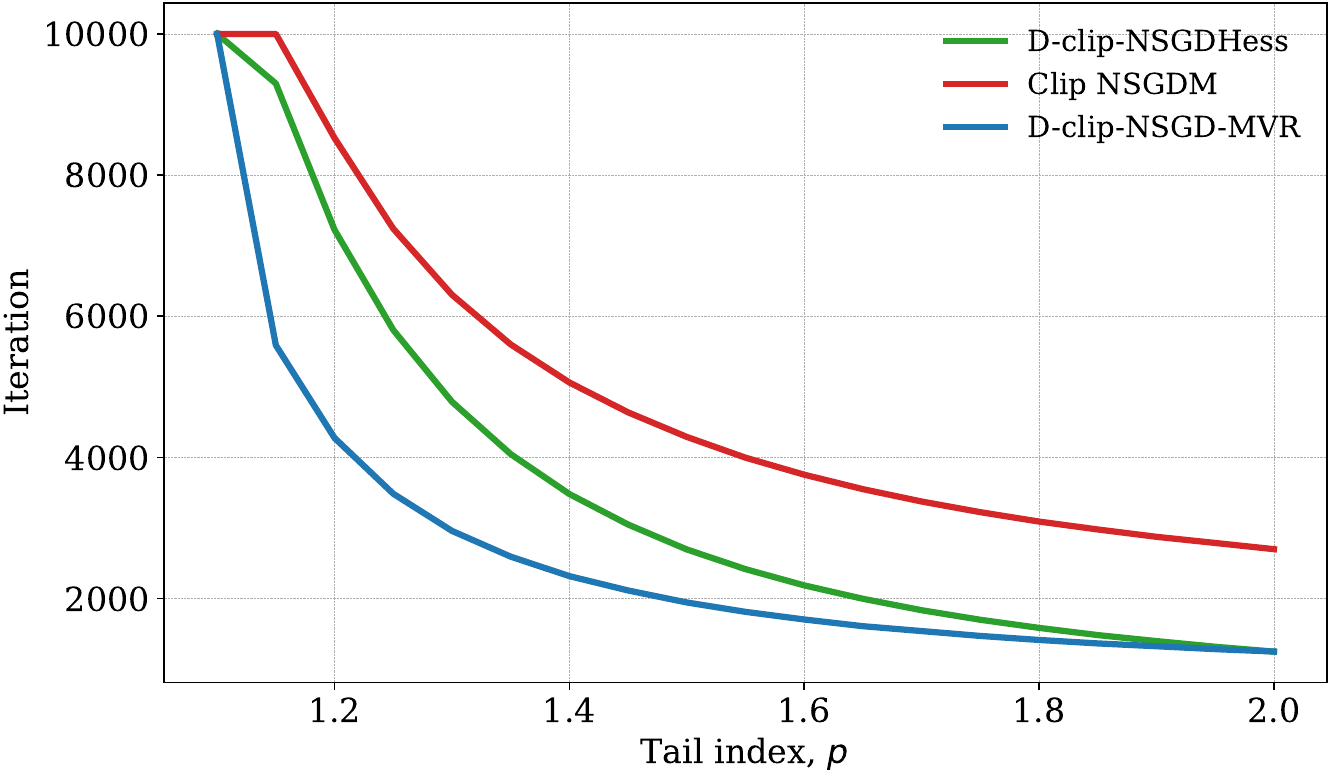} 
		\end{minipage}
		\caption{\textbf{Left:} Theoretical complexity $\mathcal{O}(\varepsilon^{-c})$ vs. tail index $p$. The y-axis represents the exponent $c$ of $\varepsilon$. \textbf{Right:} Empirical iterations required to reach a target suboptimality $F(x_t) < 1.2$ across different tail indices $p$. The empirical performance perfectly mirrors the theoretical scaling, demonstrating that \textsf{D-clip-NSGD-MVR} ($2$-WAS) provides the most robust acceleration in severe heavy-tailed regimes ($p<2$).}
		\label{fig:complexity_vs_p}
	\end{figure}
	
	To further illustrate the robustness of our analysis, we investigate how the heavy-tail index $p$ explicitly impacts sample complexity. We compare three distinct regimes: standard first-order optimization (\algname{\clip NSGDM}, \citep{NEURIPS2021_26901deb,pmlr-v195-liu23c}), second-order optimization (\textsf{D-\clip-NSGDHess}), and first-order optimization with variance reduction under the 2-WAS assumption (\textsf{D-\clip-NSGD-MVR}).
	
	\textbf{Theoretical Scaling:} In the left panel of Figure~\ref{fig:complexity_vs_p}, we plot the theoretical complexity exponents $c$ for the bound $\mathcal{O}(\varepsilon^{-c})$ as a function of the tail index $p \in (1, 2]$. As expected, in the bounded variance regime ($p=2$), the standard first-order method requires $\mathcal{O}(\varepsilon^{-4})$, while both the second-order method and the 2-WAS variance-reduced method achieve the optimal $\mathcal{O}(\varepsilon^{-3})$ complexity. However, as the noise becomes increasingly heavy-tailed ($p \to 1$), the required complexity approaches infinity across all methods. Notably, the theoretical curves reveal that under the 2-WAS assumption, our variance-reduced method (\textsf{D-clip-NSGD-MVR}) achieves a strictly better polynomial dependence on $\varepsilon$ than even the second-order method for all $p < 2$.
	
	\textbf{Empirical Validation:} To validate these theoretical scaling laws empirically, we track the exact number of iterations required for each algorithm to reach a fixed suboptimality threshold ($F(x_t) < 1.2$) as we vary $p$ from $1.1$ to $2.0$ (Figure~\ref{fig:complexity_vs_p}, Right). The empirical results perfectly mirror the theoretical predictions. \algname{\clip-NSGDM} consistently requires the most iterations to converge. Furthermore, the empirical iteration count confirms the exact crossover phenomenon predicted by our lower bounds: \textsf{D-\clip-NSGD-MVR} requires fewer iterations than the second-order \textsf{D-\clip-NSGDHess} in the severe heavy-tailed regime ($p < 2$). This compellingly demonstrates that variance reduction under the $2$-WAS assumption is a highly effective mechanism for accelerating convergence under extreme noise, completely aligning with our established theory.

\end{document}

%% file: macros.tex
\usepackage{graphicx}
\usepackage{apptools}
\usepackage[flushleft]{threeparttable}
\usepackage{array,booktabs,adjustbox}
\usepackage{makecell}
\usepackage{multirow}
\usepackage{nicefrac}
\usepackage{wrapfig}
\usepackage{caption}
\usepackage{siunitx}
\usepackage{nccmath}
\usepackage{empheq}
\usepackage{bbm}
\usepackage{suffix}
\usepackage{tabularx}
\usepackage{dirtytalk}

\usepackage{amsmath,amsfonts,bm,amssymb,mathtools}
\usepackage{amsthm}
\usepackage{thmtools}
\usepackage{thm-restate}

\usepackage{derivative}

\usepackage{tabularray}

\usepackage{circledsteps}

\usepackage{algorithmic}
\usepackage[noabbrev]{cleveref} 

\usepackage[linesnumbered, ruled, vlined, onelanguage]{algorithm2e}%
\SetCommentSty{commentalgo}%
\SetKwProg{Init}{init}{}{}%
\SetKwFor{For}{For}{do}{EndFor}%
\SetKwFor{While}{While}{do}{EndWhile}%
\SetKw{Return}{Return}%

\crefname{algocf}{alg.}{algs.}
\Crefname{algocf}{Algorithm}{Algorithms}

\usepackage{color}
\usepackage{colortbl}
\definecolor{bgcolor}{rgb}{0.76,0.88,0.50}
\definecolor{bgcolor0}{rgb}{0.93,0.99,1}
\definecolor{bgcolor1}{rgb}{0.8,1,1}
\definecolor{bgcolor2}{rgb}{0.8,1,0.8}
\definecolor{bgcolor3}{rgb}{0.50,0.90,0.50}
\usepackage{tcolorbox}
\usepackage{pifont}

\definecolor{mydarkgreen}{RGB}{39,130,67}
\definecolor{mydarkorange}{RGB}{236,147,14}
\definecolor{mydarkred}{RGB}{192,47,25}
\definecolor{ruby}{RGB}{155,17,30}
\definecolor{chili}{RGB}{191,0,0}
\definecolor{sangria}{RGB}{146,0,10}
\definecolor{burgundy}{RGB}{128,0,32} 
\definecolor{darkred}{RGB}{132,0,0} 
\definecolor{cherry}{RGB}{192,0,0} 

\definecolor{blue}{RGB}{0,0,255}

\definecolor{linen}{HTML}{FAF0E6} 

\definecolor{mydarkblue}{RGB}{20,20,192}

\definecolor{niceblue}{rgb}{0.0,0.19,0.56}

\usepackage{pifont}
\definecolor{PineGreen}{RGB}{0,110,51}
\definecolor{BrickRed}{RGB}{143,20,2}
\newcommand{\cmark}{{\color{PineGreen}\ding{51}}}%
\newcommand{\xmark}{{\color{BrickRed}\ding{55}}}%

\usepackage[textsize=tiny]{todonotes}

\usepackage{xspace}

\usepackage[scaled=0.86]{helvet}
\newcommand{\algname}[1]{{\sf #1}}

\newcommand{\norm}[1]{\left\| #1 \right\|}
\newcommand{\TVdist}[1]{\left\| #1 \right\|_{\textnormal{TV}}}
\newcommand{\sqnorm}[1]{\left\| #1 \right\|^2}
\newcommand{\abs}[1]{\left| #1 \right|}

\newcommand{\R}{\mathbb{R}} 
\newcommand{\E}[1]{\mathbb{E}\left[#1\right]}

\newcommand{\Exp}[1]{{\mathbb{E}}\left[#1\right]}
\newcommand{\ExpSub}[2]{{\mathbb{E}}_{#1}\left[#2\right]}
\newcommand{\ExpCond}[2]{{\mathbb{E}}\left[\left.#1\,\right\vert\,#2\right]}
\newcommand{\ExpSubCond}[3]{{\mathbb{E}}_{#1}\left[\left.#2\,\right\vert\,#3\right]}

\newcommand{\Proba}[1]{\mathbb{P}\left(#1\right)} 
\newcommand{\ProbCond}[2]{\mathbb{P}\left(#1\middle\vert#2\right)}

\newcommand{\cD}{\mathcal{D}}

\newcommand{\cO}{\mathcal{O}}

\newcommand{\addeq}{\addtocounter{equation}{1}}%
\newcommand\numberthis{\addeq\tag{\theequation}}

\makeatletter
\DeclareRobustCommand\widecheck[1]{{\mathpalette\@widecheck{#1}}}
\def\@widecheck#1#2{%
    \setbox\z@\hbox{\m@th$#1#2$}%
    \setbox\tw@\hbox{\m@th$#1%
       \widehat{%
          \vrule\@width\z@\@height\ht\z@
          \vrule\@height\z@\@width\wd\z@}$}%
    \dp\tw@-\ht\z@
    \@tempdima\ht\z@ \advance\@tempdima2\ht\tw@ \divide\@tempdima\thr@@
    \setbox\tw@\hbox{%
       \raise\@tempdima\hbox{\scalebox{1}[-1]{\lower\@tempdima\box
\tw@}}}%
    {\ooalign{\box\tw@ \cr \box\z@}}}
\makeatother

\newcommand{\oversetref}[2]{\overset{{\scriptscriptstyle\mathrm{#1}\!~\text{#2}}}}
\newcommand{\oversetlab}[1]{\overset{{\scriptscriptstyle\text{#1}}}}

\newcounter{relctr} 
\newcounter{relgroup} 

\everydisplay\expandafter{\the\everydisplay\setcounter{relctr}{0}}

\AtBeginDocument{\let\originallabel\label} 

\makeatletter
\newcommand\oversetrel[1]{%
    \ifnum\value{relctr}=0\relax%
      \stepcounter{relgroup}%
    \fi%
    \refstepcounter{relctr}%
    \phantomsection%
    \originallabel{rel:#1@\therelgroup}%
    \overset{\scriptscriptstyle\text{(\alph{relctr})}}%
}

\newcommand\relref[1]{%
    \@ifundefined{r@rel:#1@\arabic{relgroup}}{%
        \eqref{}%
    }{%
        \eqref{rel:#1@\arabic{relgroup}}
    }%
}
\makeatother

\theoremstyle{plain}
\newtheorem{theorem}{Theorem}[section]

\newtheorem{lemma}{Lemma}[section]
\newtheorem{corollary}{Corollary}[section]

\newtheorem{assumption}{Assumption}[section]

\theoremstyle{definition}
\newtheorem{definition}{Definition}[section]

\theoremstyle{remark}
\newtheorem{remark}{Remark}[section]

\crefname{assumption}{assumption}{assumptions}
\Crefname{assumption}{Assumption}{Assumptions}
\creflabelformat{assumption}{#2#1#3}

\crefname{condition}{condition}{conditions}
\creflabelformat{condition}{#2#1#3}

\crefname{observation}{observation}{observations}
\creflabelformat{observation}{#2#1#3}

\newcommand{\eqdef}{:=}
\WithSuffix\newcommand\eqdef*{:\!&=}
\newcommand{\reqdef}{=:}

\makeatletter
\newcommand{\vast}{\bBigg@{4}}

\def\<{\left\langle}
\def\>{\right\rangle}
\def\({\left(}
\def\){\right)}

\theoremstyle{plain}
\newtheorem{innercustomthm}{Theorem}
\newenvironment{restate-theorem}[1]
{\renewcommand\theinnercustomthm{#1}\innercustomthm}
{\endinnercustomthm}

\newtheorem{innercustomlemma}{Lemma}
\newenvironment{restate-lemma}[1]
{\renewcommand\theinnercustomlemma{#1}\innercustomlemma}
{\endinnercustomlemma}

\newtheorem{innercustomproposition}{Proposition}
\newenvironment{restate-proposition}[1]
{\renewcommand\theinnercustomproposition{#1}\innercustomproposition}
{\endinnercustomproposition}

\newcommand*{\sketchproofname}{Sketch of Proof}

\usepackage{longtable}

\newcommand{\prog}{\mathrm{prog}\,}

\newcommand{\Int}[2]{\left\{ #1, \ldots, #2 \right\}}
\newcommand{\supp}{\mathrm{supp}\,}

\newcommand*{\lbd}{\lambda}
\newcommand*{\eps}{\varepsilon}

\newcommand*{\Floor}[1]{\left\lfloor #1 \right\rfloor}

\newcommand*{\T}[1]{{#1}^{\top}}

\newcommand*{\intervalle}[4]{\mathopen{#1}#2\mathclose{},#3\mathclose{#4}}%
\newcommand*{\intff}[2]{\intervalle{[}{#1}{#2}{]}}
\newcommand*{\intof}[2]{\intervalle{(}{#1}{#2}{]}}
\newcommand*{\intfo}[2]{\intervalle{[}{#1}{#2}{)}}

\newcommand*{\ens}[1]{\left\{#1\right\}}%
\newcommand*{\enstq}[2]{\left\{#1\,:\,#2\right\}}%

\renewcommand*{\P}{\mathbb{P}}
\renewcommand*{\R}{\mathbb{R}}

\newcommand{\normop}[1]{\left\| #1 \right\|_{\mathrm{op}}}

\newcommand{\ps}[2]{\left\langle #1, #2 \right\rangle}

\renewcommand*{\abs}[1]{\left\lvert #1 \right\rvert}

\renewcommand*{\lim}{\mathop{\mathrm{lim}}\limits}

\newcommand*{\argmin}{\mathop{\mathrm{arg\,min}}}

\newcommand*{\probac}[2]{\mathop{\mathbb{P}}\left( #1 \, \vert \, #2 \right)}

\newcommand*{\Ber}{\mathrm{Ber}}

\newcommand{\clip}{{\normalfont\texttt{clip}}}
\newcommand{\sign}{\mathop{\mathrm{sign}}}